\documentclass[a4paper,12pt,reqno,oneside]{amsart}

\usepackage{setspace} 
\usepackage{typearea} 

\usepackage{amscd,amsmath,amssymb,verbatim}

\usepackage{ifthen}
\usepackage{xr-hyper}
\usepackage{paralist,cite}
\usepackage{enumitem}
\usepackage[colorlinks,backref,final,bookmarksnumbered,bookmarks]{hyperref}
\usepackage[backrefs]{amsrefs}
\usepackage{tikz-cd}

\newcommand{\arxiv}[2][]{\ifthenelse{\equal{#1}{}}
{\href{http://arxiv.org/abs/#2}{\tt arXiv:#2}}
{\href{http://arxiv.org/abs/math/#2}{\tt arXiv:math.#1/#2}}}

\theoremstyle{plain}
\newtheorem{theorem}{Theorem}[section]
\newtheorem{lemma}[theorem]{Lemma}
\newtheorem{corollary}[theorem]{Corollary}
\newtheorem{proposition}[theorem]{Proposition}

\theoremstyle{definition}
\newtheorem{example}[theorem]{Example}

\newtheoremstyle{remark}
{}{}{}{}{\itshape}{}{ }{\thmname{#1}\thmnumber{ \itshape #2.}}
\theoremstyle{remark}
\newtheorem{remark}[theorem]{Remark}

\newtheoremstyle{concise}
{}{}{}{}{\bfseries}{}{ }{\thmnumber{#2.}\thmnote{ #3.}}
\theoremstyle{concise}

\newenvironment{parts}
{

\begin{asparaenum}}
{\end{asparaenum}}

\newenvironment{embedded roster}[1][0]
{

\begin{enumerate}\setcounter{enumii}{#1}}
{\end{enumerate}}
\makeatletter
\renewcommand\p@enumii{\p@enumi}
\renewcommand{\@listii}{\leftmargin=40pt}
\makeatother

\newenvironment{roster}[1][0]
{

\begin{enumerate}\setcounter{enumi}{#1}}
{\end{enumerate}}

\newenvironment{properties}[1][0]
{

\begin{enumerate}\setcounter{enumi}{#1}}
{\end{enumerate}}

\def\incl{\subset}
\def\B{\mathcal{B}}
\def\N{\mathbb{N}} 
\def\R{\mathbb{R}} 
\def\Z{\mathbb{Z}}
\def\Q{\mathbb{Q}}
\def\x{\times}
\def\but{\setminus} 
\def\eps{\varepsilon} 
\def\phi{\varphi}
\def\emb{\hookrightarrow} \def\topcont{\hookleftarrow}
\def\invlim{\varprojlim} 
\def\xr#1{\xrightarrow{#1}} \def\xl#1{\xleftarrow{#1}} \renewcommand{\:}{\colon}
\def\imp{$\Rightarrow$}  
\DeclareMathOperator{\Int}{Int} 
\DeclareMathOperator{\id}{id}
\def\U{\mathcal{U}} \def\T{\mathcal{T}} \def\C{\mathcal{C}}
\DeclareMathOperator{\Cl}{Cl} 
\DeclareMathOperator{\st}{st} 
\DeclareMathOperator{\lcm}{lcm}
\DeclareMathOperator{\Lip}{Lip}

\def\nec#1{\raisebox{-1.5pt}{$\not\overset{\scriptstyle#1}\to$}}
\def\nee#1{\raisebox{-1.5pt}{$\not\overset{\scriptstyle#1}=$}}
\def\Nee#1{\raisebox{-1.5pt}{$\not\overset{\scriptstyle#1}{=\joinrel=}$}}
\def\neC{{\not\overset{ }\to}}
\def\neE{{\not\overset{ }=}}

\def\bydef{\mathrel{\mathop:}=}

\begin{document}
\title{Metrizable uniform spaces}
\author{Sergey A. Melikhov}
\address{Steklov Mathematical Institute of the Russian Academy of Sciences,
ul.\ Gubkina 8, Moscow, 119991 Russia}
\email{melikhov@mi-ras.ru}

\begin{abstract}
Three themes of general topology: quotient spaces; absolute retracts; and
inverse limits --- are reapproached here in the setting of metrizable uniform
spaces, with an eye to applications in geometric and algebraic topology.
The results include:

1) If $X\supset A\xr{f}Y$ is a uniformly continuous partial map of metric spaces,
where $A$ is closed in $X$, we show that the adjunction space $X\cup_f Y$ with
the quotient uniformity (hence also with the topology thereof) is metrizable,
by an explicit metric.
This yields natural constructions of cone, join and mapping cylinder in
the category of metrizable uniform spaces, which we show to coincide with those
based on subspace (of a normed linear space); on product (with a cone); and
on the isotropy of the $l_2$ metric.

2) We revisit Isbell's theory of uniform ANRs, as refined by Garg and Nhu in
the metrizable case.
The iterated loop spaces $\Omega^n P$ of a pointed compact polyhedron $P$ are
shown to be uniform ANRs.
Four characterizations of uniform ANRs among metrizable uniform spaces $X$ are given:
(i) the completion of $X$ is a uniform ANR, and the remainder is uniformly a
Z-set in the completion; (ii) $X$ is uniformly locally contractible and satisfies
the Hahn approximation property; (iii) $X$ is uniformly $\eps$-homotopy dominated
by a uniform ANR for each $\eps>0$; (iv) $X$ is an inverse limit of uniform ANRs
with ``nearly splitting'' bonding maps.

Several chapters are devoted primarily to exposition: (I) an introduction to uniform spaces, 
with a focus on the metrizable case; (V) the space of measurable functions; (VI) the space of
probability measures and other measure spaces.

\end{abstract}

\maketitle
\section{INTRODUCTION}

Although topological and uniform approaches to foundations of what was then known as Analysis Situs originated 
in the same works by M. Fr\'echet and F. Riesz (cf.\ Remark \ref{Frechet-Riesz} below), uniform spaces 
hopelessly lagged behind in development ever since, and were never taken seriously in geometric and algebraic 
topology, due in part to the lack of coherent theories of quotient spaces, ANRs, and inverse limits in 
the uniform setup.

It must be noted, however, that much of the development of geometric and algebraic topology to date has been 
confined either to ``combinatorial'' settings (such as CW complexes and simplicial sets) or to compact spaces.
Of course, in compact spaces there is no difference between topology and uniform theory.
In ``combinatorial'' settings the difference does occur, and, indeed, the topological approach is often 
much easier to deal with than the uniform one.  

Yet whenever one tries to go beyond both compact spaces and ``combinatorial'' settings, one inevitably 
encounters some painful side effects of the usual topological foundations, such as the following.

\begin{itemize}
\item The cone over $\R$ fails to be metrizable (see Example \ref{hocolim-discussion} below).
As a consequence, the class of ANRs is not closed under finite homotopy colimits.

\item The space $C(X,Y)$ of maps $X\to Y$ with the compact-open topology quickly gets wild.
When $X$ is a locally finite graph with infinitely many cycles, $C(X,S^1)$ is not locally path-connected
\cite{Sak2}*{6.2.10(8)}.
Also, when $X$ is separable metrizable but not locally compact, $C(X,\R)$ is not metrizable \cite{Fox}*{p.\ 432}.
As a consequence, the class of ANRs is not closed under finite homotopy limits.

\item Covering theory does not work over a non-locally-connected base.
(This is remedied by overlays, which are closely related to uniform structures; see \cite{M}*{7.6}.)

\item If $X$ is a non-compact metrizable space containing no isolated points, then the directed set of
open covers of $X$, ordered by refinement, contains no cofinal countable set.
This makes it much more difficult to use combinatorial approximation to study non-compact metrizable
spaces (as opposed to compact ones).%
\footnote{For example, it is awkward enough that the shape invariant homology of a space as 
simple as $\N^+\x\N$ (where $\N$ denotes the infinite countable discrete space, and $+$ denotes 
the one-point compactification) cannot be computed in ZFC, as its value depends on additional axioms 
of set theory \cite{MP}.}
\end{itemize}

\subsection{Quotients (Chapter II)}
In the present paper, the primary aim of Chapter \ref{quotients} is to show that (the topology of) 
the quotient uniformity is, after all, far nicer than quotient topology in the context of metrizable spaces.
In particular, we show that finite homotopy colimits of metrizable uniform spaces and uniformly continuous maps 
are metrizable when done uniformly; they are certainly not metrizable with the quotient topology, as we will 
now discuss.

\begin{example}\label{hocolim-discussion}
Let us consider the cone over the real line $\R$, with the topology of the quotient space $\R\x[0,1]/\R\x\{0\}$.
This topology is non-metrizable, because it does not even have a countable base of neighborhoods at 
the cone vertex.%
\footnote{Indeed, it suffices to show that $\N\x\frac1{\N^+}/\N\x\{0\}$ has no countable base of neighborhoods 
at the point $p=\N\x\{0\}$, where $\N=\{1,2,3,\dots\}\subset\R$ and 
$\frac1{\N^+}=\{1,\frac12,\frac13,\dots\}\cup\{0\}$.
A neighborhood of $p$ in the quotient corresponds to a neighborhood of $\N\x\{0\}$ in $\N\x\frac1{\N^+}$.
Suppose that $\{N_1,N_2,\dots\}$ is a countable base of neighborhoods of $\N\x\{0\}$ in $\N\x\frac1{\N^+}$.
For each $k\in\N$ let $n_k$ be the minimal number such that $(k,\frac1{n_k})\in N_k$.
Then $N\bydef \{(k,\frac 1n)\mid k\in\N,\,n>n_k\}\cup\N\x\{0\}$ is a neighborhood of $\N\x\{0\}$ in 
$\N\x\frac1{\N^+}$.
However, each $(k,\frac1{n_k})\notin N$, so $N_k\not\subset N$.
Thus $\{N_1,N_2,\dots\}$ is not a base of neighborhoods of $\N\x\{0\}$ in $\N\x\frac1{\N^+}$, which is 
a contradiction.}

Thus the quotient topology in the cone over $\R$ is not the same as the topology of the subspace 
$\R\x [0,1)\cup\{(0,1)\}$ of $\R^2$ or the topology of the subcone over $(-1,1)$ in the cone over $[-1,1]$.
\end{example}

\begin{example} Let us consider a space $Y$, a subspace $X\subset Y$ and a non-closed subset $A\subset X$,
and let $i\:X\to Y$ be the inclusion map.
Then the mapping cylinder $MC(f)$ with the quotient topology does not contain $MC(f|_A)$ as a subspace, because 
not every neighborhood of $A\x\{1\}$ in $A\x[0,1]$ extends to a neighborhood of $X\x\{1\}$ in $X\x [0,1]$ 
(see details in \cite{Sak2}*{\S4.11}).
\end{example}

\begin{remark}
Which topology on the homotopy colimits is the `right one'?
It turns out that in many situations where some actual work is being done,
the quotient topology does not do its job, and has to be replaced by something
else:

(i) In his classifying space construction $BG=(G*G*\dots)/G$, Milnor had to use 
a strong (initial) topology on his joins (including finite joins) rather 
than the weak (final) topology of the quotient \cite{Mil}.

(ii) In showing that the usual homotopy category is a closed model category in
the sense of Quillen (``The homotopy category is a homotopy category'' \cite{Str}),
A. Str\o m had to modify the quotient topology of the mapping cylinder in
order to show that if $f\:E\to B$ is a (Hurewicz) fibration, then so is
the projection $MC(f)\to B$.
His modified mapping cylinder can be identified with a subspace of Milnor's
modified join.

There are many other examples of this sort in the literature (including e.g.\ teardrop neighborhoods).
\end{remark}

It turns out that all of the aforementioned issues disappear entirely in the setting of uniform spaces,
which is the subject of \S\ref{metrizability} below.
If $X\supset A\xr{f}Y$ is a uniformly continuous partial map of metric spaces,
where $A$ is closed in $X$, we show that the adjunction space $X\cup_f Y$ with
the quotient uniformity (hence also with the topology thereof) is metrizable,
by an explicit metric (Theorem \ref{adjunction}).
This yields natural constructions of cone, join and mapping cylinder in
the category of metrizable uniform spaces, which we show to coincide with
various other natural constructions (see \S\ref{join, etc}).
In particular, we show equivalence, up to uniform homeomorphism, of a number of
definitions of join of metric spaces: one based on quotient uniformity; another
based on embedding in a normed linear space; and those based on the amalgamated
union $X\times CY\cup CX\times Y$, where the cones are defined using either any
of the previous methods or the approach of geometric group theory, based on
the isotropy of the $l_2$ metric (i.e., the Law of Cosines).

\subsection{Absolute neighborhood retracts (Chapter III)}\label{intro-ANRs}

In working with polyhedra it is convenient to separate combinatorial issues (such as
simplicial approximation and pseudo-radial projection) from topological ones, which
are well captured by the notion of an ANR.
For instance, finite homotopy limits (=homotopy inverse limits) of PL maps between
compact polyhedra are still ANRs, but no longer polyhedra in general.

The primary aim of Chapter \ref{absolute retracts} is to prepare for the treatment of
uniform polyhedra in the sequel to this paper \cite{M3} by advancing a theory of uniform ANRs
roughly to the level of the classical theory of ANRs as presented in the books
by Borsuk \cite{Bo} and Hu \cite{Hu1}.
In particular, we show that finite homotopy limits and colimits of uniform ANRs
are still uniform ANRs (Theorem \ref{holim-ANR}) --- although we will see in
the sequel that those of uniform polyhedra and ``uniformly PL'' maps are no
longer uniform polyhedra in general.
A rather naive and clumsy special case of uniform polyhedra, called ``cubohedra'',
is introduced in the present paper.

Our notion of a uniform ANR is not entirely standard.
Two best-known analogues of ANRs in the uniform world are the semi-uniform ANRs
studied by Michael and Torunczyk (see Remarks \ref{Michael&Nhu}(b) and
\ref{semi-uniform}) and the ANRUs of Isbell, which we revisit in \S\ref{ARs}.
While semi-uniform ANRs are more manageable in some respects, they are at best
a useful but technical tool, involving a mix of topological and uniform notions.
On the other hand, as long as metrizable spaces are concerned, it was realized
independently by Garg and Nhu that Isbell's ANRUs are only a part of the story
--- namely the complete part.
This understanding is, however, scarcely known, and is not well established in
the literature: Garg mentioned what we now call uniform ANRs only in passing
(so did not even give them any name); whereas Nhu's metric uniform ANRs
(see Remark \ref{Michael&Nhu}(a)), although do coincide with our uniform ANRs,
but somewhat accidentally --- for his metric uniform ARs differ from our
uniform ARs.
Above all, there seems to have been no good intuition and no readily available
technique for dealing with non-complete uniform ANRs, as compared with complete
ones (i.e.\ ANRUs).

This is now entirely changed, for we show that a metrizable uniform space is
a uniform ANR if and only if its completion is an ANRU, and the remainder
can be instantaneously taken off itself by a uniform self-homotopy of the
completion (Theorem \ref{uniform ANR}).
Moreover, in many ways uniform ANRs turn out to be easier, and not harder
than ANRUs.
In particular, we show (Theorem \ref{LCU+Hahn}) that a metrizable uniform space
is a uniform ANR if and only if it is uniformly locally contractible
(in the sense of Isbell) and satisfies the Hahn approximation property
(in the sense of Isbell).
This result, proved by an infinite process, improves on the metrizable case
of Isbell's characterization of ANRUs as those uniform spaces that are
uniformly locally contractible and satisfy the Hahn property {\it and}
the uniform homotopy extension property.
(Isbell's uniform homotopy extension property is for possibly non-closed subsets;
it follows from completeness along with the uniform homotopy extension property
for closed subsets.)

The above two characterizations of uniform ANRs are at the heart of a toolkit that
enables one to deal with uniform ANRs just as easily as with compact ANRs.
Indeed we establish uniform analogs of what appears to be the core results of
the usual theory of retracts, including Hanner's $\eps$-homotopy domination criterion
(Theorem \ref{Hanner}(a)) and J. H. C. Whitehead's theorem on adjunction spaces
(Theorem \ref{Whitehead}).
We also show (see Theorem \ref{A.3''} or Corollary \ref{A.3'}) that the space of
uniformly continuous maps from a metrizable uniform space (not necessarily compact or
locally compact!) to a uniform ANR is a uniform ANR, and extend this to maps of pairs, 
which is the nontrivial part.
In particular, this shows that the loop space, as well as iterated loop spaces
of a compact polyhedron are uniform ANRs (Corollary \ref{loops}(b)).

\subsection{Inverse limits (Chapter IV)}

One of the main advantages of metrizable uniform spaces over (non-compact)
metrizable topological spaces is that the former can be manageably described
via (countable) inverse sequences of uniform ANRs, in fact, of cubohedra.
This is the subject of Chapter \ref{inverse limits}.

The role of the (generally uncountable) resolutions of Marde\v si\'c is played by
convergent inverse sequences (see Lemma \ref{Mardesic}), thereby reducing much of
the hassle to the simple condition of convergence.
(It can be viewed as a weakening of the surjectivity of all bonding maps,
and specializes to the Mittag-Leffler condition in the case of inverse sequences
of discrete spaces.)

This enables us to generalize to metrizable uniform spaces virtually all known
theory of inverse sequences of compacta.
In particular, we show that separable metrizable complete uniform spaces
can be represented as limits of inverse sequences of uniform ANRs (Theorem
\ref{intersection of cubohedra2}).
While the finite-dimensional case is due to Isbell, in general he only knew of
such a representation by uncountable inverse spectra, not sequences.
We also establish the analogue of Milnor's lemma on extension of
a map between inverse limits to the infinite mapping telescopes
(Theorem \ref{Milnor} and Corollary \ref{Milnor2}), which amounts to a foundation
of strong shape theory (see \cite{M}).
Another noteworthy result, whose compact case the author has not seen in
the literature, is a characterization of inverse limits that are uniform ANRs
in terms of properties of the bonding maps (Theorem \ref{ANR-limit}).

\subsection{Topology of measure (Chapters V and VI)}

Uniform ANRs are much more precious objects than usual ANRs, and in particular it is still
an open problem to describe a covariant functor $F$ from the category of separable metrizable 
uniform spaces to itself such that $F(X)$ is a uniform AR containing $X$.
(One might also wish $F$ to satisfy a number of other useful properties.)

One candidate for such an $F(X)$ is the space of measurable functions $[0,1]\to X$, which is known 
to be a non-uniform AR.
We include a rather detailed treatment of this space in Chapter \ref{measurable functions},
which seems to be the first such treatment in the literature.
However, the problem whether $F(X)$ is a uniform AR remains open at the moment.

Another natural candidate would be the space of probability measures $PM(X)$, which is also
a non-uniform AR, but we observe that it is generally not a uniform AR (Proposition \ref{not-uar}).
Instead, $PM(X)$ satisfies another universal property: every uniformly continuous map $f$ from $X$ 
to a Banach space $V$ extends to a linear (hence, uniformly continuous) map from $PM(X)$ into 
the convex hull of $f(X)$.
The problem of construction of continuous, let alone uniformly continuous maps $X\to Y$, where
$X$ is, for example, a polyhedron with the metric topology (rather than the CW topology), is 
quite difficult in general, and it is hoped that $PM(X)$ can be handy for this.
Chapter \ref{prob-measures} contains a rather detailed treatment of $PM(X)$ and its variations.

\subsection{Textbooks on uniform spaces}
Unsurpassed basic references for uniform spaces are still those to the founders of the subject: 
Isbell's book \cite{I3}, which is well complemented by Chapters II, IX and X of Bourbaki's 
{\it General Topology} \cite{Bou} (let us note that the original Bourbaki group included A. Weil and 
J. Dieudonn\'e). 
See also the historic survey \cite{BHH}; further surveys exist \cite{Ho1}, \cite{Ku}.
Other specialized sources include books by A. Weil (1937; in French), J. W. Tukey \cite{Tu}, and
I. M. James \cite{Ja}.

There are also books by Naimpally--Warrack (1970) and Yefremovich--Tolpygo (2007; in Russian) on 
the closely related subject of proximity spaces (which coincide with uniform spaces in 
the metrizable case) and by H. Herrlich (1987; in German) and G. Preuss (1988) on more general notions
of nearness spaces.
Additional information can be drawn from chapters in some books on analysis and topological algebra: 
Gillman--Jerison (1960), Roelke--Dierolf (1981), W. Page (1978), N. R. Howes (1995), 
Arhangel'skii--Tkachenko (2008); and from chapters in some books on general topology: 
R. Engelking \cite{En}, Hu Sze-Tsen (1966), J. L. Kelley (1955), H. Schubert (English transl.\ 1968), 
S. Willard (1970).

Much of the modern development of uniform spaces seems to occur not within topology but, 
in particular, in Geometric Nonlinear Functional Analysis (see \cite{BL}, \cite{Ka})
and in Measure Theory (see \cite{Pac2}).

\subsection*{Acknowledgements}

I would like to thank T. Banakh, N. Brodskij, A. V. Chernavskij, A. N. Dranishnikov,
J. Dydak, O. Frolkina, M. Gugnin, B. LaBuz, R. Jimenez, J. Higes, J. Krasinkiewicz, S. Nowak,
K. Sakai, E. V. Shchepin, S. Spie\.z, J. Strom and H. Torunczyk for useful discussions.

\subsection*{Disclaimer}

I oppose all wars, including those wars that are initiated by governments at the time when 
they directly or indirectly support my research. The latter type of wars include all wars 
waged by the Russian state in the last 25 years (in Chechnya, Georgia, Syria and Ukraine) 
as well as the USA-led invasions of Afghanistan and Iraq.

\newpage
\part{REVIEW OF UNIFORM SPACES}\label{uniform spaces}

This chapter is intended to serve as an introduction to uniform spaces for the reader who has little to 
no previous acquaintance with the subject.
It appears to be rather different in viewpoint and in style from the existing introductions to uniform 
spaces in the literature, but mathematically it contains little new (apart from a few somewhat new 
results in \S\ref{finiteness}).
The chapter aims to be largely self-contained modulo straightforwardly verified facts, for which 
references are given.

\section{Metrizable uniform spaces}\label{generalities}

\subsection{Uniform continuity}

Two sequences $x_1,x_2,\dots$ and $y_1,y_2,\dots$ in a metric space $M$ are said to be {\it merging}%
\footnote{This terminology is not completely standard: while ``merging'' seems to be the most common choice 
\cite{DDF}, \cite{Bog}, other ones are also used, e.g.\  ``asymptotically approximating one another'' \cite{Pac1}.
In Russian, the meaning is neatly expressed by a single verb, ``sblizhat'sya'', which means 
``to approach one another''.
In Analysis, there is a closely related notion of ``asymptotic equality'': $f\sim g$ if $f(n)-g(n)=o\big(g(n)\big)$.}
if $d(x_n,y_n)\to 0$ as $n\to\infty$.
A map $f\:M\to N$ between metric spaces is {\it uniformly continuous} if and only if it sends any 
pair of merging sequences in $M$ into a pair of merging sequences in $N$.
(This is a straightforward reformulation of the familiar $\eps$-$\delta$ definition.)
Furthermore, it is not hard to see that $f$ is uniformly continuous if and only if $d(A,B)=0$ implies 
$d(f(A),f(B))=0$ for any $A,B\subset M$ (cf.\ \cite[II.38 and II.34]{I3}); here 
$d(A,B)=\inf\{d(a,b)\mid a\in A,\, b\in B\}$.

\subsection{Metrizable uniformity}
A {\it uniform homeomorphism} between metric spaces is a bijection that is
uniformly continuous in both directions.
Two metrics $d$ and $d'$ on a set $S$ are {\it uniformly equivalent} if $\id_S$
is a uniform homeomorphism between $(S,d)$ and $(S,d')$.
In particular, every metric $d$ is uniformly equivalent to the bounded metrics
$d'(x,y)=\min(d(x,y),1)$ and $d''(x,y)=\frac{d(x,y)}{1+d(x,y)}$.
A metrizable {\it uniformity} (or uniform structure) $u$ on $S$ is a uniform
equivalence class of metrics on $S$; each of these metrics {\it induces} $u$; and
a metrizable {\it uniform space} is a set endowed with a uniformity.
Clearly, the topology induced by a metric $d$ is determined by the uniformity
induced by $d$.

\begin{remark}\label{Frechet-Riesz}
Historically, the idea of a metrizable uniform space emerged together with those
of a metric space and a metrizable topological space.
Fr\'echet's thesis (1906, based on a series of 1904-05 papers), which introduced
metric spaces as well as a variant of topological spaces based on cluster points of
sequences, also studied an axiomatic structure midway between metric and metrizable
uniform spaces (see \cite[\S1.1]{BHH}).
Sets of axioms satisfied by the relation $d(A,B)=0$ between the subsets $A,B$ of
a metric space have been considered by F. Riesz in the same ICM talk (1908, based
on a 1906 paper), where he suggested a modification of Fr\'echet's approach based
on cluster points of sets as opposed to countable sequences (see \cite[\S1.5]{BHH}
and \cite{CV}).
\end{remark}

\subsection{Completeness}
We recall that a sequence of points $x_n$ of a metric space $M$ is called
a {\it Cauchy sequence} if for every $\eps>0$ there exists a $k$ such that
for every $j>k$, the $\eps$-neighborhood of $x_j$ in $X$ contains $x_k$.
Clearly, this notion depends only on the underlying uniform structure of $M$.
In fact, it is not hard to see that a sequence $(x_n)$ is Cauchy if and only if
every two subsequences of $(x_n)$ are merging.

A metrizable uniform space is called {\it complete} if every its Cauchy
sequence converges.
Every metrizable uniform space is a dense subset of a unique complete one,
which is called its {\it completion}; every uniformly continuous map into
a complete space uniquely extends over the completion of the domain
(see \cite{I3}).
Every {\it compactum}, i.e.\ a compact metrizable space, admits a unique uniform
structure, which is complete.
A metrizable uniform space is called {\it precompact} if its completion is compact.
Thus a subspace of a complete metrizable uniform space is precompact if and only if
its closure is compact.

\begin{lemma} \cite{Bou} \label{complete-finer} 
If $h\:X\to Y$ is a uniformly continuous homeomorphism between metric spaces and $Y$ is complete, 
then $X$ is complete.
\end{lemma}

\begin{proof} If $x_n$ is a Cauchy sequence in $X$, then $h(x_n)$ is a Cauchy sequence in $Y$.
Hence $h(x_n)$ converges, and therefore $x_n$ also converges.
\end{proof}

\subsection{Covers}\label{cover notation}
We recall that a {\it cover} (or a covering) of a set $S$ is a collection of
subsets of $S$ whose union is the whole of $S$.
A cover $C$ of $S$ is said to {\it refine} a cover $D$ of $S$ if every
$U\in C$ is a subset of some $V\in D$.
If $C$ is a cover of $S$, and $f\:T\to S$ is a map, we have the covers
$f(C)\bydef \{f(U)\mid U\in C\}$ and $f^{-1}(C)\bydef \{f^{-1}(U)\mid U\in C\}$ of $T$;
in the case where $T\incl S$ and $f$ is the inclusion map, we denote $f^{-1}(C)$
by $C\cap T$ and call it the {\it trace} of $C$ on $T$.
If $C$ and $D$ are covers of $S$, then $C\wedge D\bydef \{U\cap V\mid U\in C,\,V\in D\}$
is a cover of $S$ refining both $C$ and $D$.
Similarly one defines the {\it meet} $\bigwedge C_\lambda$ of a finite family of
covers $C_\lambda$; the meet of the empty family is the singleton cover $\{X\}$ of $X$.

If $T\incl S$ is a subset, the {\it star} $\st(T,C)$ of $T$ in a cover $C$ of $S$ is the union of
all elements of $C$ that intersect $T$.
We also abbreviate $\st(\{x\},C)$ to $\st(x,C)$.
A cover $C$ of $S$ is a {\it barycentric refinement} of a cover $D$ of $S$ if 
the cover $\{\st(x,C)\mid x\in S\}$ refines $D$.
Also, $C$ {\it star-refines} $D$ if $\{\st(U,C)\mid U\in C\}$ refines $D$.
The following implications are straightforward (compare \cite{Tu}*{V.2.16}):

\medskip
\centerline{\small $A$ barycentrically refines $B$ and $B$ barycentrically refines $C$}
\centerline{$\Downarrow$}
\centerline{\small $A$ star-refines $C$}
\centerline{$\Downarrow$}
\centerline{\small $A$ barycentrically refines $C$}
\medskip

\begin{lemma}[Tukey \cite{Tu}*{V.8.12}] \label{fully normal} 
Every open cover $C$ of a metric space $X$ has an open barycentric refinement.
\end{lemma}

\begin{proof} Let $B_\eps(x)$ denote the open $\eps$-ball about $x$.
For each $x\in X$ let $U_x$ be some element of $C$ containing $x$, and let $\eps_x>0$ be such that
$U_x$ contains $B_{4\eps_x}(x)$.
We may assume that each $\eps_x\le 1$.
Let $V_x=B_{\eps_x}(x)$ and let $D=\{V_x\mid x\in X\}$.
Let us show that $D$ barycentrically refines $C$.

Let us temporarily fix an arbitrary point $a\in X$, and let $A=\{x\in X\mid a\in V_x\}$.
Let $\eps=\sup\{\eps_x\mid x\in A\}$, which is finite since each $\eps_x\le 1$.
Then there exists a $b\in A$ such that $\eps_b>\frac23\eps$.
For each $x\in A$ we have
$V_x=B_{\eps_x}(x)\subset B_{2\eps_x}(a)\subset B_{3\eps_b}(a)\subset B_{4\eps_b}(b)$.
Hence
$\st(a,D)=\bigcup_{x\in A}V_x\subset B_{4\eps_b}(b)\subset U_b$.
\end{proof}

\subsection{Uniform covers}
A cover $C$ of a metric space $M$ is said to be {\it uniform} if there exists
a positive number $\lambda$ such that every subset of $M$ of diameter $<\lambda$
is contained in some $U\in C$; such a $\lambda$ is called a {\it Lebesgue number} of $C$.
Let us note that if a cover $C$ of $M$ is refined by the cover $C_\eps$ by all open balls
of radius $\eps$, then $C$ is uniform (with Lebesgue number $\eps$); and conversely,
every uniform cover of $M$ with Lebesgue number $\lambda$ is refined by $C_{\lambda/3}$.

\begin{lemma}[Lebesgue] \label{lebesgue} Every open cover $C$ of a compactum $K$ is uniform.
\end{lemma}

\begin{proof} Let $F$ be a finite open cover of $K$.
Given a $U\in F$, let $U'$ be the union of all elements of $F\but\{U\}$ and let $\eps=d(K\but U,\,K\but U')$.
Then $K\but U'$ lies in $U^\eps\bydef \{x\in U\mid B_\eps(x)\subset U\}$, where $B_\eps(x)$ denotes 
the closed $\eps$-ball about $x$.
Clearly, $U^\eps$ is open and $F^U\bydef (F\but\{U\})\cup\{U^\eps\}$ still covers $K$.

Now let $F_0=\{U_1,\dots,U_n\}$ be a finite subset of $C$ which covers $K$.
Assuming that $F_{i-1}$ is defined, let $F_i=F_{i-1}^{U_i}$.
Then $F_n=\{U_1^{\eps_1},\dots,U_n^{\eps_n}\}$ for some $\eps_i>0$. 
Namely, each $\eps_i=d(K\but U_i,\,K\but U'_i)$ where 
$U'_i=U_1^{\eps_1}\cup\dots\cup U_{i-1}^{\eps_{i-1}}\cup U_{i+1}\cup\dots\cup U_n$.
Let $\lambda=\min(\eps_1,\dots,\eps_n)$.
Given a set $S\subset K$ of diameter $\le\lambda$, let us pick some $x\in S$.
Then $x\in U_i^{\eps_i}$ for some $i$, and consequently $S\subset B_\lambda(x)\subset U_i$.
\end{proof}

\begin{remark} \label{lebesgue+}
The proof of Lemma \ref{lebesgue} works to prove a slightly stronger assertion (a strengthened version 
of \cite{vM}*{Lemma 1.1.1}): If $K$ is a compact subset of a metric space $X$, then for every open cover $C$ 
of $X$ there exists an open neighborhood $O$ of $K$ and a $\lambda>0$ such that the closed ball $B_\lambda(x)$ 
about $x$ lies in some element of $C$ for each $x\in O$.

Namely, the proof of Lemma \ref{lebesgue} goes through verbatim with this $C$ (note that $K\but U_i$ and $K\but U_i'$ 
are still compact, even though $U_i$ and $U_i'$ are no longer subsets of $K$), and we set $O$ to be the union of all 
elements of $F_n$.
\end{remark}

\begin{corollary} A cover $C$ of a compactum is uniform if and only if $C$ can be refined by an open cover.
\end{corollary}

It is easy to see that a map $f\:M\to N$ between metric spaces is uniformly continuous if and only if 
for every uniform cover $D$ of $N$, the cover $f^{-1}(D)$ of $M$ is uniform.
(The latter is also equivalent to saying that there exists a uniform cover $C$ of $M$ such that $f(C)$ 
refines $D\cap f(M)$.)
It follows that the property of being uniform for a cover of $M$ depends only on the underlying uniform 
structure of $M$.

\subsection{Some properties}

\begin{lemma} \label{shrinking} \cite{I3}*{IV.19}
If $\{U_\alpha\mid\alpha\in A\}$ is a uniform cover of a metric space $X$, then there exist open uniform covers 
$C$ and $\{V_\alpha\mid\alpha\in A\}$ of $X$ such that $\st(V_\alpha,C)\subset U_\alpha$ for each $\alpha\in A$.
\end{lemma}

\begin{proof} Let $\lambda$ be a Lebesgue number of $\{U_\alpha\}$.
Given a subset $S\subset X$ and an $\eps>0$, let $S^\eps=\{x\in X\mid d(x,S)<\eps\}$
and $S^{-\eps}=\{x\in X\mid d(x,\,X\but S)>\eps\}$.
Let $V_\alpha=U_\alpha^{-\lambda/4}$.
Each $V_\alpha$ is open.
If $C$ is the cover of $X$ by all open $\lambda/8$-balls, then $\st(V_\alpha,C)\subset U_\alpha$.

It remains to check that $\{V_\alpha\}$ is a uniform cover (and, in particular, a cover).
Given a subset $S\subset X$ of diameter $<\lambda/3$, the subset $S^{\lambda/3}$ is of diameter $<\lambda$
and so lies in some $U_\alpha$.
Hence $S\subset U_\alpha^{-\lambda/4}$.
Thus $\lambda/3$ is a Lebesgue number of $\{V_\alpha\}$.
\end{proof}

A family of disjoint subsets $X_\alpha\incl M$ is called {\it uniformly disjoint} if it constitutes 
a uniform cover of its union.
(In other words, if there exists an $\eps>0$ such that $d(X_\alpha,X_\beta)>\eps$ whenever 
$\alpha\ne\beta$, where $d(X,Y)=\sup\{d(x,y)\mid x\in X,\,y\in Y\}$.)
The metric space $M$ itself is called {\it uniformly discrete} if the collection of its singletons is 
uniformly discrete.

It is easy to see that $M$ is uniformly discrete if and only if for every pair of merging sequences 
in $M$, both sequences are eventually constant.
If $M$ is uniformly discrete, then all Cauchy sequences in it are eventually constant (in other words, 
$M$ is complete and non-uniformly discrete), but the converse is not true.

A neighborhood $U$ of a subset $S$ of a metrizable uniform space $X$ is called
{\it uniform} if it contains the star of $S$ in some uniform cover of $X$;
or equivalently if $S$ and $X\but U$ constitute a uniformly discrete collection.
The space $M$ is called {\it uniformly connected} if contains no subset that is
its own uniform neighborhood.

\subsection{Basis of metrizable uniformity}\label{basis}
Uniform covers can be used to axiomatize the notion of a metrizable uniform structure.
Let us call a sequence of covers $C_1,C_2,\dots$ of a set $S$ {\it fundamental}
if it satisfies

\begin{properties}
\item each $C_{n+1}$ barycentrically refines $C_n$;

\item for any distinct points $x,y\in S$ there exists an $n$ such
that no element of $C_n$ contains both $x$ and $y$.
\end{properties}

A {\it basis} for a metrizable uniformity $u$ on $S$ is a fundamental sequence of covers $C_n$ of $S$ 
such that each $C_n$ is uniform with respect to $u$, and every uniform cover $C$ of $(S,u)$ is refined 
by some $C_n$.
If $u$ is induced by a metric $d$ on $S$, then the covers $C_n$ of $S$ by the open $d$-balls of 
radius $2^{-n}$ about all points of $S$ form a {\it standard} basis of $u$.

Two fundamental sequences of covers $C_n$ and $D_n$ of $S$ are {\it equivalent} if for each $n$ there 
exists an $m$ such that $C_m$ refines $D_n$ and $D_m$ refines $C_n$.
Clearly, every two bases of $u$ are equivalent; and every fundamental sequence of covers of $S$ that is 
equivalent to a basis of $u$ is itself a basis of $u$.

\begin{theorem}\cite{AU}, \cite{Tu}*{Chapter 5}, \cite{I3}*{p.\ 8} \label{A.1} There exists
a bijection between metrizable uniformities on $S$ and equivalence classes
of fundamental sequences of covers of $S$, which assigns to a uniformity
the equivalence class of any its standard basis.
\end{theorem}

Of course, the statement in the 1923 paper of Alexandroff and Urysohn \cite{AU} is in different terms, 
even though the proof is essentially the same as that given below.
It is likely, however, that the authors must have been at least partially aware
of this interpretation of their result, since according to Fr\'echet
(1928; cf.\ \cite[p.\ 585]{BHH}), they have been thinking of avoiding the use
of metric in defining the notion of uniform continuity.

\begin{proof} It remains to show that every fundamental sequence of covers $C_n$
of $S$ is a basis of some metrizable uniformity.
To this end consider an auxiliary `pre-distance' function
$f(x,y)=\inf\{2^{-n}\mid x,y\in U\text{ for some }U\in C_{2n}\}$, and
define $d(x,y)$ to be the infimum of the sums $f(x_0,x_1)+\dots+f(x_{n-1},x_n)$
over all finite chains $x=x_0,\dots,x_n=y$ of points of $S$.
Clearly, $d$ is a pseudo-metric, i.e.\ it is symmetric, satisfies the triangle
axiom and is such that $d(x,x)=0$ for every $x\in S$.
Let $D_{2n-1}$ be the set of all subsets of $S$ of $d$-diameter at most $2^{-n}$.
Since $d(x,y)\le f(x,y)$, each $U\in C_{2n}$ also belongs to $D_{2n-1}$, thus
$C_{2n}$ refines $D_{2n-1}$.

To prove that $d$ is a metric and that $\{C_{2n}\}$ is a basis for the uniformity
induced by $d$ is suffices to show that each $D_{2n+1}$ refines $C_{2n}$.
The latter, in turn, would follow if we prove that $f(x,y)\le 2d(x,y)$.
Let us show that $f(x,y)\le 2[f(x_0,x_1)+\dots+f(x_{n-1},x_n)]$ for every
finite chain $x=x_0,\dots,x_n=y$ of points of $S$.
The case $n=1$ is clear.
Let $\ell_{[i,j]}=f(x_i,x_{i+1})+\dots+f(x_{j-1},x_j)$ whenever $0\le i<j\le n$.
Consider the maximal $k$ such that $\ell_{[0,k]}\le\frac12\ell_{[0,n]}$.
Then $\ell_{[k+1,n]}\le\frac12\ell_{[0,n]}$ as well.
On the other hand, by the induction hypothesis, $f(x_0,x_k)\le 2\ell_{[0,k]}$
and $f(x_{k+1},x_n)\le 2\ell_{[k+1,n]}$.
Thus each of the numbers $f(x_0,x_k)$, $f(x_k,x_{k+1})$, $f(x_{k+1},x_n)$
does not exceed $\ell_{[0,n]}$.
If $m$ is the least integer such that $2^{-m}\le\ell_{[0,n]}$, the pairs
$\{x_0,x_k\}$, $\{x_k,x_{k+1}\}$, $\{x_{k+1},x_n\}$ are contained in some
$U_1,U_2,U_3\in C_{2m}$.
Since $C_{2m}$ star-refines $C_{2m-2}$, we obtain that $x_0$ and $x_n$
belong to some $V\in C_{2m-2}$.
Hence $f(x_0,x_n)\le 2^{-m+1}\le 2\ell_{[0,n]}$, as required.
\end{proof}

\section{Uniform spaces} \label{general uniform spaces}

Although the main theme of the present treatise is {\it metrizable} uniform spaces, in some auxiliary 
constructions (such as quotient uniformities and semi-uniform products) we will have to temporarily 
deal with general (possibly non-metrizable) uniform spaces.
It is also easier to understand some other constructions (e.g.\ countable product) by working in 
greater generality.
Thus we will have to devote some attention to reviewing basic theory of general uniform spaces
(in this and several subsequent sections).

\subsection{Uniform structures}\label{basis'}

Let $u$ be a family of covers of a set $S$.
It is called a {\it uniformity} (or {\it uniform structure}) on $S$ if

\begin{properties}
\item every $C,D\in u$ have a common refinement $E\in U$;

\item each $C\in u$ is barycentrically refined by some $D\in u$; 

\item if a cover $C$ of $S$ is refined by some $D\in u$, then $C\in u$;

\item for any distinct points $x,y\in S$ there exists $C\in u$ such that no element of $C$ contains 
both $x$ and $y$.
\end{properties}

If $u$ satisfies only (1), (2) and (3), it is called a {\it pre-uniformity}.%
\footnote{Some authors call this a ``uniformity'', and one also satisfying (4) a ``separated uniformity''.
While (4) is indeed a direct analogue of the separation axioms for topological spaces, in practice it 
works quite differently: (a) it is essentially the only useful condition of its kind, as opposed to
the numerous $T_i$ axioms; (b) while pre-uniformities that are not uniformities often arise in various 
constructions, one can usually get rid of them, since every pre-uniform space has a canonical uniform 
quotient (see \S\ref{separatization}); (c) pre-uniformities other than uniformities do not seem to have 
much applied significance, as opposed e.g.\ to the order topology of posets, \'etal\'e spaces of sheaves, and 
the Zariski topology.}
Equivalently, a pre-uniformity on $S$ is a family of covers of $S$ that forms a filter with respect 
to barycentric refinement (cf.\ \cite[I.6]{I3}).
Let us also note that (1) is equivalent to $C,D\in u\Rightarrow C\land D\in u$ in the presence of (3).

If the family $u$ satisfies only (1), (2) and (4) (respectively, only (1) and (2)), then it is called a 
{\it (pre-)fundamental} family of covers.
If $u$ is a (pre-)fundamental family of covers, the family $\hat u$ of all covers that have a refinement in $u$
is a (pre-)uniformity, and $u$ is called its {\it basis}.

If $u$ is a (pre-)uniformity on $S$, the pair $(S,u)$ is called a {\it (pre-)uniform space} and the elements 
of $u$ are called {\it uniform covers} (with respect to $u$).

A map between two uniform spaces $f\:X\to Y$ is called {\it uniformly continuous} if for every uniform cover $C$ 
of $Y$, the cover $f^{-1}(C)$ of $X$ is uniform.
A uniformity $u$ on $S$ is {\it finer} than $u'$ if $\id_S\:(S,u)\to (S,u')$ is uniformly continuous.
If $X$ is a (pre-)uniform space and $S$ is a subset of $X$, the (pre-)uniform structure of {\it subspace} on $S$ 
is given by the covers $C\cap S$, where $C$ runs over all uniform covers of $X$.

If $X$ is a pre-uniform space, its induced topology is defined by declaring a subset $S\incl X$ open if and 
only if for each $x\in S$ there exists a uniform cover $C$ of $X$ such that $\st(x,C)\incl S$.
In other words, a base of neighborhoods of $x$ is given by the stars of $x$ in basic uniform covers of $X$.
The induced topology of every uniformity is Tychonoff (=completely regular Hausdorff=$T_{3\frac12}$) 
\cite[I.11]{I3}, and every pre-uniformity whose induced topology is $T_1$ is clearly a uniformity.

We refer to \cite{I3} for the definition and properties of complete uniform spaces.

\subsection{Reduction to the metrizable case}

As far as single (countable) coverings are concerned, all uniform spaces are like 
(separable) metrizable uniform spaces, in the following sense:

\begin{lemma} \cite[I.14]{I3}, \cite{I0}*{1.0}, \cite[3.1]{GI} \label{metric-covers}
Every (countable) uniform cover $C$ of a uniform space $X$ is refined by $f^{-1}(D)$ for some (countable) 
uniform cover $D$ of some (separable) metric space $M$ and some uniformly continuous $f\:X\to M$.
\end{lemma}

\begin{lemma} \label{shrinking2} \cite{I3}*{IV.19}
If $\{U_\alpha\mid\alpha\in A\}$ is a uniform cover of a pre-uniform space $X$, then there exist open 
uniform covers $\{V_\alpha\mid\alpha\in A\}$ and $C$ of $X$ such that $\st(V_\alpha,C)\subset U_\alpha$ 
for each $\alpha\in A$.
\end{lemma}

\begin{proof} Let $C=\{U_\alpha\}$, and let $M$, $f$ and $D$ be given by Lemma \ref{metric-covers}.
Let us fix some $\phi\:D\to C$ such that $f^{-1}(W)\subset\phi(W)$ for each $W\in D$.
Let $U'_\alpha$ be the union of all elements of $\phi^{-1}(U_\alpha)$.
Thus $f^{-1}(U'_\alpha)\subset U_\alpha$, and the cover $\{U'_\alpha\}$ of $M$ is uniform 
since it is refined by $D$.
Hence by Lemma \ref{shrinking} there exist open uniform covers $\{V'_\alpha\mid\alpha\in A\}$ and $C'$ 
of $M$ such that $\st(V'_\alpha,C')\subset U'_\alpha$ for each $\alpha\in A$.
Set $V_\alpha=f^{-1}(V'_\alpha)$ and $C=f^{-1}(C')$.
\end{proof}

In particular, every uniform cover of a pre-uniform space can be refined by an open cover
(cf.\ \cite[I.19]{I3}).
It follows that every pre-uniformity has a basis consisting of open covers.

\subsection{Families of pseudo-metrics}

Every pseudo-metric induces a pre-uniformity.
Similarly to Theorem \ref{A.1}, a pre-uniformity is pseudo-metrizable if and only if it has 
a basis that is a pre-fundamental sequence of covers $C_i$ (i.e.\ satisfies condition (1) 
in \ref{basis}).

Also similarly to Theorem \ref{A.1}, every pre-uniformity has a basis of covers
$C_{i\alpha}$, where each $C_{i\alpha}$ consists of all balls of radius $2^{-i}$
with respect to some pseudo-metric $d_\alpha$ (cf.\ \cite[proof of I.14]{I3},
\cite[proof of 8.1.10]{En}).
This yields a bijective correspondence between pre-uniformities on $S$ and
uniform equivalence classes of collections $D$ of pseudo-metrics on $S$ such that
(i) for any $d,d'\in D$ there exists a $d''\in D$ with $d''\ge\max(d,d')$;
uniformities correspond to the equivalence classes of collections $D$ such that
(ii) for each pair of distinct points $x,y\in S$ there exists a $d\in D$ such
that $d(x,y)>0$ (cf.\ \cite[8.1.18]{En}).
Two such collections $D$ and $D'$ are {\it uniformly equivalent} if
$\id\:(X,D)\to(X,D')$ is uniformly continuous in both directions.

The notion of uniform continuity can be expressed in these terms as follows: a map
$f\:(X,D)\to (Y,E)$ is uniformly continuous if and only if for each $\eps>0$ and $e\in E$
there exists a $\delta>0$ and a $d\in D$ such that $d(x,y)\le\delta$ implies
$e(f(x),f(y))\le\eps$.
A related criterion is: a function between pre-uniform spaces $f\:X\to Y$ is uniformly continuous 
if and only if for each uniformly continuous pseudo-metric $e$ on $Y$, the pseudo-metric 
$d(x,y)=e(f(x),f(y))$ is uniformly continuous (cf.\ \cite[8.1.22]{En}).

\section{Universal constructions}

We shall work with {\it concrete} categories over the category of sets, that is
``constructs'' in the terminology of {\it The Joy of Cats} \cite{AHS}.
As a bridge between sets (which we have to start from anyway) and abstract
categories (whose powerful machinery we do need), they enable a unified
treatment of constructions from Isbell \cite{I3} and Bourbaki \cite{Bou}.

Let $\U$ (resp.\ $\bar\U$) denote the category of (pre-)uniform spaces and uniformly continuous maps, 
and $\T$ the category of topological spaces and continuous maps --- all viewed as concrete categories 
over the category of sets.

\subsection{Initial uniformity} \label{initial}
Given a set $S$ and a family $F$ of maps $f_\lambda\:S\to Y_\lambda$ into pre-uniform spaces, all 
finite meets of the form $f_{\lambda_1}^{-1}(C_1)\wedge\dots\wedge f_{\lambda_k}^{-1}(C_k)$,
where each $C_i$ is a uniform cover of $Y_{\lambda_i}$, clearly form a basis of a pre-uniformity 
$u_F$ on $S$.
It is easy to see that $u_F$ is the coarsest pre-uniformity on $S$ making all the $f_\lambda$
uniformly continuous (cf.\ \cite[I.8]{I3}).
Moreover, it is not hard to see that $u_F$ is {\it initial} in $\bar\U$ with respect to $F$; that is, 
a map $g\:Z\to (S,u_F)$, where $Z$ is a pre-uniform space, is uniformly continuous if (and, 
obviously, only if) each composition $Z\xr{g} (S,u_F)\xr{f_\lambda} Y_\lambda$ is uniformly continuous
(cf.\ \cite[I.17]{I3}, where the non-trivial part of the argument is redundant).
Conversely, if a pre-uniformity on $S$ is initial with respect to $f$, then it is the coarsest 
pre-uniformity making all the $f_\lambda$ uniformly continuous (cf.\ \cite[10.43]{AHS}).

Thus we may call $u_F$ {\it the} initial pre-uniformity in $\bar\U$ with respect to $F$.
The induced topology of $u_F$ is initial in $\T$ with respect to $F$ \cite[I.16]{I3}.
Corresponding to the empty family $\emptyset$ of maps on $S$ we have the {\it anti-discrete} 
pre-uniformity $u_\emptyset=\{\{S\}\}$, which is not a uniformity (cf.\ \cite[8.3]{AHS}).

The pre-uniformity $u_F$ is a uniformity, and is initial in $\U$ with respect to $F$, provided that 
each $Y_\lambda$ is a uniform space, and $F$ is {\it point-separating}, i.e.\ for every pair of 
distinct points $x,y\in S$ there exists a $\lambda$ such that $f_\lambda(x)\ne f_\lambda(y)$ (cf.\
\cite[I.8 and I.17]{I3}).

\subsection{Finest uniformity} \label{finest}
(Pre-)uniformities on a set $S$ are ordered by inclusion, as subsets of the set of all covers of $S$.
Given a family $U$ of pre-uniformities $u_\lambda$ on $S$, the initial pre-uniformity $u_F$ 
with respect to the family $F$ of the maps $\id\:S\to (S,u_\lambda)$ coincides with the least upper 
bound $\sup U$ of the family $U$ (cf.\ \cite[\S II.1.5]{Bou}).
If at least one $u_\lambda$ is a uniformity, then so is $\sup U$.
By the above, $\sup U=\{C_1\wedge\dots\wedge C_k\mid k\in\N,\,C_i\in\bigcup_\lambda u_\lambda\}$.

A cover $C_1$ of a set $S$ is called {\it normal} with respect to a family $c$ of covers of $S$, if 
it can be included in an infinite sequence $C_1,C_2,C_3,\dots$ of covers of $S$ such that each 
$C_{i+1}$ barycentrically refines $C_i$, and each $C_i$ is refined by some element of $c$.
If $c$ is nonempty and every two elements of $c$ have a common refinement in $c$, then it is easy 
to see that the family of all covers of $S$ that are normal with respect to $c$ constitutes 
a pre-uniformity $u_c$ on $S$.
Clearly, $u_c$ is the finest among those pre-uniformities $u_\lambda$ that have a basis contained 
in $c$; in other words, $u_c=\sup U$ and $u_c\in U$, where $U$ is the family of all such $u_\lambda$ 
(cf.\ \cite[I.10]{I3}).

\subsection{Fine uniformity}
Given a Tychonoff topological space $X$, its topology is initial with respect to the family $F$ 
of all continuous maps $f_\lambda\:X\to\R$, and therefore is induced by the uniformity $u_F$, 
where $\R$ is endowed with the usual uniformity (cf.\ \cite[I.15]{I3}).
Since the set of uniformities inducing the given topology on $X$ is non-empty, there exists 
a finest such uniformity $u_X$; it consists of all covers of $X$ that are normal with respect to 
the family of all open covers of $X$ (cf.\ \cite[I.20]{I3}).
This is the {\it fine uniformity} of the Tychonoff topological space $X$.

A map from $(X,u_X)$ into a pre-uniform space is uniformly continuous if and only if it is continuous; 
and $u_X$ is characterized by this property \cite[Exer.\ IX.1.5]{Bou}.
Let us also note that a map $(X,u_F)\to\R$ is uniformly continuous if and only if it is continuous.
By \cite[Exer.\ IX.1.5]{Bou}, $u_X$ corresponds to the family of all
pseudo-metrics on $X$ that are uniformly continuous as functions $X\x X\to\R$;
whereas by \cite[Example at the end of \S IX.1.2]{Bou}, $u_F$ corresponds
to the family of all pseudo-metrics $d_\lambda(x,y)=|f_\lambda(x)-f_\lambda(y)|$.

Lemma \ref{fully normal} implies that the fine uniformity of a metrizable topological space $X$
consists of all covers that can be refined by open covers.
This uniformity is itself almost never metrizable --- specifically, it is
metrizable if and only if the set $K$ of non-isolated points of $X$ is
compact, and the complement to any uniform neighborhood of $K$ is uniformly
discrete (see \cite{At1}, \cite{Le}, \cite{Ra}).

\begin{lemma} \label{shrinking3} If $\{U_\alpha\mid\alpha\in A\}$ is an open cover of 
a metrizable space $X$, then there exist open covers $C$ and $\{V_\alpha\mid\alpha\in A\}$ of $X$ 
such that $\st(V_\alpha,C)\subset U_\alpha$ for each $\alpha\in A$.
\end{lemma}

\begin{proof} Since $X$ is metrizable, $\{U_\alpha\}$ is a uniform cover with respect to the fine uniformity 
of $X$.
Then Lemma \ref{shrinking2} yields the desired $C$ and $\{V_\alpha\}$.
\end{proof}

\begin{proof}[Alternative proof]
By Lemma \ref{fully normal}, $D\bydef \{U_\alpha\}$ is barycentrically refined by an open cover, which is in turn
barycentrically refined by an open cover $C$.
Thus there exists a function $\phi\:C\to D$ such that $\st(W,C)\subset\phi(W)$ for each $W\in C$.
Let $V_\alpha$ be the union of all elements of $\phi^{-1}(U_\alpha)$.
Thus $\st(V_\alpha,C)=\bigcup_{W\in\phi^{-1}(U_\alpha)}\st(W,C)\subset U_\alpha$.
\end{proof}

\subsection{Final pre-uniformity}
Given a set $S$ and a family $F$ of maps $f_\lambda\:Y_\lambda\to S$ from pre-uniform spaces,
let $c$ be the family of all covers $C$ of $S$ such that the cover $f_\lambda^{-1}(C)$ of $Y_\lambda$ 
is uniform for each $\lambda$; then $u^F\bydef u_c$ is a pre-uniformity on $S$.
(Here $u_c$ is the family of all covers of $S$ that are normal with respect to $c$; see \S\ref{finest}.)
Clearly $u^F$ is the finest pre-uniformity on $S$ making all the $f_\lambda$ uniformly continuous 
(cf.\ \cite[Exer.\ I.7]{I3}).
Moreover, it is easy to see that $u^F$ is {\it final} in $\bar\U$ with respect to $f$; that is, 
a map $g\:(S,u^F)\to Z$, where $Z$ is a pre-uniform space, is uniformly continuous if (and, obviously, 
only if) each composition $Y_\lambda\xr{f_\lambda} (S,u^F)\xr{g} Z$ is uniformly continuous.
(To see this, note that if $E$ barycentrically refines $D$, then $g^{-1}(E)$ barycentrically refines $g^{-1}(D)$.)
Conversely, if a pre-uniformity on $S$ is final with respect to $F$, then it is the finest pre-uniformity 
making each $f_\lambda$ uniformly continuous (this is the dualization of \cite[10.43]{AHS}).

Thus we may call $u^F$ {\it the} final pre-uniformity in $\bar\U$ with respect to $F$.
Corresponding to the empty family $\emptyset$ of maps into $X$ we have the {\it discrete} 
pre-uniformity $u^\emptyset$, which is a uniformity (cf.\ \cite[8.1]{AHS}).
Let us note that the fine uniformity of a topological space $(S,t)$ is nothing but the final uniformity 
corresponding to the family of inclusions in $S$ of all compact subspaces of $(S,t)$, considered as 
uniform spaces (with their unique uniformity).

The question when $u^F$ is a uniformity is not easy in general (see \cite[Exer.\ I.7]{I3}, 
\cite[Theorem 2.2]{GI} for partial results).
Instead, one has the following construction.

\subsection{Uniform space associated to a pre-uniform space}\label{separatization}
If $X$ is a pre-uniform space, let us write $x\sim_0 y$ if every uniform cover of $X$ contains 
an element $U$ such that $x,y\in U$.
This {\it separating} relation $\sim_0$ on $X$ is an equivalence relation: if $C$ is a uniform cover 
of $X$ such that no element of $C$ contains both $x$ and $z$, and $D$ is a uniform barycentric refinement 
of $C$, then $D$ cannot contain elements $U,V$ such that $x,y\in U$ and $y,z\in V$.
It is easy to see that $x\sim_0 y$ if and only if $d(x,y)=0$ for each uniformly continuous 
pseudo-metric on $X$.

Let $q\:X\to X/_{\sim_0}$ assign to each point its separating equivalence class.
Let $c$ be the family of all covers $C_0$ of the set $X/_{\sim_0}$ such that 
$q^{-1}(C_0)$ is uniform.
If $C$ is a uniform cover of $X$, then $q(C)$ is normal with respect to $c$.%
\footnote{Indeed, let $D$ be a uniform barycentric refinement of $C$ and $E$ be a uniform star-refinement 
of $D$.
Since $E$ is a uniform cover, $q^{-1}\big(q(U)\big)\subset\st(U,F)$ for every $U\in E$ and every uniform 
cover $F$ of $X$. 
In particular, this holds with $F=E$.
Therefore $q^{-1}\big(q(E)\big)$ refines $D$.
Also $D$ barycentrically refines $C$ and $C$ refines $q^{-1}\big(q(C)\big)$.
Therefore $q(E)$ barycentrically refines $q(C)$.
Also $q^{-1}\big(q(E)\big)$ is uniform since it is refined by $E$.}
It follows that the covers $q(C)$, where $C$ runs over all uniform covers of $X$, form 
the final pre-uniformity $u^{\{q\}}$ on $X/_{\sim_0}$; by construction it is a uniformity.
(Cf.\ \cite[\S II.3.8]{Bou}.)

\begin{remark} \label{separatization-remark}
Let us note that the corresponding story for topological spaces is more complicated.
Every topological space has the Kolmogorov quotient, or the maximal $T_0$ quotient, which is obtained by 
identifying points that are contained in the same open sets.
Also every topological space $X$ has the Hausdorffization, or the maximal $T_2$ quotient.
However, the relation of not having disjoint neighborhoods is not an equivalence relation, so one has to
consider its transitive closure $\sim$.
Worse yet, $X/_\sim$ need not be Hausdorff, so this procedure needs to be transfinitely iterated in order to 
get the Hausdorffization of $X$. 
For a pseudo-metric space $X$, its Kolmogorov quotient is metrizable, and hence coincides with the Haudorffization
and also with the underlying topological space of the associated uniform space $X/_{\sim_0}$.
\end{remark}

\subsection{Coarsest pre-uniformity}
Given a family $U$ of pre-uniformities $u_\lambda$ on a set $S$, the final pre-uniformity $u_F$ with respect to
the family $F$ of maps $\id\:(S,u_\lambda)\to S$ is the greatest lower bound $\inf U$ of the family $U$.
Alternatively, $\inf U=\sup U_*$, where the set $U_*$ of all lower bounds of $U$ 
(among all pre-uniformities on $S$) is non-empty since it contains the anti-discrete pre-uniformity 
$u_\emptyset$ (cf.\ \cite[\S II.1.5]{Bou}).
Similarly $\sup U=\inf U^*$, where the set $U_*$ of all upper bounds of $U$ (among all pre-uniformities on $S$) 
is non-empty since it contains the discrete uniformity $u^\emptyset$.
By the above, $\inf U$ consists of all covers that are normal with respect to the family 
$\bigcap_\lambda u_\lambda$.

If $c$ is a family of covers of a set $S$ such that every $C\in c$ is normal with respect to $c$, then it is 
easy to see that the family of covers $C_1\wedge\dots\wedge C_k$, where $k\in\N$ and each $C_i\in c$,
is a base of a pre-uniformity $u^c$ on $X$.
Clearly, $u^c$ is the coarsest among those pre-uniformities $u_\lambda$ that contain $c$; that is, 
$u^c=\inf U$ and $u^c\in U$, where $U$ is the family of all such $u_\lambda$ (cf.\ \cite[I.9]{I3}).

\subsection{Coarse uniformity}
A Tychonoff space $X$ admits a coarsest uniformity $u^X$ inducing its topology if and only if $X$ is 
locally compact; when $X$ is locally compact, $u^X$ coincides with the uniformity of the subspace of 
the one-point compactification of $X$, as well as with the initial uniformity with respect to 
the family of all continuous maps $X\to\R$ that vanish on the complement to a compact set
\cite[Theorem XIV]{Sa} (see also \cite[Exer.\ II.10]{I3}, \cite[Exer.\ IX.1.15]{Bou}).

A metrizable space $X$ admits a coarsest metrizable uniformity inducing its topology if and only if $X$ 
is locally compact and separable \cite[Corollary to Theorem 1]{Sh}.

\section{Basic operations} 

\subsection{Product} \label{product}
The {\it product} $\prod X_\lambda$ of uniform spaces $X_\lambda$ is their set-theoretic product $X$ 
endowed with the initial uniformity with respect to the family of projections $\pi_\lambda\:X\to X_\lambda$ 
(cf.\ \cite[10.53]{AHS}).
Thus the induced topology of $\prod X_\lambda$ is the product topology, and a cover of $\prod X_\lambda$ 
is uniform iff it is refined by $\pi_{\lambda_1}^{-1}(C_1)\wedge\dots\wedge\pi_{\lambda_k}^{-1}(C_k)$
for some uniform covers $C_1,\dots,C_k$ of some finite subcollection $X_{\lambda_1},\dots,X_{\lambda_k}$.
It is easy to check that $(X,\pi_\lambda)$ is also the product of $X_i$'s in $\U$ in the sense of 
abstract category theory (see \cite[p.\ 14]{I3}).

If $X$ and $Y$ are metrizable uniform spaces and $C_n$, $D_n$, $n=1,2,\dots$ are bases of their 
uniform covers, then clearly $E_n\bydef \pi_X^{-1}(C_n)\wedge\pi_X^{-1}(D_n)$ form a basis of uniform covers of
$X\x Y$, where $\pi_X$ and $\pi_Y$ denote the projections.
(Thus if $C_n=\{U_i\}$ and $D_n=\{V_j\}$ then $E_n=\{U_i\x V_j\}$.)
It follows that given metrics, denoted $d$, on $X$ and $Y$, then a metric $d_\infty$ on $X\x Y$ is given by 
$d_\infty\big((x,y),(x',y')\big)=\max\big(d(x,x'),\,d(y,y')\big)$.
Since $a+b\ge\max(a,b)\ge\frac12(a+b)$ for $a,b\ge 0$, it is uniformly equivalent to $d_1$, 
where $d_1\big((x,y),(x',y')\big)=d(x,x')+d(y,y')$.
Since $\max(a,b)^p\le a^p+b^p\le (a+b)^p$ for $a,b\ge 0$ and $p\ge 1$, these two metrics are also uniformly 
equivalent to $d_p$, where $d_p\big((x,y),(x',y')\big)^p=d(x,x')^p+d(y,y')^p$ for each $p\in (1,\infty)$.
That $d_p$ is a metric follows from the Minkowski inequality
$\big((a+\alpha)^p+(b+\beta)^p\big)^{1/p}\le(a^p+b^p)^{1/p}+(\alpha^p+\beta^p)^{1/p}$, where 
$a,b,\alpha,\beta\ge 0$.

Now let $X_1,X_2,\dots$ be metrizable uniform spaces and for each $i$ let $C_n^{(i)}$, $n=1,2,\dots$, 
be a basis of uniform covers of $X_i$.
Then a basis of uniform covers of $\prod X_i$ is given by
$D_n\bydef \pi_1^{-1}(C_n^{(1)})\wedge\pi_2^{-1}(C_{n-1}^{(2)})\wedge\dots\wedge
\pi_n^{-1}(C_1^{(n)})$, which is indeed a fundamental sequence.
It follows that given a metric on each $X_i$, denoted $d$, such that each $X_i$ is of diameter at most $1$, 
then the uniformity on $\prod X_i$ is induced by the metric $d_\infty$, defined by 
$d_\infty(x,y)=\sup_{i\in\N} 2^{-i}d(x_i,y_i)$, where $x=(x_i)$ and $y=(y_i)$.
It is not hard to see that the latter is uniformly equivalent to the metric $d_1$, defined by 
$d_1(x,y)=\sum_i 2^{-i}d(x_i,y_i)$ (still assuming that each $X_i$ is of diameter $\le 1$).
Indeed, clearly $d_\infty(x,y)\le d_1(x,y)$.
Conversely, if $d_\infty(x,y)<2^{-n}$, then $2^{-i}d(x_i,y_i)<2^{-n}$ for $i=1,\dots,n$, and consequently
$\sum_{i=1}^n 2^{-i}d(x_i,y_i)\le n2^{-n}$.
On the other hand, $\sum_{i=n+1}^\infty 2^{-i}d(x_i,y_i)\le\sum_{i=n+1}^\infty 2^{-i}=2^{-n}$.
Hence $d_1(x,y)<(n+1)2^{-n}$.
The two metrics are also uniformly equivalent to the metric $d_p$, defined by 
$d_p(x,y)^p=\sum_i\big(2^{-i}d(x_i,y_i)\big)^p$ (still assuming that each $X_i$ is of diameter $\le 1$), 
since $d_\infty(x,y)^p\le d_p(x,y)^p\le d_1(x,y)^p$ for each $p\in (1,\infty)$.
That $d_p$ is a metric follows from the Minkowski inequality for sequences.

Another basis of uniform covers of $\prod X_i$ is given by
$D_n\bydef \pi_1^{-1}(C_n^{(n)})\wedge\dots\wedge\pi_n^{-1}(C_1^{(n)})$.
It follows that given a metric $d$ on each $X_i$ (not necessarily bounded),
the uniformity on $\prod X_i$ is induced by the metric $d_\infty'$, defined by 
$d_\infty'(x,y)=\sup_{i\in\N}\min\{2^{-n},d(x_i,y_i)\}$, where $x=(x_i)$ and $y=(y_i)$.

\subsection{Sequential inverse limit}\label{invlimits}
Let $X_1,X_2,\dots$ be a sequence of metrizable uniform spaces.
Given uniformly continuous maps $f_i\:X_{i+1}\to X_i$ for each $i$, the {\it inverse limit} 
$L\bydef \invlim(\dots\xr{f_1}X_1\xr{f_0}X_0)$ is defined to be the subset of $\prod X_i$ consisting of 
{\it threads}, i.e.\ sequences $(x_1,x_2,\dots)$ such that each $f_i(x_{i+1})=x_i$.
The map $f^\infty_i\:L\to X_i$ is defined by restricting the projection $\pi_i\:\prod X_j\to X_i$.
The {\it bonding maps} $f_i$ have compositions $X_j\xr{f_{j-1}}\dots\xr{f_i}X_i$ denoted by $f^j_i$.
Since every two uniform covers of each $X_{i+1}$ can be refined by a single uniform cover, we conclude 
that a cover of $L$ is uniform if and only if it can be refined by the preimage of a single uniform cover 
of some $X_i$.
It is easy to check that $(L,f^\infty_i)$ is the category-theoretic inverse limit, i.e.\ every family of 
uniformly continuous maps $\phi_i\:L'\to X_i$ commuting with the bonding maps $f_i$ factors through 
a unique map $\phi\:L'\to L$ (so that each $\phi_i=f^\infty_i\phi$).

\subsection{Disjoint union}\label{disjoint union}
The {\it disjoint union} $\bigsqcup X_\lambda$ of uniform spaces $X_\lambda$ is their set-theoretic disjoint 
union $X$ endowed with the final pre-uniformity with respect to the injections
$\iota_\lambda\:X_\lambda\to X$ (cf.\ \cite[10.67(2)]{AHS}).
This pre-uniformity is obviously a uniformity (cf.\ \cite[Exer.\ I.7(i)]{I3}).
A cover $C$ of $\bigsqcup X_\lambda$ is uniform if and only if $\iota_\lambda^{-1}(C)$ is uniform for each 
$\lambda$ (indeed, every cover $C$ satisfying the latter condition is barycentrically refined by another 
such cover $\bigcup\iota_\lambda(C_\lambda)$, where each $C_\lambda$ barycentrically refines 
$\iota_\lambda^{-1}(C)$).
It is easy to check that $(X,\iota_\lambda)$ is the coproduct of $X_i$'s in
$\U$ in the sense of abstract category theory (see \cite[p.\ 14]{I3}), and that
its underlying topology is the topology of disjoint union (see \cite[II.8]{I3}).

Let us note that an infinite disjoint union of metrizable uniform spaces normally fails to be metrizable.
Finite disjoint union preserves metrizability.
Indeed, if $X$ and $Y$ are metric spaces of diameter $\le 1$, we can extend their metrics to a metric 
on the set-theoretic disjoint union of $X$ and $Y$ by $d(x,y)=1$ whenever $x\in X$, $y\in Y$; clearly, 
it induces the uniformity of the disjoint union.

A metrizable replacement of the countable disjoint union $\bigsqcup X_i$ is the inverse limit of 
the finite disjoint unions $Y_i\bydef (X_1\sqcup\dots\sqcup X_i)\sqcup\N$ and the maps $f_i\:Y_{i+1}\to Y_i$
defined by  $f_i|_{X_j}=\id$ for $j\le i$, $f_i(X_{i+1})=0$ and $f_i(i)=i-1$ for $i>0$.

When each $X_\lambda=(S_\lambda,u)$, $\lambda\in\Lambda$, is a subspace of a uniform space $B$, 
the {\it disjunction} $\amalg^B_\lambda X_\lambda$ is obtained by endowing the set-theoretic disjoint union 
$\bigsqcup S_\lambda$ with the uniformity that includes the uniform cover 
$\{U\cap S_\lambda\mid\lambda\in\Lambda,\, U\in C\}$ for each uniform cover $C$ of $B$.
Clearly, if $B$ is metrizable, then so is $\amalg^B_\lambda X_\lambda$.

\subsection{Sequential embedded direct limit}\label{dirlimits}
Let $X_1\subset X_2\subset\dots$ be uniform embeddings between metrizable uniform spaces. 
Their {\it direct limit} $X$ is their set-theoretic union $S$ endowed with the final pre-uniformity 
with respect to the injections $\iota_i\:X_i\to X$.

Suppose that each $d_i$ is a metric on $X_i$ such that 
$d_i(x,y)\le d_{i+1}(x,y)$ for all $x,y\in X_i$.
Given $x,y\in S$, let $\delta(x,y)=d_n(x,y)$, where
$n=\max\{i\mid x,y\in X_i\}$, and set
$$\delta_n(x,y)=
\inf_{x=x_0,\dots,x_n=y}\sum_{k=0}^{n-1}\delta(x_k,x_{k+1}).$$
Let us define a metric $d=\lim d_i$ on $S$ by $d(x,y)=\inf_n\delta_n(x,y)$.
Clearly, $\id\:X\to (S,d)$ is uniformly continuous for any sequence of metrics $(d_i)$ such that each 
$d_i\le d_{i+1}|_{X_i}$.
Let $\Delta$ be the set of all such sequences of metrics.
It follows from \cite[Theorem 1.4]{BR1} that the uniformity of $X$ is given by the family of metrics 
$\{\lim d_i\mid (d_i)\in\Delta\}$.

The underlying topology of $X$ coincides with the direct limit of
the underlying topologies of the $X_i$'s if each $X_i$ is locally 
compact \cite[Proposition 5.4]{BR1}, but can be strictly finer 
in general (see \cite{BR1}, \cite{BR2}).


\subsection{Embedding}
A uniformly continuous map $f\:A\to X$ of uniform spaces is called a {\it (uniform) embedding} if it is 
injective, and the uniformity on $A$ is initial with respect to $f$ (cf.\ \cite[8.6]{AHS}).
Thus if $f\:A\to X$ is an embedding, a basis of the uniformity of $A$ is given by the covers $f^{-1}(C)$, 
where $C$ runs over all uniform covers of $X$.
In fact, all uniform covers of $A$ are of this form; indeed, if $D$ is refined by $f^{-1}(C)$, then the cover
$E\bydef \{U\cup f(V)\mid V\in D,\, U\in C,\, f^{-1}(U)\incl V\}$ is refined by $C$ and satisfies $f^{-1}(E)=D$ 
since $f^{-1}(f(V))=V$ due to the injectivity of $f$.
Thus an injective map between uniform spaces is an embedding if and only if it is a uniform homeomorphism 
onto its image with the subspace uniformity.

Composition of embeddings is an embedding; and if the composition $X\xr{f} Y\to Z$ is an embedding, 
then so is $f$ \cite[8.9]{AHS}.

\subsection{Quotient} \label{quotient}
A uniformly continuous map $f\:X\to Q$ of pre-uniform spaces is called a {\it (uniform) quotient map} if it is 
surjective, and the pre-uniformity on $Q$ is final with respect to $f$ (cf.\ \cite[8.10]{AHS}).
This {\it quotient pre-uniformity} therefore consists of all covers $C_1$ of $Y$ such that $C_1$ admits 
a barycentric refinement $C_2$, which in turn admits a barycentric refinement $C_3$, etc., so that 
$f^{-1}(C_i)$ is uniform for each $i$.
Given a pre-uniform space $X$ and an equivalence relation $R$ on the underlying set of $X$, the {\it quotient} 
$X/R$ is the set of equivalence classes of $R$ endowed with the quotient pre-uniformity.

\begin{lemma} \label{qt vs tqu}
If $f\:X\to Y$ is a quotient map between uniform spaces, then every open $U\subset Y$ 
(in the topology of the quotient uniformity) is also open in the quotient topology.
\end{lemma}

\begin{proof} We need to show that $f^{-1}(U)$ is a neighborhood of each $x\in f^{-1}(U)$.
Let $y=f(x)$.
Since $U$ is a neighborhood of $y$ in $Y$, there exists a uniform cover $C$ of $Y$ such that $\st(y,C)\subset U$.
Then, in particular, $D\bydef f^{-1}(C)$ is a uniform cover of $X$.
Since $f^{-1}(U)$ contains $\st(x,D)=f^{-1}\big(\st(y,C)\big)$, it is a neighborhood of $x$.
\end{proof}

Let us note that the uniform space associated to a pre-uniform space (see \S\ref{separatization} above) is 
its quotient.
Moreover, the topology of the quotient uniformity coincides with the quotient topology in this case 
(see Remark \ref{separatization-remark}).

On the other hand, if $X$ is a uniform space, then the quotient pre-uniformity need not be a uniformity 
in general.
For instance, if $f$ has a non-closed point-inverse, then no uniformity on $Q$ can make $f$ uniformly continuous.
See \cite[Theorem 2.2]{GI} for a characterization of quotients of uniform spaces whose pre-uniformity is 
a uniformity.

Composition of quotient maps is a quotient map; and if a composition $X\to Y\xr{f} Z$ is a quotient map, then 
so is $f$ \cite[8.13]{AHS}.
Every uniformly continuous retraction is a quotient map \cite[8.12(2)]{AHS}.

A quotient of a complete metrizable uniform space need not be complete, as shown by the projection of 
$\{(x,y)\in\R^2\mid xy>0\}$ onto the $x$ axis, which is a quotient map \cite[Exer.\ II.6]{I3}.

\subsection{Graph}\label{graph}
If $f\:X\to Y$ is a possibly discontinuous map between uniform spaces, its
{\it graph} $\Gamma_f$ is the subspace $\{(x,f(x))\mid x\in X\}$ of $X\x Y$.
If $f$ is (uniformly) continuous, $\Gamma_f$ is (uniformly) homeomorphic to $X$,
via the composition $\Gamma_f\to X\x Y\to X$ of the inclusion and the projection,
whose inverse is given by $X\xr{\id_X\x f} X\x Y$.

In particular, every continuous map $f\:X\to Y$ between uniform spaces is
the composition of the homeomorphism $X\to\Gamma_f$ and the uniformly continuous
map $\Gamma_f\to X\x Y\to X$.
Similarly, every uniformly continuous map between metric spaces is
a composition of a uniform homeomorphism and a $1$-Lipschitz map.

\section{Commutativity of operations}

This section will be used rarely and may be safely omitted on the first reading.

\subsection{Mono- and epimorphisms}
{\it Monomorphisms} in $\U$ coincide with injective $\U$-morphisms; in other words, a $\U$-morphism $f\:X\to Y$ 
is injective if and only if $\U$-morphisms $g,h\:Z\to X$ are equal whenever their compositions with $f$ are equal 
\cite[II.4]{I3}.
{\it Epimorphisms} in $\U$ coincide with $\U$-morphisms with dense image; in other words, 
a $\U$-morphism $f\:X\to Y$ has dense image if and only if $\U$-morphisms $g,h\:Y\to Z$ are equal whenever 
their pre-compositions with $f$ are equal (cf.\ \cite[p.\ 15 and I.13]{I3}).

Monomorphisms in $\bar\U$ again coincide with injective $\U$-morphisms (cf.\ \cite[7.38]{AHS}); whereas 
epimorphisms in $\bar\U$ coincide with surjective $\U$-morphisms (cf.\ \cite[7.45 for $\supset$; 21.13(1) 
and 21.8(1) for $\subset$]{AHS}).

More generally, a family of uniformly continuous maps $f_\lambda\:X\to Y_\lambda$ is a {\it mono-source} in 
$\bar\U$ if and only if it is point-separating (see \S\ref{initial} concerning point-separating families); 
the former means that two maps $g,h\:Z\to X$ are equal whenever the compositions 
$Z\xr{g} X\xr{f_\lambda} Y_\lambda$ and $Z\xr{h} X\xr{f_\lambda} Y_\lambda$ are equal for each $\lambda$ 
\cite[10.8]{AHS}.
Dually, a family of uniformly continuous maps $f_\lambda\:Y_\lambda\to X$ is an {\it epi-sink} in $\bar\U$ if
and only if it is {\it jointly surjective} (i.e.\ $\bigcup_\lambda f_\lambda(Y_\lambda)=X$); the former means 
that two maps $g,h\:X\to Z$ are equal whenever the compositions $Y_\lambda\xr{f_\lambda} X\xr{g} Z$ and 
$Y_\lambda\xr{f_\lambda} X\xr{h} Z$ are equal for each $\lambda$ (cf.\ \cite[10.64 for `if'; 21.13(1) and 21.8(1)
for `only if']{AHS}).

The above remarks along with \cite[21.14 and 21.8(1)]{AHS} as well as with \cite[15.5(1)]{AHS} and it dual yield

\begin{proposition}\label{factorization}
(a) Given a family of uniformly continuous maps $f_\lambda\:X\to Y_\lambda$ between pre-uniform spaces, there 
exist a uniformly continuous surjection $h\:X\to Z$ and a point-separating family $G$ of uniformly continuous maps
$g_\lambda\:Z\to Y_\lambda$, where $Z$ has initial uniformity with respect to $G$, such that each 
$f_\lambda=g_\lambda h$.

(b) Given a family of uniformly continuous maps $f_\lambda\:Y_\lambda\to X$ between pre-uniform spaces, there 
exist a uniformly continuous injection $h\:Z\to X$ and a jointly surjective family $G$ of uniformly continuous 
maps $g_\lambda\:Y_\lambda\to Z$, where $Z$ has final uniformity with respect to $G$, such that each 
$f_\lambda=hg_\lambda$.

(c) The factorizations in (a) and (b) are unique up to uniform homeomorphism.
\end{proposition}

\begin{corollary} \cite[II.5]{I3}
Every uniformly continuous map between pre-uniform spaces is a composition of

(a) a uniformly continuous surjection and an embedding;

(b) a quotient map and a uniformly continuous injection;

(c) a quotient map, a uniformly continuous bijection and an embedding.

\noindent
Each of the three factorizations is unique up to uniform homeomorphism.
\end{corollary}

\subsection{Extremal and regular monomorphisms}
Embeddings coincide with extremal monomorphisms in $\bar\U$ and also with regular monomorphisms in $\bar\U$ 
\cite[21.13(4) and 21.8(1)]{AHS}.
A $\bar\U$-monomorphism (i.e.\ a uniformly continuous injection) $f\:A\to X$ is {\it extremal} in $\bar\U$ if, 
once $f$ factors in $\bar\U$ through a $\bar\U$-epimorphism (i.e.\ a uniformly continuous surjection) $g\:A\to B$,
this $g$ must be a uniform homeomorphism.
A uniformly continuous map $f\:A\to X$ is a {\it regular monomorphism} in $\bar\U$ if there exist 
$\bar\U$-morphisms $g,h\:X\to Y$ such that $f$ is their {\it equalizer}; that is, $gf=hf$, and any 
$\bar\U$-morphism $f'\:B\to X$ satisfying $gf'=hf'$ uniquely factors through $f$ in $\bar\U$.

It is easy to see that extremal monomorphisms in $\U$ coincide with embeddings onto closed subspaces.
To see that regular monomorphisms in $\U$ coincide with embeddings onto closed subspaces, note that if $A$ is 
a closed subspace of a uniform space $X$, then for each $x\in X\but A$ there exists a uniformly continuous map
$g_x\:X\to [0,\infty)$ such that $g_x(A)=\{0\}$ and $g_x(x)\ne 0$ \cite[I.13]{I3}; consequently, the uniformly 
continuous map $\prod g_x\:X\to\prod_{X\but A} [0,\infty)$ satisfies $\big(\prod g_x\big)^{-1}(0)=A$.

It follows that the pullback of an embedding is an embedding \cite[11.18]{AHS}.

\subsection{Extremal and regular epimorphisms}
Quotient maps coincide with extremal epimorphisms in $\bar\U$ and also with regular epimorphisms in $\bar\U$ 
\cite[21.13(5) and 21.8(1)]{AHS}.
A $\bar\U$-epimorphism (i.e.\ a uniformly continuous surjection) $f\:X\to Q$ is {\it extremal} in $\bar\U$ 
if, once $f$ factors in $\bar\U$ through a $\bar\U$-monomorphism (i.e.\ an uniformly continuous injection) 
$p\:R\to Q$, this $p$ must be a uniform homeomorphism.
A uniformly continuous map $f\:X\to Q$ is a {\it regular epimorphism} in $\bar\U$ if there exist 
$\bar\U$-morphisms $g,h\:Y\to X$ such that $f$ is their {\it coequalizer}; that is, $fg=fh$, and any 
$\bar\U$-morphism $f'\:B\to X$ satisfying $f'g=f'h$ uniquely factors through $f$ in $\bar\U$.

Extremal epimorphisms in $\U$ coincide again with quotient maps, since they coincide with extremal epimorphisms 
in $\bar\U$ as long as they are surjective --- which they have be due to the second condition in the definition 
of an extremal epimorphism.
To see that regular epimorphisms in $\U$ coincide with quotient maps, note that a quotient map $q\:X\to Q$ is 
the coequalizer of the projections of the subspace $\{(x,y)\mid q(x)=q(y)\}$ of $X\x X$ onto the factors; and
the coequalizer of a pair of maps from a pre-uniform space $Y$ into $X$ equals the coequalizer of the resulting 
maps from the uniform space associated to $X$.

It follows that the pushout of a quotient map is a quotient map (dually to \cite[11.18]{AHS}).

\subsection{Extremal mono-sources and epi-sinks}
The uniqueness part of Proposition \ref{factorization} implies the following.
A family $F$ of $\bar\U$-morphisms $f_\lambda\:X\to Y_\lambda$ is an {\it extremal} mono-source in $\bar\U$ if
and only if it is point-separating and the uniformity of $X$ is initial with respect to $F$.
The extremality means that once $F$ factors in $\bar\U$ through a $\bar\U$-epimorphism (i.e.\ a uniformly 
continuous surjection) $g\:X\to Z$, this $g$ must be a uniform homeomorphism.
Dually, a family $F$ of $\bar\U$-morphisms $f_\lambda\:Y_\lambda\to X$ is an {\it extremal} epi-sink in $\bar\U$ 
if and only if it is jointly surjective and the uniformity of $X$ is final with respect to $F$.
Here the extremality means that once $F$ factors in $\bar\U$ through a $\bar\U$-monomorphism (i.e.\ an uniformly 
continuous injection) $p\:R\to Q$, this $p$ must be a uniform homeomorphism.

Combining the above with \cite[10.26(2)]{AHS} and its dual, we obtain

\begin{corollary}
(a) Let $F$ be a family of $\bar\U$-morphisms $f_\lambda\:X\to Y_\lambda$.
Then the map $\prod F\:X\to\prod_\lambda Y_\lambda$ is an embedding if and only if $F$ is point-separating and 
the uniformity of $X$ is initial with respect to $f$.

(b) Let $F$ be a family of $\bar\U$-morphisms $f_\lambda\:Y_\lambda\to X$.
The map $\bigsqcup F\:\bigsqcup_\lambda Y_\lambda\to X$ is a quotient map if and only if $F$ is jointly surjective 
and the uniformity of $X$ is final with respect to $F$.
\end{corollary}


From \cite[10.35(5)]{AHS} we deduce

\begin{proposition}\label{quotient-coproduct} Let $f_\lambda\:X_\lambda\to Y_\lambda$
be uniformly continuous maps between pre-uniform spaces.

(a) If each $f_\lambda$ is an embedding, then so is 
$\prod f_\lambda\:\prod_\lambda X_\lambda\to\prod_\lambda Y_\lambda$.

(b) If each $f_\lambda$ is a quotient map, then so is
$\bigsqcup f_\lambda\:\bigsqcup_\lambda X_\lambda\to\bigsqcup_\lambda Y_\lambda$.
\end{proposition}

From \cite[27.15]{AHS} (see also \cite[28.14]{AHS}) and its dual we deduce

\begin{proposition}\label{quotient-product}
(b) If $g\:A\to X$ is an embedding and $Y$ is a uniform space, then $g\sqcup\id_Y\:A\sqcup Y\to X\sqcup Y$ 
is an embedding.
In particular, $\iota_2\:Y\to X\sqcup Y$ is an embedding.

(b) If $q\:X\to Q$ is a quotient map and $Y$ is a uniform space, then $q\x\id_Y\:X\x Y\to Q\x Y$ is a quotient map.
In particular, $\pi_2\:X\x Y\to Y$ is a quotient map.
\end{proposition}

In fact, a product of quotient maps is a quotient map \cite[Theorem 1]{HR} (for the case of a finite product 
see also \cite[Exer.\ III.8(c)]{I3}).
A product of two sequential direct limits is the sequential direct limit of the products \cite{BR1} (which is not 
the case for topological spaces).

For topological spaces, a topological quotient map multiplied by $\id_X$ is a quotient map as long as $X$
is locally compact \cite{En}*{3.3.17}.

\section{Function spaces}\label{function spaces}

\subsection{Definition}

If $X$ and $Y$ are uniform spaces, let $U(X,Y)$ be the set of uniformly continuous maps $X\to Y$.
Given a uniform cover $C$ of $Y$, let $U(X,C)$ be the cover of $U(X,Y)$ by the sets
$O_f\bydef \{g\mid \forall x\in X\ \exists U\in C\,f(x),g(x)\in U\}$ for all $f\in U(X,Y)$.
One can check that the covers $U(X,C)$ form a base of a uniformity on $U(X,Y)$, which therefore becomes 
a uniform space.

If $Y$ is metrizable, then so is $U(X,Y)$; specifically, if $d$ is a bounded metric on $Y$, then 
$d(f,g)=\sup_{x\in X}d(f(x),g(x))$ is a bounded metric on $U(X,Y)$.

If $Y$ is complete, then so is $U(X,Y)$ (cf.\ \cite[III.31]{I3}).
If $X$ is compact and $Y$ is separable metrizable, then $U(X,Y)$ is separable (see \cite[4.2.18]{En}).

Given subsets $A\subset X$ and $B\subset Y$, by $U\big((X,A),(Y,B)\big)$ we denote the subspace of $U(X,Y)$
consisting of maps than send $A$ into $B$.

\subsection{Uniform equicontinuity}
A subset $S\incl U(X,Y)$ is called {\it uniformly equicontinuous} (or equiuniformly continuous) if for each 
uniform cover $D$ of $Y$ there exists a uniform cover $C$ of $X$ such that for each $f\in S$, the cover 
$f^{-1}(D)$ is refined by $C$ (equivalently, $f(C)$ refines $D\cap f(X)$).

\begin{lemma} \label{UEB} \cite{Pac2}*{1.20}
(a) A set $S\subset U(X,[-1,1])$ is uniformly equicontinuous if and only if there exists 
a uniform pseudo-metric $d$ on $X$ and a $k>0$ such that all functions in $S$ are $k$-Lipschitz 
with respect to $d$.

(b) A set $S\subset U(X,\R)$ is uniformly equicontinuous and pointwise bounded if and only if there exists 
a uniform pseudo-metric $d$ on $X$ and a function $h\:X\to[0,\infty)$ that is $1$-Lipschitz with respect to $d$
such that all functions in $S$ are bounded above by $h$ and below by $-h$ and 
are $1$-Lipschitz with respect to $d$.
\end{lemma}

\begin{proof}[Proof. (a)] The ``if'' assertion is easy.
Conversely, let $d(x,y)=\sup_{f\in S} |f(x)-f(y)|$.
Then $d$ is a pseudo-metric and all $f\in S$ are 1-Lipschitz with respect to $d$.
Also, since $S$ is uniformly equicontinuous, $d$ is uniformly continuous.
\end{proof}

\begin{proof}[(b)] This is similar to (a), if we let $h(x)=\sup_{f\in S} |f(x)|$.
\end{proof}

Given a map $f\:Z\to U(X,Y)$, let $F\:Z\x X\to Y$ be defined by $F(z,x)=f(z)(x)$.
The following are equivalent \cite[III.21, III.22, III.26]{I3}:
\begin{itemize}
\item $f$ is uniformly continuous and its image is uniformly equicontinuous;
\item $F$ is uniformly continuous;
\item $F(z,*)\:X\to Y$ is a uniformly equicontinuous family and $F(*,x)\:Z\to Y$
is a uniformly equicontinuous family.
\end{itemize}

In order to characterize uniform continuity of $f$ in similar terms we need a new uniformity on $Z\x X$.

\subsection{Semi-uniform product}
The {\it semi-uniform product} $Z\ltimes X$ of a uniform space $Z$ and a metrizable uniform space $X$ is 
the set $Z\x X$ endowed with the uniformity with basis consisting of covers $\{U_i\x V^i_j\}$, where 
$\{U_i\}$ is a uniform cover of $Z$ and for each $i$, $\{V^i_j\}$ is a uniform cover of $X$
(see \cite[2.2 and 2.5]{I2}, \cite[III.23 and III.28]{I3}).
Let us note that $Z\ltimes X$ does not have to be metrizable if $Z$ and $X$ are.
$Z\ltimes X$ is uniformly homeomorphic to $Z\x X$ when $Z$ is compact or $X$ is discrete \cite[III.24]{I3} 
and is always homeomorphic to $Z\x X$ \cite[III.22]{I3}.

Given a map $f\:Z\to U(X,Y)$, let $\Phi\:Z\ltimes X\to Y$ be defined by $\Phi(z,x)=f(z)(x)$.
The following are equivalent \cite[III.22, III.26]{I3}:
\begin{itemize}
\item $f$ is uniformly continuous;
\item $\Phi$ is uniformly continuous;
\item $\Phi(z,*)\:X\to Y$ is uniformly continuous for each $z\in Z$,
and $\Phi(*,x)\:Z\to Y$ is a uniformly equicontinuous family.
\end{itemize}

\begin{remark}
$\Phi(z,*)\:X\to Y$ is uniformly continuous for each $z\in Z$ and $\Phi(*,x)\:Z\to Y$ is uniformly continuous
for each $x\in X$ if and only if $f\:Z\to U_{pw}(X,Y)$ is uniformly continuous, where $U_{pw}$ is the space 
of uniformly continuous maps with the uniformity of pointwise convergence. 
The corresponding uniformity $X\bowtie Y$ on the product $X\x Y$ need not be pre-compact even if both $X$ and $Y$ 
are compact \cite[Exer.\ III.7]{I3}.
Another symmetric version $\sup(u_{XY},u_{YX})$ of the uniformity $u_{XY}$ of $X\ltimes Y$ is studied 
in \cite{BRZ'}.
\end{remark}

\subsection{Injectivity of function spaces}
We call a uniform space $Y$ an {\it A[N]EU} if for every (not necessarily metrizable) uniform space $X$ and 
every (not necessarily closed) subspace $A\subset X$, every uniformly continuous map $A\to X$ extends to 
a uniformly continuous map $X\to Y$ [respectively, $N\to Y$, where $N$ is a uniform neighborhood of $A$ in $X$].
AEUs are also known as ``injective spaces'' \cite{I3} since they are the injective objects of the category 
of uniform spaces and uniformly continuous maps.

By Katetov's theorem (see \cite[III.9]{I3} or \cite{Gan} or \cite{It} or \cite{At3}; see also \cite{Hus},
\cite{Vi2}, \cite{Nhu1}), the unit interval $I$ is an AEU.
Isbell's finite-dimensional uniform polyhedra are known to be ANEUs \cite[1.9]{I1}, \cite[V.15]{I3}.
(Completions of infinite-dimensional cubohedra are shown to be ANEUs in Theorem \ref{cubohedron} below, 
and completions of general infinite-dimensional uniform polyhedra are treated in a sequel to this paper
\cite{M3}.)

Every ANEU is complete, using, {\it inter alia}, that the only uniform neighborhood
of $X$ in its completion $\bar X$ is the entire $\bar X$ (see \cite[V.14]{I3},
\cite[I.7]{I1}, \cite[III.8]{I3}).

\begin{theorem}\label{basic ARU} \cite[III.14]{I3} If $D$ is a discrete uniform space and $Y$ is an A[N]EU, 
then $U(D,Y)$ is an A[N]EU.

In particular, $U(D,I)$ is an AEU.
\end{theorem}

\begin{proof} If $Z$ a uniform space, it is easy to see (cf.\ \cite[III.13]{I3})
that a map $f\:Z\to U(D,Y)$ is uniformly continuous if and only if $F\:Z\x D\to Y$,
defined by $F(z,m)=f(z)(m)$, is uniformly continuous.
But since $X$ is an A[N]EU, every uniformly continuous map $A\x D\to Y$
extends to a uniformly continuous map $Z\x D\to Y$ [respectively, $N\x D\to Y$, where
$N$ is a uniform neighborhood of $A$ in $Z$], for every $A\subset Z$.
\end{proof}

\begin{remark} It is quite easy to show that $U(D,I)$ with the $l_\infty$ metric is an injective metric space, 
that is, an absolute extensor in the concrete category of metric spaces and 1-Lipschitz maps.
Indeed, $I$ is an injective metric space by using an explicit formula (Lemma \ref{lip-ext} below).
Using the metric $d(a,b)=1$ whenever $a\ne b$ on $D$ and the $l_\infty$ product metric 
$d\big((x,a),(y,b)\big)=\max\big(d(x,y),d(a,b)\big)$ on $Z\x D$, we also have that $f\:Z\to U(D,Y)$ is 
1-Lipschitz if and only if $F\:Z\x D\to Y$, defined by $F(z,m)=f(z)(m)$, is 1-Lipschitz.

By using the graph of a given map (see \S\ref{graph}) and a lemma on extension of metrics (Lemma \ref{A.M} below)
it is easy to see that every bounded injective metric space is an AEU.
(Beware that $\R$ is an injective metric space by Lemma \ref{lip-ext}, but not an AEU.)
However, our proof of Lemma \ref{A.M} depends on Theorem \ref{basic ARU}. 
\end{remark}

We note that $U(\N,I)$ is inseparable, where $\N$ denotes the infinite countable
discrete space and $I=[0,1]$.
On the other hand, $q_0\bydef U((\N^+,\infty),(I,0))$ is separable, where $+$
stands for the one-point compactification.
(We recall that the functional space $U((\N^+,\infty),(\R,0))$ is known as $c_0$.)

\begin{corollary}\label{q_0 ARU} $q_0$ is an AEU.
\end{corollary}

\begin{proof} (Compare \cite[Example 1.5]{BL}.)
We define $r\:U(\N,I)\to q_0$ by
$(r(x))_n=0$, if $x_n<d(x,q_0)$, and $(r(x))_n=x_n-d(x,q_0)$ if $x_n\ge d(x,q_0)$,
where $d(x,q_0)=\limsup x_n$.
Clearly, $r$ is a uniformly continuous retraction.
Since $U(\N,I)$ is an AEU, we infer that so is $q_0$.
\end{proof}

\begin{theorem}\label{basic ARU'} \cite[2.6]{I2}, \cite[III.25]{I3}
If $X$ is a metrizable uniform space and $Y$ is an A[N]EU, then $U(X,Y)$ is
an A[N]EU.
\end{theorem}

The proof involves non-metrizable spaces when $Y$ is metrizable.
A rather technical proof, not involving non-metrizable spaces, of the case $Y=I$ is found in \cite[III.18]{I3}.

\begin{proof} If $X$ is an AEU, every uniformly continuous map
$A\ltimes X\to Y$ extends to a uniformly continuous map $Z\ltimes X\to Y$,
for every $A\subset Z$.

The ANEU case is similar, using additionally that every uniform neighborhood
of $C\ltimes X$ in $Z\ltimes X$ contains $U\ltimes X$ for some uniform
neighborhood $U$ of $C$ in $Z$.
\end{proof}

\section{Hyperspaces}

\subsection{Hausdorff metric}
Suppose that $M$ is a metric space.
Let $K(M)$ be the set of all nonempty compact subsets of $M$, endowed with 
the following {\it Hausdorff metric}
$d_H(A,B)=\max\big\{\sup\{d(a,B)\mid a\in A\},\,\sup\{d(A,b)\mid b\in B\}\big\}$.
Let $H(M)$ be the set of all nonempty closed subsets of $M$, endowed with the metric $d'_H=\min(d_H,1)$.
If $X$ is a metrizable uniform space, this yields well-defined metrizable uniform 
spaces $K(X)$ and $H(X)$ of nonempty compact/closed subsets of $X$ --- the former of course
being a subspace of the latter (see \cite{Mi2} or \cite{I3}).
One can show that $d_H(A,B)$ equals the supremum of $|d(x,A)-d(x,B)|$ over all $x\in M$ (see \cite{Be}); 
thus there is an isometric embedding $H(M)\to U(X,[0,1])$, given by 
$F\mapsto\Big(x\mapsto\min\big(d(x,F),1\big)\Big)$.
The topology of $K(X)$ depends only on the topology of $X$; indeed, it coincides with 
the Vietoris topology, whose basic open sets are indexed by all finite collections $U_1,\dots,U_n$ of
open subsets of $Y$ and consist of all (compact) subsets of $X$ contained in $U_1\cup\dots\cup U_n$
and containing at least one point out of each $U_i$ (see \cite{Mi2}).
A subbase for the Vietoris topology is given by the families $U^-$ of all (compact) subsets of $X$ that
intersect a given open subset $U$ of $X$, and by the families $V^+$ of all (compact) subsets of $X$ that 
are contained in a given open subset $V$ of $X$ (see \cite{Be}).

If the metrizable uniform space $X$ is complete, compact or connected, then so are $K(X)$ and $H(X)$ 
(see \cite{Mi2} or \cite[II.48, II.49]{I3}).
If $X$ is separable, locally compact or locally connected, then so is $K(X)$ (see \cite{Mi2}).
In fact, if $X$ is locally path connected (and connected), then $K(X)$ is an ANR (resp.\ AR); the converse
is true for completely metrizable spaces \cite{Cu} (see also \cite{An}).
The spaces $K(X)$ and $H(X)$ are closed with respect to union indexed by a compact subset of $K(X)$ \cite{Mi2}.
$K(*)$ commutes with inverse limits \cite{Ho4}.

\subsection{Other topologies} 
Unless $X$ is compact, the topology of the hyperspace $H(X)$ is generally neither separable 
nor determined by the topology of $X$ (see \cite{Mi2}, \cite{Be}) --- in contrast to $K(X)$.
For locally compact $X$, this is remedied by the Attouch--Wetts hyperspace $H_b(X)$ of nonempty closed subsets 
of $X$.
The topology on $H_b(M)$, where $M$ is a metric space, is induced via the injective map 
$F\mapsto\big(x\mapsto d(x,F)\big)$ into the space $U_b\big(M,[0,\infty)\big)$ of maps that are uniformly 
continuous on bounded (in the metric of $M$) sets with the topology of uniform convergence on bounded sets 
--- the inverse limit of $U(M_b,[0,\infty))$ over all bounded subsets $M_b$ of $M$.
This topology is metrizable (by the inverse limit metric), completely metrizable if $M$ is complete, 
and separable as long as all bounded subsets of $M$ are precompact (see \cite{Be}).
Homeomorphic metric spaces $M$, $M'$ yield homeomorphic hyperspaces $H_b(M)$, $H_b(M')$ as long as
the given homeomorphism and its inverse send bounded sets to bounded sets and are uniformly continuous
on bounded sets (see \cite{Be}) --- which is the case if all bounded sets are precompact in both $M$ and $M'$.
Moreover, if all bounded subsets of $M$ are precompact, then the topology of $H_b(M)$ coincides 
\begin{itemize}
\item with the Wijsman topology, which is induced via the injective map $F\mapsto\big(x\mapsto d(x,F)\big)$ 
into the product $[0,\infty)^M$, that is, the space of all (possibly discontinuous) maps with the topology of 
the pointwise convergence (see \cite{Be}*{3.1.4}); 
\item with the Fell topology, whose subbase is given by the families $U^-$ of all closed subsets of $X$ that 
intersect a given open subset $U$ of $X$, and by the families $V^+$ of all nonempty closed subsets of $X$ 
that are disjoint from a given compact subset $X\but V$ of $X$ (see \cite{Be}*{5.1.10}); 
\item with the initial topology with respect to the family of functionals $\phi_K(F)=\sup_{x\in K}d(x,F)$ and 
$\psi_K(F)=\inf_{x\in K}d(x,F)$ associated to all nonempty compact subsets $K$ of $X$ (see \cite{Be}*{4.2.4}).
\end{itemize}

In general, the Fell topology is locally compact, and becomes compact upon adjoining the empty subset of $X$ 
(see \cite{Be}).

In general, the Wijsman hyperspace $H_W(M)$ of the separable metric space $M$ is separable and metrizable
(by the metric $\sum_{i=1}^\infty 2^{-i}\min\big(1,|d(x_i,A)-d(x_i,B)|\big)$, where $\{x_i\mid i\in\N\}$ is
a countable dense subset of $M$), and completely metrizable (by a different metric) if $M$ is complete
(see \cite{Be}).
However, uniformly homeomorphic metric spaces $M$ and $M'$ generally yield non-uniformly homeomorphic
$H_W(M)$ and $H_W(M')$ (see \cite{Be}).

In general, the finite Hausdorff topology is the initial topology with respect to the family of 
functionals $\phi'_K(F)=\sup_{x\in F}d(x,K)$ and $\psi_K(F)=\inf_{x\in F}d(x,K)$ associated to all
nonempty compact subsets $K$ of $X$.
The finite Hausdorff hyperspace $H_f(M)$ is Polish if the metric space $M$ is separable and complete 
\cite{HL}, and its topology is induced from the topology of pointwise convergence (or, equivalently, 
from the topology of uniform convergence on compact sets) via the injective map 
$H(M)\to U\big(K(M),[0,\infty)\x[0,\infty]\big)$, given by 
$F\mapsto\Big(K\mapsto\big(\phi'_K(F),\psi_K(F)\big)\Big)$ \cite{HL}.

\section{Finiteness conditions for covers}\label{finiteness}

\subsection{Definitions} A cover $C$ of a set $S$ is said to be of {\it multiplicity} $\le m$ if 
every $x\in S$ belongs to at most $m$ elements of $C$.
We say that a uniform space $X$ is of {\it uniform dimension} $\le d$ if every uniform cover of $X$ 
is refined by a uniform cover of multiplicity $\le d+1$.
(Isbell calls this the ``large dimension'' of $X$; however, when finite, it coincides with what 
he calls the ``uniform dimension'' \cite[Theorem V.5]{I3}.)
We call $X$ {\it uniformly finitistic}%
\footnote{Uniformly finitistic uniform spaces are also known as ``distal'' or simply ``finitistic'' 
in some of the literature.
Finitistic topological spaces (which are defined similarly, using open cover in place of uniform covers) 
are also known as ``boundedly metacompact'' in some of the literature.}
if every uniform cover of $X$ is refined by a uniform cover of finite multiplicity.

Let us recall that a topological space is called {\it (strongly) countably dimensional} if it is a union 
of countably many of its finite-dimensional (closed) subspaces.
We call a uniform space $X$ {\it uniformly countably dimensional} if it is a union of countably many of 
its subsets $X_i$ satisfying the following property: every uniform cover of $X$ is refined by a uniform 
cover $C$ of $X$ such that for each $n$, the cover $C_n\bydef \{U\in C\mid U\cap X_n\ne\emptyset\}$ of 
$\st(X_n,C)$ is of multiplicity $\le n+1$.
It may be assumed without loss of generality that $X_0\subset X_2\subset\dots$ (using that 
$C_0\cup\dots\cup C_k$ is a cover of $X_0\cup\dots\cup X_k$ of multiplicity $\le 1+2+\dots+(k+1)$).
It may also be assumed without loss of generality that each $X_i$ is closed (using that 
$\{\Int U\mid U\in C\}$ is a uniform cover of $X$ by Lemma \ref{shrinking}). 
If the $X_i$ may be chosen so that each $X_i\subset\Int X_{i+1}$, then we call the uniform space $X$ 
{\it uniformly locally finite-dimensional}.
Let us recall that a topological space $X$ is called {\it locally finite-dimensional} if each $x\in X$ 
has a finite-dimensional neighborhood. 
It is not hard to see, using \cite{Sak2}*{5.4.4}, that a metrizable space $X$ is locally 
finite dimensional if and only if $X=\bigcup_{n=0}^\infty X_n$, where each $X_n$ is a closed 
finite dimensional subset of $X$ such that $X_n\subset\Int X_{n+1}$.

A cover $\{U_\alpha\}$ of a set $X$ is called {\it point-finite} if each
$x\in X$ belongs to only finitely many $U_\alpha$'s.
Next, $\{U_\alpha\}$ is called {\it star-finite} if each $U_\beta$
meets only finitely many $U_\alpha$'s.
Following \cite[\S7]{Ho2}, we call $\{U_\alpha\}$ {\it Noetherian} if there
exists no infinite sequence $U_{\alpha_1},U_{\alpha_2},\dots$ such that
$U_{\alpha_1}\cap\dots\cap U_{\alpha_n}$ is nonempty for each $n\in\N$.
The following implications are straightforward.

\medskip
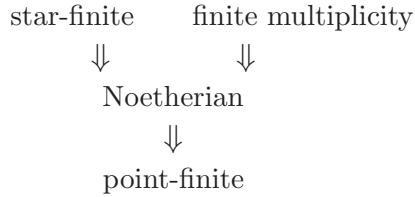
\centerline{\small finite}
\centerline{$\Downarrow$\qquad$\Downarrow$}
\centerline{\small \ \ \ \qquad star-finite\ \ \quad finite multiplicity}
\centerline{$\Downarrow$\qquad\qquad$\Downarrow$}
\centerline{\small Noetherian}
\centerline{$\Downarrow$}
\centerline{\small point-finite}
\medskip

A uniform space $X$ is called {\it point-finite (star-finite; Noetherian)}
if every uniform cover of $X$ has a point-finite (star-finite; Noetherian)
uniform refinement.
We caution the reader that in the literature, ``uniform paracompactness''
and its variations refer to a completely different type of properties,
involving arbitrary open coverings besides uniform coverings.

\begin{remark}
As for topological spaces, it is well-known that every metrizable space is paracompact and so 
in particular weakly paracompact (=metacompact), i.e.\ every its open cover admits an open point-finite
refinement (see \cite{Sak2}*{2.3.1}, \cite{BP}*{Theorem II.2.3}, \cite{En}*{4.4.1}); and that every separable
metrizable space is strongly paracompact (=hypocompact), i.e.\ every its open cover admits 
an open star-finite refinement (see \cite{En}*{5.2.4 (cf.\ 5.3.11 and 4.1.16)}, \cite{BP}*{Theorem II.2.1}).
In fact, every open cover of every paracompact space admits an open refinement with a locally
finite-dimensional nerve \cite{Sak2}*{4.9.9}; in particular, the refinement is Noetherian
(compare \cite[\S7]{Ho2}).
\end{remark}

\subsection{Basic observations}

A metrizable uniform space $X$ is precompact if and only if every uniform cover of $X$ has
a finite uniform refinement (see \cite[II.28]{I3}).
In particular, every compactum is uniformly finitistic.

It is not hard to see that every uniform cover of a separable uniform space has a countable 
uniform refinement (cf.\ \cite[2.3]{GI}).

\begin{lemma}\label{rfd-completion} The completion of every uniformly finitistic metric space 
is uniformly finitistic.
\end{lemma}

\begin{proof} This is similar to \cite[IV.23]{I3} but for the reader's
convenience we include a direct proof.
Let $C$ be a uniform cover of $X$ with a Lebesgue number $3\lambda$.
Then every subset of $X$ of diameter $\lambda$ has its $\lambda$-neighborhood
contained in some $U\in C$, and therefore is itself contained in
$U'\bydef \{x\in X\mid d(x,X\but U)>\lambda\}$.
Thus $D\bydef \{U'\mid U\in C\}$ is a uniform cover of $X$.
The cover $\bar D$ of $\bar X$ by the closures of the elements of $D$ is uniform
\cite[II.9]{I3}, and obviously its multiplicity does not exceed that of $C$.
\end{proof}

\begin{proposition}\label{finite dim vs star-finite} A uniform cover of every
uniformly $d$-dimensional (resp.\ of every uniformly finitistic) star-finite
uniform space has a star-finite uniform refinement of multiplicity $\le d+1$
(resp.\ of finite multiplicity).
\end{proposition}

For a deeper study of compatibility of properties of covers see \cite[\S1]{GI}.

\begin{proof} Let $C$ be the given uniform cover of the given space $X$.
Then $C$ has a star-finite uniform refinement $D$, which in turn has a uniform refinement $E$ of 
multiplicity $\le d+1$ (finite multiplicity).
Thus each $U\in E$ lies in some $V=f(U)\in D$.
Then $\{\bigcup f^{-1}(V)\mid V\in D\}$ is a cover of $X$ with the desired
properties.
\end{proof}

\subsection{Point-finite spaces}

It has been a long-standing open problem of Stone (1960), reiterated in \cite[Research Problem B$_3$]{I3}, 
whether every metrizable uniform space (or every uniform space indeed) is point-finite.
It has been resolved in the negative by Pelant \cite{Pe0} and Shchepin \cite{Sc} (independently).
More recently Hohti showed that the metrizable space $U(\N,I)$ fails to be point-finite \cite{Ho2}.

In contrast, the separable space $c_0\bydef U\big((\N^+,\infty),(\R,0)\big)$ is point-finite,
as observed in \cite[2.3]{Pe}; here $\N^+=\N\cup\{\infty\}$ is the one-point
compactification of the countable discrete uniform space $\N$.
Indeed, for each $\eps>0$, the covering of $c_0$ by all balls of radius
$\frac23\eps$ centered at points of the lattice $U\big((\N^+,\infty),(\eps\Z,0)\big)$ is
uniform (with Lebesgue number $\frac13\eps$) and point-finite.

\begin{theorem}[\cite{Vi}] \label{A.7} Every separable metrizable uniform space is point-finite.
\end{theorem}

\begin{proof} Let $D$ be the given uniform cover of the given space $X$,
and let $C=\{U_i\}$ be a countable star-refinement of a star-refinement of $D$.
Let us consider $W_n=\st(U_n,C)\but[\st(U_1,C)\cup\dots\cup\st(U_{n-1},C)]$
and let $V_n=\st(W_n,C)$.
Since $\{W_i\}$ is a cover of $X$, $\{V_i\}$ is a uniform cover of $X$.
Since $V_i\incl\st(\st(U_i,C),C)$, $\{V_i\}$ refines $D$.
Let us prove that each $U_i$ meets only finitely many $V_n$'s.
If $U_i\cap\st(W_n,C)\ne\emptyset$, then $\st(U_i,C)\cap W_n\ne\emptyset$,
hence $i\le n$ by the construction of $W_n$.
\end{proof}

Theorem \ref{A.7} also follows from Aharoni's theorem: every separable
metric space admits a Lipschitz (hence uniform) embedding in $c_0$
(see \cite[Theorem 3.1]{Pe}, see also \cite{BL}*{Theorem 7.11}, \cite{KaL}).
We find it easier, however, to prove Theorem \ref{A.7} directly and then use it to give
a short proof (following \cite[2.1]{Pe}) of the relevant part of Aharoni's theorem.

\begin{theorem}[Aharoni]\label{aharoni}
Each separable metrizable uniform space uniformly embeds into the function space
$q_0=U\big((\N^+,\infty),([0,1],0)\big)$, where $\N^+$ is the one-point compactification
of the countable discrete uniform space $\N$.
\end{theorem}

\begin{proof} Let $M$ be the given space with some fixed metric.
Let $C_n$ be the cover of $M$ by all closed balls of radius $2^{-n}$.
By Theorem \ref{A.7} it has a countable uniform point-finite refinement $D_n=\{V_{n1},V_{n2},\dots\}$.
Let us define $f\:M\x(\N\x\N)^+\to I$ by $f(x,n,i)=\min\big(d(x,\,M\but V_{ni}),2^{-n}\big)$ and $f(x,\infty)=0$.
(Let us note that although $V_{ni}$ is of diameter at most $2^{-n+1}$, this does not imply any bound on
$d(x,\,M\but V_{ni})$.)

Each $f(M\x\{n\}\x\N)\subset [0,2^{-n}]$.
On the other hand, since $D_n$ is point-finite, each restriction $f|_{\{x\}\x\{n\}\x\N}$ has support in 
$\{x\}\x\{n\}\x S_{xn}$ for some finite set $S_{xn}\subset\N$.
Let $T_{xn}=\{1,\dots,n\}\x(S_{x1}\cup\dots\cup S_{xn})$.
Then $f\big(\{x\}\x(\N\x\N\but T_{xn})\big)\subset [0,2^{-n-1}]$.
Hence $f(x,*)\:(\N\x\N)^+\to I$ is continuous, and therefore uniformly continuous, for each $x\in M$.
On the other hand, $f(*,y)\:M\to I$ is $1$-Lipschitz for each $y\in (\N\x\N)^+$, and therefore
these maps $f(*,y)$ form a uniformly equicontinuous family.
Thus $f$ determines a uniformly continuous map $F\:M\to U\Big(\big((\N\x\N)^+,\infty\big),(I,0)\Big)$.

If $d(x,y)>2^{-n+1}$, then no element of $D_n$ contains both $x$ and $y$.
Let $\lambda_n$ be a Lebesgue number of $D_n$.
Let $V_{ni}$ be an element of $D_n$ containing the closed $\frac12\lambda_n$-neighborhood of $x$.
Then $f(x,n,i)>\eps_n\bydef \min(\frac12\lambda_n,2^{-n})$ and $f(y,n,i)=0$.
Hence $d\big(F(x),F(y)\big)>\eps_n$.
So we have shown that $d\big(F(x),F(y)\big)\le\eps_n$ implies $d(x,y)\le 2^{-n+1}$.
Thus $F$ is injective and $F^{-1}\:F(M)\to M$ is uniformly continuous.
\end{proof}

\subsection{Star-finite spaces}
It is not hard to see that the separable metrizable space $C\N$, the cone over the countable
uniformly discrete space, is not star-finite.
(See \S\ref{join, etc} for the general definition of a cone; for the time being,
we may define $C\N$ as the subspace of $U(\N,I)$ consisting of all functions with
support in at most one point.)

If a metrizable uniform space $X$ is point-finite or star-finite, then so is the hyperspace 
of compact sets $K(X)$ \cite{Ho4}.
If $X$ is uniformly $0$-dimensional, then so is $K(X)$ (see \cite{Mi2}).

\begin{lemma}[Morita \cite{En}*{5.2.4}, \cite{BP}*{Theorem II.2.3}] \label{strongly paracompact}
Every countable open cover of a metrizable space has a star-finite countable open refinement.
\end{lemma}

\begin{proof} Let $\{U_i\mid i\in\N\}$ be the given cover of the given space $X$.
Let us choose a metric $d$ on $X$ such that the diameter of $X$ is at most $1$.
The function $f\:X\to [0,1)$, defined by $f(x)=\sup_{i\in\N} 2^{-i}d(x,\,X\but U_i)$, is continuous, 
since it is the uniform limit of the continuous functions $f_n(x)=\max_{i=1,\dots,n} 2^{-i}d(x,\,X\but U_i)$.%
\footnote{Alternatively, one can use the function $f(x)=\sum_{i\in\N}\,2^{-i-1}d(x,\,X\but U_i)$.}
Since the $U_i$ cover $X$, we have $f(x)>0$ for each $x\in X$.
On the other hand, $f(x)\le 2^{-i-1}$ for each $x\notin U_1\cup\dots\cup U_i$. 
In particular, $V_i\bydef f^{-1}\big((2^{-i-1},2^{-i+1})\big)$ lies in $U_1\cup\dots\cup U_i$.%
\footnote{Here is a more combinatorial description of $V_i$.
Let $U_{ij}=\{x\in X\mid d(x,\,X\but U_i)>2^{i-j}\}$ and $\bar U_{ij}=\{x\in X\mid d(x,\,X\but U_i)\ge 2^{i-j}\}$.
Let $W_j=U_{1j}\cup\dots\cup U_{jj}$ and $\bar W_i=\bar U_{1j}\cup\dots\cup\bar U_{jj}$.
Then $V_i=f^{-1}((2^{-i-1},1)\but[2^{-i+1},1))=W_{i+1}\but\bar W_{i-1}$.}
Since the $V_i$ cover $X$, the sets $V_i\cap U_j$, where $j\le i$, also cover $X$.
Each $V_i\cap U_j$ meets at most $3i$ of these sets $V_k\cap U_l$ (with $l\le k$), namely,
those with $k\in\{i-1,i,i+1\}$.
\end{proof}

\begin{lemma} \label{cover-extension}
Let $X$ be a metrizable space, $Z$ its closed subset and $C$ an open cover of $Z$
of multiplicity $\le m$ refining an open cover $D$ of $X$.
Then there exist an open neighborhood $Z^+$ of $Z$ and an open cover $C^+$ of $Z^+$
of multiplicity $\le m$ refining $D$ and such that $C^+\cap Z=C$.
\end{lemma}

This lemma must be well-known.
We will prove it using uniform ANRs.

\begin{proof}
For each $U\in C$ let $f_U\:Z\to [0,1]$ be defined by 
$f_U(x)=\min\big(d(x,\,Z\but U),1\big)$.
Let $f\:Z\to I^C$ be defined by $f(x)(U)=f_U(x)$.
Since $|f_U(x)-f_U(y)|\le d(x,y)$, the map $f$ is uniformly continuous with respect to 
the $l_\infty$ metric on $I^C$.
Let $Q_m\subset I^C$ consist of all functions $\phi\:C\to I$ such that
$\{U\in C\mid\phi(U)\ne 0\}$ is of cardinality $\le m$.
Since $C$ is of multiplicity $\le m$, we have $f(Z)\subset Q_m$.
By Theorem \ref{cubohedron} $Q_m$ is a uniform ANR.%
\footnote{In the application of Lemma \ref{cover-extension} in the proof of 
Theorem \ref{A.8}(d) we do not care about the value of $m$ and only need it 
to be finite, so it would suffice to use Corollary \ref{c_00}(b) here for 
our purposes there.}
Since $Z$ is closed, $f$ extends to a uniformly continuous map 
$\bar f\:Z^\oplus\to Q_m$ defined on some open neighborhood $Z^\oplus$ of $Z$.
Then each $f_U$ extends to a uniformly continuous function $\bar f_U\:Z^\oplus\to [0,1]$
defined by $\bar f_U(x)=\bar f(x)(U)$.
For each $U\in C$ let $U^\oplus=\bar f_U^{-1}\big((0,1]\big)$.
Then $U^\oplus$ is open in $Z^\oplus$ and hence in $X$, and 
$C^\oplus\bydef \{U^\oplus\mid U\in C\}$ is an open cover of $Z^\oplus$ of multiplicity $m$.
For each $U\in C$ let $U^+=U^\oplus\cap V_U$, where $V_U$ is some element of $D$ 
containing $U$.
Thus $U^+$ contains $U$ and is open in $X$.
Let $C^+=\{U^+\mid U\in C\}$ and let $Z^+$ be the union of all elements of $C^+$.
Thus $Z^+$ is an open neighborhood of $Z$ in $X$ and $C^+$ is its open cover 
of multiplicity $m$ which refines $D$.
Also we have $C^+\cap Z=C$.
\end{proof}

\begin{theorem}\label{A.8} (a) If $X$ is a separable metric space, there exists a star-finite metric 
space $Y$ and a uniformly continuous homeomorphism $h\:Y\to X$. 
Moreover:

(b) if $X$ is complete, then so is $Y$;

(c) if $X$ is of dimension $\le d$, then $Y$ can be chosen to be of uniform dimension $\le d$;

(d) if $X$ is strongly countably dimensional, then $Y$ can be chosen to be uniformly countably dimensional;

(e) if $X$ is locally finite dimensional, then $Y$ can be chosen to be uniformly locally finite dimensional.
\end{theorem}

\begin{proof}[Proof. (a)] Let $C_1,C_2,\dots$ be a basis for a metrizable uniformity $uX$ 
on the given topological space $X$.
Without loss of generality each $C_i$ is an open cover.
Since $X$ is separable, we may assume that each $C_i$ is countable.
Set $D_1=C_1$, and assume that a countable open cover $D_i$ of $X$ has been constructed.
By Lemma \ref{fully normal} there exists an open barycentric refinement $D_i'$ of $D_i$.
Then $D_i''\bydef D_i'\wedge C_{i+1}$ is an open cover of $X$ refining $C_{i+1}$
and barycentrically refining $D_i$.
By Lemma \ref{strongly paracompact}, there exists an open star-finite
refinement $D_{i+1}$ of $D_i''$.
Thus we obtain a sequence of star-finite open covers $D_1,D_2,\dots$
such that each $D_{i+1}$ refines $C_{i+1}$ and barycentrically refines $D_i$.
The former implies that the sequence $D_1,D_2,\dots$ separates points of
$X$, and thus is a basis for some star-finite metrizable
uniformity $u'X$ on the underlying set of $X$.

If a subset $A\incl X$ is open in the induced topology of $u'X$ (that is, for every $x\in A$ there exists 
an $n$ such that $\st(x,D_n)\incl A$), then $A$ is open in $X$ since the covers $D_i$ are open.
Since each $D_i$ refines $C_i$, the identity map $u'X\to uX$ is uniformly continuous.
Hence $u'X$ induces the original topology of $X$.
\end{proof}

\begin{proof}[(b)] This follows from Lemma \ref{complete-finer}.
\end{proof}

\begin{proof}[(c)] Let us modify the construction in the proof of (a).
Since $X$ is of dimension $\le d$, the open cover $D_i''$ has an open refinement $D_i'''$ 
of multiplicity $\le d+1$.
Now let us redefine $D_{i+1}$ to be an open star-finite refinement of $D_i'''$.
Then clearly $u'X$ will be of uniform dimension $\le d$.
\end{proof}

\begin{proof}[(d)] Let us again modify the construction in the proof of (a).
Since $X$ is strongly countably dimensional, it is a union of an increasing chain of
its closed subspaces $X_0\subset X_2\subset\dots$ such that $\dim X_k\le k$ for each $k$.
Then for each $k$ the open cover $D_i''\cap X_k$ of $X_k$ has an open refinement $E_k$ of 
some finite multiplicity $k+1$.
By Lemma \ref{cover-extension} there exist an open neighborhood $X_k^+$ of $X_k$ and 
an open cover $E_k^+$ of $X_k^+$ of multiplicity $\le k+1$ refining $D_i''$.
For each $U\in E_k$ let $U^\pm=U^+\but X_{k-1}$, where $X_{-1}=\emptyset$, and 
let $E_k^\pm=\{U^\pm\mid U\in E_k\}$.
Then $E_k'\bydef E_0^\pm\cup\dots\cup E_k^\pm$ is an open cover of 
$X_k'\bydef X_0^+\cup\dots\cup X_k^+$ of multiplicity $\le 1+2+\dots+(k+1)$ which refines $D_i''$.
In addition, $E_k'\cap X_j=E_j'\cap X_j$ for each $j<k$.
Finally, $D_i'''\bydef E_0'\cup E_2'\cup\dots$ is an open cover of $X$ which refines $D_i''$
and is such that each $\{U\in D_i'''\mid U\cap X_k\ne\emptyset\}$ is an open cover of
of $X_k'$ of multiplicity $\le 1+2+\dots+(k+1)$.

Now let us redefine $D_{i+1}$ to be an open star-finite refinement of $D_i'''$.
Then clearly $u'X$ will be uniformly countably dimensional.
\end{proof}

\begin{proof}[(e)] This is similar to (d).
\end{proof}

\subsection{Uniformly finitistic spaces I}

It is proved in \cite[Theorem 7.1]{Ho2} that the unit ball $Q_0=U\big((\N^+,\infty),([-1,1],0)\big)$ 
of the separable metric space $c_0$ fails to be Noetherian.
In particular, it is not uniformly finitistic and not star-finite.

In contrast, the countable product $\R^\infty$ of copies of the real line (and hence also every its subspace) 
is star-finite and uniformly finitistic: a fundamental sequence of star-finite covers of $\R^\infty$ 
of finite multiplicities is given by $f_n^{-1}(C_n)$, where $f_n\:\R^\infty\to\R^n$ is the projection and
$C_n$ is the set of the open (cubical) stars of vertices of the standard cubulation of $\R^n$ by cubes 
with edge length $2^{-n}$.

\begin{theorem}[Isbell] \label{product-embedding}
Every separable complete metric space is homeomorphic to a closed subset of the countable product $\R^\infty$.
\end{theorem}

Isbell \cite{I4}*{p.\ 246} mentions this without proof; but soon he published a proof of a parallel theorem
for arbitrary complete metric spaces (see Corollary \ref{inseparable-emb} below), whose separable case is 
slightly weaker than Theorem \ref{product-embedding}.
See \cite{M2}*{Theorem \ref{book:product-embedding}} for a simple direct proof of Theorem \ref{product-embedding}, 
based on a relative version of Tikhonov's embedding theorem.

\begin{corollary} \label{improving2}
(a) Every separable metric space $X$ is homeomorphic to a star-finite uniformly finitistic 
metric space $Y$.

(b) If $X$ is complete, then $Y$ may be chosen to be complete.
\end{corollary}

\begin{proof} Part (b) follows from Theorem \ref{product-embedding}.
Part (a) follows from (b) by considering the completion.
\end{proof}

The following example shows that Lemma \ref{strongly paracompact} and Theorem \ref{A.8} become false
if ``star-finite'' is replaced by ``uniformly finitistic''.

\begin{example} \label{Kruse-counterexample} 
(a) It is not true that every countable open cover of a metrizable space has a countable open refinement
of finite multiplicity.

Indeed, let $X=\bigsqcup_{n=0}^\infty S^n$.
Let $C_n$ be a finite open cover of $S^n$ that has no open refinement of multiplicity $\le n$.
Then $C\bydef \bigcup_{n=0}^\infty C_n$ is a countable open cover of $X$ that has no open refinement of
finite multiplicity.

(b) It is not true that every separable metric space is homeomorphic to a uniformly finitistic 
metric space by a homeomorphism whose inverse is uniformly continuous.

Indeed, let $X$ and $C$ be as in (a).
Similarly to the proof of Theorem \ref{A.8} we can modify any given metrizable uniformity on $X$ into
a metrizable uniformity $uX$ such that $C$ is a uniform cover with respect to $uX$.
Suppose that there exist a uniformly finitistic metric space $Y$ and a uniformly continuous 
homeomorphism $f\:Y\to uX$.
Then the cover $f^{-1}(C)$ is uniform, and hence has a uniform refinement $D$ of finite multiplicity.
The definition of a uniform cover implies that the interiors of the elements of $D$ form an open
cover $D'$ of $Y$.
Then $f(D')$ is an open refinement of $C$ of finite multiplicity, which is a contradiction.
\end{example}

\subsection{Partitions of unity}

\begin{lemma} \label{bary-refinement}
Every point-finite open cover $C$ of a metric space $X$ has a point-finite open barycentric refinement $D$ 
such that

(a) if $C$ is countable, then so is $D$;

(b) if $C$ is of multiplicity $\le m$, then so is $D$;

(c) if $C$ is star-finite, then so is $D$.
\end{lemma}

Every open cover of $X$ has an open barycentric refinement by Lemma \ref{fully normal}, but the proof of Lemma 
\ref{bary-refinement} below is based on a different construction.
On the other hand, Lemma \ref{bary-refinement} follows from Lemma \ref{fully normal} and 
\cite{GI}*{Proof of Theorem 1.1(b,c,d)}.

\begin{proof} First let us show that $C$ is refined by a locally finite open cover $A$.
If $C$ is star-finite, then we may set $A=C$.
In general, by Lemma \ref{shrinking3} there exist open covers $A$ and $E$ of $X$ and a bijection $\phi\:A\to C$ 
such that $\st(U,E)\subset\phi(U)$ for each $U\in A$.
If some $V\in E$ meets infinitely many distinct elements $U_i$ of $A$, then the distinct elements
$\phi(U_i)$ of $C$ all contain $V$ and in particular have a common point, which is a contradiction.
Thus $A$ is locally finite.
If $C$ is of multiplicity $\le m$, then every $V\in E$ similarly meets at most $m$ distinct element of $A$,
and in particular $A$ is of multiplicity $\le m$.

For each $U\in A$ let $f_U\:X\to [0,\infty)$ be defined by $f_U(x)=d(x,\,X\but U)$.
Next, let $\sigma\:X\to [0,\infty)$ be defined by $\sigma(x)=\sum_{V\in A}f_V(x)$, where the sum is 
always finite since $A$ is point-finite and nonzero since $A$ is an open cover.
Since $A$ is locally finite, $\sigma$ is locally equal to a finite sum and hence is continuous.
Let $\phi_U\:X\to [0,1]$ be defined by $\phi_U(x)=f_U(x)/\sigma(x)$.
These continuous functions combine into a continuous map $\phi$ from $X$ to the product $[0,1]^A$,
whose image lies in the usual geometric realization $|N|$ of the nerve $N$ of $A$.
Moreover, $A=f^{-1}(B)$, where $B$ is the cover of $|N|$ by the open stars of vertices.
Let $N'$ be the barycentric subdivision of $N$ and $B'$ be the open cover of $|N|$ by the open stars in $N'$ 
of the vertices of $N'$.
Then $B'$ barycentrically refines $B$; namely, if $x\in |N|$ lies in a simplex of $N'$ given by a flag
$\sigma_1\subset\dots\subset\sigma_r$ of simplexes of $N$, then $\st(x,B')$ lies in the open star in $N$ of
any vertex of $\sigma_1$.
Hence $D\bydef f^{-1}(B')$ barycentrically refines $A$.
The nerve of $D$ is a subcomplex of the nerve of $B'$, which in turn is isomorphic to $N'$.
Hence $D$ is point-finite.
If $A$ is of multiplicity $\le m$, then $\dim N'\le\dim N\le m-1$, and therefore $D$ is of multiplicity $\le m$.
If $A$ is star-finite, then $N$ is locally finite; hence so is $N'$, and therefore $D$ is star-finite.
\end{proof}

\begin{lemma} \label{partition of 1} {\rm (Isbell \cite{I1}*{1.1}, \cite{I3}*{IV.11})} 
Let $C$ be a uniform cover of a metric space $X$.
Then there exist a $k\in(0,\infty)$ and $k$-Lipschitz functions $f_U\:X\to [0,1]$, $U\in C$, such that 
each $f_U$ vanishes on $X\but U$ and $\sum_{U\in C}f_U(x)=1$ for each $x\in X$, where the sum has only 
countably many nonzero terms.
\end{lemma}

\begin{proof} Let $\lambda$ be a Lebesgue number of $C$.
Let us fix a well-ordering of $C$ by an ordinal $A$.
For each $\alpha\in A$ let us define $f_\alpha,g_\alpha,h_\alpha\:X\to [0,1]$ as follows:
\begin{align*}
h_\alpha(x)&=\min\big(1,\,2\lambda^{-1}d(x,\,X\but U_\alpha)\big);\\
g_\alpha(x)&=\sup\{h_\beta(x)\mid \beta<\alpha\}\text{ for }\alpha>1\text{ and }g_1(x)=0;\\
f_\alpha(x)&=g_{\alpha+1}(x)-g_\alpha(x).
\end{align*}
For each $x\notin U_\alpha$ we have $h_\alpha(x)=0$ and consequently $f_\alpha(x)=0$.
For each $\eps>0$, $\alpha>1$ and $x,y\in X$ there exist $\beta,\gamma<\alpha$ such that 
$g_\alpha(x)\le h_\beta(x)+\eps$ and $g_\alpha(y)\le h_\gamma(y)+\eps$.
We also have $g_\alpha(x)\ge h_\gamma(x)$ and $g_\alpha(y)\ge h_\beta(y)$.
Then $g_\alpha(x)-g_\alpha(y)\le h_\beta(x)-h_\beta(y)+\eps$ and 
$g_\alpha(y)-g_\alpha(x)\le h_\gamma(y)-h_\gamma(x)+\eps$, whence
$|g_\alpha(x)-g_\alpha(y)|\le 2\lambda^{-1}d(x,y)+\eps$.
Thus $g_\alpha$ is $2\lambda^{-1}$-Lipschitz.
Hence $f_\alpha$ is $4\lambda^{-1}$-Lipschitz.

For each $x\in X$, the map $\phi\:A\to[0,1]$, defined by $\phi(\alpha)=g_\alpha(x)$, is monotone. 
Hence $B\bydef \phi(A)$ is well-ordered.
Then $\psi\:B\to[0,1]$, where $\psi(\beta)$ is any rational number in the interval $(\beta,\beta+1)$, 
in injective.
Hence $B$ is countable.
Thus $f_\alpha(x)$ is nonzero only for countably many $\alpha\in A$.
By our choice of $\lambda$ there exists a $\gamma\in A$ such that $U_\gamma$ contains the open 
$\frac\lambda2$-ball about $x$.
Then $h_\gamma(x)=1$ and consequently $g_{\gamma+1}(x)=1$.
Also, for each limit $\alpha\in A$ we have $g_\alpha(x)=\sup\{g_\beta(x)\mid\beta<\alpha\}$.
It follows that $\sum_{\alpha\in A}f_\alpha=1$.
\end{proof}

\begin{remark} Since the family of functions $f_U$ given by Lemma \ref{partition of 1} is uniformly
equicontinuous, the joint map $f\:X\to l_\infty(C)$ is uniformly continuous.
Zahradn\'\i k \cite{Za} proved that for each $p\in (1,\infty)$ there exist uniformly 
continuous functions $f_U\:X\to [0,1]$, $U\in C$, such that the joint map $f\:X\to l_p(C)$ is 
uniformly continuous, each $f_U$ vanishes on $X\but U$ and $\sum_{U\in C}f_U(x)=1$ for each $x\in X$.
He also showed that such functions need not exist for $p=1$.
\end{remark}

Let $K$ be an abstract simplicial complex, that is, a family of finite subsets (``simplexes'') of a set $V$ 
(``vertices'') such that if $\sigma\in K$ and $\tau\subset\sigma$, then $\tau\in K$.
Its traditional geometric realization $|K|$ is the union of the convex hulls $\left<\sigma\right>$
of all simplexes $\sigma\in K$ in the vector space $\R[V]$ of all formal $\R$-linear combinations of 
the vertices.
The {\it metric topology} on $|K|$ is the topology of subspace of $\R[V]$ in either the $l_1$, or the
$l_2$, or the $l_\infty$ metric, or the product topology (all four topologies coincide on $|K|$;
see \cite{M3}*{Appendix} for a proof).
We call the topological space $|K|$ (with the metric topology) a {\it polyhedron}. 
(Usually polyhedra are meant to be endowed with a fixed PL structure, that is, a family of compatible 
triangulations, but we do not need this for our purposes here.)
The $l_\infty$ metric also provides a reasonable uniformity on $|K|$ when $K$ is finite dimensional \cite{I1};
the infinite dimensional case needs a different construction of geometric realization, which is the subject
of a sequel to the present paper \cite{M3}.
When $K$ is finite dimensional, we call the uniform space $|K|$ (with the uniformity of the $l_\infty$ metric)
a {\it finite dimensional uniform polyhedron}.

\begin{corollary} \label{partition-emb}
Every uniformly finitistic metric space $X$ uniformly embeds in a countable
product of finite dimensional uniform polyhedra. 
\end{corollary}

\begin{proof} Let $C_1,C_2,\dots$ be a basis of uniformity of $X$ such that each $C_i$ is of finite
multiplicity.
Then in particular each $C_i$ is point-finite, so its nerve $N_i=\{S\subset C_i\mid\bigcap S\ne\emptyset\}$
is a simplicial complex.
It is easy to see that the joint map $f_i\:X\to l_\infty(C_i)$ of the functions $f_U$, $U\in C_i$, given by Lemma
\ref{partition of 1} has its image contained in the traditional geometric realization $|N_i|$.
Since $C_i$ is of finite multiplicity, $|N_i|$ is a finite dimensional uniform polyhedron.
The cover $D_i$ of $|N_i|$ by the open stars of vertices is a uniform cover, and $f^{-1}(D_i)$ refines $C_i$.

Since the $f_i$ are uniformly continuous, so is their joint map $f\:X\to P\bydef \prod_{i=1}^\infty |N_i|$.
Given $x,y\in X$, there exists an $i$ such that no element of $C_i$ contains both $x$ and $y$.
Then $f_i(x)\ne f_i(y)$, and consequently $f(x)\ne f(y)$.
Thus $f$ is injective.
It remains to show that the inverse map $g\:f^{-1}(X)\to X$ is uniformly continuous.
Given a uniform cover $C$ of $X$, it is refined by some $C_i$.
We can represent $f_i$ as the composition $X\xr{g^{-1}}f(X)\xr{\iota}P\xr{\pi_i}|N_i|$, where $\iota$ is 
the inclusion and $\pi_i$ is the projection.
Then $E_i=(\pi_i\iota)^{-1}(D_i)$ is a uniform cover of $f(X)$ such that $C$ is refined by $f^{-1}(E_i)=g(E_i)$.
In other words, $g^{-1}(C)$ is refined by $E_i$.
Hence $g^{-1}(C)$ is a uniform cover.
Thus $g$ is uniformly continuous.
\end{proof}

\subsection{Uniformly finitistic spaces II}

\begin{theorem}[Isbell \cite{I6}*{Theorem 3.6}]\label{improving}
(a) Every metric space $X$ is homeomorphic to a uniformly finitistic metric space $Y$.

(b) If $X$ is complete, then $Y$ may be chosen to be complete.
\end{theorem}

The Zentralblatt reviewer of \cite{I6}, Arthur Kruse, was not convinced by Isbell's proof of 
Theorem \ref{improving}, but later wrote a note \cite{Kr} where he retracted his doubts.
In that note he also attempted to reprove Theorem \ref{improving}(b) in a more geometric way --- but 
succeeded in proving only the weaker assertion (a), as we will see below.

Although Isbell's proof is fine (even if terse) as it stands, we reproduce it with a bit more detail 
for the reader's convenience.
There is a minor simplification (in our argument, barycentric refinement suffices in place of Isbell's 
star-refinement).

Let us note that Corollary \ref{improving2} follows from Theorem \ref{improving} and Theorem \ref{A.8};
in fact, Corollary \ref{improving2} follows directly from the proof of Theorem \ref{improving}.

\begin{proof} Let $X$ be the given metric space.
By considering its completion, we may assume that $X$ is complete (thus (a) reduces to (b)). 
Let $C_n$, $n=1,2,\dots$, be a basis for the metrizable uniformity of $X$ consisting of open covers
(for example, any standard basis).
Thus the unordered family of open covers $\{C_n\}$ {\it induces the topology} of $X$ in the sense that 
every neighborhood $U$ of every $x\in X$ contains $\st(x,C_n)$ for some $n$.
(The openness of the covers $C_n$ automatically yields the converse, i.e., $U$ is a neighborhood of $x$ if 
it contains $\st(x,C_n)$ for some $n$.)
Also, a sequence of points $x_i\in X$ is convergent if it is {\it $\{C_n\}$-Cauchy}, in 
the sense that for each $n$ there exist an $m$ and an $U\in C_n$ such that $x_i\in U$ for all $i\ge m$.

\begin{lemma} \label{RFDization}
There exists a countable family $\{G_m\}$ of open covers of $X$ of finite multiplicities that 
induces the topology of $X$ and is such that every $\{G_m\}$-Cauchy sequence is convergent.
\end{lemma}

\begin{proof}[Proof: The case of separable $X$]
Let us fix some $n$.
Since $X$ is separable, $C_n$ has a countable subcover $C'_n$.
By Lemma \ref{strongly paracompact} $C'_n$ has a countable star-finite open refinement $D_n$.
Let us call a finite sequence $U_1,\dots,U_r$ of elements of $D_n$ such that $U_i\cap U_{i+1}\ne\emptyset$ 
a {\it chain of length $r$ connecting $U_1$ and $U_r$}.
Given $U,V\in D_n$, let us write $U\sim V$ if $U$ and $V$ can be connected by a chain of finite length.
Then $D_n$ is partitioned into the equivalence classes $D_{nk}$ of the equivalence relation $\sim$. 
Let us fix one representative $U_{nk1}\in D_{nk}$ in each equivalence class.
Let $D_{nkj}$ be the set of all elements of $D_{nk}$ that can be connected to $U_{nk1}$ by a chain of
length $j$ but not by a chain of length $j-1$.
Since $D_n$ is star-finite, $D_{nkj}$ is finite.
Let $U_{nkj}$ be the union of all elements of $D_{nkj}$.
If $U_{nkj}\cap U_{nk'j'}\ne\emptyset$, then some $U\in D_{nkj}$ meets some $U'\in D_{nk'j'}$, which implies
that $k=k'$ and $|j-j'|\le 1$.
Therefore the open cover $\{U_{nkj}\}$ of $X$ is of multiplicity $\le 2$ (and star-finite). 
By Lemma \ref{shrinking3} it has an open refinement $E_n=\{V_{nkj}\}$ such that the closure $\bar V_{nkj}$ 
of each $V_{nkj}$ lies in $U_{nkj}$.%
\footnote{Here is a simple ad hoc construction of $E_n$.
Let $V_{nk1}$ and $V_{nk1}'$ be disjoint open neighborhoods of the closed sets 
$U_{nk1}\but U_{nk2}$ and $X\but U_{nk1}$.
Assuming that $V_{nk,i-1}$ has been defined, let $V_{nki}$ and $V_{nki}'$ be disjoint open neighborhoods of 
the closed sets $U_{nki}\but (V_{nk,i-1}\cup U_{nk,i+1})$ and $X\but U_{nki}$.}
Then $F_{nkj}\bydef D_{nkj}\cup\{X\but\bar V_{nkj}\}$ is a finite open cover of $X$.

Let $G_{nkj}=E_n\land F_{nkj}$.
Since $E_n$ and $F_{nkj}$ are of finite multiplicities, so is $G_{nkj}$.
Let us show that the family $\{G_{nkj}\}$ induces the topology of $X$.
If $U$ is a neighborhood of an $x\in X$, then $U$ contains $\st(x,C_n)$ for some $n$.
Let $V_{nkj}$ be any member of $E_n$ containing $x$.
Then $\st(x,G_{nkj})\subset\st(x,F_{nkj})=\st(x,D_{nkj})\subset\st(x,D_n)\subset\st(x,C_n)\subset U$, where 
the equality is due to $x\in V_{nkj}$.
Finally, suppose that a sequence of points $x_i\in X$ is $\{G_{nkj}\}$-Cauchy.
Then it is also $\{E_n\}$-Cauchy and $\{F_{nkj}\}$-Cauchy. 
So for each $n$ there exists an $m'$ and a $V_{nkj}\in E_n$ such that $x_i\in V_{nkj}$ for all $i\ge m'$.
Also there exists an $m$ and a $V\in F_{nkj}$ such that $x_i\in V$ for all $i\ge m$.
Then $V_{nkj}\cap V\ne\emptyset$, so $V\in D_{nkj}\subset D_n$.
Hence $V\subset U$ for some $U\in C_n$.
So $x_i\in U$ for all $i\ge m$.
Thus $x_i$ is $\{C_i\}$-Cauchy, and therefore it is convergent.
\end{proof}

\begin{proof}[General case]
Let us fix an $n$.
By Stone's theorem (see \cite{En}*{4.4.1}) $C_n$ is refined by an open cover $B_n$ such that 
$B_n=\bigcup_{i\in\N}B_{ni}$, where each $B_{ni}$ is a {\it discrete family} in the sense that each $x\in X$ has 
an open neighborhood $U_x$ that meets at most one element of $B_{ni}$.
By Lemma \ref{fully normal}, the cover $\{U_x\mid x\in X\}$ of $X$ has an open barycentric refinement $H_n$.
It is easy to see that the closure $\bar U$ of each $U\in B_{ni}$ lies in $U^+\bydef \st(U,H_n)$, and that
the elements of $B_{ni}^+\bydef \{U^+\mid U\in B_{ni}\}$ are pairwise disjoint.
Since $B_{ni}$ is a discrete family, the closure $\bar W_{ni}$ of the union $W_{ni}$ of all elements of $B_{ni}$
coincides with $\{\bar U\mid U\in B_{ni}\}$ and therefore is contained in the union of all elements of $B_{ni}^+$.
Therefore $A_{ni}\bydef B_{ni}^+\cup\{X\but\bar W_{ni}\}$ is an open cover of $X$ multiplicity $\le 2$.
On the other hand, $I_n\bydef \{W_{ni}\mid i\in\N\}$ is a countable open cover of $X$.
By Lemma \ref{strongly paracompact} it has a countable star-finite open refinement $J_n$.
Then by Lemma \ref{bary-refinement}(a,c) below, $J_n$ has a countable star-finite open barycentric refinement 
$D_n$.
Using this $D_n$, we define $E_n$, $F_{nkj}$ and $G_{nkj}$ as before.

Let $G_{nkji}=G_{nkj}\land A_{ni}$.
Since $G_{nkj}$ and $A_{ni}$ are of finite multiplicities, so is $G_{nkji}$.
Let us show that the family $\{G_{nkj}\}$ induces the topology of $X$.
If $U$ is a neighborhood of an $x\in X$, then $U$ contains $\st(x,C_n)$ for some $n$.
Let $V_{nkj}$ be any member of $E_n$ containing $x$.
Then by the previous argument $\st(x,G_{nkj})\subset\st(x,D_n)$.
Let $W_{ni}$ be any member of $I_n$ containing $\st(x,D_n)$.
Then $\st(x,G_{nkji})\subset\st(x,G_{nkj})\subset\st(x,D_n)\subset W_{ni}$.
On the other hand, $\st(x,G_{nkji})\subset\st(x,A_{ni})=\st(x,B_{ni}^+)$, where the equality is due to 
$x\in W_{ni}$.
Since the elements of $B_{ni}^+$ are pairwise disjoint, $\st(x,B_{ni}^+)$ consists of a single element
$U^+\in B_{ni}^+$.
Thus $\st(x,G_{nkji})\subset W_{ni}\cap U^+=U\in B_{ni}$.
In particular, $\st(x,G_{nkji})\subset\st(x,B_{ni})\subset\st(x,B_n)\subset\st(x,C_n)\subset U$.

Finally, suppose that a sequence of points $x_i\in X$ is $\{G_{nkji}\}$-Cauchy.
Then it is also $\{A_{ni}\}$-Cauchy and $\{G_{nkj}\}$-Cauchy. 
By the previous argument, for each $n$ there exists an $m'$ and a $V'\in D_n$ such that $x_i\in V'$ 
for all $i\ge m'$.
Let $W_{ni}$ be any member of $I_n$ containing $V'$.
Then there exists an $m''$ and a $V''\in A_{ni}$ such that $x_i\in V''$ for all $i\ge m''$.
We have $V'\cap V''\ne\emptyset$, so $V''\in B_{ni}^+$.
Thus $V''=V^+$ for some $V\in B_{ni}$.
Hence $V'\cap V''\subset W_{ni}\cap V^+=V$.
Since $V\in B_n$, we have $V\subset U$ for some $U\in C_n$.
So $x_i\in U$ for all $i\ge m\bydef \max(m',m'')$.
Thus $x_i$ is $\{C_i\}$-Cauchy, and therefore it is convergent.
\end{proof}

Let us resume the proof of Theorem \ref{improving}.
Let $G_m$, $m=1,2,\dots$, be the covers given by Lemma \ref{RFDization}.
Set $D_1=G_1$, and assume that an open cover $D_n$ of $X$ of finite multiplicity has been constructed.
By Lemma \ref{bary-refinement}(b) there exists an open barycentric refinement $D_n'$ of $D_n$ of finite 
multiplicity.
Then $D_{n+1}\bydef D_n'\wedge G_{n+1}$ is an open cover of $X$ of finite multiplicity refining $G_{n+1}$
and star-refining $D_n$.
Since the family $\{G_i\}$ induces the topology of $X$ and each $G_n$ is refined by $D_n$, the family 
$\{D_i\}$ also induces the topology of $X$.
If a sequence of points $x_k\in X$ is $\{D_i\}$-Cauchy, then it is $\{G_i\}$-Cauchy and hence convergent.
Thus $D_1,D_2,\dots$ is a basis of a uniformly finitistic complete metrizable uniformity on 
the underlying topological space of $X$.
\end{proof}

Theorem \ref{improving}(b) and Corollary \ref{partition-emb} have the following

\begin{corollary}[Isbell \cite{I6}*{Theorem 3.6}]\label{inseparable-emb}
Every complete metric space is homeomorphic to a closed subset of a countable product
of finite dimensional uniform polyhedra.
\end{corollary}

Without ``closed'', this was also proved in \cite{Kow}.

\subsection{Uniformly finitistic spaces III}

Given a set $S$ and a $p\in [1,\infty]$, let $\Lambda_p(S)$ be the convex hull of 
$\{0\}\cup\{\delta_s\mid s\in S\}$ in $l_p(S)$, where $\delta_s$ is the indicator function of $\{s\}$, 
and let $\bar\Lambda_p(S)$ be the closure of $\Lambda_p(S)$.
Thus $\Lambda_p(S)$ consists of all finite linear combinations $\sum_{i=1}^n\lambda_i\delta_{s_i}$, where 
each $\lambda_i\ge 0$ and $\sum_{i=1}^n\lambda_i\le 1$.
The set of such linear combinations with at most $n$ summands will be denoted $\Lambda_p(S)^{(n)}$.
It is not hard to see that $\bar\Lambda_p(S)$ consists of all real sequences $\lambda=(\lambda_s)$ such that 
each $\lambda_s\ge 0$ and $\sum_{s\in S}\lambda_s\le 1$.

\begin{lemma} \label{Kruse}
If $p\in (1,\infty]$, there exists a uniform embedding 
$f\:\bar\Lambda_p(S)\to\prod_{n=1}^\infty\Lambda_p(S)^{(n)}$.
In particular, $\bar\Lambda_p(S)$ is uniformly finitistic.
\end{lemma}

This result is essentially due to Kruse \cite{Kr}, who defined the map $f$ and stated that in the case $p=2$
it is a homeomorphism onto a closed subset.

\begin{proof} 
For each $\lambda\in\bar\Lambda_p(S)$ and each $n\in\N$, the inequality $\lambda_s>\frac1{n+1}$
is satisfied for at most $n$ elements $s\in S$, since $\sum\lambda_s$ is bounded above by $1$.  
Let us define $f_n\:\bar\Lambda_p(S)\to\Lambda_p(S)^{(n)}$ by 
$\big(f_n(\lambda)\big)_s=\max(\lambda_s-\frac1{n+1},\,0)$, and let $f$ be defined by 
$f(\lambda)=\big(f_n(\lambda)\big)$.
We have $\big|\big(f_n(\lambda)-f_n(\mu)\big)_s\big|\le|\lambda_s-\mu_s|=|(\lambda-\mu)_s|$ for each $s$.
Therefore each $f_n$ is uniformly continuous; hence so is $f$.
If $\lambda\ne\mu$, then there there exist an $s\in S$ such that $\lambda_s\ne\mu_s$ and an $n\in\N$ such that
$\min(\lambda_s,\mu_s)\ge\frac1{n+1}$.
Then $\big(f_n(\lambda)-f_n(\mu)\big)_s=\lambda_s-\mu_s\ne 0$; hence $f$ is injective.   

It remains to show that $f^{-1}\:f\big(\bar\Lambda_p(S)\big)\to\bar\Lambda_p(S)$ is uniformly continuous.
A metric on $\prod_{n=1}^\infty\Lambda_p(S)^{(n)}$ is given by 
$d(x,y)=\sup\{2^{-i}\min\big(d(x_i,y_i),1\big)\mid i\in\N\}$, where $x=(x_i)$ and $y=(y_i)$.
Thus $d(x,y)<2^{-2n}$ implies $2n$ inequalities $d(x_i,y_i)<2^{i-2n}$, $i=1,\dots,2n$, including 
$d(x_n,y_n)<2^{-n}$.
Hence $d\big(f(\lambda),f(\mu)\big)<2^{-2n}$ implies $d\big(f_n(\lambda),f_n(\mu)\big)<2^{-n}$.

In the case $p=\infty$ we have $d\big(\lambda,f_n(\lambda)\big)\le\frac1{n+1}$ 
and similarly for $\mu$.
Therefore $d\big(f(\lambda),f(\mu)\big)<2^{-2n}$ implies $d\big(\lambda,\mu)<2^{-n}+\frac2{n+1}$.

In the case $p<\infty$ we still have $a_s\bydef \big|\lambda_s-\big(f_n(\lambda)\big)_s\big|\le\frac1{n+1}<\frac1n$.
Since the set $T\bydef \{s\in S\mid\lambda_s\ge\frac1{n+1}\}$ contains at most $n$ elements, we have
$\sum_{s\in T}a_s^p\le n\frac1{n^p}=\frac1{n^{p-1}}$.
On the other hand, for each $s\notin T$ we have $a_s^p=\lambda_s^p=\frac1{n^p}(n\lambda_s)^p\le
\frac1{n^p}n\lambda_s=\frac1{n^{p-1}}\lambda_s$.
Hence $\sum_{s\notin T} a_s^p\le\frac1{n^{p-1}}\sum_{s\notin T}\lambda_s\le\frac1{n^{p-1}}$.
Thus $d\big(\lambda,f_n(\lambda)\big)\le(\frac2{n^{p-1}})^{1/p}$ and similarly for $\mu$.
Therefore $d\big(f(\lambda),f(\mu)\big)<2^{-2n}$ implies $d\big(\lambda,\mu)<2^{-n}+2(\frac2{n^{p-1}})^{1/p}$.
\end{proof}

\begin{theorem} \label{l_2-embedding}

Let $p\in [1,\infty)$, let $X$ be a metric space and let $S$ be an infinite dense subset of $X$.

(a) $X$ embeds in $\bar\Lambda_p(S)$ by an embedding $g$.

(b) $g^{-1}\:g(X)\to X$ is uniformly continuous when $p=1$.

(c) If $X$ is uniformly finitistic, then it uniformly embeds in $\bar\Lambda_p(S)$.
\end{theorem}

The proof is a variation on Dowker's construction \cite{Dow}*{Lemma 1}.

Part (a) yields an alternative proof of Theorem \ref{improving}(a).

It was wrongly asserted in \cite{Kr} that part (b) also holds for $p=2$.%
\footnote{This assertion, if it were true, would have yielded an alternative proof of Theorem \ref{improving}(b), 
using Lemma \ref{Kruse}.
Let us note that Dowker's construction does yield an embedding of $X$ in $l_2(S)$ with uniformly 
continuous inverse (cf.\ \cite{Sak}*{Proof of 6.2.4}).}
In fact, the assertion of (b) is false for each $p\in(1,\infty)$ by Example \ref{Kruse-counterexample} 
and Lemma \ref{Kruse}.%
\footnote{Even though $g^{-1}\:g(X)\to X$ need not be uniformly continuous for $p\in (1,\infty)$, one could 
hope that it always sends Cauchy sequences to Cauchy sequences.
But this does not seem to be the case.
Let $e_p\:l_1(S)\to l_p(S)$ be the identity inclusion, $(x_s)\mapsto (x_s)$.
It is well known and easy to see that $e_p$ uniformly continuous.
However, $e_p^{-1}$ is not continuous even when restricted to $\bar\Lambda_p(S)$. 
Indeed, let $x_n=(x_{ni})$, where $x_{ni}=\frac1n$ for $i=1,\dots,n$ and $x_{ni}=0$ for $i>n$.
Then $x_n\to 0$ in $l_p(S)$ but $x_n\not\to 0$ in $l_1(S)$.
Moreover, since $e_p$ is a continuous injection, $x_n$ has no limit in $l_1(S)$.
Thus $e_p^{-1}$ sends a Cauchy sequence in $\bar\Lambda_p(S)$ to a non-Cauchy sequence.} 

\begin{proof}[Proof. (a)] For each $n=1,2,\dots$ let $C_n$ be the cover of $X$ by all closed balls of radius 
$2^{-n}$.
Let $D_n$ be a (non-uniform) locally finite open refinement of $C_n$ (which exists since metrizable spaces 
are paracompact, see \cite{En}*{4.4.1}).
By considering one $U\in D_n$ for each $s\in S$ so that $s\in U$, we may assume that $D_n=\{V_{n\alpha}\}$ 
has the same cardinality as $S$.
Since $S$ is infinite, $\bigcup_{n\in\N} D_n$ also has the same cardinality as $S$, so we may think of $S$ 
as the set of all pairs $(n,\alpha)$.

Let $f_{n\alpha}\:X\to[0,\infty)$ be defined by $f_{n\alpha}(x)=d(x,\,X\but V_{n\alpha})$.
Each $f_{n\alpha}$ is a continuous function such that $f_{n\alpha}^{-1}(0)=X\but V_{n\alpha}$.
Let $\sigma_n\:X\to[0,\infty)$ be defined by $\sigma_n(x)=\sum_\alpha f_{n\alpha}(x)$.
Since $D_n$ is point-finite, the sum is actually finite for each $x\in X$, and since $D_n$ covers $X$, 
$\sigma_n(x)>0$ for each $x\in X$.
Since $D_n$ is locally finite, for each $x\in X$ the sum is finite on some neighborhood of $x$, and 
consequently $\sigma_n$ is continuous at $x$.
Then each $g_{n\alpha}\:X\to[0,2^{-n}]$, defined by 
$g_{n\alpha}(x)=2^{-n}\frac{f_{n\alpha}(x)}{\sigma_n(x)}$, is continuous.
These functions $g_{n\alpha}$ combine into a map $g\:X\to\R^S$, where $\R^S$ is the set of all real sequences 
$\lambda=(\lambda_{n\alpha})$.

For each $x\in X$ and $n\in\N$ we have $\sum_\alpha g_{n\alpha}(x)=2^{-n}$.
Hence $\sum_{n,\alpha}g_{n\alpha}(x)=1$.
Since $p\ge 1$, we also get $\sum_{n,\alpha}g_{n\alpha}(x)^p\le 1$.
Hence $g(x)\in l_p(S)$, moreover $g(x)\in\bar\Lambda_p(S)$.

Let us show that $g$ is injective and also a closed map onto $g(X)$.
Given a closed set $Z\subset X$ and an $x\in X$ such that $x\notin Z$, we have $d(x,Z)>2^{-n+1}$ for 
some $n\in\N$.
Let $V_{n\alpha}\in D_n$ be such that $x\in V_{n\alpha}$.
Since $D_n$ refines $C_n$, we have $V_{n\alpha}\cap Z=\emptyset$.
Hence $g_{n\alpha}(z)=0$ for each $z\in Z$, whereas $k\bydef g_{n\alpha}(x)>0$.
Then $d\big(g(x),g(Z)\big)\ge k>0$.
Therefore $g(x)$ does not lie in the closure of $g(Z)$.
This shows that $g(Z)$ is closed in $g(X)$ and also, by considering $Z=\{z\}$, that $g$ is injective.

Let us show that $g$ is continuous.
Given an $x\in X$ and an $\eps>0$, let $n$ be such that $2^{-n+2}<\eps^p$.
Let $U$ be a neighborhood of $x$ that meets only a finite family $T_x$ of elements of $D_1\cup\dots\cup D_n$.
Thus $g_{k\alpha}$ vanishes on $U$ if $k\le n$ and $(k,\alpha)\notin T_x$.
For each $(k,\alpha)\in T_x$, since $g_{k\alpha}$ is continuous, there exists a neighborhood $V$ of $x$ in $U$ 
such that $|g_{k\alpha}(y)-g_{k\alpha}(x)|<\frac\eps{2|T_x|}$ for all $y\in V$.
Then $\sum_{k\le n,\alpha}|g_{k\alpha}(y)-g_{k\alpha}(x)|<\frac{\eps^p}2$ and
$\sum_{k>n,\alpha}|g_{k\alpha}(y)-g_{k\alpha}(x)|\le\sum_{k>n,\alpha}g_{k\alpha}(x)+
\sum_{k>n,\alpha}g_{k\alpha}(y)=2^{-n}+2^{-n}<\frac{\eps^p}2$.
Since each $|g_{k\alpha}(y)-g_{k\alpha}(x)|\le g_{k\alpha}(y)+g_{k\alpha}(x)\le 2^{-k}+2^{-k}\le 1$, we have
$\sum_{k,\alpha}|g_{k\alpha}(y)-g_{k\alpha}(x)|^p\le\sum_{k,\alpha}|g_{k\alpha}(y)-g_{k\alpha}(x)|<\eps^p$.
Hence $d\big(g(x),g(y)\big)<\eps$.
\end{proof}

\begin{proof}[(b)] 
Suppose that $d(x,y)>2^{-n+1}$ for some $x,y\in X$.
Since $D_n$ refines $C_n$, no element of $D_n$ contains both $x$ and $y$.
Hence $S_x\bydef \{V\in D_n\mid x\in V\}$ and $S_y$ are disjoint.
We have $g_{n\alpha}(x)=0$ for each $\alpha\notin S_x$ and $g_{n\alpha}(y)=0$ for each $\alpha\notin S_y$.
Then $\sum_\alpha |g_{n\alpha}(x)-g_{n\alpha}(y)|=
\sum_{\alpha\in S_x} g_{n\alpha}(x)+\sum_{\alpha\in S_y} g_{n\alpha}(y)=2^{-n}+2^{-n}$.
Hence $d\big(g(x),g(y)\big)\ge 2^{-n+1}$.
Thus for each $\eps>0$ there exists a $\delta>0$, namely, $\delta=2^{-n+1}$, where $2^{-n+1}\le\eps$ 
and $n\in\N$, such that $d\big(g(x),g(y)\big)<\delta$ implies $d(x,y)<\eps$.
\end{proof}

\begin{proof}[(c)]  
For each $n=1,2,\dots$ let $C_n$ be the cover of $X$ by all closed balls of radius $2^{-n}$.
Let $D'_n$ be an uniform refinement of $C_n$ of a finite multiplicity $\mu_n$.
By Lemma \ref{shrinking} there exist open uniform covers $D_n$ and $E_n$ of $M$ such that 
the cover $\{\st(U,E)\mid U\in D_n\}$ refines $D'_n$.
Every element $V$ of $E_n$ meets at most $\mu_n$ elements of $D_n$ (for otherwise $V$ would be contained in 
more than $\mu_n$ elements of $D'_n$).
Let $\lambda$ be a Lebesgue number of $D_n$.
By considering one $U\in D_n$ for each $s\in S$ so that $U$ contains the closed $\frac\lambda3$-ball about $x$, 
we may assume that $D_n=\{V_{n\alpha}\}$ has the same cardinality as $S$.
Since $S$ is infinite, $\bigcup_{n\in\N} D_n$ also has the same cardinality as $S$, so we may think of $S$ 
as the set of all pairs $(n,\alpha)$.

Let $f_{n\alpha}\:X\to [0,1]$ be the $k$-Lipschitz functions given by Lemma \ref{partition of 1}.
Thus each $f_{n\alpha}$ vanishes on $X\but V_{n\alpha}$ and $\sum_\alpha f_{n\alpha}(x)=1$ for each $x\in X$.
Let $g_{n\alpha}\:X\to[0,1]$ be defined by 
$g_{n\alpha}(x)=2^{-n/p}c_nf_{n\alpha}(x)^{1/p}$, where $c_n=(2^{1/p}-1)\mu_n^{-1}$.
Let us note that $\mu_n\ge 1$, whence $c_n\le 1$.
The functions $g_{n\alpha}$ combine into a map $g\:X\to\R^S$, where $\R^S$ is the set of all real sequences 
$\lambda=(\lambda_{n\alpha})$.

For each $x\in X$ and $n\in\N$ we have $\sum_\alpha f_{n\alpha}(x)=1$.
Hence $\sum_\alpha g_{n\alpha}(x)^p=2^{-n}c_n^p\le 2^{-n}$.
Therefore $\sum_{n,\alpha}g_{n\alpha}(x)^p\le 1$.
Thus $g(x)\in l_p(S)$.

Moreover, each $f_{n\alpha}(x)\le 1$, so $g_{n\alpha}(x)\le 2^{-n/p}c_n$.
Hence $\sum_\alpha g_{n\alpha}(x)\le 2^{-n/p}c_n\mu_n$, using that the latter sum contains at most $\mu_n$ 
nonzero terms.
Therefore $\sum_{n,\alpha} g_{n\alpha}(x)\le\frac{2^{-1/p}}{1-2^{-1/p}}c_n\mu_n=1$.
Thus $g(x)\in\bar\Lambda_p(S)$.

Let us show that $g$ is injective and $g^{-1}\:g(X)\to X$ is uniformly continuous.
Suppose that $d(x,y)>2^{-n+1}$ for some $x,y\in X$.
Since $D_n$ refines $C_n$, no element of $D_n$ contains both $x$ and $y$.
Hence $S_x\bydef \{V\in D_n\mid x\in V\}$ and $S_y$ are disjoint.
We have $f_{n\alpha}(x)=0$ for each $\alpha\notin S_x$ and $f_{n\alpha}(y)=0$ for each $\alpha\notin S_y$.
Therefore $\sum_\alpha |f_{n\alpha}(x)^{1/p}-f_{n\alpha}(y)^{1/p}|^p=
\sum_{\alpha\in S_x} f_{n\alpha}(x)+\sum_{\alpha\in S_y} f_{n\alpha}(y)=2$.
Consequently $\sum_\alpha|g_{n\alpha}(x)-g_{n\alpha}(y)|^p=2^{-n+1}c_n^p$.
Therefore $d\big(g(x),g(y)\big)^p\ge 2^{-n+1}c_n^p$.
This shows that $g$ is injective, and that for each $\eps>0$ there exists a $\delta>0$, namely,
$\delta=2^{(-n+1)/p}c_n$, where $2^{-n+1}\le\eps$ and $n\in\N$, such that $d\big(g(x),g(y)\big)<\delta$ 
implies $d(x,y)<\eps$.

Let us show that $g$ is uniformly continuous.
Given an $\eps>0$, let $n$ be such that $2^{-n+2}<\eps^p$.
Let $E=E_1\land\dots\land E_n$ and $r=\mu_1+\dots+\mu_n$.
The function $\phi\:[0,1]\to[0,1]$, defined by $\phi(x)=x^{1/p}$, is continuous and hence uniformly continuous.
So there exists a $\gamma>0$ such that $|a-b|<\gamma$ implies $|a^{1/p}-b^{1/p}|<\frac{\eps}{(2r)^{1/p}}$ 
for $a,b\in[0,1]$.
Let $\delta=\min(\frac\gamma k,\lambda)$, where $\lambda$ is a Lebesgue number of $E$.
Suppose that $x,y\in X$ satisfy $d(x,y)\le\delta$.
Then they both belong to some $U\in E$.
Let $T_U$ be the set of those elements of $D_1\cup\dots\cup D_n$ that meet $U$.
Thus $g_{k\alpha}$ vanishes on $U$ if $k\le n$ and $(k,\alpha)\notin T_U$.
For each $(k,\alpha)\in T_U$, since $d(x,y)\le\frac\gamma k$ and $f_{k\alpha}$ is $k$-Lipschitz, we have
$|f_{k\alpha}(y)-f_{k\alpha}(x)|\le\gamma$.
Therefore $|f_{k\alpha}(y)^{1/p}-f_{k\alpha}(x)^{1/p}|^p\le\frac{\eps^p}{2r}$, and consequently
$|g_{k\alpha}(y)-g_{k\alpha}(x)|^p\le 2^{-k/p}c_k\frac{\eps^p}{2r}\le\frac{\eps^p}{2r}$.
Since $E$ refines $E_k$ for $k\le n$, and each element of $E_k$ meets at most $\mu_k$ elements of $D_k$,
we have $|T_U|\le r$.
Thus $\sum_{k\le n,\alpha}|g_{k\alpha}(y)-g_{k\alpha}(x)|^p\le\frac{\eps^p}2$.
Also, $\sum_{k>n,\alpha}|g_{k\alpha}(y)-g_{k\alpha}(x)|^p\le\sum_{k>n,\alpha}g_{k\alpha}(x)^p+
\sum_{k>n,\alpha}g_{k\alpha}(y)^p=2^{-n+1}c_n^p\le 2^{-n+1}<\frac{\eps^p}2$.
Thus $d\big(g(x),g(y)\big)^p<\eps^p$.
\end{proof}

\subsection{Some embedding results}

\begin{remark}\label{A.2*}
Let us mention some other known embedding theorems.
\begin{roster}
\item (cf.\ \cite[II.19]{I3}, \cite[1.1(i)]{BL}, \cite{BP}*{II.1.1}, \cite{Sak}*{2.3.9}, 
\cite{M2}*{\ref{book:wojdyslawski}}) 
Every metric space $M$ isometrically embeds in the space $U_b(M,\R)$ of bounded uniformly 
continuous functions on $M$ with metric induced by the norm $||f||=\sup_{x\in M} |f(x)|$.
The embedding $e$ is given for example by $e(x)(y)=d(x,y)-d(x_0,y)$ for some fixed $x_0\in M$. 
The same formula also defines an isometric embedding of $M$ in $l_\infty(dM)=U_b(dM,\R)$, where
$dM$ is $M$ with the discrete uniformity.
If $M$ has diameter $\le 1$, then $e$ can also be defined by $e(x)(y)=d(x,y)$, and its image lies in $U(M,I)$.

\item (Banach--Mazur; see \cite[4.5.21]{En}, \cite[Corollary II.1.2]{BP})
Every separable metric space of diameter $\le 1$ isometrically embeds in $U(I,I)$.

\item (Tikhonov; see e.g.\ \cite{M2}*{\ref{book:product-embedding2}}) 
Every compactum topologically (hence, uniformly) embeds in the Hilbert cube. 

\item (Arens--Eels and Blumenthal--Klee; see \cite{AE}, \cite{BP}*{Corollary II.1.1}, \cite{Sak2}*{Theorem 6.2.1}; 
see also \S\ref{prob-measures} below) 
Every (complete) separable metric space $M$ is isometric to a linearly independent subset
of a (complete) separable normed vector space.

\item (Enflo et al.; see \cite{Enf}, \cite{BL}*{8.17 and 9.21}) $q_0$ (and hence $c_0$) does not uniformly 
embed in $l_2$ nor in any other Hilbert space $l_2(S)$.
Each $l_p$, $p>2$, also does not uniformly embed in any Hilbert space.
See \cite{GLZ}*{p.\ 5} for further information.

\item The unit ball of each $L_p(\mu)$, $1\le p<\infty$, is uniformly homeomorphic to that of the Hilbert space
$L_2(\mu)$ (see \cite{BL}*{9.1}).
Hence $q_0$ and $l_p$, $p>2$, also do not uniformly embed in this unit ball.
In particular, $l_p$ for $2<p<\infty$ does not uniformly embed in its own unit ball.
In contrast, $l_2$ uniformly embeds in its unit ball \cite{BL}*{8.11} and so does $c_0$ (by Aharoni's theorem).
\end{roster}
\end{remark}

\subsection{Uniform local compacta}
By a {\it local compactum} we mean a locally compact separable metrizable space
or equivalently a metrizable topological space that is a countable union of
compacta $X_i$ such that each $X_i\incl\Int X_{i+1}$; or equivalently
the complement to a point in a compactum; or equivalently the complement to
a compactum in a compactum (see e.g.\ \cite[3.3.2, 3.8.C, 3.5.11]{En}).

By a {\it uniform local compactum} we mean a metrizable uniform space that
has a countable uniform cover by compacta.
A map $f$ from a uniform local compactum $X$ into a metric space is uniformly
continuous if and only if it is continuous and every two proper maps $\phi,\psi\:\N\to X$
such that $d(\phi(n),\psi(n))\to 0$ as $n\to\infty$ satisfy
$d(f\phi(n),f\psi(n))\to 0$ as $n\to\infty$.
Every closed subset of a finite-dimensional Euclidean space is a uniform
local compactum.

Clearly, every uniform local compactum is complete and its underlying
topological space is a local compactum.
The converse to the latter is false: $\N\x\N$ with the metric
$d((m,n),(m',n'))=1$ if $m\ne m'$ and $d((m,n),(m,n'))=\frac1m$ is not
a uniform local compactum, although it is topologically discrete.
However, an ANRU (see \S\ref{ARs}) whose underlying topological space is
a local compactum is a uniform local compactum \cite[5.4]{I2}.

\begin{proposition}\label{A.9}
A local compactum is homeomorphic to a uniform local compactum.
\end{proposition}

A version for not necessarily metrizable spaces goes back to A. Weil
(cf.\ \cite[p.\ 590]{BHH}, see also \cite[Exer.\ IX.1.15(d)]{Bou}).

\begin{proof}[Proof]
Let $d$ be some metric on the one-point compactification of the given local compactum $X$.
Then $d'(x,y)=d(x,y)+|\frac1{d(x,\infty)}-\frac1{d(y,\infty)}|$ is a complete metric on $X$, inducing 
the same topology as $d$, moreover every ball (of finite radius) in $(X,d')$ is compact.
\end{proof}

The constructed uniform structure is not canonical: it depends in general on
the choice of $d$ in its uniform equivalence class.

\begin{theorem}[{\cite[1.15]{I1}, \cite[Prop.\ 1]{Ri}}]\label{A.10}
Uniform local compacta are star-finite.
\end{theorem}

\begin{proof} Let $Q=\{q_1,q_2,\dots\}$ be a countable dense subset of $X$,
and suppose that closed $3\eps$-balls in $X$ are compact.
Define $Z\incl Q$ by $q_1\in Z$, and $q_{i+1}\in Z$ unless $q_{i+1}$ is contained
in the union of the open $\eps$-balls about the points of
$\{q_1,\dots,q_i\}\cap Z$.
On the other hand, let $\{V_\alpha\}$ be the cover of $X$ by the open
$\eps/2$-balls about all $x\in X$.
If $K$ is the closed $3\eps$-ball about some $z_0\in Z$, its cover
$\{V_\alpha\cap K\}$ has a finite subcover $\{W_j\}$.
Since $Z$ contains no pair of points at distance $<\eps$ from each other, each
$W_j$ contains at most one $z\in Z$.
So $K\cap Z$ is finite.
Then the cover of $X$ by the $\frac32\eps$-balls about the points of $Z$ is
star-finite; clearly, it is also uniform.
\end{proof}

\subsection{Bornological conditions}\label{bornology}
Let $X$ be a uniform space.

We recall that $X$ is precompact if and only if for each uniform cover $C$ of
$X$ there exists a finite set $F\incl X$ such that $\st(F,C)=X$.
It is easy to see, by using completions, that a uniformly continuous image of a precompact
space is always precompact.

A subset $S$ of $X$ is called {\it $\R$-bounded} if for each uniform cover $C$
of $X$ there exists a finite set $F\incl X$ and a positive integer $n$ such that
$S\incl\st^n(F,C)$; here $\st^0(F,C)=F$, and $\st^{n+1}(F,C)=\st(\st^n(F,C),C)$
(see \cite{Hu2}, \cite[Exer.\ II.4.7]{Bou}; an anticipatory definition is found
in \cite{Hu3} and in an older edition of \cite{Bou}).
It is not hard to see that $\R$-bounded subsets of a uniformly locally precompact
space are precompact (cf.\ \cite[1.18]{He1}).
On the other hand, every uniformly contractible uniform space (for instance,
$U(X,I)$ for every uniform space $X$) is $\R$-bounded (as a subset of itself).

It turns out that $S\incl X$ is $\R$-bounded if and only if every uniformly
continuous map of $X$ into the real line $\R$ with the Euclidean uniformity is
bounded on $S$ \cite[1.14]{He1}
(see also \cite[Theorem 2]{At1}, \cite[Theorem 7]{At2}).
In particular, $\R$-boundedness is preserved under uniformly continuous
surjections.
In fact, $S$ is $\R$-bounded if and only if it has finite diameter with respect
to every uniformly continuously pseudo-metric on $X$; when $X$ is metrizable,
$S$ is $\R$-bounded if and only if it has finite diameter with respect to every
metric on $X$ \cite[1.12, 1.13]{He1}.
Further references on $\R$-bounded spaces include \cite{He2}, \cite{FR} and
those therein.

A uniform space $X$ is called {\it fine-bounded} if every uniformly
continuous map from $X$ to a fine uniform space $F$ has precompact image.
Equivalently (see \cite{Ta2}), if every uniformly continuous map from $X$
to $\R$ with the fine uniformity is bounded.

A uniform space $X$ is called {\it $\N$-bounded} if it admits no countable
cover by uniformly disjoint sets (see \cite{Ta2}).
Equivalently, if every uniformly continuous map of $X$ into the countable
uniformly discrete space $\N$ is bounded.
Clearly,

\medskip
\centerline{
compact\ \imp\ precompact\ \imp\ $\R$-bounded\ \imp\ fine-bounded\ \imp\ $\N$-bounded.}
\medskip

The fine uniformity on $\R$ is $\N$-bounded but not fine-bounded, and
the Euclidean uniformity on $\R$ is fine-bounded but not $\R$-bounded.
However, metrizable $\N$-bounded spaces are fine-bounded \cite{Ta1}
(see also \cite[proof of Theorem 3]{RT}).

\newpage
\part{QUOTIENTS AND QUOTIENT MAPS}\label{quotients}

\section{Metrizability of adjunction spaces}\label{metrizability}

\subsection{Bi-uniformly continuous and uniformly open maps}
A map $f\:X\to Y$ between uniform spaces is called {\it bi-uniformly continuous} if the images of uniform 
covers of $X$ form a basis of the uniformity of $Y$.%
\footnote{Bi-uniformly continuous maps are found under this name in \cite{Pla} and under various other 
names in \cite{Hi2}, \cite{Vi}, \cite{Ja}, \cite{BDLM1}.}
The following are easily seen to be equivalent for a map $f\:X\to Y$ between uniform
spaces:
\begin{itemize}
\item $f$ is bi-uniformly continuous;
\item a cover of $Y$ is uniform if and only if it is the image of a uniform cover
of $X$;
\item $f$ is uniformly continuous and sends every uniform cover to a uniform cover.
\end{itemize}

A uniform quotient map that is not bi-uniformly continuous is
e.g.\ $S^1\sqcup S^1\to S^1\vee S^1$.

A map $f\:X\to Y$ between uniform spaces is called {\it uniformly open} if it is
surjective and uniformly continuous, and for every uniform cover $C$ of $X$ there
exists a uniform cover $D$ of $Y$ such that for each $x\in X$,
$\st(f(x),D)\subset f(\st(x,C))$.
The latter condition can be interpreted as a ``chain lifting property''
(compare \cite{BDLM1}): given $x'$ and $y'$ lying in a single element of $D$,
if $x'=f(x)$ then $y'=f(y)$, where $x$ and $y$ lie in a single element of $C$.

\begin{example}\label{circles} Let $X_n=\{(x,y)\in\R^2\mid x^2+y^2=n^2\}$, and let
$X=\bigcup_{n\in\N} X_n$ with the Euclidean uniformity.
(So the $X_n$ form a uniformly disjoint collection.)
Let $f\:X\to X_1$ be the radial projection.
Then $f$ is bi-uniformly continuous, but not uniformly open.
\end{example}

\begin{proposition} \cite{Hi2}, \cite{Ja}, \cite{BDLM1} \label{simple-quotient}
Let $f\:X\to Y$ be a map.

(a) If $f$ is uniformly open, then it is bi-uniformly continuous.

(b) If $f$ is bi-uniformly continuous, then it is a quotient map.
\end{proposition}

\begin{proof}[Proof. (a)]
Given a cover $E$ of $X$, let $C$ be a barycentric refinement of $E$, and let $D$ be
given by the definition of uniform openness.
Then each element of $f(E)$ contains $f(\st(x,C))$ for some $x\in X$, which
in turn contains $\st(f(x),D)$ and consequently an element of $D$.
Thus $f(E)$ is refined by the uniform cover $D$.
\end{proof}

\begin{proof}[(b)] Clearly $f$ is surjective.
Since $f$ is uniformly continuous, every uniform cover $D_1$ of $Y$ is included in
a sequence of covers $D_i$ of $Y$ such that $f^{-1}(D_i)$ is uniform and $D_{i+1}$
barycentrically refines $D_i$ for each $i$.
Conversely, if $D_1$ is a cover of $Y$ such that $f^{-1}(D_1)$ is uniform, then
$D_1=f(f^{-1}(D_1))$ is uniform.
\end{proof}

The following aims to clarify/strengthen \cite[pp.\ 24-25]{Ja} and
\cite[2.9]{BDLM1}.

\begin{proposition}\label{simple-quotient2} Let $f\:X\to Y$ be a quotient map.

(a) $f$ is bi-uniformly continuous if and only if for every uniform cover $D$ of $X$
there exists a uniform cover $E$ of $X$ such that $f(E)$ barycentrically refines $f(D)$.

(b) $f$ is uniformly open if and only if for every uniform cover $C$ of $X$ there
exists a uniform cover $E$ of $X$ such that for each $x\in X$,
$\st(f(x),f(E))\subset f(\st(x,C))$.
\end{proposition}

The latter condition in (b) can be reinterpreted in terms of the relation $R$ such that
$Y=X/R$ \cite[2.15]{Ja}: given $\tilde x$ and $\tilde y$ lying in a single element of $E$,
if $\tilde xRx$ then $\tilde yRy$, where $x$ and $y$ lie in a single element of $C$.

A similar interpretation of the condition in (a) is possible (see \cite[2.13]{Ja}).

\begin{proof}[(a)] The assertion is equivalent to saying that $f$ is bi-uniformly
continuous if and only if the images of uniform covers of $X$ form a base of some
uniformity $u$ on the underlying set of $Y$.
In this form, ``only if'' is trivial, and ``if'' follows by observing that $f$ is
bi-uniformly continuous with respect to $u$, and then $u$ is automatically
the quotient uniformity by Proposition \ref{simple-quotient}(b).
\end{proof}

\begin{proof}[(b)]
The ``only if'' direction follows by taking $E=f^{-1}(D)$.
To prove the converse it suffices to show that $f$ is bi-uniformly continuous
(for the latter would imply that $f(E)$ is uniform).
We will do it using (a).
Given a uniform cover $D$ of $X$, let $C$ be its barycentric refinement.
Then each element of $f(D)$ contains $f(\st(x,C))$ for some $x\in X$, which
in turn contains $\st(f(x),f(E))$.
Thus $f(E)$ barycentrically refines $f(D)$.
\end{proof}

\subsection{Topology of quotient uniformity}
A uniformly open map is obviously open in the underlying topologies, and in
particular it is a quotient map in the topological sense.

It is not hard to see that a bi-uniformly continuous map with compact
point-inverses is a topological quotient map \cite[Theorem 2]{Hi2}.
Both hypotheses are essential:

\begin{example} \label{quotient topology}
Consider the following subsets of the plane:
$A=\{(i,0)\mid i\in\N\}$ and $B=\{(i,\frac1i)\mid i\in\N\}$.
Let $X=A\cup B$ with the Euclidean uniformity, and let $Y=X/A$ with
the quotient uniformity.
Then the uniform quotient map $f\:X\to Y$ is bi-uniformly continuous, and
$Y$ is (uniformly) homeomorphic to the one-point compactification of $A$.
Thus the topology of $Y$ is different from the quotient topology, which
is discrete.
\end{example}

\begin{example} \label{quotient topology-2}
Let $A=\{\frac1n\mid n\in\N\}\subset\R$ and
$B=\{(\frac1n,1+\frac1n)\mid n\in\N\}\subset\R^2$.
Let $X=A\x[0,1]\cup B\cup\{(0,0)\}$ with the Euclidean uniformity, and let
$Y$ be the adjunction space $X\cup_{\pi^+} (A\cup\{0\})$ with the quotient uniformity
(see \S\ref{adjunction space} for the definition of adjunction space),
where $\pi^+\:A\x [0,1]\cup\{(0,0)\}\to A\cup\{0\}$ is the continuous extension of 
the projection $A\x [0,1]\to A$.
Then $0$ is a cluster point of the image of $B$ in $Y$; on the other hand, $B$ is closed in $X$, hence so is 
the image of $B$ in the adjunction space with the quotient topology.
\end{example}

\subsection{Type $n$ quotient maps}
We say that a map $f\:X\to Y$ between uniform spaces is
a {\it quotient map of finite type} if it is uniformly continuous and surjective, and there exists 
an $n$ such that for every uniform cover $C$ of $X$, the cover $f_n(C)$ is uniform, where $f_0(C)$ 
is the cover of $Y$ by singletons and $f_i(C)=f\bigg(\st\Big(f^{-1}\big(f_{i-1}(C)\big),C\Big)\bigg)$.
Specifically we will say that $f$ is of {\it type $n$}.

Type $1$ quotient maps were introduced in \cite{V}.
Bi-uniformly continuous maps are obviously type $1$ quotient maps.
The converse does not hold, by considering the map $S^1\sqcup S^1\to S^1\vee S^1$.

\begin{theorem}\label{finite type}
Let $f\:X\to Y$ be a quotient map of finite type.
Then

(a) $f$ is a quotient map;

(b) if $X$ is pseudo-metrizable, then so is $Y$.
\end{theorem}

\begin{proof}[Proof. (a)] We need to show that the uniformity of $Y$ includes every
sequence of covers $D_i$ such that $f^{-1}(D_i)$ is uniform and $D_{i+1}$ star-refines $D_i$ for each $i$.
(That all uniform covers of $Y$ are of the form $D_1$ follows since $f$ is uniformly
continuous by the hypothesis.)
Indeed, if $C=f^{-1}(D_n)$, then $f_1(C)=D_n$, $f_{i+1}(C)=\st(f_i(C),D_n)$, and
so $f_n(C)$ refines $D_1$.
\end{proof}

\begin{proof}[(b)] We will use Theorem \ref{A.1}.
Let $E_1,E_2,\dots$ be a pseudo-fundamental sequence of uniform covers for $X$
(i.e., each $E_{i+1}$ barycentrically refines $E_i$, and each uniform cover of $X$ is
refined by some $E_i$), and let $D_1$ be a uniform cover of $Y$.
There exist uniform covers $D_i$, $i=2,3,\dots$ of $Y$ such that each
$D_{i+1}$ star-refines $D_i$.
Then $f^{-1}(D_n)$ is refined by some $E_m$.
Hence $f_n(E_m)$ refines $D_1$.
Thus every uniform cover of $Y$ is refined by $f_n(E_m)$ for some $m$.
So the uniformity of $Y$ has a countable basis, and consequently a basis that is
a pseudo-fundamental sequence of covers.
\end{proof}

\begin{corollary} \label{orbit space metrization}
If $f\:X\to Y$ is bi-uniformly continuous and $X$ is pseudo-metrizable, then
so is $Y$.
\end{corollary}

\subsection{Order $n$ quotient maps}
A map $f\:X\to Y$ between two uniform spaces will be called an {\it order $n$ quotient map,}
where $2n\in\N\cup\{0\}$, if it is uniformly continuous and surjective, and for each uniform
cover $C$ of $X$ there exists a uniform cover $D$ of $Y$ such that for any pair of points
$x_0,x_{2n+1}\in X$ such that $f(x_0)$ and $f(x_{2n+1})$ lie in a single element of $D$
there exists a sequence $x_1,x_2,\dots,x_{2n}$ of points in $X$
such that $f(x_{2i})=f(x_{2i+1})$; and $x_{2i+1}$ and $x_{2i+2}$ lie in a single
element of $C$ for each $i$.
When $n$ is integer, we call such a sequence a {\it $C$-chain of length $n$ between}
$f(x_0)$ and $f(x_{2n+1})$.

It is easy to see that order $1/2$ quotient maps coincide with uniformly open maps.
If a map is bi-uniformly continuous, clearly it is a quotient map of order $1$.

\begin{proposition}\label{finite order}
(a) A type $n$ quotient map is an order $2n$ quotient map.

(b) An order $n$ quotient map is a type $n$ quotient map.
\end{proposition}

\begin{proof} Let $f\:X\to Y$ be a uniformly continuous map.
Let $C$ be a uniform cover of $X$.
Given an $x\in Y$, the element $f_n(x)$ of $f_n(C)$ consists of all $y\in Y$
such that there exists a $C$-chain of length $n$ between $x$ and $y$.
Thus every pair of points $y$ and $y'$ lying in a single element $f_n(x)$ of
$f_n(C)$ is connected by a $C$-chain of length $2n$.

Conversely, suppose that $D$ is a uniform cover of $Y$ such that every pair of
points lying in a single element of $D$ is connected by a $C$-chain of length $n$.
Given a $U\in D$, pick some $x\in U$.
Then $U\subset f_n(x)$.
Hence $D$ refines $f_n(C)$.
\end{proof}

\subsection{The $d_n$ metric}
Let $(S,d)$ be a metric space, $Q$ a set, and $f\:S\to Q$ a surjection.
Let
$$d_n(x,y)=\inf_{x=x_0,\dots,x_n=y}\sum d(f^{-1}(x_i),f^{-1}(x_{i+1}))$$ and
$d_\infty(x,y)=\inf_{n\in\N} d_n(x,y)$.
Clearly, $d_\infty$ is a pseudo-metric on $Q$.

\begin{lemma} \cite{Ma} \label{dn=dinfty}
$d_\infty=d_n$ if and only if $d_n$ is a pseudo-metric.
\end{lemma}

\begin{proof}
If $d_n$ satisfies the triangle axiom, then
$d_n=d_k$ for all $k>n$.
Conversely, $d_n=d_{2n}$ implies the triangle axiom for $d_n$.
\end{proof}

\begin{proposition} \label{"uc"}
Let $(S,d)$ be a metric space, $Q$ a set, and $f\:S\to Q$ a surjection.

Then $f\:(S,d)\to (Q,d_\infty)$ is an order $n$ quotient map if and only if
$(Q,d_\infty)\xr{\id}(Q,d_n)$ is ``uniformly continuous''.
\end{proposition}

Here ``uniformly continuous'' is in quotes since $d_n$ need not be a pseudo-metric;
the meaning of the term is nevertheless the usual one: for each $\eps>0$ there
exists a $\delta>0$ such that every pair of points $y,y'\in Q$ at
$d_\infty$-distance $<\delta$ satisfies $d_n(y,y')<\eps$.

\begin{proof} $f\:(S,d)\to (Q,d_\infty)$ is an order $n$ quotient map if and only if
for each $\eps>0$ there exists a $\delta>0$ such that every pair of points
$y,y'\in Q$ at $d_\infty$-distance $<\delta$ is connected by a $C_\eps$-chain
of length $n$, where $C_\eps$ is the cover of $(S,d)$ by balls of radius $\eps$.

If $d_n(y,y')<\eps$, then $y$ and $y'$ are connected by a $C_\eps$-chain of length $n$.
The latter, in turn, implies that $d_n(y,y')<n\eps$.
\end{proof}

From Lemma \ref{dn=dinfty} and Proposition \ref{"uc"} we obtain

\begin{corollary} \label{dn}
Let $(S,d)$ be a metric space, $Q$ a set, and $f\:S\to Q$ a surjection.

If $d_n$ is a metric, then $f\:(S,d)\to (Q,d_n)$ is an order $n$ quotient map.
\end{corollary}

Taking into account Proposition \ref{finite order}(b) and Theorem \ref{finite type}(a),
we obtain the following

\begin{theorem} \label{metrization lemma}
Let $(S,d)$ be a metric space, $Q$ a set, and $f\:S\to Q$ a surjection.

(a) {\rm (Marxen \cite{Ma})} If $d_n$ is a metric, then $f\:(S,d)\to (Q,d_n)$ is
a uniform quotient map.

(b) If $(Q,d_\infty)\xr{\id}(Q,d_n)$ is ``uniformly continuous'', then
$f\:(S,d)\to (Q,d_\infty)$ is a uniform quotient map.
\end{theorem}

The proof of part (a) given above seems to be simpler than the original proof.
Part (b) is used in a sequel to this paper dealing with uniform polyhedra \cite{M3}.

\begin{corollary}[Marxen \cite{Ma}]\label{A.5f} Let $X$ be a metrizable uniform
space and $S_\alpha$ a uniformly discrete family of closed subsets of $X$.
Then the quotient $Q$ of $X$ by the equivalence relation whose only
non-singleton equivalence classes are the $S_\alpha$ is metrizable.

Moreover, if $d$ is a metric on $X$ then $d_2$ is a metric on $Q$.
\end{corollary}

\begin{proof} It is easy to see that $d_\infty=d_2$
(cf.\ \cite{Ma}; see also \cite{V}) and that $d_2(x,y)\ne 0$ whenever $x\ne y$.
\end{proof}

\subsection{Extension of bounded metrics}

\begin{lemma}[{Nhu \cite[2.5]{Nhu3}}] \label{A.M} If $X$ is a metrizable uniform space and $A\subset X$ 
is a closed subset, then every bounded metric inducing the uniform structure of $A$ extends to 
a bounded metric inducing the uniform structure of $X$.
\end{lemma}

The proof given below depends on Corollary \ref{A.5f}, in contrast to Nhu's paper,
which the author became aware of long after writing up the proof below.
An extension theorem for bounded uniformly continuous pseudo-metrics (not necessarily inducing the given
uniform structures) is well-known \cite[Lemma 1.4]{I1} (a direct proof), \cite[III.16]{I3};
see \cite{Ya}*{A.5} and \cite{Nhu3} concerning the unbounded case.
An extension theorem for (possibly unbounded) metrics inducing the given topologies was proved by 
F. Hausdorff (see \cite{Sak2}*{6.2.3}, \cite{Nhu3}*{\S1}).
See also \cite{BB} and \cite{Hus}.

\begin{proof} We may assume that the diameter of $A$ does not exceed $1$.
Then there exists an isometric embedding of $A$ into $U(A,I)$ (see Remark \ref{A.2*}(i)),
which in turn extends to a uniformly continuous map $f\:X\to U(A,I)$
(see Theorem \ref{basic ARU}).%
\footnote{Note that $U(A,I)$ is inseparable (unless $A$ is compact).
One can remain within the realm of separable spaces here (as long as $X$ itself
is separable) at the cost of using a more subtle uniform embedding
(see Theorem \ref{aharoni}) along with a more subtle extension
(see Corollary \ref{q_0 ARU}).}
Pick some metric on $X/A$, which exists by Corollary \ref{A.5f}, and consider
the $l_\infty$ product metric (or the $l_1$ product metric) on
$U(A,I)\x(X/A)$.
Then $g=f\x q\:X\to U(dA,I)\x (X/A)$ is injective, uniformly continuous, and
isometric on $A$.
It suffices to prove that its inverse is uniformly continuous.
Suppose that $x_i,y_i\in X$ are such that $d(g(x_i),g(y_i))\to 0$ as
$i\to\infty$ but $d(x_i,y_i)$ is bounded below by some $\eps>0$.
By passing to a subsequence and interchanging $x_i$ with $y_i$, we may assume
that either
\begin{roster}
\item $d(x_i,A)\to 0$ and $d(y_i,A)\to 0$ as $i\to\infty$, or
\item $d(x_i,A)\to 0$, while $d(y_i,A)$ is bounded below by some $\delta>0$, or
\item $d(x_i,A)$ and $d(y_i,A)$ are bounded below by some $\delta>0$.
\end{roster}
In the second case, $d(q(x_i),q(y_i))$ is bounded below by $\delta/2$ for
sufficiently large $i$, hence $d(g(x_i),g(y_i))$ is bounded below.
In the third case, let $Z$ be the complement in $X$ to the open
$\delta$-neighborhood of $A$.
Then $q|_Z$ is a uniform embedding, so $d(q(x_i),q(y_i))$ is bounded below
since $d(x_i,y_i)$ is.
In the first case, there exist $x_i',y_i'\in A$ such that $d(x_i,x_i')\to 0$
and $d(y_i,y'_i)\to 0$ as $i\to\infty$.
Then $d(x'_i,y'_i)$ is bounded below by $\eps/2$ for sufficiently large $i$.
Hence $d(f(x'_i),f(y'_i))$ is bounded below.
By the triangle axiom, $d(f(x_i),f(y_i))$ is bounded below.
Thus $d(g(x_i),g(y_i))$ is bounded below.
\end{proof}

\subsection{Amalgamated union} Let $f\:A\to X$ and $g\:A\to Y$ be
embeddings between uniform spaces with closed images.
The pushout of the diagram $X\xl{f} A\xr{g} Y$ is called the {\it amalgamated
union} of $X$ and $Y$ along the copies of $A$, and is denoted $X\cup_A Y$;
a more detailed notation is $X\cup_h Y$, where $h=gf^{-1}$ is the uniform
homeomorphism between $f(A)$ and $g(A)$.
Thus $X\cup_A Y$ is the quotient of $X\sqcup Y$ by the equivalence
relation $x\sim y$ if $\{x,y\}=\{f(a),g(a)\}$ for some $a\in A$.
It is not hard to see that $X$ and $Y$ (and consequently also $A$) can be
identified with subspaces of $X\cup_A Y$.

\begin{corollary}\label{A.5g}
Let $X$ and $Y$ be metrizable uniform spaces and $h$ is a uniform homeomorphism
between closed subsets $A\incl Y$ and $B\incl Z$, then the amalgamated union
$X\cup_h Y$ is metrizable.

Moreover, there exists a metric on $X\sqcup Y$ such that $h$ is an isometry; and if
$d$ is any such metric, then $d_2$ is a metric on $X\cup_h Y$.
\end{corollary}

The metric $d_2$ on the amalgamated union (not identified as a metric of the
uniform quotient) has been considered by Nhu \cite{Nhu1}, \cite{Nhu2} and
elsewhere \cite{BH}, \cite{BBI}, \cite{Ve}.

On the other hand, in Corollary \ref{amalgam-metrizable} below we give
an alternative proof of the metrizability of the amalgamated union, which is more
direct (based on Theorem \ref{A.1}) but does not produce any nice explicit metric.
It is, however, the explicit $d_2$ metric that will be crucial for applications
of Corollary \ref{A.5g}.

\begin{proof} The existence of a metric such that $h$ is an isometry
follows from Lemma \ref{A.M}.
If $h$ is an isometry, then it is easy to see that $d_\infty=d_2$ and that
$d_2(x,y)\ne 0$ whenever $x\ne y$ (cf.\ \cite[I.5.24]{BH}).
\end{proof}

\subsection{Adjunction space}\label{adjunction space}
Let $X$ and $Y$ be uniform spaces, $A\incl X$ and $f\:A\to X$ a uniformly
continuous map.
The {\it adjunction space} $X\cup_f Y$ is the pushout (in the category of
pre-uniform spaces) of the diagram $X\supset A\xr{f}Y$.
In other words, $X\cup_f Y$ is the quotient of $X\sqcup Y$ by the
equivalence relation $x\sim y$ if $x\in A$ and $f(x)=y$.
Its equivalence classes are $[y]=\{y\}\cup f^{-1}(y)$ for $y\in Y$ and
$[x]=\{x\}$ for $x\in X\but A$.

We are now ready to prove the main result of this section.

\begin{theorem}\label{adjunction} Let $X$ and $Y$ be metrizable (complete)
uniform spaces and $A$ a closed subset of $X$.
If $f\:A\to Y$ is a uniformly continuous map, then $X\cup_f Y$ is metrizable
(and complete).

Moreover, there exists a metric $d$ on $X\sqcup Y$ such that $f$ is
$1$-Lipschitz; for any such metric, $d_3$ is a metric on $X\cup_f Y$.
\end{theorem}

We note that this includes Corollaries \ref{A.5f} and \ref{A.5g} as special cases, apart from
the explicit metrics.

It might be possible to prove the metrizability of the adjunction space, without
producing any nice explicit metric, by a more direct method, akin to the proof of
Corollary \ref{amalgam-metrizable} below.
In fact, G. L. Garg gave a (rather complicated) construction of a metrizable
uniform space with properties resembling those of the adjunction space
\cite{Gar}; in his Zentralblatt review of Garg's paper, J. R. Isbell claims,
without giving any justification, that ``the author constructs the pushout of
a closed embedding and a surjective morphism in the category of metrizable
uniform spaces''.
We note that due to the nature of Garg's construction, a proof of Isbell's claim
would be unlikely to produce any nice explicit metric on the adjunction space.
It is, however, the explicit $d_3$ metric that will be crucial for the applications
of Theorem \ref{adjunction} in \S\ref{join, etc} below.

\begin{proof}
Let $d_X$ and $d_Y$ be some bounded metrics on $X$ and $Y$.
Let $D(a,b)=d_X(a,b)+d_Y\big(f(a),f(b)\big)$ for all $a,b\in A$.
Then $D$ is a bounded metric on $A$; clearly it is uniformly equivalent to
the restriction of $d_X$ over $A$; and $d_Y(f(a),f(b))\le D(a,b)$ for all
$a,b\in A$.
(Note that $D$ is nothing but the restriction of the $l_1$ product metric
to the graph of $f$.)
By Lemma \ref{A.M} $D$ extends to a metric, also denoted $D$, on the uniform
space $X$.
We may assume that $X$ and $Y$ have diameters $\le 2$ with respect to $D$ and $d_Y$. 
Define a metric $d$ on $X\sqcup Y$ by $d(x,x')=D(x,x')$ if $p,q\in X$,
by $d(y,y')=d_Y(y,y')$ if $y,y'\in Y$ and by $d(x,y)=1$ whenever $x\in X$
and $y\in Y$.
Then $d(f(a),f(b))\le d(a,b)$ for $a,b\in A$.
Since $A$ is closed, $d(x,A)>0$ for every $x\in X\but A$.
We have $$d_1([y],[z])=\min\big(d(y,z),d(f^{-1}(y),f^{-1}(z))\big)=d(y,z)$$
for all $y,z\in f(X)$, and it follows that
$$d_\infty([y],[z])=d_1([y],[z])=d(y,z)$$
for all $y,z\in Y$.
Therefore $Y$ is identified with a subspace of $X\cup_f Y$, and it is easy to see
that $$d_\infty(x,[y])=d_2(x,[y])=\inf_{z\in f(X)}d(x,f^{-1}(z))+d(z,y)\qquad
\ge d(x,A)>0$$
for all $x\in X\but A$ and $y\in Y$ and
\begin{multline*}
d_\infty([x],[x'])=d_3([x],[x'])=\min(d(x,x'),\inf_{y,z\in f(X)}
d(x,f^{-1}(y))+d(y,z)+d(f^{-1}(z),x'))\\
\ge\min(d(x,x'),d(x,A)+d(x',A))>0
\end{multline*}
for any pair of distinct $x,x'\in X\but A$.
Thus $d_3$ is a metric (see Lemma \ref{dn=dinfty}).
Hence by Theorem \ref{metrization lemma}(a), $X\cup_f Y$ is metrizable.

Now suppose that $X$ and $Y$ are complete.
Let $q_1,q_2,\dots$ be a Cauchy sequence in $Q$.
Suppose that $d_3(q_n,Y)$ is bounded below.
Then each $q_n=\{x_n\}$, where $x_n\in X\but A$ and $d(x_n,A)$ is bounded
below.
Then $x_n$ is a Cauchy sequence in $X$ and so converges to an $x\in X\but A$.
Then $q_n$ converges to $\{x\}$ in $Q$.
If $d_3(q_n,Y)$ is not bounded below, then by passing to a subsequence we
may assume that $d_3(q_n,Y)\to 0$ as $n\to\infty$.
Then there exist $y_1,y_2,\dots\in Y$ such that $d_3(q_n,[y_n])\to 0$ as
$n\to\infty$.
Then $[y_n]$ is a Cauchy sequence in $Q$, whence $y_n$ is a Cauchy
sequence in $Y$ and so has a limit $y\in Y$.
Then $[y_n]$ converges to $[y]$ in $Q$.
Since $q_1,[y_1],q_2,[y_2],\dots$ is a Cauchy sequence and it has a subsequence
converging to $[y]$, it converges to $[y]$ itself.
\end{proof}

\subsection{More on topology of quotient uniformity}

Topological quotient maps include not only open maps, but also closed maps.
Closed continuous maps with compact point inverses are also known as perfect maps.
A map between metrizable spaces is perfect if and only if it is proper, i.e.\ the preimage of every compact set
is compact (see \cite{Sak2}*{2.1.6}, \cite{Dav}*{\S3}).
It is well-known and easy to see that a surjection with compact point-inverses is closed if and only if 
the associated decomposition of the domain is upper semi-continuous (cf.\ \cite{Dav}, see also \cite{Bing}).
It is well-known that a perfect image of a metrizable space is metrizable (see \cite{En}*{4.4.15}).

\begin{theorem} \label{perfect}
If $f\:X\to Y$ is a type $n$ quotient map which is also a perfect map, then the topology of $Y$ 
(i.e.\ the topology of the quotient uniformity) coincides with the quotient topology.
\end{theorem}

\begin{proof} By Lemma \ref{qt vs tqu} every open subset of $Y$ is also open in the quotient topology.
Conversely, let $U$ be a subset of $Y$ open in the quotient topology and let $y\in U$.
We need to show that $U$ contains $\st(y,D)$ for some uniform cover $D$ of $Y$.
Let us fix some metric on $X$.

Let $K=f^{-1}(Y)$ and $U_1=U$.
Assuming that $U_i$ has been defined, let $V_i=f^{-1}(U_i)$.
Since $K$ is compact and $V_i$ is open, $d(K,V_i)>0$.
Let $\eps_i=d(K,V_i)/2$.
Let $W_i$ be the open $\eps_i$-neighborhood of $K$.
Since $f$ is closed, $U_{i+1}\bydef Y\but f(X\but W_i)$ is open.

Let us note that $V_{i+1}\subset W_i$ and the open $\eps_i$-neighborhood of $W_i$ lies in $V_i$.
Hence the $\eps_i$-neighborhood of $V_{i+1}$ lies in $V_i$.
Let $\eps=\min(\eps_1,\dots,\eps_{2n})$ and let $C$ be the cover of $X$ by all open balls of radius $\eps/2$.
Then $\st(V_{i+1},C)\subset V_i$ for each $i\le 2n$.
Let $D=f_n(C)$.
Then $\st(y,D)\subset\st(U_{2n+1},D)\subset U$.
\end{proof}

On the other hand, if $f\:X\to Y$ is a non-closed type $1$ quotient map with compact point-inverses, 
the topology of $Y$ may differ from the quotient topology (by Example \ref{quotient topology-2}, whose 
quotient map is easily seen to be of type $1$).

Theorem \ref{finite type}(a), Proposition \ref{finite order}(b), Corollary \ref{dn} and 
Theorem \ref{perfect} yield

\begin{corollary} \label{dn-perfect}
If $d_n$ is a metric and the map $(S,d)\to (Q,\,d_n)$ is perfect, then 
the topology of $d_n$ coincides with the quotient topology.
\end{corollary}

\subsection{More on type $n$ quotient maps}
If $X$ is a compact space and $f\:X\to Y$ is a continuous surjection, then $f$
is a topological quotient map (since it is closed) and hence also a uniform
quotient map (if a composition $X\xr{f}Y\xr{g}Z$ is uniformly continuous,
then $g$ is continuous, and therefore uniformly continuous).
If additionally the quotient pre-uniformity on $Y$ is a uniformity, then $f$ is
a type $1$ quotient map \cite{V}.

Type $1$ and finite type quotient maps are closed under composition:

\begin{proposition}
The composition of a type $m$ quotient map and a type $n$
quotient map is a type $mn$ quotient map.
\end{proposition}

\begin{proof} Let $h\:X\xr{f}Y\xr{g}Z$ be the given
composition.
Given a uniform cover $C$ and a cover $D$ of $X$, let
$f_1(D,C)=\st(f^{-1}(f(D)),C)$ and $f_{i+1}(D,C)=f_1(f_i(D,C),C)$.
Then $h_1(D,C)$ is refined by $f_1(D,C)$, hence $h_i(D,C)$ is refined by
$f_i(D,C)$.
Writing $X$ for the cover of $X$ by singletons, we obtain that
$f(h_m(X,C))$ is refined by $f(f_m(X,C))=f_m(C)$ and therefore is uniform.
Let $C'$ be a barycentric refinement of $f(h_m(X,C))$.

Given a cover $D'$ of $Y$, let $f_1(D',C)=f(\st(f^{-1}(D'),C))$
and let $f_{i+1}(D',C)=f_1(f_i(D',C),C)$.
Then $f_m(D',C)$ is refined by $\st(D',C')$.
If $D''$ is a cover of $Z$, then $h_m(D'',C)$ is refined by
$g(f_m(g^{-1}(D''),C))$.
The latter is in turn refined by the cover $g(\st(g^{-1}(D''),C'))=g_1(D'',C')$.

Finally, let $h_1^m(D'',C)=h_m(D'',C)$ and
$h_{i+1}^m(D'',C)=h_1^m(h_i^m(D'',C),C)$.
Then $h_n^m(D'',C)$ is refined by $g_n(D'',C')$.
Hence $h_{mn}(C)=h_{mn}(Z,C)=h^m_n(Z,C)$ is refined by
$g_n(Z,C')=g_n(C')$ and therefore is uniform.
\end{proof}

Here is an example of a finite type quotient map that is not a type $1$ quotient map.
Other such examples arise in the sequel to this paper \cite{M3}, from representation of uniform
polyhedra as quotients of a uniformly disjoint family of simplices (via
Theorem \ref{metrization lemma}(b) with $n=3$).

\begin{example}\label{type 2}
Let $X_+$ and $X_-$ be non-uniformly discrete metrizable uniform spaces.
Let $T=\bigcup_{i\in\N}\{i\}\x[-\frac1i,\frac1i]\subset\R^2$ with the Euclidean
uniformity.
Let $X=X_+\x X_-\x T$.
For $\eps\in\{+,-\}$, let $A_\eps=\bigcup_{i\in\N}\{i\}\x\{\eps\frac1i\}\subset T$.
Let $f_\eps\:X_+\x X_-\x A_\eps\to X_\eps\x A_\eps$ be the projection.
Finally, set $A=X_+\x X_-\x(A_+\cup A_-)$ and $Y=X_+\x A_+\sqcup X_-\x A_-$, and
let $f=f_+\cup f_-\:A\to Y$.

Thus the adjunction space $X\cup_f Y$ is a union of a uniformly disjoint
family of subspaces $X_i$, where each $X_i$ is uniformly homeomorphic to
the join $X_+*X_-$; however, the family of homeomorphisms cannot be chosen to
be uniformly equicontinuous.

By Theorem \ref{adjunction}, Corollary \ref{dn} and Proposition \ref{finite order}(b),
the quotient map $X\to X\cup_f Y$ is of type $3$.

However, it is not a type $1$ quotient map.
Indeed, the diameter of $X_i$ tends to zero as $i\to\infty$.
On the other hand, if $C_\delta$ is the cover of $X$ by balls of radius $\delta$,
then none of the $X_i$'s will be contained in a single element of the cover
$f_1(C_\delta)$ for sufficiently small $\delta>0$.
\end{example}

\subsection{More on metrizability of quotients}

\begin{proposition} \label{Pelant-Husek}
Let $X$ be a metrizable uniform space with a basis
$C_1,C_2,\dots$ of uniformity, and let $Y$ be a pre-uniform space
and $f\:X\to Y$ a quotient map.
Consider the covers $D_{i0}=f(C_i)$ and
$D_{i,j+1}=\{\st(\{y\},D_{ij})\mid y\in Y\}$ of $Y$.
Given an increasing $\phi\:\N\to\N$, write
$D_\phi=D_{\phi(0),0}\cup D_{\phi(1),1}\cup\dots$.

\begin{parts}
\item A basis of the pre-uniformity of $Y$ consists of
the $D_\phi$ for all increasing $\phi\:\N\to\N$.

\item $Y$ is pseudo-metrizable iff there exist increasing
$\phi_1,\phi_2,\dots\:\N\to\N$ such that for every increasing $\phi\:\N\to\N$
there exists an $n\in\N$ such that $D_{\phi_n}$ refines $D_\phi$.
\end{parts}
\end{proposition}

An analogue of (a) in the language of pseudo-metrics is
established in \cite{PH}.

\begin{proof}[Proof.(a)] Each $D_\phi$ is star-refined by
$D_{\phi\sigma}$, where $\sigma(k)=k+1$.
Each $D_\phi$ and $D_\psi$ are refined by $D_\chi$, where
$\chi(n)=\max(\phi(n),\psi(n))$.
Thus the $D_\phi$ for increasing $\phi$ form a pre-fundamental family of
covers of $Y$.
Finally, suppose that $E_0,E_1,\dots$ is a sequence of covers of $Y$ such that
each $E_{i+1}$ star-refines $E_i$ and each $f^{-1}(E_i)$ is uniform.
Then $f^{-1}(E_i)$ is refined by $C_{\phi(i)}$ for some $\phi\:\N\to\N$, or
equivalently $E_i$ is refined by $f(C_{\phi(i)})=D_{\phi(i),0}$.
Then $E_0$ is refined by $D_\phi$.
\end{proof}

\begin{proof}[(b)] Suppose that $Y$ is metrizable, and
let $E_1,E_2,\dots$ be a basis of its uniformity.
Then by (a) each $E_k$ is refined by $D_{\phi_k}$ for some increasing
$\phi_k\:\N\to\N$, and each $D_\phi$ is refined by some $E_n$.

Conversely, suppose that there exist $\phi_i$ as specified.
Then by (a), every uniform cover of $Y$ is refined by some $D_\phi$,
and by the hypothesis, $D_\phi$ is refined by $D_{\phi_n}$ for some $n$.
Also for each $m$ there exists an $n$ such that $D_{\phi_m\sigma}$ is refined
by $D_{\phi_n}$, where $\sigma(k)=k+1$ so $D_{\phi_m}$ is star-refined by
$D_{\phi_m\sigma}$.
By a renumbering we may assume that $n=m+1$ without loss of generality.
Then $D_{\phi_1},D_{\phi_2},\dots$ is a basis of the pre-uniformity of
$Y$.
\end{proof}

Let us state an analogue of Proposition \ref{Pelant-Husek}(b) in terms of pseudo-metrics.

\begin{proposition}[Marxen \cite{Ma}] \label{metrization lemma-2}
Let $X$ be a metrizable uniform space and $f\:X\to Y$ a quotient map.
If $Y$ is metrizable, then its uniformity is induced by the metric $d_\infty$ for some
metric $d$ on $X$.
\end{proposition}

\begin{proof} Write $Y=(Q,u)$, and suppose that the quotient uniformity $u$ is induced
by a metric $d_Q$.
Then $d'(x,y)\bydef d_Q(f(x),f(y))$ is a uniformly continuous pseudo-metric on $X$.
The pseudo-metric $d'_\infty$ on $Q$ is defined as above in the case of a metric.
Since $d_Q$ satisfies the triangle axiom, $d'_\infty=d_Q$.
If $d$ is a metric on $X$, then $d''\bydef d'+d$ is a metric, which is uniformly
equivalent to $d$, and $d''_\infty\ge d'_\infty=d_Q$.
Since $d_Q$ is a metric (rather than just a pseudo-metric), so is $d''_\infty$.
Finally, $d''_\infty$ is uniformly continuous on $(Q,d_Q)=(Q,u)$,
and therefore is uniformly equivalent to $d_Q$.
\end{proof}

\begin{remark}
The proof of Proposition \ref{metrization lemma-2} can be modified into a short proof
of Theorem \ref{metrization lemma}(b) under the additional hypothesis that
the quotient uniformity on $Q$ is metrizable.

Indeed, in the notation of the proof of Proposition \ref{metrization lemma-2}, we note
that $d''_\infty\ge d_\infty$.
Then to prove that $d_\infty$ is uniformly equivalent to $d''_\infty$ (and hence
induces $u$), it suffices to show that $d''_\infty$ is uniformly continuous on
$(Q,d_\infty)$.
Now $D(x,y)\bydef d''_\infty(f(x),f(y))$ is a uniformly continuous pseudo-metric
on $X$.
So for each $\eps>0$ there exists a $\delta>0$ such that $d(x,y)<\delta$
implies $D(x,y)<\frac1n\eps$.
By the hypothesis there exists a $\gamma>0$ such that
$d_\infty(v,w)<\gamma$ implies $d_n(v,w)<\frac12\delta$.
Then there exist $x_0,y_1,x_1,\dots,y_n\in X$ such that $f(x_0)=v$,
$f(x_i)=f(y_i)$ for $i=1,\dots,n$, $f(y_n)=w$ and
$d(x_0,y_1)+\dots+d(x_{n-1},y_n)<\delta$.
In particular, each $d(x_i,y_{i+1})<\delta$.
Then $D(x_i,y_{i+1})<\frac1n\eps$; whereas $D(y_i,x_i)=0$ for each $i$.
Hence $d''_\infty(v,w)\le D(x_1,y_n)\le \eps$.
\end{remark}

\section{Cone, join and mapping cylinder}\label{join, etc}

\subsection{Uniform join}
Let $X$ and $Y$ be uniform spaces.
We define their {\it join} to be the quotient $X\x Y\x [-1,1]/\sim$,
where $(x,y,t)\sim(x',y',t')$ if either $t=1$ and $x=x'$ or $t=-1$ and $y=y'$.
Thus we have the pushout diagram
$$\begin{CD}
X\x Y\x\{-1,1\}@.\qquad\subset\qquad@.X\x Y\x[-1,1]\\
@V\pi VV@.@VVV\\
X\sqcup Y@.\longrightarrow@.X*Y,
\end{CD}$$
where $\pi(x,y,-1)=x$ and $\pi(x,y,1)=y$.
By Lemma \ref{quotient-product}(b) and Lemma \ref{quotient-coproduct}(b), $\pi$
is a quotient map.
So if $X$ and $Y$ are metrizable or complete, then by Theorem \ref{adjunction}
so is $X*Y$.
More specifically, given metrics, denoted $d$, on $X$ and $Y$, then a metric on
$X*Y$ is given by
$$d_3([(x,y,t)],[(x',y',t')])=\min
\left\{\begin{matrix}
d(x,x')+d(y,y')+|t-t'|,\\
d(x,x')+(t+1)+(t'+1),\\
d(y,y')+(1-t)+(1-t'),\\
(2-|t-t'|)+2
\end{matrix}\right\}.$$
(In the case $t,t'\in(-1,1)$, the four options correspond respectively to the chains
$$\begin{aligned}
&(x,y,t)&&\to&&& (x',y',t');\\
&(x,y,t)\to& (x,*,-1)&\to& (x',*,-1)&\to& (x',y',t');\\
&(x,y,t)\to& (*,y,1)&\to& (*,y',1)&\to& (x',y',t');\\
&(x,y,t)\to& (*,y,1)&\to& (x',*,-1)&\to& (x',y',t');
\end{aligned}$$
other cases involve subchains of these chains; it is only in the last chain
that none of the three steps can be combined together.)
The formula for $d_3$ implies that if $X$ and $Y$ are metrizable uniform
spaces, and $A\subset X$ and $B\subset Y$ are their subspaces, then $A*B$ is
a subspace of $X*Y$.
In particular, $X*\emptyset$ and $\emptyset*Y$ are subspaces of $X*Y$.
Since they are uniformly disjoint, their union may be identified with
$X\sqcup Y$.

\subsection{Uniform cone}
The uniform join $X*pt$ is denoted $CX$ and called the {\it cone} over
the uniform space $X$.
Thus $CX$ may be viewed as the quotient space $X\x [0,1]/X\x\{1\}$.
Given a metric $d$ on $X$, by Corollary \ref{A.5f}, a metric on $CX$ is given by
$$d_2([(x,t)],[(x',t')])=\min\{d(x,x')+|t-t'|,\,(1-t)+(1-t')\}.$$
If the diameter of $X$ is $\le 2$, then $(X,d)$ is isometric to
the subset $X\x\{0\}$ of $(CX,d_2)$.

\begin{lemma} \label{join} Let $X$ and $Y$ be metrizable uniform spaces.
Then $X*Y$ is uniformly homeomorphic to $CX\x Y\cup_{X\x Y}X\x CY$.
\end{lemma}

\begin{proof} Let us fix some metrics on $X$ and $Y$ such that $X$ and $Y$
have diameters $\le 2$.
Viewing $CX$ as $X\x [0,1]/X\x\{1\}$ and $CY$ as $Y\x [-1,0]/Y\x\{-1\}$,
a metric on $CX\x Y$ is given by
$$d([(x,y,t)],[(x',y',t')])=\min\{d(x,x')+d(y,y')+|t-t'|,\,d(y,y')+(1-t)+(1-t')\},$$
and a metric on $X\x CY$ is given by
$$d([(x,y,t)],[(x',y',t')])=\min\{d(x,x')+d(y,y')+|t-t'|,\,d(x,x')+(t+1)+(t'+1)\}.$$
Since the diameters of $X$ and $Y$ are $\le 2$, the metrics on $CX\x Y$ and
$X\x CY$ both induce the same metric $d(x,x')+d(y,y')$ on $X\x Y$.
Then it is easy to see that the metric on $CX\x Y\cup_{X\x Y}X\x CY$ given by
Corollary \ref{A.5g} coincides (exactly, not just up to uniform equivalence) with the above
metric $d_3$ of $X*Y$.
\end{proof}

\subsection{Euclidean cone} 
Let $d$ be a metric on a set $X$ such that the diameter of $X$ does not
exceed $\pi$.
Then $X\x [0,1]/X\x\{0\}$ can be endowed with a metric $d_E$ modelled on
the Law of Cosines that holds in Euclidean spaces:
$d_E([(x,t)],\,[(y,s)])^2=t^2+s^2-2ts\cos d(x,y)$.
(Think of $X$ as a subset of the Euclidean unit sphere in some Euclidean space.)
It is well-known $d_E$ is indeed a metric, and that it is complete if $d$ is
(see e.g.\ \cite[I.5.9]{BH}, \cite{BBI}).
It is not hard to see that the uniform equivalence class of $d_E$ does not
depend on the choice of $d$ in its uniform equivalence class.

Note that $d_E([(x,t)],[(y,t)])=2t\sin\frac{d(x,y)}2$, which
is bounded below by $\frac12td(x,y)$ and above by $td(x,y)$.
This is somewhat ``smoother'' than with our standard cone metric:
$d_2([(x,1-t)],[(y,1-t)])=\min\{2t,d(x,y)\}$.

\begin{lemma}\label{euclidean cone} If $X$ is a metrizable uniform space,
the Euclidean cone over $X$ is uniformly homeomorphic to $CX$ keeping $X$ fixed.
\end{lemma}

\begin{proof}[Proof]
Given a metric $d$ on $X$, we have the Euclidean metric $d_E$ on
$X\x[0,1]/X\x\{0\}$ and the metric $d_2$ on $CX=X\x[0,1]/X\x\{1\}$.
Define a metric $d_e$ on $CX$ by $d_e([(x,t)],[(y,s)])=d_E([(x,1-t)],[(y,1-s)])$.
Note that if $v$ is the cone vertex, $d_2(v,[(x,t)])=1-t=d_e(v,[(x,t)])$.

Let $x_i,y_i\in CX$ be two sequences.
By passing to subsequences and interchanging $x_i$ with $y_i$, we may assume
that either
\begin{roster}
\item $d(x_i,v)\to 0$ and $d(y_i,v)\to 0$ as $i\to\infty$, or
\item $d(x_i,v)\to 0$, while $d(y_i,v)$ is bounded below by some $\delta>0$, or
\item $d(x_i,v)$ and $d(y_i,v)$ are bounded below by some $\delta>0$,
\end{roster}
where $d(*,v)=d_2(*,v)=d_e(*,v)$.
We claim that $d_2(x_i,y_i)\to 0$ as $i\to\infty$ if and only if
$d_e(x_i,y_i)\to 0$ as $i\to\infty$.
Indeed, both hold in the case (i); none holds in the case (ii); and to do with
the case (iii) in suffices to show that the restrictions of $d_2$ and $d_e$
over $X\x [0,1-\eps]$ are uniformly equivalent for each $\eps>0$.

It is clear that the restriction of $d_2$ over $X\x [0,1-\eps]$ is uniformly
equivalent to the $l_1$ product metric for each $\eps>0$.
The restriction of $d_e$ over $X\x [0,1-\eps]$ is uniformly equivalent to
the $l_2$ product metric for each $\eps>0$, using that
$$d_E([(x,t)],[(y,s)])^2=(t-s)^2+2ts(1-\cos d(x,y))=
(t-s)^2+4ts\sin^2\frac{d(x,y)}2,$$
where $ts\in[\eps^2,1]$, and $\frac z2\le\sin z\le z$ as long as
$z\in [0,\frac\pi2]$.
\end{proof}

\begin{corollary} \label{cone-complete} If $X$ is complete, then so is $CX$.
\end{corollary}

\subsection{Rectilinear cone and join} Let $V$ be a vector space.
Given two subsets $S,T\incl V$, their {\it rectilinear join} $S\cdot T$ is
the union of $S$, $T$, and all straight line segments with one endpoint
in $S$ and another in $T$.
Obviously, rectilinear join is associative.

We define the {\it rectilinear cone} $cS=(S\x\{0\})\cdot\{(0,1)\}\incl V\x\R$.
We identify $S$ with $S\x\{0\}\incl cS$.

Given two vector spaces $V$ and $W$ and subsets $S\incl V$ and $T\incl W$, their
{\it independent rectilinear join}
$ST=(S\x\{0\}\x\{-1\})\cdot (\{0\}\x T\x\{1\})\incl V\x W\x\R$.
We identify $S,T$ with $S\emptyset,\emptyset T\incl ST$.

\begin{lemma}\label{rectilinear join}
Let $V$ and $W$ be normed vector spaces, and let $X\incl V$ and $Y\incl W$.
Then there exists a uniform homeomorphism between $c(XY)$ and $cX\x cY$, taking
$XY$ onto $cX\x Y\cup X\x cY$.
\end{lemma}

\begin{proof} We have $cX=(X\x\{0\})\cdot\{(0,1)\}\incl V\x I$ and
$cY=(Y\x\{0\})\cdot\{(0,1)\}\incl W\x I$, where $I=[0,1]\incl\R$, and
$$c(XY)=(X\x\{0\}\x\{(-1,0)\})\cdot(\{0\}\x Y\x\{(1,0)\})\cdot\{(0,0,0,1)\}
\incl V\x W\x\Delta,$$
where $\Delta=c(\{-1\}\cdot\{1\})$ is the convex hull of $(-1,0)$, $(1,0)$ and
$(0,1)$ in $\R^2$.
Define $\phi\:I\x I\to\Delta$ by $(0,0)\mapsto(0,0)$, $(1,0)\mapsto(1,0)$,
$(0,1)\mapsto(-1,0)$, $(1,1)\mapsto(0,1)$ and by extending linearly to the
convex hull of $(0,0)$, $(1,0)$, $(1,1)$ and to the convex hull of
$(0,0)$, $(0,1)$, $(1,1)$.
Clearly $\phi$ is a homeomorphism, and hence (by compactness) a uniform
homeomorphism.
Therefore $f\:V\x I\x W\x I\to V\x W\x\Delta$, defined by
$f(w,t,v,s)=(w,v,\phi(t,s))$ is a uniform homeomorphism.
It is easy to see that $f^{-1}(c(XY))=cX\x cY$ and $f^{-1}(XY)=cX\x Y\cup X\x cY$.
\end{proof}

\begin{remark}
A less economical, but more elegant variant of the independent rectilinear join
$ST$ would be $(S\x\{0\}\x\{(0,1\})\cdot(\{0\}\x T\x\{(1,0)\})\incl V\x W\x\R^2$.
This has the advantage of being visualizable via discrete probability measures,
as the author learned from T. Banakh.
In the notation of \S\ref{prob-measures} below, this version of independent 
rectilinear join $X*Y$ can be
identified with the subspace of $PM_\delta(X\sqcup Y)$ consisting of the linear
combinations $t\delta_x+(1-t)\delta_y$, where $x\in X$, $y\in Y$ and $t\in [0,1]$.
\end{remark}

\subsection{Rectilinear vs.\ uniform cone}

\begin{lemma}\label{rectilinear cone}
Let $V$ be a normed vector space and $X$ a bounded subset of $V$.
Then the rectilinear cone $cX$ is uniformly homeomorphic to $CX$ keeping $X$ fixed.
\end{lemma}

The proof is based on Lemma \ref{euclidean cone}.

\begin{proof}
Let us consider the $l_\infty$ norm $||(v,t)||=\max\{||v||,|t|\}$ on $V\x\R$.
By scaling the norm of $V$ we may assume that $X$ lies in the unit ball $Q$
of $V$.
Without loss of generality $X=Q$.
Let $f\:Q\x I\to cQ$ send $Q\x\{0\}$ onto the cone vertex $\{(0,1)\}$, and
each cylinder generator $\{x\}\x I$ linearly onto the cone generator $c\{x\}$.

If $x\in Q$, the cone generator $c\{x\}$ has $\R$-length $1$ (i.e.\ its
projection onto $\R$ has length $1$) and $V$-length $\le 1$.
Therefore if $y\in Q$ and $t,s\in I$, the straight line segment
$[f(x,t),f(y,s)]$ with endpoints $f(x,t)$ and $f(y,s)$ has $\R$-length
$|s-t|$; and the segment $[f(z,\max(t,s)),f(z,\min(t,s))]$, where $z=x$ if
$t>s$ and $z=y$ if $t<s$, has $V$-length $\le |s-t|$.
Since $[f(x,\max(t,s)),f(y,\max(t,s))]$ is parallel to $V$, it has $V$-length
$\max(t,s)D$, where $D=||x-y||$.
Hence by the triangle inequality, $[f(x,t),f(y,s)]$ has $V$-length between
$\max(t,s)D-|s-t|$ and $\max(t,s)D+|s-t|$.
Thus $S\bydef ||f(x,t)-f(y,s)||$ is bounded below by
$\max(|s-t|,\max(t,s)D-|s-t|)$ and above by $\max(t,s)D+|s-t|$.

On the other hand, note that $1-\frac{D^2}2\le\cos D\le 1-\frac{D^2}8$
due to $\frac D4\le\sin\frac D2\le\frac D2$ for $D\in[0,\frac\pi2]$.
If we take $d(x,y)=||x-y||$ in the definition of $d_E$, these inequalities
imply that $E^2\bydef d_E([(x,t)],[(y,s)])^2$ is bounded above by $tsD^2+(s-t)^2$
and below by $\frac14tsD^2+(s-t)^2$.

We consider two cases.
When $|s-t|\le\frac12\max(t,s)D$, we have
$S\ge\max(t,s)D-|s-t|\ge\frac12\max(t,s)D$.
When $|s-t|\ge\frac12\max(t,s)D$, we have $S\ge |s-t|\ge\frac12\max(t,s)D$.
Thus in either case $S\ge\frac12\max(t,s)D\ge\frac12\sqrt{ts}\,D$ and
$S\ge |s-t|$.
Hence $4S^2+S^2\ge tsD^2+(s-t)^2\ge E^2$ and so $3S\ge E$.

We consider two cases.
When $|s-t|\le\frac12\max(t,s)$, we get
$ts\ge\max(t,s)^2-\max(t,s)|s-t|\ge\frac12\max(t,s)^2$ and therefore
$E^2\ge\frac14tsD^2\ge\frac18\max(t,s)^2D^2$.
When $|s-t|\ge\frac12\max(t,s)$, we get
$E^2\ge (s-t)^2\ge\frac14\max(t,s)^2\ge\frac18\max(t,s)^2D^2$.
In either case, $E\ge\frac14\max(t,s)D$ and $E^2\ge (s-t)^2$.
Hence $4E+E\ge\max(t,s)D+|s-t|\ge S$.
\end{proof}

\subsection{Rectilinear vs.\ uniform join}

\begin{lemma}\label{amalgam-product1}
If $X$, $Y$ and $Z$ are uniform spaces, $x\in X$ and $y\in Y$, then
the subspace $X\x\{y\}\x Z\cup\{x\}\x Y\x Z$ of $X\x Y\x Z$ is uniformly
homeomorphic to the amalgamated union $X\x Z\cup_{\{x\}\x Z=\{y\}\x Z}Y\x Z$.
\end{lemma}

This can be seen as a purely category-theoretic fact.

\begin{proof} The amalgamated union maps onto the subspace in the obvious way.
Given maps $f\:X\x\{y\}\x Z\to W$ and $g\:\{x\}\x Y\x Z\to W$ agreeing on
$\{x\}\x\{y\}\x Z$, one defines $f\cup g$ using the projections
of $X\x Y\x Z$ onto $X\x Z$ and onto $Y\x Z$.
\end{proof}

\begin{lemma}\label{join-amalgam}
Let $V$ and $W$ be normed vector spaces, and let $X\incl V$ and $Y\incl W$
be bounded subsets.
Then the subspace $cX\x Y\cup X\x cY$ of $cX\x cY$ is uniformly homeomorphic to
the amalgam $cX\x Y\cup_{X\x Y} X\x cY$.
\end{lemma}

\begin{proof}
Lemma \ref{rectilinear cone} implies that a uniform neighborhood of $X$ in $cX$
is uniformly homeomorphic to $X\x I$.
Denote this uniform neighborhood by $U_X$.
Then the uniform neighborhood $U_X\x U_Y$ of $X\x Y$ in $cX\x cY$ is uniformly
homeomorphic to $X\x Y\x I\x I$.
Hence by Lemma \ref{amalgam-product1}, the subspace $U_X\x Y\cup X\x U_Y$ of
$U_X\x U_Y$ is uniformly homeomorphic to the amalgam $U_X\x Y\cup_{X\x Y} X\x U_Y$.
The definition of the amalgam as a pushout yields a uniformly continuous
bijection $f\:cX\x Y\cup_{X\x Y} X\x cY\to cX\x Y\cup X\x cY$ such that $f^{-1}$
is uniformly continuous on $cX\x Y$ and on $X\x cY$.
By the above, $f^{-1}$ is also uniformly continuous on $U_X\x Y\cup X\x U_Y$.
Since these three subsets of $cX\x Y\cup X\x cY$ form its uniform cover, $f^{-1}$
is uniformly continuous.
\end{proof}

\begin{theorem}\label{two joins}
Let $V$ and $W$ be normed vector spaces, and let $X\incl V$ and $Y\incl W$ be
bounded subsets.
Then the independent rectilinear join $XY$ is uniformly homeomorphic to $X*Y$.
\end{theorem}

\begin{proof}
By Lemma \ref{rectilinear join}, $XY$ is uniformly homeomorphic to
the subspace $cX\x Y\cup X\x cY$ of $cX\x cY$.
By Lemma \ref{join-amalgam}, the latter is uniformly homeomorphic to
the amalgam $cX\x Y\cup_{X\x Y} X\x cY$.
By Lemma \ref{rectilinear cone} the latter is in turn uniformly homeomorphic to
the amalgam $CX\x Y\cup_{X\x Y} X\x CY$.
Finally, by Lemma \ref{join}, the latter is uniformly homeomorphic to $X*Y$.
\end{proof}

\begin{corollary} Let $V$ and $W$ be normed vector spaces, $X$ and $Y$ metrizable
uniform spaces, and $f\:X\emb V$ and $g\:Y\emb W$ embeddings onto bounded subsets.
Then the uniform homeomorphism type of the independent rectilinear join $f(X)g(Y)$
does not depend on the choices of $V$, $W$, $f$ and $g$.
\end{corollary}

\begin{corollary} Given metrizable uniform spaces $X$ and $Y$, there exists
a uniform homeomorphism between $C(X*Y)$ and $CX\x CY$ taking $X*Y$ onto
$CX\x Y\cup X\x CY$.
\end{corollary}

\begin{proof} Pick bounded metrics on $X$ and $Y$.
Then there exists an isometric embedding of $X$ onto a bounded subset of
the normed vector space $U_b(X,\R)$ (see Remark \ref{A.2*}(i)), and similarly for $Y$.
Then Theorem \ref{two joins} and Lemmas \ref{rectilinear cone},
\ref{rectilinear join} yield uniform homeomorphisms
$C(X*Y)\to C(XY)\to c(XY)\to cX\x cY\to CX\x CY$.
\end{proof}

\subsection{Mapping cylinder}
Given a uniformly continuous map $f\:X\to Y$ between uniform spaces, its {\it mapping cylinder} $MC(f)$ 
is the adjunction space $X\x I\cup_{f'}Y$, where $f'$ is the partial map 
$X\x I\supset X\x\{1\}=X\xr{f} Y$.

\begin{lemma}\label{mapping cylinder} Let $f\:X\to Y$ be a uniformly continuous map
between metrizable uniform spaces.

(a) If $A$ is a subspace of $X$, then $MC(f|_A)$ is a subspace of $MC(f)$.

(b) $MC(f)$ is uniformly homeomorphic to the image of
$\Gamma_f\x I\cup X\x Y\x\{1\}$ under the map $X\x Y\x I\to CX\x Y$.
\end{lemma}

\begin{proof}[Proof. (a)] Given bounded metrics $d_X$ on $X$ and $d$ on $Y$, define
a new metric $d$ on $X$ by $d(x,x')=d_X(x,x')+d(f(x),f(x'))$.
Clearly, $d$ is bounded and uniformly equivalent to $d_X$, and $d(f(x),f(y))\le d(x,y)$
for all $x,y\in X$.
Using $d$ we define the $l_1$ product metric on $X\x I$, which is bounded; and
hence the metric of disjoint union, also denoted $d$, on $X\x I\sqcup Y$.
Then $d(f'(p),f'(q))\le d(p,q)$ for all $p,q\in X\x\{1\}$, and thus from Theorem
\ref{adjunction} we have the following $d_3$ metric on $MC(f)$:
$d_3([y],[y'])=d(y,y')$, $d_3([(x,t)],[y])=(1-t)+d(f(x),y)$ and
$$d_3([(x,t)],[(x',t')])=\min\{d(x,x')+|t-t'|,\,(1-t)+(1-t')+d(f(x),f(x'))\}$$
for all $y,y'\in Y$, $x,x'\in X$ and $t,t'\in I$.
Clearly, this coincides with the similar $d_3$ metric on $MC(f|_A)$.
\end{proof}

\begin{proof}[(b)] The uniform homeomorphism between $X$ and $\Gamma_f$ yields
a uniform homeomorphism between $MC(f)$ and $MC(\pi|_{\Gamma_f})$, where
$\pi\:X\x Y\to Y$ is the projection.
On the other hand, by the proof of Lemma \ref{join}, $CX\x Y$ is uniformly
homeomorphic to $MC(\pi)$.
Hence the assertion follows from (a).
\end{proof}

It is natural to conjecture that if $\Delta$ is a finite diagram of
metrizable uniform spaces and uniformly continuous maps, its homotopy colimit
is a metrizable uniform space.

\newpage
\part{THEORY OF RETRACTS}\label{absolute retracts}

\section{Classical theory}\label{ARs}

\subsection{Definition}
As long as we have finite coproducts, as well as embeddings and quotient maps in
a concrete category $\C$ over the category of sets, we also have in $\C$
the notions of map extension and neighborhood.
A diagram $X\overset{a}{\topcont}A\xr{f}Y$, where $f$ is a $\C$-morphism
and $a$ is a $\C$-embedding is called a {\it partial $\C$-morphism}, and is
said {\it extend} to a $\C$-morphism $\bar f\:X\to Y$ if the composition
$A\overset{a}{\emb}X\xr{\bar f}Y$ equals $f$.
Given a $\C$-embedding $g\:X\emb Y$, its decomposition into a pair of
$\C$-embeddings $X\overset{n}{\emb} N\overset{n'}{\emb} Y$ is
a {\it neighborhood} of $g$ if there exists a $\C$-embedding $g'\:X'\to Y$
such that $g\sqcup g'\:X\sqcup X'\to Y$ is a $\C$-embedding and
$n\sqcup g'\:N\sqcup X'\to Y$ is a $\C$-quotient map.

An AE($\C$) is any $\C$-object $Y$ such that every partial
$\C$-morphism $X\overset{a}\topcont A\xr{f}Y$ extends to a
$\C$-morphism $X\to Y$.
If $Y$ satisfies this for the special case of $f=\id_Y$, then $Y$ is
called an AR($\C$).
Absolute extensors are also known as ``injective objects'', cf.\
\cite[p.\ 39]{I3}, \cite[9.1; see also 9.6]{AHS}

An ANE($\C$) is any $\C$-object $Y$ such that for every
partial $\C$-morphism $X\overset{a}\topcont A\xr{f}Y$ there exists
a $\C$-neighborhood $A\overset{i}{\emb}N\overset{j}{\emb}X$ of $a$
such that the partial $\C$-morphism $N\overset{i}\topcont A\xr{f}Y$
extends to a $\C$-morphism $N\to Y$.
If $Y$ satisfies the above for the special case of $f=\id_Y$, then $Y$ is
called an ANR($\C$).
Note the obvious implications
$$\begin{CD}
\text{AE($\C$)}\ \ @.\Rightarrow\ \text{ANE($\C$)}\\
\Downarrow\ @.\ \ \Downarrow\\
\text{AR($\C$)}\ \ @.\Rightarrow\ \text{ANR($\C$).}
\end{CD}$$

\subsection{Comparison}
Among full subcategories of the category $\U$ of uniform spaces and uniformly
continuous maps are the categories $\mathcal{CM}$ of compact metrizable spaces
(topological, or equivalently, uniform); $\mathcal{CU}$ of complete uniform spaces;
$\mathcal{MU}$ of metrizable uniform spaces; $\mathcal{SU}$ of uniform spaces whose
topology is separable; their pairwise intersections $\mathcal{MCU}$, $\mathcal{SMU}$,
$\mathcal{SCU}$; and the category
$\mathcal{SMCU}=\mathcal{SU}\cap\mathcal{MU}\cap\mathcal{CU}$
of {\it Polish uniform} spaces.

The A[N]R($\mathcal{CM}$) coincide with the A[N]E($\mathcal{CM}$)
(see \cite[III.3.2(k)]{Hu1}); and also with A[N]Rs in the notation of Borsuk
\cite{Bo}, which are the same as compact A[N]Rs in the notation of
Hu \cite{Hu1} (see \cite[V.1.2]{Bo} or \cite[III.5.3]{Hu1}).

By definition, the A[N]E($\U$) are the same as the A[N]EU of \S\ref{function spaces}.
Following Isbell \cite{I1}, \cite{I2}, we also abbreviate A[N]R($\U$) to A[N]RU.
Using Remark \ref{A.2*}(i), it is easy to see that these are in fact the same as
the A[N]EU, cf.\ \cite[p.\ 111]{I1}, \cite[V.14]{I3}.
We further note that restricting the class of $\U$-embeddings to all
extremal/regular monomorphisms of $\U$, that is, embeddings onto {\it closed}
subsets (in the spirit of the non-uniform A[N]R theory \cite{Bo}, \cite{Hu1})
does not affect the definitions of an A[N]RU and A[N]EU
(see \cite[III.8]{I3}, \cite[proof of I.7]{I1}).
Every A[N]RU is an A[N]R($\mathcal M$), and more generally an A[N]R for
the category of collectionwise normal spaces \cite[4.4]{I2}.

Every A[N]RU is an A[N]R($\mathcal{CU}$) since it is complete.
Conversely, every A[N]R($\mathcal{CU}$) is an A[N]RU by considering
the completion of the domain (see \cite[II.10]{I3}).
Similar (and simpler) arguments show that the A[N]R($\mathcal{MCU}$)
coincide with the A[N]R($\mathcal{MU}$); the A[N]R($\mathcal{SMCU}$)
coincide with the A[N]R($\mathcal{SMU}$); etc.

It is known that the following properties are equivalent: being an A[N]R($\mathcal{MU}$); being 
an A[N]E($\mathcal{MU}$); being a metrizable A[N]RU (see \cite[p.\ 111]{I1}; concerning the 
omitted proof of \cite{I1}*{Corollary 1.5} see \cite{Ya}*{Appendix, Proof of Lemma A.1} or 
\cite{Hus}*{Proof of Proposition 22}).

\begin{lemma}\label{A.2}
(a) The following properties are equivalent: being an A[N]R($\mathcal{SMCU}$);
being an A[N]E($\mathcal{SMCU}$); being a Polish A[N]RU.

(b) \cite[p.\ 624]{I2}
The following properties are equivalent: being an A[N]R($\mathcal{CM}$);
being a compact metrizable A[N]RU.
\end{lemma}

\begin{proof}[Proof. (a)]
The implications
$$\begin{CD}
\text{Polish A[N]EU}\ \ @.\Rightarrow\
\text{A[N]E($\mathcal{SMCU}$)}\\
\Downarrow\ @.\ \ \Downarrow\\
\text{Polish A[N]RU}\ \ @.\Rightarrow\ \text{A[N]R($\mathcal{SMCU}$)}
\end{CD}$$
are obvious.
Thus it suffices to show that if $Y$ is an A[N]R($\mathcal{SMCU}$), then it is an
A[N]EU.
By the hypothesis $Y$ is a Polish uniform space.
Then by Aharoni's Theorem \ref{aharoni} it uniformly embeds onto a closed
subset of $q_0$.
By Corollary \ref{q_0 ARU}, $q_0$ is an AEU.
Hence $f\:A\to Y\incl q_0$ extends to a uniformly continuous map
$\bar f\:X\to q_0$.
Then $\bar f$ composed with a uniformly continuous retraction of
[a uniform neighborhood of $Y$ in] $q_0$ onto $Y$ is the required extension.
\end{proof}

\begin{proof}[(b)]
Obviously, a product of AEU's is an AEU; in particular, the Hilbert cube $I^\infty$
is an AEU.
On the other hand, it is well-known that every compactum embeds in $I^\infty$.
Hence every A[N]R($\mathcal{CM}$) is a [neighborhood] retract of $I^\infty$,
and therefore (similarly to the proof of (a)) it is an A[N]EU.
\end{proof}

\begin{remark}
The implication A[N]R($\mathcal{SMCU}$)\imp A[N]E($\mathcal{SMCU}$) in (a)
could be alternatively proved by embedding into $U(Y,I)$ (see Remark \ref{A.2*}(i)).
The full assertion of (a) could be alternatively proved using the Banach--Mazur
Theorem (see Remark \ref{A.2*}(ii).
\end{remark}

\subsection{Examples from Functional Analysis}

\begin{remark} \label{Banach} Let us mention some results on ARUs and ANRUs
arising from Banach (=complete normed vector) spaces and Frech\'et
(=complete metrizable locally convex vector) spaces.

(a) If a Banach space $V$ is an ANRU, then every [bounded] closed convex
body $B$ in $V$ is an ANRU [resp.\ ARU], cf.\ \cite[comments preceding 3.1]{I2}.

Indeed, by translating we may assume without loss of generality that $0$
is an interior point of $B$.
Then the Minkowski functional $||v||_B=\inf\{r>0\mid v/r\in B\}$ yields
a uniformly continuous retraction $v\mapsto v/\max\{1,||v||_B\}$.
(Note that if $B$ is the unit ball, $||v||_B=||v||$.)
Hence $B$ is an ANRU.

If in addition $B$ is bounded, then $v\mapsto (1-t)v$, $t\in[0,1]$, is
a uniformly continuous null-homotopy of $B$ in itself.
So $B$ is a uniformly contractible ANRU, hence an ARU \cite[1.11]{I1} (see also
a somewhat different proof in Theorem \ref{uniform AR} below).
In fact, $v\mapsto\frac{||v||}{||v||_B}v$ is a uniform homeomorphism
of $B$ onto the unit ball.

(b) A [bounded] complete metric space that is a Lipschitz absolute retract is
an ANRU [resp.\ ARU], see Remark \ref{Michael&Nhu}(a,c) below.
We note the following implications for a Banach space (see
\cite[proof of the second assertion of 3.1]{I2}, \cite[pp.\ 31, 32]{BL},
\cite[Theorem 1, Propositions 2, 3]{Nhu4}, \cite[p.\ 7]{La}):

\medskip
\centerline{\small
$1$-Lipschitz AR\ \imp\ injective as a Banach space\ \imp\ Lipschitz AR}
\medskip

\noindent
and for a separable Banach space:

\medskip
\centerline{\small
injective as a Banach space\ \imp\ injective as a separable Banach space\ \imp\
Lipschitz AR.}
\medskip

The space $L_\infty(\mu)$ for every measure $\mu$, in particular, the space
$l_\infty$ of bounded real sequences, is a $1$-Lipschitz AR
\cite[p.\ 32]{BL}, and hence an ANRU.
In particular, its unit ball is an ARU by (a).
In general, Banach spaces that are $1$-Lipschitz ARs are precisely the spaces
$U(K,\R)$, where $K$ is an extremally disconnected compact Hausdorff space
(Nachbin--Goodner--Kelley--Hasumi; see references in \cite[p.\ 31]{BL} and
\cite[p.\ 7]{La}).

The space $c_0$ is not injective as a Banach space, for it is not
a bounded linear retract of $l_\infty$ (R. S. Philips, 1940; see
\cite[proof of I.2.f.3]{LT}).
However, $c_0$ is injective as a separable Banach space (see \cite[I.2.f.4]{LT})
and also is a uniform retract of $l_\infty$ \cite[Example 1.5]{BL},
\cite[Theorem 6(a)]{Li}.
Either way we get that it is an ANRU, and hence its unit ball
$Q_0=U((\N^+,\infty),([-1,1],0))$ is an ARU.
Note that $(x_1,x_2,\dots)\mapsto(|x_1|,|x_2|,\dots)$ is a uniformly continuous
retraction of $Q_0$ onto $q_0$.

(c) A Banach space $V$ is an ANRU (but not necessarily a Lipschitz AR) if it
has uniformly normal structure
\cite[1.26]{BL}, that is, if there exists a $\gamma<1$ such that every convex
subset of $E$ of diameter $1$ is contained in a ball of radius $\gamma$.
The Banach spaces $L_p(\mu)$ with $1<p<\infty$ and every measure $\mu$,
including the sequence spaces $l_p$ with $1<p<\infty$, have uniformly normal
structure  (see \cite[A.9]{BL} and \cite[II.1.f.1, ``in addition'' part]{LT}).
In particular, they are ANRUs, though this has more direct proofs
\cite[Theorem 8]{Li} (correcting a mistake in \cite[3.1, proof of (c)]{I2}),
\cite[1.29]{BL}.
All their closed subspaces are also Banach spaces with uniformly normal structure.
In particular, the unit balls of all these spaces are ARUs by (a).

The unit ball of $L_1(\mu)$ for every measure $\mu$ is uniformly homeomorphic to
those of $L_p(\mu)$ for $1<p<\infty$ (see \cite[Remark to 1.29]{BL}), hence is
also an ARU.
In particular, the unit ball of $l_1$ is an ARU; in fact, according to
\cite[Remarks after the proof of 3.2]{I3}, $l_1$ itself is an ANRU.
However, there exists a closed subspace of $l_1$ whose unit ball is not an ARU \cite{Ka2}*{8.1}.
A necessary condition for a Banach space $V$ to have its unit ball uniformly homeomorphic to that of an $L_2(\mu)$ 
is that the unit ball $Q_0$ of $c_0$ does not embed in $V$ \cite{BL}*{9.21}.
The slightly stronger necessary condition that the subset $\bigcup_n (Q_0\cap l_\infty^n)\x\{n\}$ of $Q_0\x\N$ 
does not uniformly embed in $V$, while not sufficient in general \cite[9.23]{BL}, is sufficient within 
two large classes of Banach spaces \cite[9.4, 9.7]{BL}.

(d) Given a normed vector space $V$ identified with a subspace of the countably
dimensional vector space $\bigoplus_{i=1}^\infty\R$ and given normed vector spaces
$V_1,V_2,\dots$, let us write $\bigoplus_V V_i$ for the normed vector space
$f^{-1}(V)$, $||x||=||f(x)||_V$, where
$f(x_1,x_2,\dots)=(||x_1||_{V_1},||x_2||_{V_2},\dots)$.
Each $\bigoplus_{c_0} V_i$ is isomorphic to a subspace of $c_0$
(see \cite[(5.2)]{Li}).

The Banach space $(l^1_p\oplus l^2_p\oplus\dots)_{c_0}$, where
$1\le p\le\infty$ and $l^n_p$ is the $n$-dimensional vector space with
the $l_p$-norm, is an ANRU \cite[proof of Theorem 13]{Li}.

On the other hand, none of the following Banach spaces, nor even their unit balls,
is an ANRU: $(l_{p_1}\oplus l_{p_2}\oplus\dots)_{c_0}$ and
$(l_{p_1}\oplus l_{p_2}\oplus\dots)_{l_p}$, as well as
$(l^{n_1}_{p_1}\oplus l^{n_2}_{p_2}\oplus\dots)_{c_0}$
and $(l^{n_1}_{p_1}\oplus l^{n_2}_{p_2}\oplus\dots)_{l_p}$,
where $1<p_k<\infty$ for each $k$ and $p_k\to\infty$ as $k\to\infty$;
in addition $1\le p\le\infty$ and $(\log n_k)/p_k\to\infty$ as $k\to\infty$
\cite[Corollary 1 of Theorem 10 and subsequent Remark 2]{Li}
(see also \cite[Example 1.30]{BL}).

(e) If $V$ is a closed convex subset of a Banach space, the space $H_{cb}(V)$ of all
closed bounded nonempty convex subsets of $V$ with the Hausdorff metric is an ANRU
(in fact, an absolute Lipschitz retract) \cite[3.2]{I2} (cf.\ \cite[Lemma 3]{Li};
for a different proof see \cite[1.7]{BL}).

(f) If $B$ is the unit ball of a Banach space, and $X$ is a uniformly finitistic 
uniform space, then every uniformly continuous partial
map $X\supset A\to B$ extends to a uniformly continuous map $X\to B$ \cite{Vi2}.

(g) Frech\'et spaces are characterized as limits of inverse sequences of
Banach spaces (in the category of topological vector spaces), see
\cite[Proposition I.6.4]{BP}.
If a Frech\'et space $V$ is an ANRU, then every closed (bounded) convex body in $V$
is an ANRU (ARU).
Indeed, if $V$ is the limit of an inverse sequence $\dots\xr{f_1}V_1\xr{f_0}V_0$
of Banach spaces $V_i$ and continuous linear maps, then each $f^\infty_i\:V\to V_i$
is continuous, so its kernel $K_i$ is closed.
Without loss of generality $0$ is an interior point of $B$.
Then $B_i\bydef B+K_i$ is a closed convex body in $V$ that is absorbing (i.e.\ for every
$v\in V$ there exists an $r$ such that $v/r\in B_i$), and then similarly to (a),
$B_i$ is a uniform retract of $V$, and in particular of $B_{i-1}$.
Since $B$ is the inverse limit of the $B_i$, it follows (see the proof of
Theorem \ref{ANR-limit} below) that $B$ is a uniform retract of $B_{[0,\infty]}$,
and therefore (see the proof of Theorem \ref{LCU+Hahn} below) also of $V$.
The remainder of the proof is similar to that in (a).

(h) If $V$ is an injective object in the category of Frech\'et spaces, then
every closed bounded convex body $B$ in $V$ is an ARU.
Indeed, $V$ is injective as a locally convex space (see \cite[Lemma 0]{DO}), and
$B$ is a uniform retract of $V$ by the argument in (g).
Now the assertion follows from \cite[Theorem 1.6]{Nhu1}.

Injective Frech\'et spaces include injective Banach spaces (see \cite[Lemma 0]{DO})
and hence their countable products, and these appear to be all known examples
(see \cite{Chi} and \cite{DD}).
\end{remark}

\section{Uniform ANRs}\label{uniform-anrs}

\subsection{Definition}
Isbell's proof that ANRUs are complete uses non-metrizable test spaces such as
the ordinal $\omega_1$ in the order topology in an essential way, even when
the given ANRU is itself metrizable.
Relevance of this argument for the purposes of geometric topology is questionable;
one would not be comfortable using this construction as a basis for geometrically
substantial results.
An alternative, more transparent approach (implicit in the papers of Garg \cite{Gar}
and Nhu \cite{Nhu1}, \cite{Nhu2}; see also \cite{Ya}) results from replacing
embeddings with extremal (or regular) epimorphisms (i.e.\ embeddings onto closed
subsets) in the definitions of A[N]R($\mathcal{MU}$) and A[N]E($\mathcal{MU}$).

Thus by a {\it uniform A[N]E} we mean a metrizable uniform space $Y$ such that given
a closed subset $A$ of a metrizable uniform space $X$, every uniformly continuous
$f\:A\to Y$ extends to a uniformly continuous map defined on [a uniform neighborhood
of $A$ in] $X$.
Similarly, by a {\it uniform A[N]R} we mean a metrizable uniform space $Y$ such that
for every uniform embedding $i$ of $Y$ onto a closed subspace of a metrizable uniform
space $Z$ there exists a uniformly continuous retraction of [a uniform neighborhood
of $Y$ in] $Z$ onto $i(Y)$.
Obviously uniform A[N]Rs include all metrizable A[N]RUs and no other complete spaces.

\begin{proposition}[Garg \cite{Gar}]\label{Garg}
A metrizable uniform space $Y$ is a uniform A[N]R if and only if it is a uniform A[N]E.
\end{proposition}

\begin{proof}
This follows from Theorem \ref{adjunction} by the usual adjunction space argument (see \cite[\S III.3]{Hu1}).
\end{proof}

The above proof is much easier than the original one in \cite{Gar}.
A still easier proof, which uses a singular version of the adjunction space (and hence does not use 
the metrizability of the genuine adjunction space) is given in \cite[Appendix]{Ya}.

\subsection{Homotopy completeness}
We say that a uniform space $X$ is {\it homotopy complete} if there exists
a uniform homotopy $H\:\bar X\x I\to\bar X$, where $\bar X$ is the completion of $X$,
such that $H(x,0)=x$ and $H(\bar X\x (0,1])\incl X$.
Note that if $X$ is homotopy complete, then it is uniformly $\eps$-homotopy
equivalent to its completion, for each $\eps>0$.

\begin{remark} 
A subspace $A$ of a separable metrizable topological space $Y$ such that there
exists a homotopy $h_t\:Y\to Y$ satisfying $h_0=\id$ and $h_t(Y)\cap A=\emptyset$
for $t>0$ is called by various authors a ``Z-set'', a ``homotopy negligible set'' 
or an ``unstable set''.
It is well-known that under this assumption, $Y$ is an ANR if and only if
$Y\but A$ is an ANR (cf.\ \cite[1.1(iv)]{Fe}, \cite[Proposition 1.2.1,
Exercise 1.2.16]{BRZ}).
Beware that Z-sets in a more usual sense (as above but with $t\in\{1/i\mid i\in\N\}\cup\{0\}$) 
are something else in general \cite[Exercise 1.2.11]{BRZ} but the same in ANRs
\cite[Theorem 1.4.4]{BRZ}).
\end{remark}

\begin{theorem}\label{uniform ANR} Suppose that $X$ is a metrizable uniform space.
Then $X$ is a uniform A[N]R if and only if it is homotopy complete and its
completion is an A[N]RU.
\end{theorem}

A few months after having written up the proof of Theorem \ref{uniform ANR},
the author learned that its analogue for semi-uniform ANRs (see
Remark \ref{Michael&Nhu}(b)) had been known \cite{Sak}.

\begin{proof} For the ``if'' direction, we consider the case of ARUs; the case
of ANRUs is similar.
Let $Y$ be a metrizable uniform space and $A$ a closed subset of $Y$.
Then $D\:Y\to I$ defined by $D(y)=\min\{d(y,A),1\}$ satisfies $D^{-1}(0)=A$.
Given a uniformly continuous $f\:A\to X$, the hypothesis yields a uniformly
continuous extension $\bar f\:Y\to\bar X$.
Let $F$ be the composition $Y\xr{\bar f\x D}\bar X\x I\xr{H}\bar X$.
Then $F|_A=f$, and $F(Y\but A)\incl H(\bar X\x (0,1])\incl X$.
Thus $F$ is a uniformly continuous extension of $f$ with values in $X$.

Conversely, suppose $X$ is a uniform A[N]R.
Let $Z_t$ denote the subspace $X\x\{0\}\cup\bar X\x (0,t]$ of $\bar X\x [0,1]$.
Then $X\x\{0\}$ is a closed subset of $Z_1$.
Since $X$ is a uniform ANR, $\id_X$ extends to a uniformly continuous map
$Z_\eps\to X$ for some $\eps>0$.
The latter has a unique extension over the completions, which yields the
required homotopy.

It remains to show that $\bar X$ is an A[N]RU, or equivalently an
A[N]RU($\mathcal{MU}$); here we consider only the case of ARUs, the case of ANRUs
being similar.
Let $Y$ be a metrizable uniform space and $A$ a subset of $Y$.
Given a uniformly continuous $f\:A\to\bar X$, we uniquely extend it to a
uniformly continuous $\bar f\:\bar A\to\bar X$.
Recall that the homotopy $H$ has been constructed; consider the composition
$F\:\bar A\x I\xr{\bar f\x\id_I}\bar X\x I\xr{H}\bar X$.
Then $F$ sends $\bar A\x (0,1]$ into $X$, and so $F|_{\bar A\x (0,1]}$ extends to
a uniformly continuous map $\bar F\:\bar Y\x(0,1]\to X$.
If a point $(a,t)$ of $\bar A\x I$ is close to a point $(y,s)$ of $\bar Y\x (0,1]$,
then they are both close to the point $(a,s)$ of the intersection (by considering
the $l_\infty$ metric on the product); thus
$F\cup\bar F\:\bar A\x I\cup\bar Y\x(0,1]\to\bar X$ is uniformly continuous.
Define $D\:\bar Y\to I$ by $D(y)=\min\{d(y,\bar A),1\}$; since $\bar A$ is closed
in $\bar Y$, we have $D^{-1}(0)=\bar A$.
Then the composition of the graph
$\Gamma_D\:\bar Y\to\bar A\x\{0\}\cup\bar Y\x (0,1]$ with $F\cup\bar F$
provides the required extension of $f$.
\end{proof}

\subsection{Examples and comparison}

\begin{corollary} (a) \cite[Example 2]{Gar}, \cite[Remark 2.8]{Nhu1}
The half-open interval $(0,1]$ and the open interval $(0,1)$ are uniform ARs.

(b) $[-1,0)\cup(0,1]$ is not a uniform ANR.
\end{corollary}

\begin{corollary} \label{c_00}
(a) The space $c_{00}\subset c_0$ of finite sequences of reals is a uniform ANR.

(b) The space $q_{00}\subset q_0$ of finite sequences of points of $[0,1]$
is a uniform AR.
\end{corollary}

\begin{proof}[Proof. (b)]
Define $H_t\:I\to I$, $t\in I=[0,1]$, by $H_t(x)=\max\{0,1-(1-x)(1+t)\}$.
Thus $H_t$ is fixed on the endpoints, $H_0=\id$, and $H_1$ sends
$[0,\frac12]$ to $0$.
Next define a homotopy $h_t\:U(\N,I)\to U(\N,I)$ by $h_t(f)=H_tf$.
Clearly $H_t$ is uniformly continuous, $h_0=\id$, and
$h_t(q_0)\incl q_{00}$ for $t>0$.
Now $q_0$ is an ARU by Corollary \ref{q_0 ARU}, so $q_{00}$ is a uniform ANR
by Theorem \ref{uniform ANR}.
\end{proof}

\begin{proof}[(a)] This is similar to (b), using the extension of $H_t(s)$
by $H_t(s)=-H_t(-s)$ for $s\in [-1,0]$ and by $H_t(s)=s$ for $s\notin [-1,1]$.
\end{proof}

\begin{remark}\label{CW} Let $\R^\omega$, resp.\ $I^\omega$ 
denote the direct limits (in the category of uniform spaces) of 
the inclusions 
$\R^1\subset\R^2\subset\dots$ and $[0,1]^1\subset [0,1]^2\subset\dots$.
Then $\id\:\R^\omega\to c_{00}$ and $\id\:I^\omega\to q_{00}$ are
uniformly continuous.
The composition $c_{00}\xr{h_1}c_{00}\xr{\id}\R^\omega$ is continuous,
using that every point of $q_{00}$ has a neighborhood that $h_1$ sends 
into some $\R^n$.
However, the composition $q_{00}\xr{h_1}q_{00}\xr{\id}I^\omega$
is not uniformly continuous, by considering the metric
$\lim nd|_{I^n}$ (see \S\ref{dirlimits}), where $d$ is the sup
metric on $I^\N$. 
\end{remark}

\begin{corollary}\label{separable ANR}
Let $X$ be a separable metrizable uniform space.
The following are equivalent:

(i) $X$ is a uniform A[N]R;

(ii) $X$ is a [neighborhood] retract of every separable metrizable uniform space
containing it as a closed subset.

(iii) $X$ is a [neighborhood] extensor for every closed pair of separable metrizable uniform spaces.
\end{corollary}

\begin{proof} Obviously (i)\imp(ii) and (iii)\imp(ii).
The implication (ii)\imp(i) follows from Lemma \ref{A.2}(a),
Theorem \ref{uniform ANR}, and the separable analogue of Theorem \ref{uniform ANR},
which is proved by the same argument.
The implication (i)\imp(iii) follows from Proposition \ref{Garg}; alternatively,
the implication (ii)\imp(iii) follows from the separable analogue of Proposition
\ref{Garg}, which is proved by the same argument.
\end{proof}

\subsection{Uniform local contractibility}
We call a metric space $X$ {\it uniformly locally contractible} if for each $\eps>0$
there exists a $\delta>0$ such that for every uniform space $Y$, every two
$\delta$-close uniformly continuous maps $f,g\:Y\to X$ are uniformly
$\eps$-homotopic (that is, are joined by a uniformly continuous homotopy
$Y\x I\to X$ that is $\eps$-close to the composition
$Y\x I\xr{{\rm projection}}Y\xr{f}X$).
Obviously uniform local contractibility does not depend on the choice of
the metric on $X$ in its uniform equivalence class.

\begin{remark}
If $X$ is uniformly locally contractible, then for each $\eps>0$ there
exists a $\delta>0$ such that $\delta$-balls in $X$ contract within their
concentric $\eps$-balls, by a uniformly equicontinuous family of
null-homotopies (see \cite[4.2]{I2}).
\end{remark}

\begin{remark}\label{higher homotopies}
If $X$ is uniformly locally contractible, then for each $\eps>0$
there exists a $\delta>0$ such that for each $n$ and every uniform space $Y$,
every $(n-1)$-sphere $F\:Y\x S^{n-1}\to X$ of maps $Y\to X$ that are within
$\delta$ of each other bounds an $n$-ball $\bar F\:Y\x B^n\to X$ of
maps $Y\to X$ that are within $\eps$ of each other.
Indeed, the definition of uniform local contractibility yields an $\eps$-homotopy
$h_t$ between $F$ and the composition of the projection $Y\x\partial I^n\to Y$
and $F|_{Y\x pt}\:Y\to X$; then $\bar F$ can be defined by $\bar F(y,[(r,\phi)])=h_{1-r}(y,\phi)$, 
where $[(r,\phi)]$ is the image of
$(r,\phi)\in [0,1]\x S^{n-1}$ in $B^n=[0,1]\x S^{n-1}/(\{0\}\x S^n)$.
\end{remark}

\begin{lemma}\label{LCU} Let $X$ be a (separable; compact) metric space.
Then $X$ is uniformly locally contractible if either

(a) $X$ is a uniform ANR, or

(b) for each (separable; compact) metrizable uniform space $Y$ and every $\eps>0$
there exists a $\delta>0$ such that every two $\delta$-close uniformly continuous
maps $f,g\:Y\to X$ are uniformly $\eps$-homotopic.
\end{lemma}

We note that the hypothesis of (b) is a weakening of the condition of uniform
local contractibility, with restrictions imposed on $Y$ and with $\delta$ allowed
to depend on $Y$.
It follows from (b) that $X$ is uniformly locally contractible if and only if
it is LCU in the sense of Isbell (whose $\delta$ is allowed to depend on $Y$)
\cite{I2}.
That ANRUs are LCUs was known to Isbell \cite[4.2]{I2}; his proof of this fact
is rather different (via functional spaces).

\begin{proof}
Let $U_i$ be the $\frac1i$-neighborhood of the diagonal in $X\x X$ in
the $l_\infty$ product metric on $X\x X$, and let $f_i,g_i\:U_i\subset X\x X\to X$
be the two projections.
We shall show that, under the hypothesis of either (a) or (b), for each $\eps>0$
there exists an $i$ such that $f_i$ and $g_i$ are uniformly $\eps$-homotopic.

(a) Let $Y_n=\bigsqcup_{i\ge n}U_i\x[0,\frac1i]$, and let
$A_n=\bigsqcup_{i\ge n}U_i\x\{0,\frac1i\}$.
Define $f\:A_1\to X$ via $f_i$ on $U_i\x\{0\}$ and via $g_i$ on
$U_i\x\{\frac1i\}$.
Since $X$ is a uniform ANR, $f$ has a uniformly continuous extension $\bar f$
over a uniform neighborhood $U$ of $A_1$ in $Y_1$.
This $U$ contains the $\frac1n$-neighborhood of $A_1$ for some $n$, in the $l_\infty$
product metric on $Y\subset X\x X\x\N\x [0,1]$, which in turn contains $Y_n$.
Moreover, for each $\eps>0$ there exists an $m\ge n$ such that $\bar f$ takes
$\frac1m$-close points into $\eps$-close points.
Then $\bar f$ restricted to $U_m\x[0,\frac1m]$ is a uniform
$\eps$-homotopy between $f_m$ and $g_m$.

(b) Consider $F=\bigsqcup_{i\in\N} f_i$ and
$G_n\bydef \bigsqcup_{i\in[n]} f_i\sqcup\bigsqcup_{i\in\N\but[n]}g_i$ both mapping
$Y\bydef \bigsqcup U_i$ into $X$.
Note that $d(F,G_n)\to 0$ as $n\to\infty$.
Then by the hypothesis of (b), for each $\eps>0$ there exists an $n$ such that $F$
and $G_n$ are uniformly $\eps$-homotopic.
Then also $f_{n+1}$ and $g_{n+1}$ are uniformly $\eps$-homotopic.

We have thus shown, in both (a) and (b), that for each $\eps>0$ there exists
a uniform $\eps$-homotopy $h_t\:U_i\to X$ between $f_i$ and $g_i$ for some $i$.
Now given $\frac1i$-close maps $f,g\:Y\to X$, the image of $f\x g\:Y\to X\x X$
lies in $U_i$.
Then the composite homotopy $Y\xr{f\x g}U_n\xr{h_t}X$ is a uniform $\eps$-homotopy
between $f$ and $g$.
\end{proof}

\subsection{Locally equiconnected compacta}

\begin{lemma} \label{LCU-rel-lemma} A metric space $X$ is uniformly locally contractible if and only if
for each $\eps>0$ there exists a $\delta>0$ such that for every uniform space $Y$
and a subset $A$ of $Y$, every two $\delta$-close uniformly continuous maps
$f,g\:Y\to X$ that agree on $A$ are uniformly $\eps$-homotopic keeping $A$ fixed.
\end{lemma}

\begin{proof} Since $X$ is uniformly locally contractible, for each $i$ there exists a $\gamma_i>0$ such that 
for every uniform space $Z$, every two $\gamma_i$-close uniformly continuous maps $\phi,\psi\:Z\to X$ are 
uniformly $\frac\eps i$-homotopic; also, for each $i$ there exists a $\beta_i>0$ such that for every uniform 
space $Z$, every two $\beta_i$-close uniformly continuous maps $\phi,\psi\:Z\to X$ are uniformly 
$\gamma_i$-homotopic.
Set $\delta=\beta_0$.
Since $f$ and $g$ are uniformly continuous, there exists an $\alpha_i>0$ such that
both $f$ and $g$ take $\alpha_i$-close points to $\frac{\beta_i}2$-close ones.
Let $U_i$ be the open $\alpha_i$-neighborhood of $A$ in $Y$ for $i>0$, and let $U_0=Y$.

Then for $i>0$, each $y\in U_i$ is $\alpha_i$-close to some $z\in A$,
and consequently each of $f(y)$, $g(y)$ is $\frac{\beta_i}2$-close to $f(z)=g(z)$.
Thus $f(y)$ is $\beta_i$-close to $g(y)$ for each $y\in U_i$, $i>0$.
Also, $f$ is $\beta_0$-close to $g$ since $\delta=\beta_0$.
Hence $f|_{U_i}$ is uniformly $\gamma_i$-homotopic to $g|_{U_i}$ for each $i$; 
let $h_i\:U_i\x I\to X$ be the homotopy.
Let us now define $H_i\:U_{i+1}\x\partial(I^2)\to X$ by
$H_i(y,t,0)=f(y)$, $H_i(y,t,1)=g(y)$, $H_i(y,0,t)=h_i(y,t)$ and $H_i(y,1,t)=h_{i+1}(y,t)$.
Then $H_i$ is $\gamma_i$-close to $f\x\id_{\partial(I^2)}$.
Hence they are uniformly $\frac\eps i$-homotopic.
It follows that $H_i$ extends to a uniformly continuous map $\bar H_i\: U_{i+1}\x I^2\to X$
that is $\frac\eps i$-close to $f\x\id_{I^2}$.

Let $U=Y\x\{1\}\cup A\x\{0\}\cup\bigcup_{i\in\N}U_i\x [\frac1{i+1},\frac1i]$, let 
$F=f\x\id_I|_U$ and let $G=g\x\id_I|_U$.
Then by combining the $\bar H_i$ and extending by continuity to $A\x\{0\}\x I$ we get a map 
$H\:U\x I\to X$, $\eps$-close to $F\x\id_I$ and such that $H(y,0)=F(y)$, $H(y,1)=G(y)$,
and $H(a,t)=H(a,0)$ for all $a\in A\x\{0\}$.
To see that $H$ is uniformly continuous, let $\gamma>0$ be given; then there exists an $\alpha>0$
such that $F\x\id_I$ takes $\alpha$-close points to $\frac\gamma3$-close ones, an $i$ such that
$\frac\eps i\le\frac\gamma3$, and a $\beta>0$ such that $H$ takes $\beta$-close points in 
$U\x I\cap Y\x [1/i,1]\x I$ to $\gamma$-close ones.
Since the restrictions of $H$ and $F\x\id_I$ to $U\x I\cap Y\x[0,1/i]\x I$ are $\frac\eps i$-close,
it follows that $H$ takes $\min(\alpha,\beta)$-close points to $\gamma$-close ones.   

Finally, let $\phi\:Y\to I$ be such that $\phi^{-1}(0)=A$ and $\phi^{-1}([0,\frac1i))=U_i$ for each $i$.
Then the image of $\id_Y\x\phi\:Y\to Y\x I$ lies in $U$.
Hence $(\id_Y\x\phi)\x\id_I\:Y\x I\to U\x I$ composed with $H$ is the desired uniformly continuous 
$\eps$-homotopy between $f$ and $g$ keeping $A$ fixed.
\end{proof}

\begin{remark}\label{LCU-rel}
Due to Lemma \ref{LCU-rel-lemma}, uniformly locally contractible compacta coincide with locally 
equiconnected compacta \cite[Theorem 2.5]{Du2}.
Whether locally equiconnected compacta are ANRs is a long standing open problem going back to Fox \cite{Fox}.
Its special case is the so-called ``compact AR problem'' of whether a compact
convex subset of a metrizable topological vector space is an AR.

Uniformly locally contractible compacta are also considered in \cite{Pa}.
\end{remark}

\begin{corollary} The following are equivalent for a metric space $X$:

(i) $X$ is uniformly locally contractible;

(ii) $\Delta_X\subset X\x X$ is a uniform deformation retract of some neighborhood;

(iii) $\Delta_X\subset X\x X$ is a uniform strong deformation retract of some neighborhood.
\end{corollary}

\begin{proof} See \cite{Du2}*{proof of Theorem 2.1}.
\end{proof}

\subsection{Hahn property}
Following Isbell \cite{I2}, we say that a metric space $X$ satisfies the {\it Hahn
property} if for each $\eps>0$ there exists a $\delta>0$ such that for every
uniform space $Y$ and every $\delta$-continuous map $\Xi\:Y\to X$ (that is,
a possibly discontinuous map such that there exists a uniform cover $C$ of $Y$
such that $\Xi$ sends every element of $C$ into a set of diameter at most $\delta$),
there exists a uniformly continuous map $f\:Y\to X$ that is $\eps$-close to $\Xi$.
Obviously the satisfaction of the Hahn property does not depend on the choice of
the metric on $X$ in its uniform equivalence class.

\begin{lemma}\label{Hahn} Let $X$ be a (separable; compact) metric space.
Then $X$ satisfies the Hahn property if either

(a) $X$ is a uniform ANR, or

(b) for each (separable; compact) metrizable uniform space $Y$ and every $\eps>0$
there exists a $\delta>0$ such that every $\delta$-continuous map $Y\to X$
is $\eps$-close to a uniformly continuous map $Y\to X$.
\end{lemma}

We note that the hypothesis of (b) is a weakening of the Hahn condition, with
restrictions imposed on $Y$ and with $\delta$ allowed to depend on $Y$.
That ANRUs satisfy the Hahn property was shown by Isbell \cite[4.2]{I2}; this
along with Theorem \ref{uniform ANR} implies (a).

\begin{proof}[Proof of (b)] Let $F$ be a (separable; compact) metric ARU containing
$X$ (see Theorem \ref{aharoni} and Remark \ref{A.2*}(i,iii)).
Let $U_i$ be the $\frac1i$-neighborhood of $X$ in $F$, and let $U=\bigsqcup U_i$.
Pick a $\frac1{3i}$-continuous retraction $\xi_i\:U_i\to X$ (cf.\ the proof of
Theorem \ref{LCU+Hahn}), and define $\Xi_n\:U\to X$ via $\xi_i$ on each $U_i$
for $i\ge n$, and by a constant map on $U_1\sqcup\dots\sqcup U_{n-1}$.
Thus $\Xi_n$ is $\frac1{3n}$-continuous.
By the hypothesis, for each $\eps>0$ there exists an $n\ge\frac1\eps$ such that
$\Xi_n$ is $\frac\eps2$-close to a uniformly continuous map $U\to X$.
Then $\xi_n$ is $\frac\eps2$-close to a uniformly continuous map $r_n\:U_n\to X$.

By (a), $F$ satisfies the Hahn property.
So for each $n$ there exists a $\delta>0$ such that every $\delta$-continuous map
$\Xi\:Y\to X\subset F$ is $\frac1n$-close to a uniformly continuous map
$f\:Y\to U_n\subset F$.
Then $\Xi$ is $\eps$-close to the composition $Y\xr{f}U_n\xr{r_n}X$.
\end{proof}

\begin{theorem}\label{approximate polyhedron}
The following are equivalent for a compactum $X$:

(i) $X$ satisfies the Hahn property;

(ii) for each $\eps>0$ there exist a compact polyhedron $P$ and continuous maps
$X\xr{f}P\xr{g}X$ whose composition is $\eps$-close to $\id_X$;

(iii) for each $\eps>0$ there exist a compactum $Y$ satisfying the Hahn property
and continuous maps $X\xr{f}Y\xr{g}X$ whose composition is $\eps$-close to $\id_X$.
\end{theorem}

This means that compacta satisfying the Hahn property coincide with
the ``approximate ANRs'' of Noguchi and Clapp \cite{Cl}, which are also known as
``NE-sets'' after Borsuk \cite{Bo2} and as compacta that are
``approximate polyhedra'' in the sense of Marde\v si\'c \cite{Mard}.

We shall generalize Corollary \ref{approximate polyhedron} to all separable metrizable
uniform spaces in Theorem \ref{approximate cubohedron} (and, in a more direct way,
also in the sequel to this paper dealing with uniform polyhedra \cite{M3}).

\begin{proof}
(i) implies (ii) using that $X$ is the inverse limit of an inverse sequence of compact
polyhedra (see also Lemma \ref{A.12}(a) below).
Since compact polyhedra are ANRs, hence uniform ANRs, and in particular satisfy
the Hahn property, (ii) implies (iii).
The implication (iii)\imp (i) is an easy exercise.
\end{proof}

The following is straightforward to verify:

\begin{proposition}
A metrizable uniform space $X$ satisfies the Hahn property if and only if
the completion $\bar X$ satisfies the Hahn property, and $\id_{\bar X}$ is
$\eps$-close to a uniformly continuous map $\bar X\to X$ for each $\eps>0$.
\end{proposition}

\subsection{Uniform ANR=ULC+Hahn property}

\begin{theorem} \label{LCU+Hahn} Let $X$ be a metrizable uniform space.
Then $X$ is a uniform ANR if and only if $X$ is uniformly locally contractible
and satisfies the Hahn property.
\end{theorem}

This improves on the metrizable case of Isbell's characterization of ANRUs
\cite[4.3]{I2}, which additionally involved the homotopy extension property
(see Lemma \ref{Borsuk} below).

A few months after having written up the proof of Theorem \ref{LCU+Hahn},
the author discovered that the compact case is contained in \cite[Theorem 3.2]{Du2}, 
and a rather similar characterization of
non-uniform ANRs is stated without proof in \cite[Theorem 3.7]{AR} (some relevant
ideas can be found also in Theorems 3.2 and 4.8 in \cite{AR}).

A strengthened form of Theorem \ref{LCU+Hahn} is given in Theorem
\ref{approaching uniform local contractibility}(c) below.

\begin{proof} If $X$ is a uniform ANR, it is uniformly locally contractible by
Lemma \ref{LCU}(a) and satisfies the Hahn property by Lemma \ref{Hahn}(a).

Conversely, suppose $Y$ is a metric space, $A$ is a closed subset of
$Y$, and $f\:A\to X$ is a uniformly continuous map.
Fix some metric on $X$.
Let $\gamma_i=\delta_{LCU}(2^{-n})$ be given by the uniform local contractibility
of $X$ with $\eps=2^{-n}$; let $\beta_i=\delta_{Hahn}(\gamma_i)$ be given by
the Hahn property of $X$ with $\eps=\gamma_i$; and let $\alpha_i=\delta_f(\beta_i)$
be such that $f$ sends $3\alpha_i$-close points into $\beta_i$-close points.
We may assume that $\alpha_{i+1}\le\alpha_i$, $\gamma_{i+1}\le\gamma_i$ and
$\beta_i\le\gamma_i$ for each $i$.
Let $U_i$ be the $\alpha_i$-neighborhood of $A$ in $Y$; thus $A=\bigcap U_i$.

Define $\Xi_i\:U_i\to X$ by sending every $x\in U_i$ into $f(x'_i)$, where
$x'_i\in A$ is $\alpha_i$-close to $x$.
(When $Y$ is separable, it is easy to find such an $x'_i$ using only the countable
axiom of choice.)
Then $\Xi_i$ sends $\alpha_i$-close points $x$, $y$ into $\beta_i$-close
points $f(x'_i)$, $f(y'_i)$ (using that $x'_i$ is $3\alpha_i$-close to $y'_i$).
By the Hahn property, there exists a uniformly continuous map
$f_i\:U_i\to X$ such that $f_i$ is $\gamma_i$-close to $\Xi_i$; in particular,
$f_i|_A$ is $\gamma_i$-close to $f$.
Given an $x\in U_{i+1}$, it is $\alpha_i$-close to both $x'_i$ and $x'_{i+1}$,
which are therefore $2\alpha_i$-close to each other.
Hence $\Xi_i(x)=f(x'_i)$ is $\beta_i$-close to $f(x'_{i+1})=\Xi_{i+1}(x)$.
Thus $f_i|_{U_{i+1}}$ is $3\gamma_i$-close to $f_{i+1}$.
Then by the uniform local contractibility, $f_i|_{U_{i+1}}$ is uniformly
$2^{-i}$-homotopic to $f_{i+1}$.
These homotopies combine into a uniformly continuous map $H$ on the extended
mapping telescope $U_{[1,\infty]}\bydef A\cup\bigcup_{i\in\N}U_i\x[2^{-i-1},2^{-i}]
\incl U_1\x[0,1]$ of the inclusions $\dots\subset U_2\subset U_1$, which
restricts to the identity on $A$.

Pick a uniformly continuous function $\phi\:U_1\to [0,1]$ such that
$\phi^{-1}(0)=A$ and $\phi^{-1}([0,2^{-i}])\incl U_i$.
For instance, $\phi(x)=\alpha d(x,A)$ works, where
$\alpha\:[0,\alpha_1]\to[0,1]$
is a monotonous homeomorphism such that $\alpha^{-1}(2^{-i})\le\alpha_i$
for each $i$.
Then $\id\x\phi\:U_1\to U_1\x[0,1]$ embeds $U_1$ into $U_{[1,\infty]}$
and restricts to the identity on $A$.
Hence the composition $U_1\xr{\id\x\phi}U_{[1,\infty]}\xr{H}X$ is a uniformly
continuous extension of $f$.
\end{proof}

\subsection{Hanner-type characterizations}
As a consequence of Theorem \ref{LCU+Hahn} we get the following uniform analogues of Hanner's two
characterizations of ANRs (\cite{Han}*{7.1, 7.2}, \cite{Hu1}*{IV.5.3, IV.6.3}, \cite{Sak2}*{6.6.2}).

\begin{corollary}\label{Hanner}
(a) A metric space is a uniform ANR if it is uniformly $\eps$-homotopy dominated
by a uniform ANR for each $\eps>0$.

(b) Suppose that a metric space $X$ is uniformly embedded in a uniform ANR $Z$.
Then $X$ is a uniform ANR if and only if for each $\eps>0$ there exists
a uniformly continuous map $\phi$ of a uniform neighborhood of $X$ in $Z$ into $X$
such that $\phi|_X$ is uniformly $\eps$-homotopic to $\id_X$ with values in $X$.
\end{corollary}

\begin{proof}[Proof. (b)] The ``only if'' assertion follows using that $X\x\{0\}$
is closed in $X\x\{0\}\cup Z\x(0,1]$ (compare with Theorem \ref{uniform ANR}).

To prove the ``if'' assertion, fix an $\eps>0$ and let $\delta$ be such that
that the uniform neighborhood $U$ provided by the hypothesis contains
the $\delta$-neighborhood of $X$ in $Z$, and the map $\phi$ provided by
the hypothesis is $(\delta,\eps)$-continuous.

By Lemma \ref{LCU}(a), there exists a $\gamma>0$ such that for every
uniform space $Y$, every two $\gamma$-close uniformly continuous maps
$f,g\:Y\to X$ are uniformly $\delta$-homotopic with values in $Z$.
Then the homotopy actually has values in $U$, and therefore $\phi f$ and $\phi g$
are uniformly $\eps$-homotopic with values in $X$.
On the other hand, by the hypothesis they are uniformly $\eps$-homotopic
to $f$ and $g$, respectively, with values in $X$.
Thus $X$ is uniformly locally contractible.

By Lemma \ref{Hahn}(a), there exists a $\beta>0$ such that for every uniform space
$Y$, every $\beta$-continuous map $\Xi\:Y\to X$ is $\delta$-close to a uniformly
continuous map $h\:Y\to Z$.
Then $h$ actually has values in $U$, and then $\phi h\:Y\to X$ is $\eps$-close
to $\phi\Xi$.
By the hypothesis, the latter is in turn $\eps$-close to $\Xi$.
Thus $X$ satisfies the Hahn property.

So we infer from Theorem \ref{LCU+Hahn} that $X$ is a uniform ANR.
\end{proof}

\begin{proof}[(a)] Let $X$ be the given space.
We may assume that it is embedded in a uniform ANR (for instance, in some ARU) $Z$.
Given an $\eps>0$, the hypothesis provides a uniform ANR $Y$ that uniformly
$\eps$-homotopy dominates $X$.
Thus we are given uniformly continuous maps $u\:X\to Y$ and $d\:Y\to X$
such that the composition $X\xr{u}Y\xr{d}X$ is uniformly $\eps$-homotopic to
the identity.
Let $\delta$ be such that $d$ is $(\delta,\eps)$-continuous.
Since $Y$ is a uniform ANR, and $X\x\{0\}$ is closed in $X\x\{0\}\cup Z\x(0,1]$,
$u$ is uniformly $\delta$-homotopic to the restriction of a uniformly
continuous map $\bar u\:U\to Y$ for some uniform neighborhood $U$ of $X$ in $Z$.
Then $du$ is uniformly $\eps$-homotopic to $d\bar u|_X$.
On the other hand, $du$ is uniformly $\eps$-homotopic to the identity.
So we infer from (b) that $X$ is a uniform ANR.
\end{proof}

\begin{remark}
The Ageev--Repov\v s characterization of ANRs \cite[Theorem 3.7]{AR}, unfortunately, does not seem to admit
a direct uniform analogue (even if using non-standard metrizable uniformities on $X\x[0,\infty)$).
\end{remark}

\subsection{Uniformly finite-dimensional uniform ANRs}
Another consequence of Theorem \ref{LCU+Hahn} is the following

\begin{theorem}\label{rfd-ANR} A uniformly finite dimensional metric space 
is a uniform ANR if and only if it is uniformly locally contractible.
\end{theorem}

\begin{proof}
Let $X$ be a uniformly locally contractible uniformly $n$-dimensional space.
Then its completion $\bar X$ is still uniformly $n$-dimensional (see the proof of
Lemma \ref{rfd-completion}).
By \cite[Theorem V.34]{I3}, $\bar X$ is the limit of an inverse sequence of
$n$-dimensional uniform polyhedra $P_i$.
Since $X$ is uniformly locally contractible, every uniformly equicontinuous family of small
spheroids $S^k\to X$ bounds a uniformly equicontinuous family of small
singular disks $D^{k+1}\to X$.
Using this, it is easy to construct maps $r_i\:P_i\to X$ by induction on
skeleta, so that for each $\eps>0$ there is an $i$ such that the
composition $X\subset\bar X\xr{p^\infty_i}P_i\xr{r_i}X$ is $\eps$-close
to the identity.
By \cite[1.9]{I1}, each $P_i$ is a uniform ANR and so (see Lemma
\ref{Hahn}(a)) satisfies the Hahn property.
Since each $p^\infty_i\:\bar X\to P_i$ is uniformly continuous, we infer
that $X$ also satisfies the Hahn property.
By Theorem \ref{LCU+Hahn}, $X$ is a uniform ANR.
\end{proof}

\begin{remark}\label{equi-LCU}
Borsuk's example of a locally contractible compactum $X$ that is not an ANR
(see \cite{Bo}) is easily seen to be not uniformly locally contractible
(cf.\ \cite[Theorem 4.1]{Du2}).
Thus uniform local contractibility of a uniformly finitistic metric space $X$ 
does not follow from uniform local connectedness of each of the functional spaces 
$F_n=U(S^n\x\N,X)$.
By the proof of Theorem \ref{rfd-ANR}, it does follow when $X$ is
uniformly finite-dimensional.
\end{remark}

\begin{remark} By the proof of Theorem \ref{rfd-ANR} (see also
Part \ref{inverse limits}), if $X$ is a uniformly locally contractible metrizable
uniform space, then for every uniformly finite-dimensional uniform space $Y$,
every uniformly continuous partial map $Y\supset A\to X$ extends to a uniformly
continuous map $Y\to X$.
\end{remark}

\subsection{Cubohedra}\label{cubohedra} Let $Q$ be the triangulation of
the real line $\R$ with vertices precisely in all integer points.
Then $Q^n$ is a cubulation of the Euclidean space $\R^n$, and
$Q^\omega\bydef \bigcup_{n=0}^\infty Q^n\x\{0\}$ is a cubulation of the space
$c_{00}\subset c_0$ of finite sequences of reals.

By a {\it cubohedron} we mean any subcomplex of $Q^\omega$, viewed as
a (separable metrizable) uniform space.

\begin{theorem} \label{cubohedron} Every cubohedron is a uniform ANR.
\end{theorem}

\begin{proof} Let $H_t\:[0,1]\to[0,1]$ be as in the proof of Corollary
\ref{c_00}(b).
Define $g_t\:[-\frac12,\frac12]\to[-\frac12,\frac12]$
by $g_t(x)=\operatorname{sign}(x)\frac12 H_t(2|x|)$,
and extend $g_t$ to a homotopy $G_t\:\R\to\R$ by $G_t(x)=g_t(x-[x])$,
where $[x]$ denotes the nearest integer to $x$ (the least one when
there are two).
Thus $G_t$ is fixed on all integers and half-integers, $G_0=\id$, and
$G_1$ sends the $\frac14$-neighborhood of each integer onto that integer.

Finally, define $h_t\:U(\N,\R)\to U(\N,\R)$ by $h_t(f)=G_tf$.
Thus $h_t$ sends every cube of $Q^\omega$ into itself, for each $t$, and
$h_1$ sends the $\frac14$-neighborhood of each subcomplex of $Q^\omega$
into that subcomplex.

The homotopy $h^n_t(x)=2^{-n}h_t(2^n x)$ exhibits a similar behavior with respect
to the denser lattice $(2^{-n}Q)^\omega$.
Every subcomplex $K$ of $Q^\omega$ is also a subcomplex of $(2^{-n}Q)^\omega$,
so $K$ is invariant under $h_t^n$, and $h_1^n$ sends the $2^{n-2}$-neighborhood
of $K$ into $K$.
Since $c_{00}$ is a uniform ANR by Corollary \ref{c_00}(a), we get from
Corollary \ref{Hanner}(b) that so is $K$.
\end{proof}

\begin{theorem}\label{approximate cubohedron}
The following are equivalent for a separable metrizable uniform space $X$:

(i) $X$ satisfies the Hahn property;

(ii) for each $\eps>0$ there exist a cubohedron $P$ and uniformly continuous maps
$X\xr{f}P\xr{g}X$ whose composition is $\eps$-close to $\id_X$;

(iii) for each $\eps>0$ there exist a metrizable uniform space $Y$ satisfying
the Hahn property and uniformly continuous maps $X\xr{f}Y\xr{g}X$ whose
composition is $\eps$-close to $\id_X$.
\end{theorem}

\begin{proof} Every cubohedron is a uniform ANR (Theorem \ref{cubohedron})
and so satisfies the Hahn property (Lemma \ref{Hahn}(a)), whence (ii)\imp (iii).
The implication (iii)\imp(i) is easy.

Let us show that (i)\imp(ii).
By Aharoni's theorem \ref{aharoni}, $X$ uniformly embeds in $q_0$.
Given a $\delta>0$, let $P_\delta$ be the minimal subcomplex of $(\delta Q)^\omega$
whose completion $\bar P_\delta$ contains $X$.
Then every point of $\bar P_\delta$ is $\delta$-close to a point of $X$.
Hence if $X$ satisfies the Hahn property, given an $\eps>0$, there exists a
$\delta>0$ and a uniformly continuous map $f\:\bar P_\delta\to X$ such that
$f|_X$ is $\frac\eps2$-close to $\id_X$.
On the other hand, since $P_\delta$ is homotopy complete (or alternatively since
its satisfies the Hahn property), for each $\gamma>0$ there exists a uniformly
continuous map $g\:\bar P_\delta\to P_\delta$ that is $\gamma$-close to
$\id_{\bar P_\delta}$.
It follows that the composition
$X\subset\bar P_\delta\xr{g}P_\delta\xr{f|}X$
is $\eps$-close to the identity, if $\gamma$ is small enough.
\end{proof}

Theorem \ref{approximate cubohedron}(i)\imp(ii) and Lemma \ref{LCU}(a) imply
the ``only if'' direction of the following result, whose ``if'' direction is
a special case of Theorem \ref{Hanner}(a).

\begin{corollary} \label{Hanner2} A separable metric space is a uniform ANR
if and only if it is uniformly $\eps$-homotopy dominated by a cubohedron for each $\eps>0$.
\end{corollary}

\subsection{Homotopy theory}
The following uniform Borsuk homotopy extension lemma is proved in
a straightforward way, similarly to \cite[1.10]{I1}, \cite[VI.17]{I3}.

\begin{lemma}\label{Borsuk} Suppose that $Y$ is a uniform ANR, $X$ is a metrizable uniform space and $A$ is 
a closed subset of $X$.
If a uniformly continuous homotopy $h\:A\x I\to Y$ extends over $X\x\{0\}$, then that extension extends 
over $X\x I$ (all extensions being uniformly continuous).
\end{lemma}

\begin{theorem}\label{uniform AR} A metrizable uniform space is a uniform AR
if and only if it is a uniform ANR and is uniformly contractible.
\end{theorem}

This is parallel to \cite[1.11]{I1}, apart from using the metrizability of the cone.

\begin{proof} If $Y$ is a uniform AR, by definition it is a uniform ANR.
The cone $CY$ is metrizable, so $\id_Y$ extends to a uniformly continuous map
$CY\to Y$; the composition $Y\x I\to CY\to Y$ is then a uniform null-homotopy of
$\id_Y$.

Conversely, let $A$ be a closed subset of a metrizable uniform space $X$.
If $Y$ is uniformly contractible, every uniformly continuous map $f\:A\to Y$
is uniformly homotopic to a constant map, which extends over $X$.
If additionally $Y$ is a uniform ANR, then by Lemma \ref{Borsuk} the homotopy
extends over $X\x I$; in particular, $f$ extends over $X$.
\end{proof}

\begin{lemma}\label{deformation retract}
Let $X$ be a metrizable uniform space and $A$ a closed subset of $X$.

(a) Suppose that $A$ is a uniform ANR.
Then $A$ is a uniform deformation retract of $X$ if and only if the
inclusion $A\emb X$ is a uniform homotopy equivalence.

(b) Suppose that $X$ is a uniform ANR.
Then $A$ is a uniform deformation retract of $X$ if and only if
$A$ is a uniform strong deformation retract of $X$.
\end{lemma}

This is proved by usual arguments \cite[1.4.10 and 1.4.11]{Sp} making use of
the uniform Borsuk lemma (Lemma \ref{Borsuk}).

\begin{lemma}\label{pointed} If $f\:X\to Y$ is a uniform homotopy equivalence
of uniform ANRs, then $f$ is a pointed uniform homotopy equivalence of
$(X,x)$ and $(Y,f(x))$ for each $x\in X$.
\end{lemma}

This is proved by usual arguments (see e.g.\ \cite[Lecture 4, Propoposition 1]{Pos})
making use of the uniform Borsuk lemma (Lemma \ref{Borsuk}).

\subsection{Alternative approaches}

\begin{remark}\label{Michael&Nhu}
Several variations of the notion of a uniform A[N]R exist in the literature.

(a) Nhu \cite{Nhu1}, \cite{Nhu2} uses closed isometric, rather than uniform,
embeddings to define what may be termed {\it metric uniform A[N]Rs}; these are
also considered in \cite{BL} (see their Remark (i) to Proposition 1.2).
For bounded metrics the distinction is vacuous by Lemma \ref{A.M}, but in general
it is not: the real line with its usual metric is a metric uniform AR; whereas
every uniform AR is uniformly contractible by Theorem \ref{uniform AR} and
therefore is $\R$-bounded (cf.\ \S\ref{bornology}).
Nevertheless, metric uniform ARs of finite diameter are precisely uniform ARs with
a choice of a metric \cite{Nhu1}; and (arbitrary) metric uniform ANRs are precisely
uniform ANRs with a choice of a metric \cite[Appendix]{Ya}.

(b) Michael calls a map $f\:X\to Y$ of metric spaces {\it uniformly continuous at}
a closed subset $A\subset X$ if for each $\eps>0$ there exists a $\delta>0$ such $f$ is
$(\delta,\eps)$-continuous on the $\delta$-neighborhood of $A$ \cite{Mi}.
Michael \cite{Mi} and Torunczyk \cite{To} (see also \cite[p.\ 193]{Nhu2}) use
continuous retractions of uniform neighborhoods that are uniformly continuous at
their target to define what may be termed {\it semi-uniform A[N]Rs}.
These lie between usual (metrizable) A[N]Rs and our uniform A[N]Rs; in fact every
A[N]R is the underlying space of a semi-uniform A[N]R \cite{Mi}, \cite{To} 
(see also \cite{Sak2}*{Theorem 6.8.11}).
Sakai established the analogue of Theorem \ref{uniform ANR} for semi-uniform ANRs
\cite{Sak} (see also \cite{Sak2}*{Corollary 6.8.10}).

(c) Lipschitz, and 1-Lipschitz ANRs and ARs have been studied (see \cite{La}).
Note that a Lipschitz A[N]R is a metric uniform A[N]R; in particular, Lipschitz ANRs
are uniform ANRs, and bounded Lipschitz ARs are uniform ARs.
According to \cite[p.\ 65]{I5}, 1-dimensional topologically complete ARs
are metrizable as 1-Lipschitz ARs (see \cite{Pl} for a proof in the compact case).
On the other hand, Isbell showed that 2-dimensional non-collapsible compact polyhedra
are not metrizable as 1-Lipschitz ARs \cite{I5}.
It appears to be unknown whether every ANR is homeomorphic to a Lipschitz ANR
(cf.\ \cite{Ho3}).
However, Hohti showed that every LC$_n$ compactum can be remetrized so as to be
Lipschitz $n$-LC for all $n>0$ \cite{Ho3}.
\end{remark}

\begin{remark}\label{semi-uniform} We mention some facts relating to
semi-uniform ANRs.

(a) Similarly to the proof of Theorem \ref{LCU+Hahn}, a metrizable uniform
space is a semi-uniform ANR if and only if it is semi-uniformly locally
contractible and satisfies the weak Hahn property.
We call a metric space $M$ {\it semi-uniformly locally contractible} if
for each $\eps>0$ there exists a $\delta>0$ such that every two $\delta$-close
continuous maps of a metrizable topological space into $M$ are $\eps$-homotopic.
(That semi-uniform ANRs satisfy this property but not conversely was observed
by Michael \cite{Mi}.)
We say that $M$ satisfies the {\it weak Hahn property} if for each $\eps>0$
there exists a $\delta>0$ such that for every $\gamma>0$, every
$(\gamma,\delta)$-continuous map $f$ of a metric space $N$ into $M$ is
$\eps$-close to a continuous map.
(Note that the metric on $N$ is irrelevant, and the hypothesis on $f$ is equivalent
to saying that for every convergent sequence $(x_n)$ in $N$, the set of cluster points
of the sequence $(f(x_n))$ has diameter $<\delta$.)

(b) Similarly to the proof of Theorem \ref{rfd-ANR}, the following are equivalent
for a uniformly finite dimensional metric space $X$: (i) $X$ is a semi-uniform ANR;
(ii) $X$ is semi-uniformly locally contractible; (iii) for each $\eps>0$ there exists
a $\delta>0$ such that every continuous map $S^n\to X$ with image of diameter
$\le\delta$ bounds a continuous map $B^{n+1}\to X$ with image of diameter $\le\eps$.
\end{remark}

\section{Homotopy limits and colimits}

\subsection{Adjunction space}

\begin{theorem}\label{Whitehead} Let $X$, $Y$ and $A$ be uniform A[N]Rs, where
$A$ is a closed subset of $X$.
If $f\:A\to Y$ is a uniformly continuous map, then the adjunction space
$X\cup_f Y$ is a uniform A[N]R.
\end{theorem}

This is a uniform analog of J. H. C. Whitehead's theorem (see
\cite[Theorem V.9.1]{Bo}).
Our proof is a modification of Hanner's proof of Whitehead's theorem
\cite[Theorem 8.2]{Han}, \cite[Theorem VI.1.2]{Hu1}.

\begin{lemma} Assume the hypothesis of Theorem \ref{Whitehead} and fix
a metric on $X\cup_f Y$.
Then for each $\eps>0$ there exists a uniform $\eps$-homotopy
$H_t\:X\cup_f Y\to X\cup_f Y$ keeping $Y$ fixed and such that $H_0$ is the identity,
and $H_1$ retracts a uniform neighborhood of $Y$ in $X\cup_f Y$ onto $Y$.

Moreover, $H_t$ lifts to a uniform homotopy $h_t\:X\to X$ keeping $A$ fixed and
such that $h_0$ is the identity and $h_1$ retracts a uniform neighborhood
of $A$ in $X$ onto $A$.
\end{lemma}

\begin{proof}
Since $A$ is a uniform ANR, there is a uniform retraction $r$ of
a uniform neighborhood $U$ of $A$ in $X$ onto $A$.
The projection $X\x\{0\}\cup A\x I\to X$ combines with $r$ into a
uniformly continuous map $R\:X\x\{0\}\cup A\x I\cup U\x\{1\}\to X$.
Since $X$ is a uniform ANR, the latter extends to a uniformly continuous map
$\bar R$ on a uniform neighborhood $W$ of $X\x\{0\}\cup A\x I\cup U\x\{1\}$
in $X\x I$.
This $W$ contains the region $\Phi$ below the graph of a function
$\phi\:(X,V)\to ([\delta,1],\{1\})$, where $V$ is a uniform neighborhood of $A$
in $U$ and $\delta>0$.
Then $h_t\:X\to X$, defined by $h_t(x)=\bar R(x,t\phi(x))$, is a uniform homotopy,
keeping $A$ fixed, between $\id_X$ and an extension $\bar r$ of $r$ over $X$.
Since $\Phi$ may be assumed to be contained in any uniform neighborhood of
$X\x\{0\}\cup A\x I$ given in advance, $h_t$ may be assumed to be an
$\eps$-homotopy.

Consider the self-homotopy $H_t$ of $X\cup_f Y$ defined by $H_t(y)=y$ for
each $y\in Y$ and all $t\in I$ and by $H_t([x])=[h_t(x)]$ for $x\in X$ and all
$t\in I$.
It is well-defined since $h_t$ fixes $A$.
To prove that $H_t$ is uniform, we consider the $d_3$ metric as in Theorem
\ref{adjunction}.
Given $x,x'\in X$, we have
$d_3([x],[x'])=\min\{d(x,x'),\inf_{a,a'\in A}d(x,a)+d(f(a),f(a'))+d(a',x')\}$.
Hence if $([x],t)$ is $\alpha$-close to $([x'],t')$, then $|t-t'|\le\alpha$, and
either $d(x,x')\le\alpha$, or there exist $a,a'\in A$ with $d(x,a)\le\alpha$,
$d(x',a')\le\alpha$ and $d(f(a),f(a'))\le\alpha$.
Suppose $h_t$ is $(\alpha,\beta)$-continuous, viewed as a map $X\x I\to X$,
where $\beta\ge\alpha$.
Then either $h_t(x)$ is $\beta$-close to $h_{t'}(x')$, or $|t-t'|\le\beta$ and
there exist $a,a'\in A$ such that $h_t(x)$ is $\beta$-close to $h_t(a)=a$,
$h_t(x')$ is $\beta$-close to $h_t(a')=a'$, and $d(f(a),f(a'))\le\beta$.
Therefore $d_3([h_t(x)],[h_{t'}(x')])\le 3\beta$.
Thus $H_t$ is a uniform homotopy.

It remains to observe that the image of $V\sqcup Y$ in $X\cup_f Y$ is
a uniform neighborhood of $Y$ in $X\cup_f Y$, by considering the $d_3$ metric.
\end{proof}

\begin{proof}[Proof of Theorem \ref{Whitehead}]
We only consider the case of ANRs; the case of ARs is similar, and alternatively it
can be deduced from the case of ANRs using Theorem \ref{uniform AR}.

We may identify $X\cup_f Y$ with a closed subset of a uniform ANR $Z$.
We are going to apply Theorem \ref{Hanner}(b); to this end, fix an $\eps>0$, and
feed it into the preceding lemma.
Let $U_Y$ be the uniform neighborhood of $A$ in $X$ provided by the preceding lemma,
and let $h_t$ and $H_t$ be the homotopies provided by the preceding lemma.
Define $\bar f\:X\to X\cup_f Y$ by $\bar f(x)=[x]$.
Let $U_X$ be a uniform neighborhood of $X\but U_Y$, uniformly disjoint from $A$;
then $\bar f$ uniformly embeds $U_X$.
Write $V=X\cup_f Y$, $V_A=\bar f(U_X\cap U_Y)$, $V_X=\bar f(U_X)$ and
$V_Y=Y\cup\bar f(U_Y)$.
Let $Z_X$ and $Z_Y$ be uniformly disjoint open uniform neighborhoods respectively
of $V_X\but V_A$ and of $V_Y\but V_A$ in $Z$, and set $Z_A=Z\but (Z_X\cup Z_Y)$.

Since $A$ is a uniform ANR, $h_1\bar f^{-1}|_{V_A}$ extends to a uniformly
continuous map $\phi_A\:W_A\to A$, where $W_A$ is a uniform neighborhood of
$V_A$ in $V_A\cup Z_A$.
Since $V_A$ is a uniform neighborhood of $Z_A\cap V$ in $V$, there exists
a uniform neighborhood $N_A$ of $V$ in $Z$ such that $N_A\cap Z_A\subset W_A$.
Since $X$ is a uniform ANR,
$(h_1\bar f^{-1}|_{V_X})\cup\phi_A\:V_X\cup W_A\to X$ extends
to a uniformly continuous map $\phi_X\:W_X\cup W_A\to X$, where $W_X$ is
a uniform neighborhood of $V_X$ in $Z_X\cup V_X$.
Since $V_X$ is a uniform neighborhood of $Z_X\cap V$ in $V$, there exists
a uniform neighborhood $N_X$ of $V$ in $Z$ such that $N_X\cap Z_X\subset W_X$.
Since $Y$ is a uniform ANR,
$H_1|_{V_Y}\cup f\phi_A\:V_Y\cup W_A\to Y$ extends to a uniformly continuous map
$\phi_Y\:W_Y\cup W_A\to Y$, where $W_Y$ is a uniform neighborhood of
$V_Y$ in $Z_Y\cup V_Y$.
Since $V_Y$ is a uniform neighborhood of $Z_Y\cap V$ in $V$, there exists
a uniform neighborhood $N_Y$ of $V$ in $Z$ such that $N_Y\cap Z_Y\subset W_Y$.

Since $Z_X$ and $Z_Y$ are uniformly disjoint, so are $W_X$ and $W_Y$, and
therefore the map $(f\phi_X)\cup \phi_Y\:W_X\cup W_A\cup W_Y\to V$ is well-defined
and uniformly continuous.
By construction it restricts to $H_1$ on $V$.
On the other hand, $W_X\cup W_A\cup W_Y$ contains the uniform neighborhood
$N_X\cap N_A\cap N_Y$ of $V$ in $Z$.
So we infer from Theorem \ref{Hanner}(b) that $V$ is a uniform ANR.
\end{proof}

\begin{corollary}\label{Nhu}
If $X$ and $Y$ are uniform A[N]Rs each containing a closed copy of a uniform
A[N]R $A$, then the amalgamated union $X\cup_A Y$ is a uniform A[N]R.
\end{corollary}

Modulo Corollary \ref{A.5g}, this also follows from the results of Nhu \cite{Nhu1},
\cite{Nhu2}, who used the $d_2$ metric on the underlying set of $X\cup_A Y$ but
did not identify it as a metric of the quotient uniformity.

\subsection{Functional space}

\begin{theorem}\label{A.3''} If $Y$ is a metrizable uniform space, $B\incl Y$, and
$X$ and $A\incl X$ are uniform A[N]Rs, then $U((Y,B),(X,A))$ is a uniform A[N]R.
\end{theorem}

For an alternative proof see Corollary \ref{A.3'} below.

\begin{proof} We consider the case of ANRUs; the case of ARUs follows using
Theorem \ref{uniform AR}.

By Theorem \ref{uniform ANR}, the completion $\bar X$ of $X$ is an ANRU.
Hence by Theorem \ref{basic ARU'}, $U(Y,\bar X)$ is an ANRU.
Again by Theorem \ref{uniform ANR}, there exists a uniform homotopy $h_t$ of
$\bar X$ such that $h_0=\id$ and $h_t(\bar X)\subset X$ for $t>0$.
Now $H_t\:f\mapsto h_tf$ is a uniform homotopy of $U(Y,\bar X)$ such that
$H_0=\id$ and $H_t(U(Y,\bar X))\subset U(Y,X)$ for $t>0$.
Hence by Theorem \ref{uniform ANR}, $U(Y,X)$ is a uniform ANR.

Since $A$ is a uniform ANR, by Theorem \ref{Hanner}(b) for each $\eps>0$ there exists
uniformly continuous map $g\:U\to A$, where $U$ is a uniform neighborhood of $A$
in $X$ such that $g|_A$ is uniformly $\eps$-homotopic to $\id_A$ with values in $A$.
Then $G\:f\mapsto gf$ is a uniformly continuous map
$U((Y,B),(X,U))\to U((Y,B),(X,A))$ such that $G|_{U((Y,B),(X,A))}$ is uniformly
$\eps$-homotopic to $\id_{U((Y,B),(X,A))}$ with values in $U((Y,B),(X,A))$.
Since $U((Y,B),(X,U))$ is a uniform neighborhood of $U((Y,B),(X,A))$ in
$U(Y,X)$, again by Theorem \ref{Hanner}(b), $U((Y,B),(X,A))$ is a uniform ANR.
\end{proof}

\begin{remark}\label{path components} Under the hypothesis of Theorem \ref{A.3''},
the path components of $U((X,A),(Y,B))$ form a uniformly disjoint collection
(see \cite[V.16]{I3} and its proof).
In particular, each path component of $U((X,A),(Y,B))$ is a uniform ANR.
It follows, for instance, that the subspace of self-homotopy equivalences
in $U((X,A),(X,A))$ is a uniform ANR.
\end{remark}

Note that Theorem \ref{A.3''} implies Corollary \ref{q_0 ARU}.
Here is another consequence:

\begin{corollary}\label{loops} If $P$ is a uniform ANR, then the iterated loop space
$\Omega^n(P,pt)=U((S^n,pt),(P,pt))$ is a uniform ANR.
\end{corollary}

\subsection{Mapping cocylinder}

\begin{theorem}\label{holim-ANR}
Let $\Delta$ be a finite diagram of uniform ANRs and uniformly continuous maps.
Then the homotopy limit and the homotopy colimit of $\Delta$ are uniform ANRs.
\end{theorem}

\begin{proof}
We explicitly consider only the mapping cylinder and the mapping cocylinder,
i.e.\ the homotopy limit and colimit of a single map; the general case is
established similarly.

Let $f\:X\to Y$ be a uniformly continuous map between uniform ANRs.
Then $MC(f)$ is the adjunction space $X\x I\cup_{f'}Y$, where $f'$ is the map
$X\x\{0\}=X\xr{f}Y$, and therefore a uniform ANR by Theorem \ref{Whitehead}.

The {\it mapping cocylinder} of $f$ is
$MCC(f)=\{(x,p)\in X\x U(I,Y)\mid f(x)=p(0)\}$.
For each $\eps>0$, we define a uniform embedding $j_\eps$ of $MCC(f)$
into $MCC'(f)\bydef U((I,\{0\},\{1\}),(MC(f),X,Y))$ by
$$j_\eps(x,p)(t)=
\begin{cases}
[(x,t/\delta)],& t\le\delta,\\
[p(\frac{t-\delta}{1-\delta})],&t\ge\delta,
\end{cases}
$$
where $\delta=\delta(p,\eps)$ is given by Lemma \ref{uniform modulus} below, and
the square brackets denote the equivalence class in $MC(f)=X\x I\sqcup Y/\sim$.
By Theorem \ref{A.3'}, either generalized from pairs to triples or coupled with
Remark \ref{path components}, $MCC'(f)$ is a uniform ANR.
On the other hand, we have a uniformly continuous map $r\:MCC'(f)\to MCC(f)$ given
by $r(p)=(p(0),\pi p)$, where $\pi\:MC(f)\to Y$ is the projection.
Clearly, $rj_\eps$ is uniformly $\eps$-homotopic to the identity.
Hence $MCC(f)$ is a uniform ANR by Corollary \ref{Hanner}(a).
\end{proof}

\begin{lemma} \label{uniform modulus} Let $X$ and $Y$ be metric spaces.
Then for each $\eps>0$ there exists a uniformly continuous function
$\delta\:U(X,Y)\to (0,1]$ such that each $p\in U(X,Y)$ is
$(\delta(p),\eps)$-continuous.
\end{lemma}

\begin{proof}
Let $Z_n$ be the set of all $(2^{-n},\frac\eps3)$-continuous maps in $Z\bydef U(X,Y)$.
Clearly, $Z_0\subset Z_1\subset\dots$, and $\bigcup Z_n=Z$.
Let $U_n$ be the $(1-2^{-n-1})\frac\eps3$-neighborhood of $Z_n$.
Thus each $p\in U_n$ is $(2^{-n},\eps)$-continuous.
Since $Z_n\subset Z_m$ for $m>n$, the $(2^{-n-1}-2^{-m-1})\frac\eps3$-neighborhood
of $U_n$ lies in $U_m$.
For each $p\in U_{n+1}\but U_n$ let $r_n(p)=\max(\frac3\eps d(p,U_n),2^{-n-2})$,
and let $\delta(p)=2^{-n-1}-r_n(p)$.
Since $p\in U_{n+1}$ and $\delta(p)\le 2^{-n-1}$, we infer that $p$ is
$(\delta(p),\eps)$-continuous.

By construction $\delta$ is $\frac3\eps$-Lipschitz on all pairs $(p,p')$ with
$p,p'\in U_{n+1}\but U_n$.
On the other hand, since $U_m$ contains
the $(2^{-n-1}-2^{-m-1})\frac\eps3$-neighborhood of $U_n$, we have
$d(p,U_n)+d(p,\,Z\but U_m)>(2^{-n-1}-2^{-m-1})\frac\eps3$ for all
$p\in U_m\but U_n$.
Hence $r'_n(p)+q_m(p)\ge 2^{-n-1}-2^{-m-1}$, where $r'_n(p)=\frac3\eps d(p,U_n)$
and $q_m(p)=\frac3\eps d(p,\,Z\but U_m)$.
If $p\notin U_{n+1}$, we additionally have
$q_m(p)\ge 2^{-n-2}-2^{-m-1}$, and it follows that
$r_n(p)+q_m(p)\ge 2^{-n-1}-2^{-m-1}$.
Therefore $\delta(p)=2^{-n-1}-r_n(p)\le 2^{-m-1}+q_m(p)$ for all
$p\in U_{n+1}\but U_n$.
Given a $p'\in U_{m+1}\but U_m$, we have $\delta(p')=2^{-m-1}-r_m(p')$, and therefore
$$0\le\delta(p)-\delta(p')\le q_m(p)+r_m(p')\le
\tfrac3\eps[d(p,\,Z\but U_m)+d(p',U_m)]\le\tfrac3\eps[d(p,p')+d(p',p)].$$
Thus $\delta$ is $\frac6\eps$-Lipschitz on all pairs $(p,p')$ with
$p'\in U_{m+1}\but U_m$ and $p\in U_{n+1}\but U_n$ with $m>n$.
But all pairs are either of this type or of the similar type with $m=n$,
which has already been treated.
Thus $\delta$ is $\frac6\eps$-Lipschitz.
\end{proof}

We next turn to an alternative proof of Theorem \ref{A.3''} (with a similar result for 
non-metrizable ARUs and ANRUs as a byproduct), avoiding the infinite construction in the proof 
of Theorem \ref{LCU+Hahn} but involving a study of amalgamated unions of non-metrizable 
uniform spaces.

\subsection{Non-metrizable amalgam}

\begin{lemma}\label{amalgam}
A cover of $X\cup_A Y$ is uniform if and only if it is refined by
a cover of the form $C+D\bydef \{\st(z,\,C\cup D)\mid z\in X\cup_A Y\}$, where
$C$ is a uniform cover of $X$ and $D$ a uniform cover of $Y$.
\end{lemma}

\begin{proof}
Let $E$ be a uniform cover of $X\cup_A Y$.
Then $E$ is barycentrically refined by a cover $E_*$ such that $C\bydef E_*\cap X$ is a uniform
cover of $X$ and $D\bydef E_*\cap Y$ is a uniform cover of $Y$.
Then $C\cup D$ refines $E_*$, hence $C+D$ refines $E$.

Conversely, let $C$ be a uniform cover of $X$ and $D$ a uniform cover of $Y$.
We need to construct a sequence of covers $E_0,E_1,\dots$ of $X\cup Y$ such that
$E_0=C+D$, each $E_{i+1}$ barycentrically refines $E_i$, and each $E_i\cap X$ is a uniform
cover of $X$ and each $E_i\cap Y$ is a uniform cover of $Y$.
First note that $(C+D)\cap X$ itself is uniform, for it is refined by
$\{\st(x,C)\mid x\in X\}$, which in turn is refined by $C$.
Similarly $(C+D)\cap Y$ is uniform.

Let $C_*$ be a uniform cover of $X$ star-refining $C$ and let $D_*$ be
a uniform cover of $Y$ star-refining $D$.
Then $C_A\bydef C_*\cap A$ and $D_A\bydef D_*\cap A$ are uniform covers of $A$,
hence so is $F_*\bydef C_*\wedge D_*=C_A\wedge D_A$.
Since $C_A$ is a uniform cover of $A$, it is of the form $C_Y\cap A$
for some uniform cover $C_Y$ of $Y$.
Similarly $D_A$ is of the form $D_X\cap A$ for some uniform cover $D_X$ of $X$.
Then $C_*'\bydef C_*\wedge D_X$ is a uniform cover of $X$ star-refining $C$ and
$D_*'\bydef D_*\wedge C_Y$ is a uniform cover of $Y$ star-refining $D$.
In addition, $C_*'\cap A=F_*=D_*'\cap A$.

Next, let $C_{**}$ be a uniform cover of $X$ star-refining $C_*'$ and let
$D_{**}$ be a uniform cover of $Y$ star-refining $D_*'$.
Let $F_{**}=C_{**}\wedge D_{**}$, and define $C_{**}'$ and $D_{**}'$ similarly to
the above.
Then $C_{**}'$ is a uniform cover of $X$ star-refining $C_*'$ and $D_{**}'$ is
a uniform cover of $Y$ star-refining $D_*'$; in addition,
$C_{**}'\cap A=F_{**}=D_{**}'\cap A$.

We claim that $C_{**}'+D_{**}'$ barycentrically refines $C+D$; iterating the construction
of $C_{**}'$ and $D_{**}'$ would then yield the required sequence
$E_1,E_2,\dots$ (with $E_1=C_{**}'+D_{**}'$).
Given a $z\in X\cup_A Y$, we will show that
$\st(z,\,C_{**}'+D_{**}')$ lies in $\st(z',C)\cup\st(z',D)=\st(z',C\cup D)$
for some $z'\in X\cup_A Y$.
By symmetry we may assume that $z\in X$.
Let $U$ be an element of $C_{**}'+D_{**}'$ containing $z$.
Then $U=\st(w,C_{**}')\cup\st(w,D_{**}')$ for some $w\in X\cup_A Y$.
We consider two cases.

I. First suppose that $z\notin\st(A,C_{**}')$.
Then $w\in X\but A$, hence $U=\st(w,C_{**}')$.
Then $U$ is contained in an element of $C_*'$, which is in turn contained
in an element of $C$.
Thus $U\incl\st(z,C)$ and we may set $z'=z$.

II. It remains to consider the case where $z\in\st(A,C_{**}')$.
Then $z\in\st(z',C_{**}')$ for some $z'\in A$.
Also since $z\in U$, either $z\in\st(w,C_{**}')$ or $z\in\st(w,D_{**}')$.
We consider these two cases.

1. If $z\in\st(w,D_{**}')$ then $z\in Y$, whence $z\in A$.
Since $z$ and $z'$ are contained in one element of $C_{**}'$ and also in
$A$, they are contained in one element of $F_{**}$, hence in one element of
$D_{**}'$.
Thus $z\in\st(z',D_{**}')$ and $w\in\st(z,D_{**}')$, whence
$w\in\st(z',D_*')$.
In particular, $w\in Y$.
We consider two cases.

a) If $w\notin A$, then $U=\st(w,D_{**}')$.
Then $U\incl\st(z',D)$ and we are done.

b) If $w\in A$, then $w\in\st(z',F_*)$, and therefore $w\in\st(z',C_*)$.
Then $U=\st(w,C_{**}')\cup\st(w,D_{**}')$ is contained in
$\st(z',C)\cup\st(z',D)$ and we are done.

2. If $z\in\st(w,C_{**}')$ then also $w\in\st(z,C_{**}')$, and therefore
$w\in\st(z',C_*')$.
In particular, $w\in X$.
We consider two cases.

a) If $w\notin A$, then $U=\st(w,C_{**}')$.
Then $U\incl\st(z',C)$ and we are done.

b) If $w\in A$, then $w\in\st(z',F_*)$, and therefore $w\in\st(z',D_*)$.
Then $U=\st(w,C_{**}')\cup\st(w,D_{**}')$ is contained in
$\st(z',C)\cup\st(z',D)$ and we are done.
\end{proof}

We note that Lemma \ref{amalgam} yields an alternative proof of
Corollary \ref{amalgam-metrizable}, apart from the explicit metric:

\begin{corollary}\label{amalgam-metrizable} If $X$ and $Y$ are metrizable uniform
spaces, every amalgamated union $X\cup_A Y$ is metrizable.
\end{corollary}

\begin{proof}
Let $C_1,C_2,\dots$ be a basis of the uniformity of $X$ and $D_1,D_2,\dots$ be a
basis of the uniformity of $Y$.
If $C$ is a uniform cover of $X$ and $D$ is a uniform cover of $Y$, then there
exists an $i$ such that $C_i$ refines $C$ and $D_i$ refines $D$.
Then $C_i+D_i$ refines $C+D$.
Hence every uniform cover of $X\cup_A Y$ is refined by one of
$C_1+D_1,\,C_2+D_2,\dots$.
By Theorem \ref{A.1}, $X\cup_A Y$ is metrizable.
\end{proof}

\begin{corollary}\label{amalgam-product2}
Let $X$ and $Y$ be uniform spaces and $A\incl X$ and $B\incl Y$ closed subspaces.

(a) The subspace $X\x B\cup A\x Y$ of $X\x Y$ is uniformly homeomorphic to
the amalgamated union $X\x B\cup_{A\x B}A\x Y$.

(b) If $X$ is metrizable, then the natural maps $B\ltimes A\to B\ltimes X$,
$B\ltimes A\to Y\ltimes A$ and $Y\ltimes A\cup_{B\ltimes A}B\ltimes X\to Y\ltimes X$
are uniform embeddings.
\end{corollary}

Part (a), whose proof is similar to (b) (but easier), will not be used below;
the metrizable case of (a) can also be deduced from Corollary \ref{A.5g}.

The first two assertions of (b) are proved in \cite{I3}.

\begin{proof}[Proof of (b)]
Since $X$ is metrizable, a base of uniform covers of $B\ltimes X$ is given
by covers of the form $\{U_\alpha\x\bar V_\alpha^i\}$, where $\{U_\alpha\}$ is
a uniform cover of $B$ and for each $i$, $\{\bar V_\alpha^i\}$ is a uniform cover
of $X$.
It follows that $B\ltimes A\to B\ltimes X$ is a uniform embedding
(cf.\ \cite[III.29]{I3}).
Since $A$ is metrizable, a base of uniform covers of $Y\ltimes A$ is given
by covers of the form $\{\bar U_\alpha\x V_\alpha^i\}$, where $\{\bar U_\alpha\}$
is a uniform cover of $Y$ and for each $i$, $\{V_\alpha^i\}$ is a uniform cover
of $A$.
It follows that $B\ltimes A\to Y\ltimes A$ is a uniform embedding
(cf.\ \cite[III.25]{I3}).
Since $X$ is metrizable, a base of uniform covers of $Y\ltimes X$ is
$\{\bar U_\alpha\x\bar V_\alpha^i\}$, where $\{\bar U_\alpha\}$ is a uniform
cover of $Y$ and for each $i$, $\{\bar V_\alpha^i\}$ is a uniform cover of $X$.
Then another such base is given by covers of the form
$\bar W_\alpha^i\bydef 
\{\st(p,\{\bar U_\alpha\x\bar V_\alpha^i\})\mid p\in Y\ltimes X\}$.
On the other hand, by Lemma \ref{amalgam}, a base of uniform covers of
$Y\ltimes A\cup_{B\ltimes A}B\ltimes X$ is given by
$W_\alpha^i\bydef \{\bar U_\alpha\x V_\alpha^i\}+\{U_\alpha\x\bar V_\alpha^i\}$
in the same notation as above.
Since each $W_\alpha^i=\bar W_\alpha^i\cap (Y\ltimes A\cup B\ltimes X)$, we obtain
that the injective map $Y\ltimes A\cup_{B\ltimes A}B\ltimes X\to Y\ltimes X$
(furnished by the categorical definition of a pushout) is a uniform embedding.
\end{proof}

\subsection{Functional space II}

\begin{theorem}\label{A.3}
If $X$ is a metrizable uniform space, $A$ is a subset of $X$,
and $Y$ and $B\incl Y$ are A[N]RUs, then $U((X,A),(Y,B))$ is an A[N]RU.
\end{theorem}

The case $A=\emptyset$ is known (see Theorem \ref{basic ARU'}); the proof of
the general case is based on the same idea but additionally employs
Corollary \ref{amalgam-product2}(b).

\begin{proof} We consider the ARU case; the ANRU case is similar (cf.\ the proof
of Theorem \ref{basic ARU'}).
Since $B$ is an ARU, it is complete, and therefore closed in $Y$.
Then without loss of generality $A$ is closed in $X$ (else it can be replaced by its
closure).
Pick a pair $(Z,C)$ of uniform spaces with $C$ closed in $Z$ and a uniformly continuous $f\:C\to U((X,A),(Y,B))$.
We now apply Corollary \ref{amalgam-product2}(b).
Since $C\ltimes A\to C\ltimes X$ is a uniform embedding, $f$ determines a uniformly
continuous map $\Phi\:(C\ltimes X,\,C\ltimes A)\to (Y,B)$.
Since $C\ltimes A\to Z\ltimes A$ is a uniform embedding and $B$ is an ARU,
the restriction $\psi\:C\ltimes A\to B$ of $\Phi$ extends to a uniformly continuous
$\bar\psi\:Z\ltimes A\to B$.
Since $Z\ltimes A\cup_{C\ltimes A}C\ltimes X\to Z\ltimes X$ is a uniform embedding
and $Y$ is an ARU, $\Phi\cup_\psi\bar\psi\:Z\ltimes A\cup_{C\ltimes A}C\ltimes X\to Y$
extends to a uniformly continuous $\bar\Phi\:Z\ltimes X\to Y$.
This map $\bar\Phi\:(Z\ltimes X,\,Z\ltimes A)\to (Y,B)$ determines a uniformly
continuous $\bar f\:Z\to U((X,A),(Y,B))$ extending $f$.
\end{proof}

\begin{corollary}\label{A.3'} If $Y$ is a metrizable uniform space, $B\incl Y$, and
$X$ and $A\incl X$ are uniform A[N]Rs, then $U((Y,B),(X,A))$ is a uniform A[N]R.
\end{corollary}

For an alternative proof, see Theorem \ref{A.3''}.

\begin{proof} We consider the case of uniform ANRs.
The case of uniform ARs follows e.g.\ using Theorem \ref{uniform AR}.

We first construct a uniformly continuous homotopy $h_t$ of
$(\bar X,\bar A)$ such that $h_0=\id$ and $h_t(\bar X,\bar A)\incl (X,A)$ for $t>0$.
Let $h^A_t$ and $h^X_t$ be uniformly continuous homotopies of $\bar A$ and $\bar X$
such that $h^A_0=\id$, $h^X_0=\id$, $h^A_t(\bar A)\incl A$ and
$h^X_t(\bar X)\incl X$ for $t>0$ (see Lemma \ref{uniform ANR}).
Define $h^1_t\:\bar X\to\bar X$ by $h^1_t(x)=h^X_{td(x,\bar A)}(x)$, where
$d$ is a metric on $\bar X$ bounded by $1$.
Then $h^1\:\bar X\x I\to\bar X$ is uniformly continuous, $h^1_0=\id$,
$h^1_t|_{\bar A}=\id$, and $h^1_t(\bar X\but\bar A)\incl X$ for $t>0$.

On the other hand, since $X$ is a uniform ANR and $\bar A\cap X$ is closed in $X$,
the homotopy $h^A_t|_{\bar A\cap X}$ can be extended to a uniformly continuous
homotopy $H_t\:U\to X$, where $U$ is a closed uniform neighborhood of $A$ in $X$,
so that $H_0=\id$.
The latter in turn extends by continuity to a uniformly continuous homotopy
$\bar H_t\:\bar U\to\bar X$, which necessarily restricts to $h^A_t$ on $\bar A$.
Define $r\:\bar X\x I\to\bar U\x I\cup\bar X\x\{0\}$ by
$r(x,t)=(x,t(1-ud(x,A)))$, where $u=1/d(\bar U,\bar X\but\bar U)$.
Then $r$ is the identity on $\bar A\x I$, sends $\bar U\x I$ into itself,
and $(\bar X\but\bar U)\x I$ into $(\bar X\but\bar U)\x\{0\}$.
Define $h^2_t\:\bar X\to\bar X$ by $h^2_t(x)=\bar H_t(r(x,t))$ for
$x\in\bar U$ and by $h^2_t(x)=x$ for $x\notin\bar U$.
Then $h^2_0\:\bar X\x I\to\bar X$ is uniformly continuous, $h^2_0=\id$,
$h^2_t(X)=h^2_t(U)\cup h^2_t(X\but U)\incl X$ for all $t$, and
$h^2_t(\bar A)\incl A$ for $t>0$.

We set $h_t=h^2_th^1_t$.
Then $h_0=\id$, $h_t(\bar X)=h^2_t(h^1_t(\bar A)\cup h^1_t(\bar X\but\bar A))
\subset h^2_t(\bar A\cup X)\subset X$ for $t>0$, and
$h_t(\bar A)=h^2_t(h^1_t(\bar A))=h^2_t(\bar A)\subset A$ for $t>0$.

Now define a homotopy $\phi_t$ of $U((Y,B),(\bar X,\bar A))$ by $\phi_t(f)=h_tf$.
Then $\phi_t$ is uniformly continuous, $\phi_0=\id$ and $\phi_t$ sends
$U((Y,B),(\bar X,\bar A))$ into $U((Y,B),(X,A))$ for $t>0$.
It follows that $U((Y,B),(\bar X,\bar A))$ is a completion of $U((Y,B),(X,A))$.
Now the assertion follows from Theorems \ref{uniform ANR} and \ref{A.3}.
\end{proof}

\newpage
\part{INVERSE LIMITS}\label{inverse limits}

\section{Convergence and stability}\label{invlim stability}

The definition of the inverse limit of metrizable uniform spaces 
$X_1,X_2,\dots$ and uniformly continuous maps $f_i\:X_{i+1}\to X_i$
is reviewed in \S\ref{invlimits}.
It is easy to check that if each $X_i$ is complete; separable; 
point-finite; star-finite; or Noetherian, then so is their inverse limit.
And if each $X_i$ is a uniform local compactum and each $f_i$ is proper
(i.e.\ the preimage of every compactum is a compactum), then 
the inverse limit is a uniform local compactum.

\subsection{$\eps$-separating maps}
Let $f\:X\to Y$ be a map between metric spaces.
We recall that $f$ is uniformly continuous iff for each $\eps>0$ there
exists a $\delta>0$ such that $f$ is {\it $(\delta,\eps)$-continuous},
that is, sends $\delta$-close (=at most $\delta$-close) points into
$\eps$-close points.
Note that the Hahn property (see \S\ref{uniform-anrs}) involved $\eps$-continuous
maps, i.e.\ maps that are $(\delta,\eps)$-continuous for some $\delta>0$.
Dually, we say that $f$ is {\it $(\eps,\delta)$-separating} if $\delta$-close
points have $\eps$-close point-inverses; and {\it $\eps$-separating} if it is
$(\eps,\delta)$-separating for some $\delta>0$.
Note that when $X$ is compact, the latter is equivalent to the more familiar
notion of an ``$\eps$-map'', which is that $f^{-1}(x)$ is of diameter $<\eps$
for every $x\in X$.
$\eps$-Separating maps were known to Isbell \cite{I4}, who called them simply
``$\eps$-mappings''.

\begin{lemma}\label{A.11} Given an inverse sequence of metric spaces $X_i$ and
uniformly continuous maps $p_i$, for each $\eps>0$ there exists an $i$ such that
$p^\infty_i\:\invlim X_j\to X_i$ is $\eps$-separating.
\end{lemma}

\begin{proof} Let $C$ be the cover of $\invlim X_j$ by all sets of
diameter $\eps$.
Since $C$ is uniform, it is refined by $(p^\infty_i)^{-1}(C_i)$ for some
uniform cover $C_i$ of $X_i$.
If $\lambda$ is the Lebesgue number of $C_i$, we obtain that $p^\infty_i$
is $(\eps,\lambda)$-separating.
\end{proof}

\subsection{Extended mapping telescope} \label{mapping telescope}
Let us consider an inverse sequence $X=(\dots\xr{f_1}X_1\xr{f_0}X_0)$ of uniformly
continuous maps between metrizable uniform spaces.
Let $X_{\N\cup\infty}=\invlim(\dots\xr{}X_2\sqcup X_1\sqcup X_0
\xr{\id\sqcup\id\sqcup f_1}X_1\sqcup X_0\xr{\id\sqcup f_0}X_0)$.
Given any subset $J\incl\N\cup\{\infty\}$, we denote by $X_J$ the preimage of $J$
under the projection $X_{\N\cup\infty}\to\N\cup\{\infty\}$.
Note that $X_\infty$ is identified with $\invlim X$, and its complement $X_{\N}$
is homeomorphic to the topological space $\bigsqcup_{i\in\N} X_i$.
If each $X_i$ is complete, $X_{\N\cup\infty}$ is the completion of $X_{\N}$.

We may further define $X_{[0,\infty]}$ to be the inverse limit of the finite
mapping telescopes
$X_{[0,n]}=MC(f_0)\cup_{X_1}MC(f_1)\cup_{X_2}\dots\cup_{X_{n-1}}MC(f_{n-1})$
and the obvious retractions $X_{[0,n+1]}\to X_{[0,n]}$.
If $J\incl [0,\infty]$, by $X_J$ we denote the preimage of $J$ under the obvious
projection $X_{[0,\infty]}\to [0,\infty]$.
Clearly, $X_{\N\cup\infty}$ is same as before.

\subsection{Convergent inverse sequences}
Let us call the inverse sequence $X$ {\it convergent} if every uniform
neighborhood of $\invlim X$ in $X_{\N\cup\infty}$ (or equivalently in
$X_{[0,\infty]}$) contains all but finitely many of $X_i$'s.
We say that $X$ is {\it Cauchy} if for every $\eps>0$ there exists a $k$
such that for every $j>k$, the $\eps$-neighborhood of $X_j$ in $X_{\N}$
(or equivalently in $X_{[0,\infty)}$) contains $X_k$.
Clearly, these notions depend only on the underlying uniform structures.
Here is an example of a Cauchy inverse sequence that is divergent:
$\dots\incl(0,\frac14]\incl(0,\frac12]\incl(0,1]$.
The inverse sequence $\dots\incl[2,\infty)\incl[1,\infty)\incl[0,\infty)$ fails
to be Cauchy.

The close analogy with the definition of a convergent/Cauchy sequence can be
formalized.
It is easy to see that the inverse sequence $X$ is convergent (Cauchy) iff the
sequence of the closed subsets $X_i$ of $X_{\N\cup\infty}$ is convergent (Cauchy)
in the hyperspace $H(X_{\N\cup\infty})$, or equivalently in $H(X_{[0,\infty]})$.
This implies parts (a) and (b) of the following lemma.

\begin{lemma}\label{A.12} \cite{M} Let $X=(\dots\xr{f_1}X_1\xr{f_0}X_0)$ be an
inverse sequence of uniformly continuous maps between metric spaces.

\begin{parts}
\item If each $X_i$ is compact, $X$ is convergent.

\item If $X$ is convergent, it is Cauchy; the converse holds when
each $X_i$ is complete.

\item $X$ is convergent if and only if for each $i$, each uniform
neighborhood of $f^\infty_i(\invlim X)$ in $X_i$ contains all but finitely many of
$f^j_i(X_j)$'s.

\item $X$ is Cauchy if and only if for each $i$ and every $\eps>0$
there exists a $k$ such that for every $j>k$, the $\eps$-neighborhood of
$f^j_i(X_j)$ in $X_i$ contains $f^k_i(X_k)$.

\item If each $X_i$ is uniformly discrete, $X$ is convergent if and
only if it satisfies the Mittag-Leffler condition: for each $i$, the images of
$X_k$ in $X_i$ stabilize, i.e.\ there exists a $j>i$ such that
$p^k_i(X_k)=p^j_i(X_j)$ for each $k>j$.

\item If $X$ is convergent and each $X_i$ is non-empty (resp.\
uniformly connected), then $\invlim X$ is non-empty (resp.\ uniformly connected).

\item If $Y=(\dots\xr{g_1}Y_1\xr{g_0}Y_0)$ is another inverse
sequence of uniformly continuous maps between metric spaces, $h_i\:X_i\to Y_i$
are surjections commuting with the bonding maps and $X$ is convergent (Cauchy),
then $Y$ is convergent (resp.\ Cauchy).
\end{parts}
\end{lemma}

It should be noted that the map $\invlim X\to\invlim Y$ in (g) is
not necessarily a surjection, in contrast to the compact case.
For instance, take $X_i=\N\bydef \{0,1,\dots\}$ and $Y_i=[i]\bydef \{0,\dots,i-1\}$
(discrete uniform spaces), and let each $f_i$ be the identity map, and
$g_i\:[i+1]\to [i]$ and $h_i\:\N\to [i]$ the retractions with the only
non-degenerate point-inverses being those of $i-1$.
Then $\invlim Y$ is homeomorphic to the one-point compactification of $\N$,
where the remainder point is not in the image of $\invlim X$.

See \cite{M} for an alternative proof of (a) and a more detailed proof of (c).

\begin{proof}[Proof] Parts (a) and (b) have been proved above.
Parts (c) and (d) follow using that each
$p^\infty_i\:X_{\{i,i+1,\dots,\infty\}}\to X_i$ (which is the restriction of
$p^\infty_i\:X_{\N\cup\infty}\to X_1\sqcup\dots\sqcup X_i$)
is uniformly continuous and $\eps_i$-separating, where $\eps_i\to 0$ as
$i\to\infty$, by Lemma \ref{A.11}.
Parts (e) and (f) follow from (c); part (g) follows from (c) and (d).
\end{proof}

\subsection{Applications}
From Lemma \ref{A.12}(b,d) we immediately obtain (compare \cite{M})

\begin{corollary}[Bourbaki's Mittag-Leffler Theorem] \label{A.13} Let $L$ be
the limit of an inverse sequence $X=(\dots\xr{f_1}X_1\xr{f_0}X_0)$ of uniformly
continuous maps between complete metrizable uniform spaces.
If each $f_i(X_{i+1})$ is dense in $X_i$, $f^\infty_0(L)$ is dense in $X_0$.
\end{corollary}

As observed by V. Runde, a special case of Corollary \ref{A.13} is (compare \cite{M})

\begin{corollary}[Baire Category Theorem] \label{A.14}
The intersection of a countable collection of dense open sets in a complete
metrizable uniform space is dense.
\end{corollary}

The following result will be used in the proof of Theorem \ref{KR-complete}.

\begin{proposition} \label{corona}
Let $X$ be a complete metric space.
For each $i\in\N$ let $\eps_i>0$, let $K_i$ be a compact subset of $X$ and let $Z_i$ be the closed 
$\eps_i$-neighborhood of $K_i$.

(a) If $\eps_i\to 0$, then $Z\bydef \bigcap_{i=1}^\infty Z_i$ is compact.

(b) For each $n$ let $Y_n=\bigcap_{i=1}^n Z_i$.
If $\sum_{i=1}^\infty\eps_i<\infty$, then the inverse sequence of inclusions $\dots\to Y_2\to Y_1$ is convergent.

(c) Suppose that $X$ is embedded, as a topological space, in a completely metrizable space $Q$.
For each $i$ let $\bar Y_i$ be the closure of $Y_i$ in $Q$.
If $\sum_{i=1}^\infty\eps_i<\infty$, then $\bigcap_{i=1}^\infty\bar Y_i=Z$.
\end{proposition}

\begin{proof}[Proof. (a)] Since $Z$ is closed, it is complete.
Also each $Z_i$ can be covered by finitely many $2\eps_i$-balls with centers in $K_i$.
Hence $Z$, being a subset of $Z_i$, is covered by finitely many sets of diameter at most $2\eps_i$,
and therefore also by finitely many balls of radius $4\eps_i$.
Thus $Z$ is totally bounded and complete, and therefore it is compact.
\end{proof}

\begin{proof}[(b)]
By the proof of (a) each $Y_n$, being a subset of $Z_n$, admits a finite cover $C_n$ by sets of diameter 
at most $2\eps_i$.
Let us define maps $f_i\:C_{i+1}\to C_i$ by sending every $U\in C_{i+1}$ to some $V\in C_i$ such that
$U\cap V\ne\emptyset$.
Since the $C_i$ are finite, the inverse sequence $\dots\to C_2\to C_1$ satisfies the Mittag-Leffler condition.
Hence for each $n$ there exists an $m>n$ such that the image of $f^m_n\:C_m\to C_n$ equals that of
$f^\infty_n\:C\to C_n$, where $C=\invlim C_i$.

Given a $y_m\in Y_m$, it is contained in some $U_m\in C_m$.
For $i=n,\dots,{m-1}$ let $U_i=f^m_i(U_m)$, and let $y_i\in U_i\cap U_{i+1}$.
Then each $d(y_i,y_{i+1})\le2\eps_{i+1}$, and therefore $d(y_m,y_n)\le2(\eps_{n+1}+\dots+\eps_m)<2\delta_{n+1}$,
where $\delta_k=\sum_{i=k}^\infty\eps_i$.

Let $(V_i)\in C$ be some thread with $V_n=U_n$.
For each $i\ge n$ let $x_i\in V_i\cap V_{i+1}$.
Then $d(y_n,x_n)\le 2\eps_n$ and each $d(x_i,x_{i+1})\le 2\eps_{i+1}$.
Since $\delta_k\to 0$ as $k\to\infty$, the sequence $x_n,x_{n+1},\dots$ is fundamental.
Since $X$ is complete, it converges to some $x\in X$.
For each $i\ge n$ the sequence $x_i,x_{i+1},\dots$ lies in $Y_i$.
Since $Y_i$ is closed, its limit $x$ also lies in $Y_i$.
Hence $x\in\bigcap_{i=n}^\infty Y_i=Z$.
We have $d(x_n,x)\le 2\delta_{n+1}$.
Since $d(y_m,y_n)\le 2\delta_{n+1}$ and $d(y_n,x_n)\le 2\eps_n$, we get 
$d(y_m,x)<4\delta_{n+1}+2\eps_n<4\delta_n$.
Thus $Y_m$ lies in the $4\delta_n$-neighborhood of $Z$.

Given an $\eps>0$, let $n$ be such that $4\delta_n<\eps$, and let $m$ be such that the image of $f^m_n$
equals that of $f^\infty_n$.
Then $Y_m$ lies in the $\eps$-neighborhood of $Z$.
\end{proof}

\begin{proof}[(c)]
Since $X$ is completely metrizable, it is a $G_\delta$ subset of $Q$.
Hence $Q\but X$ can be represented as $\bigcup_{i=1}^\infty Q_i$, where each $Q_i$ is closed in $Q$.
Given an $n$, let $R_n$ be a closed neighborhood of $Z$ in $Q\but Q_n$.
Since $Z$ is compact, its neighborhood $R_n\cap X$ in $X$ contains the $\eps$-neighborhood of $Z$
in the original metric on $X$.
Since the inverse sequence $\dots\to Y_2\to Y_1$ is convergent, the latter neighborhood contains some $Y_m$.
In particular, $\bar Y_m$ lies in $R_n$ and so is disjoint from $Q_n$.
Hence $\bigcap_{i=1}^\infty\bar Y_m$ is disjoint from $\bigcup_{i=1}^\infty Q_i$, and therefore lies in $X$.
\end{proof}

\subsection{Maps between inverse limits}
The following result says, in particular, that inverse limits are stable under
sufficiently small perturbations of inverse sequences.

\begin{proposition}\label{A.15} Let $\dots\xr{p^2_1}X_1\xr{p^1_0}X_0$ and
$\dots\xr{q^2_1}Y_1\xr{q^1_0}Y_0$ be inverse sequences of uniformly continuous
maps between metric spaces, where the $Y_i$ are complete, and let $X$ and $Y$
be their inverse limits.
Let $f_i\:X_i\to Y_i$, $i=0,1,\dots$, be uniformly continuous maps.
A uniformly continuous map $f\:X\to Y$ such that each $q^\infty_if$ is
$2\beta_i$-close to $f_i p^\infty_i$

\begin{parts}
\item exists, provided that

\begin{embedded roster}\item
{\tt
$f_ip^{i+1}_i$ and $q^{i+1}_if_{i+1}$ are $\alpha_i$-close for each $i$, where

\item $\alpha_i>0$ is such that $q^i_j$ is
$(\alpha_i,2^{j-i}\beta_j)$-continuous for each $j\le i$;
}
\end{embedded roster}

\item is unique, if in addition to (i) and (ii) the following holds:

\begin{embedded roster}[2]\item
{\tt
each $\beta_i>0$ is such that $q^\infty_i$ is $(\delta_i,9\beta_i)$-separating,
where $\delta_i>0$ is such that $q^\infty_i$ is $\delta_i$-separating
and $\delta_i\to 0$ as $i\to\infty$;}
\end{embedded roster}

\item is a uniform homeomorphism onto its image, if

\begin{embedded roster}[3]\item
{\tt
each $f_i$ is $(\gamma_i,5\beta_i)$-separating, where

\item each $\gamma_i>0$ is such that $p^\infty_i$ is
$(\eps_i,\gamma_i)$-separating, where $\eps_i>0$ is such that
$p^\infty_i$ is $\eps_i$-separating and $\eps_i\to 0$ as $i\to\infty$;
}
\end{embedded roster}

\item is surjective, if in addition to (i)--(v) the following holds:
\begin{embedded roster}[5]\item
{\tt
every $y'_i\in Y_i$ is $\alpha_i$-close to some $y_i\in f_i(X_i)$;

\item $X$ is complete and $\dots\xr{p^2_1}X_1\xr{p^1_0}X_0$ is convergent.
}
\end{embedded roster}
\end{parts}
\end{proposition}

Admittedly the statement of Proposition \ref{A.15} is rather cumbersome.
For some purposes it becomes more revealing if simplified in one or both of
the following ways.

\begin{corollary}\label{A.15'}
Proposition \ref{A.15}(a,b,c,d) holds

\begin{parts}
\item[{\rm (e)}] if each $X_i$ has diameter $\le 1$, and condition (v) is
replaced by

\begin{embedded roster}
\item[{\tt\quad(v$'$)}] {\tt each $\gamma_i>0$ is such that $p^i_j$ is
$(\gamma_i,2^{j-i})$-continuous for all $j\le i$;}
\end{embedded roster}

\item[{\rm (f)}] if each $Y_i$ has diameter $\le 1$, and condition (iii) is
replaced by

\begin{embedded roster}
\item[{\tt\quad(iii$'$)}] {\tt each $\beta_i>0$ is such that $q^i_j$ is
$(9\beta_i,2^{j-i})$-continuous for $j\le i$;}
\end{embedded roster}

\item[{\rm (g)}] when (e) and (f) are combined.
\end{parts}
\end{corollary}

\begin{proof}
If each $X_i$ has diameter $\le 1$, we may endow $X$ with the metric
$d((x_i),(x'_i))=\sup\{2^{-i}d(x_i,x'_i)\mid i\in\N\}$ and take $\eps_i=2^{-i}$.
Then (v) follows from (v$'$).

Similarly, if each $Y_i$ has diameter $\le 1$, we may endow $Y$ with the metric
$d((y_i),(y'_i))=\sup\{2^{-i}d(y_i,y'_i)\mid i\in\N\}$ and take $\delta_i=2^{-i}$.
Then (iii) follows from (iii$'$).
\end{proof}

In the compact case, a version of Corollary \ref{A.15'}(g) was obtained by
Rogers \cite{Ro}.
His proof is by a different method, reducing the general case to the case where
$p^i_j$ and $q^i_j$ are embeddings, and involving what appears to be
a substantial use of compactness.

We note the following regarding the proof of Proposition \ref{A.15}.
The proof of (d) will not be simplified if the $f_i$ are assumed to be
surjective.
If the $f_i$ are only assumed to be continuous, rather than uniformly continuous,
then the hypotheses of (a) and (c) still imply that $f$ is a homeomorphism onto
its image and $f^{-1}$ is uniformly continuous.
The constant $9$ in (iii) and (iii$'$) is relevant for (d), but can be replaced
by $5$ for the purposes of (b).

\begin{proof}[Proof. (a)] Let us consider the compositions
$F^{(i)}_j\:X\xr{p^\infty_i}X_i\xr{f_i}Y_i\xr{q^i_j}Y_j$.
By (i) and (ii), $F^{(i+1)}_j$ and $F^{(i)}_j$ are $2^{j-i}\beta_j$-close for
every $i\ge j$.
Since $Y_j$ is complete, $F^{(j+k)}_j$ uniformly converge to
a map $F_j\:X\to Y_j$.
Since $2^0+2^{-1}+\dots=2$, it is $2\beta_j$-close to
$F^{(j)}_j=f_jp^\infty_j$.
Each $f_i$ is uniformly continuous, hence so is each $F^{(j+k)}_j$ and
consequently their uniform limit $F_j$.
Since each $F^{(i)}_j=q^{j+1}_jF^{(i)}_{j+1}$, we obtain
$F_j=q^{j+1}_jF_{j+1}$.
Then by the universal property of inverse limits
there exists a uniformly continuous map $f\:X\to Y$ such that
$q^\infty_jf=F_j$ for each $j$.
In particular, $q^\infty_jf$ is $2\beta_j$-close to $f_jp^\infty_j$
for each $j$.
\end{proof}

\begin{proof}[(b)]
Given another map $f'\:X\to Y$ such that $q^\infty_jf'$ is
$2\beta_j$-close to $f_jp^\infty_j$ for each $j$, we have that
$q^\infty_jf'$ is $4\beta_j$-close to $q^\infty_jf$ for each $j$.
Hence by (iii), $f'$ is $\delta_j$-close to $f$ for each $j$,
i.e.\ $f'=f$.
\end{proof}

\begin{proof}[(c)] By (iv) and (v), each $f_jp^\infty_j$ is
$(\eps_j,5\beta_j)$-separating.
By the hypothesis it is $2\beta_j$-close to $q^\infty_jf$, so the latter must be
$(\eps_j,5\beta_j-2\beta_j-2\beta_j)$-separating.
In particular, $q^\infty_jf$ is $\eps_j$-separating; hence so is $f$.
Since this holds for every $j$, and $\eps_j\to 0$ as $j\to\infty$, $f$ must
be injective.
On the other hand, since $q^\infty_jf$ is $(\eps_j,\beta_j)$-separating,
and the uniformly continuous $q^\infty_j$ is $(\lambda_j,\beta_j)$-continuous
for some $\lambda_j>0$, $f$ is $(\eps_j,\lambda_j)$-separating.
Then $f^{-1}\:f(X)\to X$ is $(\lambda_j,\eps_j)$-continuous for every $j$.
Since $\eps_j\to 0$ as $j\to\infty$, we conclude that $f^{-1}$ is uniformly
continuous.
\end{proof}

\begin{proof}[(d)] Pick some $y\in Y$.
By (vi), each $q^\infty_i(y)$ is $\alpha_i$-close to some
$y_i=f_i(x_i)$ for some $x_i\in X_i$.
Let $\mu_j$ be such that $f_j$ is $(\mu_j,\alpha_j)$-continuous.
By the convergence hypothesis and Lemma \ref{A.12}(c), there exists
an $i=\phi(j)$ such that $p^i_j(X_i)$ is contained in the $\mu_j$-neighborhood
of $p^\infty_j(X)$.
Then $p^i_j(x_i)$ is $\mu_j$-close to $p^\infty_j(x_{(j)})$ for some
$x_{(j)}\in X$.
Hence $y^i_j\bydef f_jp^i_j(x_i)$ is $\alpha_j$-close to
$y^{(j)}_j\bydef f_jp^\infty_j(x_{(j)})$.
Now let us consider an arbitrary $k<j$.
By (ii), $q^j_k(y^i_j)$ is $\frac12\beta_k$-close to $q^j_k(y^{(j)}_j)$.
On the other hand, by (i) and (ii), $q^j_kf_j$ is $2\beta_k$-close to
$f_kp^j_k$ (and even $(2\beta_k-\iota)$-close for some $\iota>0$, using that
$1+\frac12+\dots+2^{k-j}<2$); hence $q^j_k(y^i_j)$ is $2\beta_k$-close to
$y^i_k$, and $q^j_k(y^{(j)}_j)$ is $2\beta_k$-close to
$y^{(j)}_k\bydef f_kp^\infty_k(x_{(j)})$.
Summing up, $y^{(j)}_k$ is $\frac92\beta_k$-close to $y^i_k$.
By the above, the latter is in turn $\alpha_k$-close, hence by (ii)
$\beta_k$-close to $y^{(k)}_k$.
Thus $y^{(j)}_k$ is $\frac{11}2\beta_k$-close to $y^{(k)}_k$.
Since by (iv) and (v), $f_kp^\infty_k$ is $(\eps_k,5\beta_k)$-separating,
we conclude that $x_{(j)}$ is $\eps_k$-close to $x_{(k)}$.
Consequently $x_{(1)},x_{(2)},\dots$ is a Cauchy sequence, and since
$X$ is complete, it converges to some $x\in X$.

By (b) we have $q^\infty_jf=F_j$ in the notation in the proof of (a),
where $F_j(x)$ is the limit of $F^{(i)}_j(x)$'s as $i\to\infty$, each
$F^{(i)}_j(x)$ in turn being the limit of $F^{(i)}_j(x_{(l)})$'s as $l\to\infty$.
We have $F^{(i)}_j(x_{(l)})=q^i_jf_ip^\infty_i(x_{(l)})=
q^i_j(y^{(l)}_i)$.
From (i) and (ii), $q^i_j(y^{(l)}_i)$ is $2\beta_j$-close to $y^{(l)}_j$.
Without loss of generality, $l>j$; then by the above, $y^{(l)}_j$ is
$\frac92\beta_j$-close to $y^m_j$, where $m=\phi(l)$.
From (i) and (ii), $y^m_j$ is $2\beta_j$-close to $q^m_j(y_m)$.
By our choice of $y_m$, this $y_m$ is $\alpha_m$-close to $q^\infty_m(y)$,
hence by (ii), $q^m_j(y_m)$ is $\frac12\beta_j$-close to $q^\infty_j(y)$
(using that $m\ge l<j+1$).
To summarize, $F^{(i)}_j(x_{(l)})=q^i_j(y^{(l)}_i)$ is $9\beta_j$-close to
$q^\infty_j(y)$; in fact, they are even $(9\beta_j-\iota)$-close for some
$\iota>0$.
Hence $F^{(i)}_j(x)$ is $(9\beta_j-\frac\iota2)$-close to $q^\infty_j(y)$
for each $i$.
Then $F_j(x)$ is $9\beta_j$-close to $q^\infty_j(y)$.
Since $F_j=q^\infty_jf$, by (iii), $f(x)$ is $\delta_j$-close to $y$ for
each $j$.
Since $\delta_j\to 0$ as $j\to\infty$, we obtain that $f(x)=y$.
\end{proof}

The following is a direct consequence of Corollary \ref{A.15'}(a,b,f);
the compact case (apart from the uniqueness) is due to Mioduszewski \cite{Mio}.

\begin{corollary}\label{miod1} Let be $\dots\xr{q_1}Y_1\xr{q_0}Y_0$ an inverse
sequence of uniformly continuous maps between complete metric spaces, and
let $Y$ be its inverse limit.
Then there exists a sequence of $\beta_i^*>0$ such that for each sequence of
$\beta_i\in (0,\beta_i^*]$ there exists a sequence of $\alpha_i>0$ such that
the following holds.
Suppose $\dots\xr{p_1}X_1\xr{p_0}X_0$ is an inverse sequence of uniformly
continuous maps between metrizable uniform spaces, and $X$ is its inverse limit.
If $n_i$ is a non-decreasing unbounded sequence of natural numbers, and
$f_i\:X_{n_i}\to Y_i$ are uniformly continuous maps such that the diagram
$$\begin{CD}
X_{n_{i+1}}@>f_{i+1}>>Y_{i+1}\\
@Vp^{n_{i+1}}_{n_i}VV@Vq_iVV\\
X_{n_i}@>f_i>>Y_i
\end{CD}$$ $\alpha_i$-commutes for each $i$, then there exists a unique uniformly
continuous map $f\:X\to Y$ such that the diagram
$$\begin{CD}
X@>f>>Y\\
@Vp^\infty_{n_i}VV@Vq^\infty_iVV\\
X_{n_i}@>f_i>>Y_i
\end{CD}$$
$\beta_i$-commutes for each $i$.
\end{corollary}

\section{Inverse sequences of uniform ANRs}

\subsection{Inverse limits of cubohedra}

\begin{lemma}\label{intersection of cubohedra}
Every separable metrizable complete uniform space is the limit of a convergent
inverse sequence of uniform embeddings between completed cubohedra.
\end{lemma}

\begin{proof} By Aharoni's theorem \ref{aharoni}, the given space
uniformly embeds onto a subset $X$ of $q_0\subset c_0$.
For each $n$, the cubohedron $(2^{-n}Q)^\omega$ (see \S\ref{cubohedra})
cubulates $c_{00}$, and so is dense in $c_0$.
Let $P_n$ be the minimal subcomplex of $(2^{-n}Q)^\omega$ such that its completion
$\bar P_n$ contains $X$.
(For instance, if $X=\{(1,\frac12,\frac14,\dots,)\}$, then
$P_0=\{1\}\x[0,1]^\omega$, $P_1=\{1\}\x\{\frac12\}\x[0,\frac12]^\omega$, etc.)
If $x\in\bar P_n$, then $x$ is $2^{-n}$-close to a point of $X$.
Since $X$ is closed in $q_0$, we get that $X=\bigcap_{n=0}^\infty\bar P_n$.
Clearly, the inverse sequence of $\bar P_n$ is convergent.
\end{proof}

Since cubohedra are uniform ANRs (Theorem \ref{cubohedron}), we get
(see Theorem \ref{uniform ANR})

\begin{theorem}\label{intersection of cubohedra2}
Every separable metrizable complete uniform space is the limit of
a convergent inverse sequence of separable complete uniform ANRs.
\end{theorem}

We shall prove a somewhat stronger and more flexible form of this result, by
considering nerves of point-finite covers, in a sequel to this paper dealing
with uniform polyhedra \cite{M3}.

The uniformly finitistic case of Theorem \ref{intersection of
cubohedra2} is due to Isbell \cite[Theorem V.34]{I3} (see also \cite[Lemma 1.6]{CI},
\cite[7.2]{I2}).
It was also known to him that every complete uniform ANR is the inverse
limit of an uncountable inverse spectrum of ANRUs \cite[7.1]{I2}.

\begin{remark}
The metrizable complete uniform space $U(\N,I)$ is not the limit of any inverse
sequence of separable metrizable uniform spaces, since it is not point-finite.
\end{remark}

\begin{proposition}\label{limit of cubohedra}
Every separable complete metric space is homeomorphic to the limit of a convergent inverse sequence 
of uniformly continuous maps between finite-dimensional locally finite cubohedra.
\end{proposition}

This is a version of Isbell's result \cite{I6}*{Corollary 3.7}, \cite{I4}*{Theorem 3}
(see also \cite[Lemma 1.6]{CI}, \cite[Theorem V.34]{I3}).
We refer to \cite{M2}*{Remark \ref{book:nerves and cubes}} for a discussion of related results and their proofs.

\begin{proof} By Theorem \ref{product-embedding}, the given space is homeomorphic
to a closed subset $X$ of the infinite product $\R^\infty$.
Let $r_n\:\R^\infty\to\R^n$ be the projection onto the first $n$ factors, and let
$P_n$ be the cellular neighborhood of $r_n(X)$ in the cubohedron $(2^{-n}Q)^n$ (see \S\ref{cubohedra}).
Clearly, each projection $p_n\:\R^{n+1}\to\R^n$ sends $P_{n+1}$ into $P_n$, and the inverse sequence
$\dots\to P_1\to P_0$ is convergent, with limit $X$.
\end{proof}

\subsection{Resolutions}

\begin{lemma}\label{Mardesic}
Let $X$ be the limit of a convergent inverse sequence of metrizable
uniform spaces $X_i$ and uniformly continuous maps $p_i$, and let $Y$ be
a metric space.

(a) Suppose that $Y$ satisfies the Hahn property.
Then for every uniformly continuous map $f\:X\to Y$ and each $\eps>0$ there exists
a $j$ and a uniformly continuous map $g\:X_j\to Y$ such that $f$ is $\eps$-close
to the composition $X\xr{p^\infty_j}X_j\xr{g}Y$.

(b) Suppose that $f,g\:X_i\to Y$ are uniformly continuous maps such that the two
compositions $X\xr{p^\infty_i}X_i\overset{f}{\underset{g}{\rightrightarrows}}Y$
are $\eps$-close.
Then for each $\delta>0$ there exists a $k$ such that the two compositions
$X_k\xr{p^k_i}X_i\overset{f}{\underset{g}{\rightrightarrows}}Y$ are
$(\eps+\delta)$-close.
\end{lemma}

The conclusion resembles the definition of a resolution from shape theory.
A version of the compact case of (a) is found already in the Eilenberg--Steenrod book 
\cite[Theorem 11.9]{ES}.

A special case of (a), with the convergence hypothesis strengthened to
surjectivity of all bonding maps, was known to Isbell \cite[7.4]{I2}.

\begin{proof}[(a)] Let $\delta=\delta_{Hahn}(\eps/2)$ be given by the Hahn
property for $\eps_{Hanh}=\eps/2$.
Let us fix some metrics on $X$ and the $X_i$'s.
There exists a $\gamma>0$ such that $f$ is $(\gamma,\delta)$-continuous.
Then by Lemma \ref{A.11} there exists an $i$ such that $p^\infty_i\:X\to X_i$ is
$\gamma$-separating.
Hence $p^\infty_i$ is $(\gamma,3\beta)$-separating for some $\beta>0$.
Then by Lemma \ref{A.12}(c) there exists a $j$ such that the
$\beta$-neighborhood of $p^\infty_i(X)$ contains $p^j_i(X_j)$.

Now every $x\in U_\beta$ is $\beta$-close to $p^\infty_i(\phi(x))$ for some
$\phi(x)\in X$.
(If $X_i$ is separable, the definition of $\phi\:U_\beta\to X$ requires only
the countable axiom of choice.
We do not require that $\phi p^\infty=\id_X$.)
Moreover, if $y$ is $\beta$-close to $x$, then $p^\infty_i(\phi(y))$ is
$3\beta$-close to $p^\infty_i(\phi(x))$.
Hence $\phi(y)$ is $\gamma$-close to $\phi(x)$.
Thus $\phi$ is $(\beta,\gamma)$-continuous.

The composition $U_\beta\xr{\phi}X\xr{f}Y$ is $(\beta,\delta)$-continuous,
and so is $(\eps/2)$-close to a uniformly continuous map $\psi\:U_\beta\to Y$.
Hence also the composition $X\xr{p^\infty_i}U_\beta\xr{\phi}X\xr{f}Y$ is
$(\eps/2)$-close to the composition $X\xr{p^\infty_i}U_\beta\xr{\psi}Y$.
Since $f$ is $(\gamma,\beta)$-separating, the composition
$X\xr{p^\infty_i}U_\beta\xr{\phi}X$ is $\gamma$-close to $\id_X$, and
therefore the composition $X\xr{p^\infty_i}U_\beta\xr{\phi}X\xr{f}Y$ is
$\delta$-close to $f$.
We may assume that $\delta\le\eps/2$.
Then we conclude that $f$ is $\eps$-close to the composition
$X\xr{p^\infty_j}X_j\xr{p^j_i}U_\beta\xr{\psi}Y$.
\end{proof}

\begin{proof}[(b)]
Let us fix some metric on the $X_i$, and pick a $\gamma>0$ such that $f$ and $g$
are $(\gamma,\delta/2)$-continuous.
By Lemma \ref{A.12}(c) there exists a $k$ such that the $\gamma$-neighborhood of
$p^\infty_i(X)$ contains $p^k_i(X_k)$.
Given an $x\in X_k$, pick a $z\in X$ such that $p^\infty_i(z)$ is $\gamma$-close
to $p^k_i(x)$.
Then the $f$- and $g$-images of $p^k_i(x)$ are $(\delta/2)$-close to
those of $p^\infty_i(z)$, which are in turn $\eps$-close to each other.
\end{proof}

\subsection{Combinatorial approximation}

\begin{theorem}\label{miod3} Let $\dots\xr{q_1}Y_1\xr{q_0}Y_0$ be an inverse
sequence of uniformly continuous maps between metric spaces satisfying
the Hahn property, and let $Y$ be its inverse limit.
Suppose $f\:X\to Y$ is a uniformly continuous map, where $X$ is the limit
of a convergent inverse sequence $\dots\xr{p_1}X_1\xr{p_0}X_0$ of uniformly
continuous maps between metrizable uniform spaces.
Then for each sequence of $\alpha_i>0$ there exist an increasing sequence of
natural numbers $n_i$, and uniformly continuous maps $f_i\:X_{n_i}\to Y_i$
such that the diagrams
$$\begin{CD}
X_{n_{i+1}}@>f_{i+1}>>Y_{i+1}\\
@Vp^{n_{i+1}}_{n_i}VV@Vq_iVV\\
X_{n_i}@>f_i>>Y_i
\end{CD}
\qquad\text{and}\qquad
\begin{CD}
X@>f>>Y\\
@VVp^\infty_{n_i}V@VVq^\infty_iV\\
X_{n_i}@>f_i>>Y_i
\end{CD}$$
$\alpha_i$-commute for each $i$.
\end{theorem}

The compact case is due essentially to Mioduszewski \cite{Mio}.

\begin{proof} By Lemma \ref{Mardesic}(a), there exist an $n_0$ and a uniformly
continuous $f_0\:X_{n_0}\to Y_0$ such that the composition
$X\xr{p^\infty_{n_0}}X_{n_0}\xr{f_0}Y_0$ is $(\alpha_0/3)$-close to
$X\xr{f}Y\xr{q^\infty_0}Y_0$.
Similarly for each $\alpha_1'>0$ there exist an $n_1'>n_0$ and an
$f_1'\:X_{n_1'}\to Y_1$ such that the composition
$X\xr{p^\infty_{n_1'}}X_{n_1'}\xr{f_1'}Y_1$ is $\alpha_1'$-close to
$X\xr{f}Y\xr{q^\infty_1}Y_1$.
Let $\alpha_1'<\alpha_1/3$ be such that $q^1_0$ is
$(\alpha_1',\alpha_0/3)$-continuous.
Then the compositions $X\xr{p^\infty_{n_0}}X_{n_0}\xr{f_0}Y_0$
and $X\xr{p^\infty_{n_1'}}X_{n_1'}\xr{f_1'}Y_1\xr{q^1_0}Y_0$ are
$(2\alpha_0/3)$-close.
Hence by Lemma \ref{Mardesic}(b), there exists an $n_1\ge n_1'$ such that
the compositions $X_{n_1}\xr{p^{n_1}_{n_0}}X_{n_0}\xr{f_0}Y_0$
and $X_{n_1}\xr{p^{n_1}_{n_1'}}X_{n_1'}\xr{f_1'}Y_1\xr{q^1_0}Y_0$ are
$\alpha_0$-close.
We define $f_1$ to be the composition
$X_{n_1}\xr{p^{n_1}_{n_1'}}X_{n_1'}\xr{f_1'}Y_1$, and proceed similarly.
\end{proof}

\begin{theorem}\label{Milnor}
Let $\dots\xr{q_1}Y_1\xr{q_0}Y_0$ be an inverse sequence of uniformly
continuous maps between uniform ANRs, and let $Y$ be its inverse limit.
Suppose $f\:X\to Y$ is a uniformly continuous map, where $X$ is the limit
of a convergent inverse sequence $\dots\xr{p_1}X_1\xr{p_0}X_0$ of uniformly
continuous maps between metrizable uniform spaces.

Then there exists an increasing sequence $n_i$ and a level-preserving
uniformly continuous extension
$f_n\:X_{n_{[0,\infty]}}\to Y_{[0,\infty]}$ of $f$.

Moreover, given another such extension $f'_n$, there exists an increasing
subsequence $l_i$ of $n_i$ such that the compositions
$X_{l_{[0,\infty]}}\xr{p^l_n}X_{n_{[0,\infty]}}
\overset{f_n}{\underset{f'_n}{\rightrightarrows}}Y_{[0,\infty]}$
are uniformly homotopic through level-preserving extensions
$X_{l_{[0,\infty]}}\to Y_{[0,\infty]}$ of $f$.
\end{theorem}

\begin{proof}
By Lemmas \ref{Hahn} and \ref{LCU}, each $Y_i$ satisfies the Hahn property and is
uniformly locally contractible.
Let $\alpha_i=\delta_{LCU}(2^{-i})$ be given by the uniform local
contractibility of $Y_i$ corresponding to $\eps=2^{-i}$.
The first assertion now follows from Theorem \ref{miod3}.

The moreover part is established by similar arguments, but replacing

$\bullet$ Lemma \ref{Hahn}(a) with Lemma \ref{LCU}(a);

$\bullet$ Lemma \ref{LCU}(a) with its $1$-parameter version (see Remark
\ref{higher homotopies});

$\bullet$ Lemma \ref{Mardesic}(b) with Lemma \ref{Mardesic}(a);

$\bullet$ Lemma \ref{Mardesic}(b) with itself applied to $X\x I$.
\end{proof}

\subsection{Extended mapping telescope}
We now sketch a slightly different approach.

\begin{theorem}\label{telescope}
Let $\dots\xr{q_1}Y_1\xr{q_0}Y_0$ be an inverse sequence of uniformly
continuous maps between uniform ANRs, and let $Y$ be its inverse limit.
Then the mapping telescope $Y_{[0,\infty)}$ and the extended mapping telescope
$Y_{[0,\infty]}$ are uniform ANRs.

Moreover, if $Y_0$ is a uniform AR, then $Y_{[0,\infty)}$ and
$Y_{[0,\infty]}$ are uniform ARs.
\end{theorem}

\begin{proof} Each $Y_{[i,i+1]}$ is a uniform ANR by Theorem \ref{Whitehead},
so each $Y_{[0,i]}$ is a uniform ANR by Corollary \ref{Nhu}.
Now $Y_{[0,\infty)}$ and $Y_{[0,\infty]}$ are uniformly $2^{-i}$-homotopy dominated by
$Y_{[0,i]}$ for each $i$, whence they are uniform ANRs by Corollary
\ref{Hanner}(a).
The moreover part follows from Theorem \ref{uniform AR}, since $Y_{[0,\infty)}$
and $Y_{[0,\infty)}$ uniformly deformation retract onto $Y_0$.
\end{proof}

Theorem \ref{telescope} immediately implies the following result, which is easily
seen to be equivalent (cf.\ \cite[Lemma 2.5]{M}) to Theorem \ref{Milnor}.

\begin{corollary}\label{Milnor2}
Let $\dots\xr{q_1}Y_1\xr{q_0}Y_0$ be an inverse sequence of uniformly
continuous maps between uniform ANRs, where $Y_0$ is a uniform AR, and
let $Y$ be its inverse limit.
Suppose $f\:X\to Y$ is a uniformly continuous map, where $X$ is the limit
of a convergent inverse sequence $\dots\xr{p_1}X_1\xr{p_0}X_0$ of uniformly
continuous maps between metrizable uniform spaces.

Then there exists a uniformly continuous extension
$f_{[0,\infty]}\:X_{[0,\infty]}\to Y_{[0,\infty]}$ of $f$ sending
$X_{[0,\infty)}$ into $Y_{[0,\infty)}$.
Moreover, every two such extensions are uniformly homotopic through such
extensions.
\end{corollary}

Most of the compact case of Theorem \ref{telescope} and Corollary \ref{Milnor2}
was proved by J. Milnor (1961; published 1995) and rediscovered in mid-70s
independently by J. Krasinkiewicz; Y. Kodama; Chapman--Siebenmann; and
Dydak--Segal (see references in \cite[\S2]{M}).

\begin{theorem}\label{telescope retraction}
Let $\dots\xr{p_1}X_1\xr{p_0}X_0$ be a convergent inverse sequence of uniformly
continuous maps between uniform ANRs, where $X_0$ is a uniform AR [resp.\ no
condition on $X_0$], and let $X$ be its inverse limit.

Then $X$ is a uniform A[N]R if and only if it is a uniform retract of
$X_{[0,\infty]}$ [resp.\ of $X_{[n,\infty]}$ for some $n$].
\end{theorem}

A similar characterization of non-uniform ANRs is found in \cite[Theorem 1]{Sak}.

\begin{proof} The `if' direction follows from Theorem \ref{telescope}.
But let us sketch an alternative proof, avoiding the use of Theorem \ref{Whitehead}.
Given a metrizable uniform space $Y$ and a closed subset $A\subset Y$, we may
embed $Y$ in the mapping telescope $U_{[0,\infty]}$ of appropriate uniform
neighborhoods $U_i$ of $A$ in $Y$ (with $U_0=Y$) like in the end of the proof
of Theorem \ref{LCU+Hahn}.
Given a uniformly continuous map $A\to X$, by Theorem \ref{Milnor} it extends
to a uniformly continuous map $U_{[0,\infty]}\to X_{[0,\infty]}$.
The required extension is now given by the composition
$Y\subset U_{[0,\infty]}\to X_{[0,\infty]}\to X$ [resp.\ by its restriction of
the form $U_m\subset U_{[m,\infty]}\to X_{[n,\infty]}\to X$].

The ``only if'' direction follows by the definition of a convergent inverse
sequence.
\end{proof}

\section{Uniform homeomorphisms of inverse limits}

\subsection{Passage to the limit}

\begin{theorem}\label{miod2} Let $\dots\xr{p_1}X_1\xr{p_0}X_0$ and
$\dots\xr{q_1}Y_1\xr{q_0}Y_0$ be inverse sequences of uniformly continuous
maps between complete metric spaces, and let $X$ and $Y$ be their inverse
limits.
Then there exists a sequence of $\alpha_i>0$ such that the following holds.
Suppose that there exist non-decreasing unbounded sequences of natural numbers
$n_i$ and $m_i$, and uniformly continuous maps $f_i\:X_{n_i}\to Y_{m_i}$ and
$g_i\:Y_{m_i}\to X_{n_{i-1}}$ such that the diagrams
$$\begin{CD}
X_{n_{i+1}}@>f_{i+1}>>Y_{m_{i+1}}\\
@Vp^{n_{i+1}}_{n_i}VV@Vq^{m_{i+1}}_{m_i}VV\\
X_{n_i}@>f_i>>Y_{m_i}
\end{CD}
\qquad\text{and}\qquad
\begin{CD}
X_{n_i}@<g_{i+1}<<Y_{m_{i+1}}\\
@VVp^{n_i}_{n_{i-1}}V@VVq^{m_{i+1}}_{m_i}V\\
X_{n_{i-1}}@<g_i<<Y_{m_i}
\end{CD}$$
respectively $\alpha_{m_i}$- and $\alpha_{n_{i-1}}$-commute, and the compositions
$Y_{m_{i+1}}\xr{g_{i+1}}X_{n_i}\xr{f_i}Y_{m_i}$ and
$X_{n_i}\xr{f_i}Y_{m_i}\xr{f_{i-1}}X_{n_{i-1}}$ are respectively
$\alpha_{m_i}$- and $\alpha_{n_{i-1}}$-close to the bonding maps, for each $i$.
Then $X$ and $Y$ are uniformly homeomorphic.

Moreover, there exists a sequence of $\beta_i^*>0$ such that for each sequence
of $\beta_i\in (0,\beta_i^*]$, the $\alpha_i$ can be chosen so that there exists
a unique uniform homeomorphism $h\:X\to Y$ such that the diagrams
$$\begin{CD}
X@>h>>Y\\
@Vp^\infty_{n_i}VV@Vq^\infty_{m_i}VV\\
X_{n_i}@>f_i>>Y_{m_i}
\end{CD}
\qquad\text{and}\qquad
\begin{CD}
X@<h^{-1}<<Y\\
@VVp^\infty_{n_{i-1}}V@VVq^\infty_{m_i}V\\
X_{n_{i-1}}@<g_i<<Y_{m_i}
\end{CD}$$
respectively $\beta_{m_i}$- and $\beta_{n_{i-1}}$-commute for each $i$.
\end{theorem}

The compact case (apart from the uniqueness) is due to Mioduszewski \cite{Mio}.

\begin{proof} It suffices to prove the moreover assertion.
We may assume that the $\alpha_i$ are such that $p^l_k$ is
$(\alpha_l,2^{k-l}\beta_k)$-continuous for each $k$ and each $l>k$.
It follows that every diagram of the form
$$\begin{CD}
X_{n_j}@<g_{j+1}<<Y_{m_{j+1}}\\
@Vp^{n_j}_{n_i}VV@Vq^{m_j}_{m_i}VV\\
X_{n_i}@>f_i>>Y_{m_i}
\end{CD}$$
$\alpha_{m_i}+\beta_{m_i}$-commutes, since it splits into $j-i$ square diagrams
and one triangular diagram as in the hypothesis.

On the other hand, Corollary \ref{miod1} yields uniformly continuous maps
$f\:X\to Y$
and $g\:Y\to X$ satisfying the desired conditions in place of $h$ and $h^{-1}$.
It remains to show that $fg=\id_Y$ and $gf=\id_X$.
Let $\gamma_i$ be such that $f_i$ is $(\gamma_i,\beta_{m_i})$-continuous.
We may assume that the $\beta_i$ are such that each $p^l_k$ is
$(\beta_l,\frac1{l-k})$-continuous.
Then for each $i$ there exists a $j$ such that $p^{n_j}_{n_i}$ is
$(\beta_{n_j},\gamma_i)$-continuous.
Then the two compositions
$$\begin{CD}
@.Y\\
@.@Vq^\infty_{m_{j+1}}VV\\
X_{n_j}@<g_{j+1}<<Y_{m_{j+1}}\\
@Vp^{n_j}_{n_i}VV@.\\
X_{n_i}@>f_i>>Y_{m_i}
\end{CD}
\qquad\text{and}\qquad
\begin{CD}
X@<g<<Y\\
@Vp^\infty_{n_j}VV@.\\
X_{n_j}@.\\
@Vp^{n_j}_{n_i}VV@.\\
X_{n_i}@>f_i>>Y_{m_i}
\end{CD}$$
are $\beta_{m_i}$-close.
Since $f$ satisfies the desired condition on $h$, the left-hand composition
is in turn $\beta_{m_i}$-close to the composition
$Y\xr{g}X\xr{f}Y\xr{q^\infty_{m_i}}Y_{m_i}$.
On the other hand, by the above the right-hand composition is
$(\alpha_{m_i}+\beta_{m_i})$-close to $Y\xr{q^\infty_{m_i}}Y_{m_i}$.

To summarize, $Y\xr{fg}Y\xr{q^\infty_{m_i}}Y_{m_i}$ is
$(3\beta_{m_i}+2\alpha_{m_i})$-close to $Y\xr{q^\infty_{m_i}}Y_{m_i}$.
We may assume that each $\alpha_i\le\beta_i$, and that each the $\beta_i$
are such that each $q^\infty_i$ is $(\eps_i,5\beta_i)$-separating for some
zero-convergent sequence of $\eps_i>0$.
Thus $fg$ is $\eps_i$-close to the identity for each $i$, whence it is the
identity.
Similarly $gf$ is the identity.
\end{proof}

Theorem \ref{miod2} immediately implies the following well-known result, whose
compact case is due to M. Brown \cite{Br}.

\begin{corollary}[{\cite[remark to Lemma B]{I4}, \cite[Exer.\ IV.7(a)]{I3},
\cite{Cu}}]\label{Brown}
Let $X$ be the limit of an inverse sequence $\dots\xr{p_1}X_1\xr{p_0}X_0$
of uniformly continuous maps between complete metric spaces.
Suppose that for each $i$ we are given a sequence of uniformly continuous
maps $q_{i1},q_{i2},\dots$ uniformly convergent to $p_i$.
Then there exists a sequence of $n_i\in\N$ such that for each sequence of
$m_i\ge n_i$, the limit $Y_{(m_i)}$ of the inverse sequence
$\invlim(\dots\xr{q_{1{m_1}}}X_1\xr{q_{0{m_0}}}X_0)$ is uniformly homeomorphic
to $X$.
\end{corollary}

Let us note that under the additional hypothesis that $\dots\xr{p_1}X_1\xr{p_0}X_0$
is convergent (which holds if the $X_i$ are compact), Corollary \ref{Brown}
follows directly from Corollary \ref{A.15'}(a,c,d,g).

\subsection{Combinatorial approximation}

\begin{theorem}\label{miod4} Let $\dots\xr{p_1}X_1\xr{p_0}X_0$ and
$\dots\xr{q_1}Y_1\xr{q_0}Y_0$ be convergent inverse sequences of uniformly
continuous maps between metric spaces satisfying the Hahn property, and
suppose that their inverse limits $X$ and $Y$ are uniformly homeomorphic by
a homeomorphism $h$.
Then for each sequence of $\alpha_i>0$ there exist increasing sequences of
natural numbers
$n_i$ and $m_i$, and uniformly continuous maps $f_i\:X_{n_i}\to Y_{m_i}$ and
$g_i\:Y_{m_i}\to X_{n_{i-1}}$ such that the diagrams
$$\begin{CD}
X_{n_{i+1}}@>f_{i+1}>>Y_{m_{i+1}}\\
@Vp^{n_{i+1}}_{n_i}VV@Vq^{m_{i+1}}_{m_i}VV\\
X_{n_i}@>f_i>>Y_{m_i}
\end{CD}
\qquad\text{and}\qquad
\begin{CD}
X@>h>>Y\\
@Vp^\infty_{n_i}VV@Vq^\infty_{m_i}VV\\
X_{n_i}@>f_i>>Y_{m_i}
\end{CD}$$
$\alpha_{m_i}$-commute, and the diagrams
$$\begin{CD}
X_{n_i}@<g_{i+1}<<Y_{m_{i+1}}\\
@VVp^{n_i}_{n_{i-1}}V@VVq^{m_{i+1}}_{m_i}V\\
X_{n_{i-1}}@<g_i<<Y_{m_i}
\end{CD}
\qquad\text{and}\qquad
\begin{CD}
X@<h^{-1}<<Y\\
@VVp^\infty_{n_{i-1}}V@VVq^\infty_{m_i}V\\
X_{n_{i-1}}@<g_i<<Y_{m_i}
\end{CD}$$
$\alpha_{n_{i-1}}$-commute, and the compositions
$Y_{m_{i+1}}\xr{g_{i+1}}X_{n_i}\xr{f_i}Y_{m_i}$ and
$X_{n_i}\xr{f_i}Y_{m_i}\xr{f_{i-1}}X_{n_{i-1}}$ are respectively
$\alpha_{m_i}$- and $\alpha_{n_{i-1}}$-close to the bonding maps, for each $i$.
\end{theorem}

A version of Theorem \ref{miod4}, with compact case due to Freudenthal, is found in \cite[Theorem 1]{I4}; 
a closer version of the compact case is also found in \cite{Mio}.

The proof of Theorem \ref{miod4} employs the same ideas as that of
Theorem \ref{miod3}, and we leave the details to the reader.

\begin{theorem} \label{miod4'} Let $\dots\xr{p_1}X_1\xr{p_0}X_0$ and
$\dots\xr{q_1}Y_1\xr{q_0}Y_0$ be convergent inverse sequences of uniformly
continuous maps between complete metric spaces satisfying the Hahn property,
and let $X$ and $Y$ be their inverse limits.
Then there exists a sequence of $\alpha_i^*>0$ such that for each sequence of
$\alpha_i\in(0,\alpha_i^*]$ the following are equivalent:

(i) $X$ and $Y$ are uniformly homeomorphic;

(ii) there exist non-decreasing unbounded sequences of natural numbers
$n_i$ and $m_i$, and uniformly continuous maps $f_i\:X_{n_i}\to Y_{m_i}$ and
$g_i\:Y_{m_i}\to X_{n_{i-1}}$ such that the diagrams
$$\begin{CD}
X_{n_j}@<g_{j+1}<<Y_{m_{j+1}}\\
@Vp^{n_j}_{n_i}VV@Vq^{m_{j+1}}_{m_i}VV\\
X_{n_i}@>f_i>>Y_{m_i}
\end{CD}
\qquad\text{and}\qquad
\begin{CD}
X_{n_j}@>f_j>>Y_{m_j}\\
@VVp^{n_j}_{n_{i-1}}V@VVq^{m_j}_{m_i}V\\
X_{n_{i-1}}@<g_i<<Y_{m_i}
\end{CD}$$
respectively $\alpha_{m_i}$- and $\alpha_{n_{i-1}}$-commute, for each $i$ and
each $j\ge i$.
\end{theorem}

\begin{proof}[(i)\imp(ii)] Let $n_i$, $m_i$, $f_i$ and $g_i$ be given by
Theorem \ref{miod4}.
We may assume that the $\alpha_i$ are such that $p^l_k$ is
$(\alpha_l,2^{k-l}\alpha_k)$-continuous for each $k$ and each $l>k$.
Each diagram in condition (ii) splits into $j-i$ square diagrams and one
triangular diagram which approximately commute by Theorem \ref{miod4},
and the assertion follows.
\end{proof}

\begin{proof}[(ii)\imp(i)]
Let $\gamma_i$ be such that $f_i$ is $(\gamma_i,\alpha_{m_i})$-continuous.
We may assume that the $\alpha_i$ are such that each $p^l_k$ is
$(\alpha_l,\frac1{l-k})$-continuous.
Then for each $i$ there exists a $j$ such that $p^{n_j}_{n_i}$ is
$(\alpha_{n_j},\gamma_i)$-continuous.
The composition $X_{n_{j+1}}\xr{f_{j+1}}Y_{m_{j+1}}\xr{g_{j+1}}X_{n_j}$
is $\alpha_{n_j}$-close to the bonding map, whereas the composition
$Y_{m_{j+1}}\xr{g_{j+1}}X_{n_j}\xr{p^{n_j}_{n_i}}X_{n_i}\xr{f_i}Y_{m_i}$
is $\alpha_{m_i}$-close to the bonding map.
It follows that the diagram
$$\begin{CD}
X_{n_{j+1}}@>f_{j+1}>>Y_{m_{j+1}}\\
@Vp^{n_{j+1}}_{n_i}VV@Vq^{m_{j+1}}_{m_j}VV\\
X_{n_i}@>f_i>>Y_{m_i}
\end{CD}$$
$2\alpha_{m_i}$-commutes.
Similarly, for each $i$ there exists a $j$ so that the diagram
$$\begin{CD}
X_{n_j}@<g_{j+1}<<Y_{m_{j+1}}\\
@VVp^{n_j}_{n_{i-1}}V@VVq^{m_{j+1}}_{m_i}V\\
X_{n_{i-1}}@<g_i<<Y_{m_i}
\end{CD}$$
$2\alpha_{n_{i-1}}$-commutes.
Hence after an appropriate thinning out of indices, Theorem \ref{miod2}
applies to produce a uniform homeomorphism between $X$ and $Y$.
\end{proof}

\section{Theory of retracts for inverse limits}

\subsection{Characterizing the Hahn property}

\begin{theorem}\label{Hahn-limit}
Let $\dots\xr{p_1}X_1\xr{p_0}X_0$ be a convergent inverse sequence of uniformly
continuous maps between complete metric spaces satisfying the Hahn property,
and let $X$ be its inverse limit.
Then there exists a sequence of $\alpha_i^*>0$ such that for each sequence of
$\alpha_i\in (0,\alpha_i^*]$ the following are equivalent:

(a) $X$ satisfies the Hahn property;

(b) for each $i$ there exists a $j\ge i$ such that for each $k>j$ there exists
a uniformly continuous map $s_i\:X_j\to X_k$ such that the composition
$X_j\xr{s_i}X_k\xr{p^k_i}X_i$ is $\alpha_i$-close to $p^k_i$.
\end{theorem}

The proof of Theorem \ref{Hahn-limit} is similar to (and easier than) that of the next result.

\subsection{Characterizing uniform ANRs}

\begin{theorem}\label{ANR-limit}
Let $\dots\xr{p_1}X_1\xr{p_0}X_0$ be a convergent inverse sequence of uniformly
continuous maps between complete uniform ANRs, and let $X$ be its inverse limit.
Then there exists a sequence of $\alpha_i^*>0$ such that for each sequence of
$\alpha_i\in (0,\alpha_i^*]$ the following are equivalent:

(a) $X$ is a uniform ANR;

(b) for each $i$ there exists a $j\ge i$ such that for each $k>j$ there exists
an $l\ge k$ and a uniformly continuous map $s_i\:X_j\to X_k$ such that
the composition $X_j\xr{s_i}X_k\xr{p^k_i}X_i$ is $\alpha_i$-close to $p^j_i$, and
$p^l_k$ is uniformly homotopic to the composition $X_l\xr{p^l_j}X_j\xr{s_i}X_k$ by
a homotopy $h_i$ such that the composition $X_l\x I\xr{h_i}X_k\xr{p^k_i}X_i$
is $\alpha_i$-close to the motionless homotopy of $p^l_i$.
\end{theorem}

\begin{proof} Let $\alpha_n^*$ be bounded above by the $\alpha_n$ in Corollary \ref{miod1}
(upon setting $\beta_n=\beta_n^*$ in there).

\medskip\noindent
{\it (a)\imp (b).} We may assume that the $\alpha_n$'s are chosen so that each
$p^n_m$ is $(\alpha_n,\alpha_m/2)$-continuous.
Since each $X_n$ is uniformly locally contractible (see Lemma \ref{LCU}(a)),
there exists a $\gamma_n>0$ such that every two $\gamma_n$-close uniformly
continuous maps into $X_n$ are uniformly $\alpha_n$-homotopic.

Suppose that $X$ is a uniform ANR.
Given an $i$, let $\eps_i>0$ be such that $p^\infty_i$ is
$(\eps_i,\gamma_i/2)$-continuous.
Since $X$ is uniformly locally contractible, there exists a $\delta_i>0$ such
that every two $4\delta_i$-close maps into $X$ are uniformly $\eps_i$-homotopic.
Since the inverse sequence is convergent and $X$ satisfies the Hahn property
(see Lemma \ref{Hahn}(a)), there exists an $m_i\ge i$ and a uniformly continuous
map $r_i\:X_{m_i}\to X$ such that the composition
$X\xr{p^\infty_{m_i}}X_{m_i}\xr{r_i}X$ is $\delta_i$-close to the identity.
Since $p^\infty_i$ is $(\delta_i,\gamma_i/2)$-continuous due to $\delta_i\le\eps_i$,
we get that $p^\infty_i$ is $\frac{\gamma_i}2$-close to the composition
$X\xr{p^\infty_{m_i}}X_{m_i}\xr{r_i}X\xr{p^\infty_i}X_i$.
Since the inverse sequence is convergent, by
Lemma \ref{Mardesic}(b) there exists a $j=j(i)\ge m_i$ such that $p^j_i$ is
$\gamma_i$-close to the composition
$X_j\xr{p^j_{m_i}}X_{m_i}\xr{r_i}X\xr{p^\infty_i}X_i$.
The latter coincides with the composition
$X_j\xr{s_i}X_k\xr{p^j_i}X_i$, where $s_i$ denotes
the composition $X_j\xr{p^j_{m_i}}X_{m_i}\xr{r_i}X\xr{p^\infty_k}X_k$.
Thus $p^j_i$ and the composition $X_j\xr{s_i}X_k\xr{p^k_i}X_i$ are joined
by a uniform $\alpha_i$-homotopy $\nu_i\:X_j\x I\to X_i$.

Now the composition $X\xr{p^\infty_{m_k}}X_{m_k}\xr{r_k}X$ is $\delta_k$-close to
the identity, which is in turn $\delta_i$-close to the composition
$X\xr{p^\infty_{m_i}}X_{m_i}\xr{r_i}X$.
Since $m_k\ge m_i$ and the inverse sequence is convergent, by
Lemma \ref{Mardesic}(b) there exists an $l=l(i,k)\ge m_k$ such that the compositions
$X_l\xr{p^l_{m_i}}X_{m_i}\xr{r_i}X$ and
$X_l\xr{p^l_{m_k}}X_{m_k}\xr{r_k}X$ are $2(\delta_i+\delta_k)$-close.
Therefore, due to $\delta_k\le\delta_i$, they are joined by a uniform
$\eps_i$-homotopy $\lambda_i\:X_l\x I\to X$.

The compositions $h_i'\:X_l\x I\xr{\lambda_i}X\xr{p^\infty_k}X_k$ and
$h_i''\:X_l\x I\xr{p^l_{j'}\x\id_I}X_{j'}\x I\xr{\nu_k}X_k$, where $j'=j(k)$,
concatenate to form a uniform homotopy $h_i\:X_l\x I\to X_k$ between $p^l_k$ and
the composition $X_l\xr{p^l_{m_i}}X_{m_i}\xr{r_i}X\xr{p^\infty_k}X_k$, which
is the same as $X_l\xr{p^l_j}X_j\xr{s_i}X_k$.
Since $p^\infty_i$ is $(\eps_i,\alpha_i/2)$-continuous due to $\gamma_i\le\alpha_i$,
the composition $X_l\x I\xr{h_i'}X_k\xr{p^k_i}X_i$ is
an $\frac{\alpha_i}2$-homotopy.
Since $p^k_i$ is $(\alpha_k,\alpha_i/2)$-continuous, the composition
$X_l\x I\xr{h_i''}X_k\xr{p^k_i}X_i$ is also an $\frac{\alpha_i}2$-homotopy.
Thus the composition $X_l\x I\xr{h_i''}X_k\xr{p^k_i}X_i$ is an $\alpha_i$-homotopy.
\end{proof}

\begin{proof}[(b)\imp(a)]
Let us write $j=n_i$, and choose $k=n_{i+1}$.
Let us write $l=l_{i+1}$.
We may assume that each $l_{i+1}\ge l_i$.

The uniform homotopy between $p^{l_{i+1}}_{n_{i+1}}$ and the composition
$X_{l_{i+1}}\xr{p^{l_{i+1}}_{n_i}}X_{n_i}\xr{s_i}X_{n_{i+1}}$ yields a uniformly
continuous map $\phi_i\:MC(p^{l_{i+1}}_{l_i})\to X_{n_{i+1}}$ that restricts to
$p^{l_{i+1}}_{n_{i+1}}$ on $X_{l_{i+1}}$ and to the composition
$X_{l_i}\xr{p^{l_i}_{n_i}}X_{n_i}\xr{s_i}X_{n_{i+1}}$ on $X_{l_i}$.
Let us write $Y_i=X_{l_i}$ so that $MC(p^{l_{i+1}}_{l_i})$ becomes $Y_{[i,\,i+1]}$.
For each $j\ge i+1$ let $\phi^i_j$ denote the composition
$Y_{[i,\,i+1]}\xr{\phi_i}X_{n_{i+1}}\xr{s_i}\dots\xr{s_{j-1}}X_{n_j}$.
Then $\phi^0_j,\dots,\phi^{j-1}_j$ combine into a uniformly continuous map
$\Phi_j\:Y_{[0,j]}\to X_{n_j}$.
By the hypothesis each composition
$Y_{[i,\,i+1]}\xr{\phi^i_{j+1}}X_{n_{j+1}}\xr{p^{n_{j+1}}_j}X_j$ is $\alpha_j$-close
to the composition $Y_{[i,\,i+1]}\xr{\phi^i_j}X_{n_j}\xr{p^{n_j}_j}X_j$.
Hence the composition $\Psi_j\:Y_{[0,j]}\xr{\Phi_j}X_{n_j}\xr{p^{n_j}_j}X_j$
is $\alpha_j$-close to the composition $Y_{[0,j]}\subset
Y_{[0,\,j+1]}\xr{\Psi_{j+1}}X_{j+1}\xr{p^{j+1}_j}X_j$.

The telescope's bonding map $q_j\:Y_{[0,\,j+1]}\to Y_{[0,j]}$ restricts to
the identity over $Y_{[0,j]}$ and to the projection
$\pi\:MC(p^{l_{j+1}}_{l_j})\to X_{l_j}$ over $Y_{[j,\,j+1]}$.
By the hypothesis, the composition
$Y_{[j,\,j+1]}\xr{\pi}Y_j\xr{p^{l_j}_{n_j}}X_{n_j}\xr{p^{n_j}_j}X_j$
is $\alpha_j$-close to the composition
$Y_{[j,\,j+1]}\xr{\phi_j}X_{n_{j+1}}\xr{p^{n_{j+1}}_{j+1}}X_{j+1}\xr{p_j}X_j$.
Hence the following diagram $\alpha_j$-commutes:
$$\begin{CD}
Y_{[0,\,j+1]}@>\Psi_{j+1}>>X_{j+1}\\
@Vq_jVV@Vp_jVV\\
Y_{[0,j]}@>\Psi_j>>X_j.
\end{CD}$$
Thus by Corollary \ref{miod1} we obtain a uniformly continuous map
$\Psi\:Y_{[0,\infty]}\to X$, which by construction (or alternatively by
the uniqueness in Corollary \ref{miod1}) restricts to the identity on $X$.
The assertion now follows from Corollary \ref{telescope retraction}.
\end{proof}

We say that a uniformly continuous map $f\:X\to Y$ is a {\it uniform
$\eps$-homotopy equivalence} over a uniformly continuous map $p\:Y\to B$
if there exists a uniformly continuous map $s\:Y\to X$ and
uniform homotopies $h_X\:\id_X\sim sf$ and $h_Y\:\id_Y\sim fs$ such that
the compositions $Y\x I\xr{h_Y}Y\xr{p}B$ and $X\x I\xr{h_X}X\xr{f}Y\xr{p}B$
are $\eps$-homotopies.

\begin{corollary}\label{homology manifolds}
Let $\dots\xr{p_1}X_1\xr{p_0}X_0$ be a convergent inverse sequence of uniformly
continuous maps between complete uniform ANRs, and let $X$ be its inverse limit.
Then there exists a sequence of $\alpha_i>0$ such that if $n_i$ is a sequence
such that each $p^{n_{i+1}}_{n_i}\:X_{n_{i+1}}\to X_{n_i}$ is a uniform
$\alpha_i$-homotopy equivalence over $p^{n_i}_i\:X_{n_i}\to X_i$, then $X$ is
a uniform ANR.
\end{corollary}

The compact case of Corollary \ref{homology manifolds} appears in
\cite[Lemma 3.1]{BFMW}.

\subsection{Characterizing uniform local contractibility}
Let $X$ be a separable metrizable complete uniform space.
Let $\dots\xr{p_1}X_1\xr{p_0}X_0$ be a convergent inverse sequence of uniformly
continuous maps between metric spaces satisfying the Hahn property with inverse
limit $X$ (see Theorem \ref{intersection of cubohedra2}).

We say that $X$ is {\it approachingly locally contractible} if for each $i$
and $\eps>0$ there exists a $j\ge i$ and a $\delta>0$ such that for each
$k\ge j$ there exists an $l\ge k$ such that for each $x\in X$, the map
$(p^l_j)^{-1}(U_\delta(p^\infty_j(x)))\xr{p^l_k|}(p^k_i)^{-1}(U_\eps(p^\infty_i(x)))$
is null-homotopic.

We say that $X$ is {\it approachingly uniformly locally contractible} if for
each $i$ and $\eps>0$
there exists a $j\ge i$ and a $\delta>0$ such that for each $k\ge j$ there exists
an $l\ge k$ such that for each uniform space $Y$ and every two uniformly continuous
maps $f,g\:Y\to X_l$ whose compositions with $p^l_j$ are $\delta$-close, their
compositions with $p^l_k$ are joined by a uniform homotopy $H\:Y\x I\to X_k$ whose
composition with $p^k_i$ is an $\eps$-homotopy.

It can be seen using Theorem \ref{miod4} that the two properties just defined do
not depend on the choice of the inverse sequence.
``Approachingly'' can be interpreted to have the meaning ``in the sense of uniform
strong shape''.

\begin{proposition} \label{approaching local contractibility}
Let $X$ be a compactum.

(a) If $X$ is locally contractible, then it is approachingly locally contractible.

(b) If $X$ is finite-dimensional, then the converse holds.
\end{proposition}

Part (a) easily follows from Theorem \ref{miod1}, or
alternatively from Theorem \ref{ANR-limit}.

Part (b) is well-known (see \cite[Theorem 6.1]{M}).

\begin{theorem} \label{approaching uniform local contractibility}
Let $X$ be a separable metrizable complete uniform space.

(a) If $X$ is uniformly locally contractible, then it is approachingly uniformly
locally contractible.

(b) If $X$ is uniformly finite-dimensional, then the converse holds.

(c) If $X$ is approachingly uniformly locally contractible and satisfies
the Hahn property, then it is a uniform ANR.
\end{theorem}

Part (c) is a strengthened form of Theorem \ref{LCU+Hahn} (except that it
assumes that $X$ is separable and complete).

\begin{proof}[Proof. (a)] For a metric space $M$ and a $\beta>0$, let $M^\beta$
denote $\{(x,y)\in M\x M\mid d(x,y)<\beta\}$ and let $\pi_1$ and $\pi_2$ stand for
the two projections $M\x M\to M$ or their restrictions.
By the proof of Lemma \ref{LCU}, uniform local contractibility of $X$ is
equivalent to the following: for each $\eps>0$ there exists a $\delta>0$ such that
the maps $\pi_h\:X^\delta\to X$, $h=1,2$, are uniformly $\eps$-homotopic.
Using Lemma \ref{A.11}, the latter is in turn equivalent to the following:
\smallskip

(i) for each $i$ and $\eps>0$ there exists a $j\ge i$ and a $\delta>0$
such that the maps $\pi_h\:(p^\infty_j\x p^\infty_j)^{-1}(X_j^\delta)\to X$,
$h=1,2$, are joined by a uniform homotopy $H$ such that the composition
$(p^\infty_j\x p^\infty_j)^{-1}(X_j^\delta)\x I\xr{H}X\xr{p^\infty_i}X_i$ is an
$\eps$-homotopy.
\smallskip

Similarly to the above application of the proof of Lemma \ref{LCU}, approaching
uniform local contractibility is equivalent to the following:
\smallskip

(ii) for each $i$ and $\eps>0$ there exists a $j\ge i$ and a $\delta>0$
such that for each $k\ge j$ there exists an $l\ge k$ such that the compositions
$(p^l_j\x p^l_j)^{-1}(X_j^\delta)\xr{\pi_h}X_l\xr{p^l_k}X_k$, $h=1,2$, are
joined by a uniform homotopy $H$ such that the composition
$(p^l_j\x p^l_j)^{-1}(X_j^\delta)\x I\xr{H}X_k\xr{p^k_i}X_i$ is an
$\eps$-homotopy.
\smallskip

Using Theorem \ref{miod3} and that $\dots\to X_1\to X_0$ is convergent it
is easy to see that (i) implies (ii).
\end{proof}

\begin{proof}[(b)]
Conversely, we argue as in the proof of Theorem \ref{rfd-ANR}.
If the uniform dimensions of the $X_i$'s are bounded by $n$, then upon replacing
the functions $j=j(i)$ and $l=l(k)$ in the definition of approaching uniform local
contractibility by their $(n+1)$st iterates, we will be able to construct a map
$s_i\:X_j\to X_k$ as in Theorem \ref{ANR-limit} by inducting on the skeleta of $X_j$.
Similarly, the desired homotopy $h_i$, which can be viewed as a map
$MC(p^l_j)\to X_k$, can be constructed by inducting on the skeleta of $MC(p^l_j)$.
\end{proof}

\begin{proof}[(c)]
By Theorem \ref{intersection of cubohedra2} we may represent $X$ as the
limit of a convergent inverse sequence of complete uniform ANRs (and
not just spaces satisfying the Hahn property).
Now if the inverse sequence satisfies the condition in the definition of
approaching uniform local contractibility and the equivalent of the Hahn property
(in Theorem \ref{Hahn-limit}), then it is not hard to see that it also satisfies
the equivalent of the uniform ANR condition (in Theorem \ref{ANR-limit}),
albeit with different $j(i)$ and $l(k)$.
\end{proof}

\newpage
\part{THE SPACE OF MEASURABLE FUNCTIONS} \label{measurable functions}

\section{Review of measurable functions}

Much of this section is rather close to standard textbook material, scattered in books on Real Analysis, 
Measure and Integration, Probability and Functional Analysis.
Textbooks usually treat the case where $X$ is the real line or, sometimes, a normed vector space.
We need $X$ to be a general Polish space; some basic theory for this case is covered in Schwarz's textbook
\cite{Sch}.
The further story is mentioned by Bessaga and Pelczynski \cite{BP1}, \cite{BP}*{\S VI.7}, but is
not treated in any detail there.

\subsection{Lebesgue measure}
Let us briefly review the basics of Lebesgue measure on $I=[0,1]$.
For every open $U\subset I$, the intersection $U\cap (0,1)$ is a countable union of its connected components, 
which are open intervals $(a_i,b_i)$, and one defines $\mu(U)=\sum_i b_i-a_i$.
It is clear that $\mu(U\cup V)\le\mu(U)+\mu(V)$ for any open $U,V\subset I$, which turns into equality
if $U\cap V=\emptyset$.

For each {\it elementary} set, i.e.\ an open set $E\subset I$ such that $E\cap (0,1)$ is a finite union of 
open intervals, $\mu(I\but\Cl E)=1-\mu(E)$.
This is not true for arbitrary open sets, e.g.\ if $U_n\subset I$ the union of the open 
$\frac1{4^{n+1}}$-neighborhoods of all dyadic rationals of the form $\frac k{2^n}$, $k=0,\dots,2^n$, 
and $U=\bigcup_{n=0}^\infty U_n$, then $\mu(U_n)=\frac1{2^{n+1}}$ and consequently $\mu(U)\le\frac12$; 
however, $I\but\Cl U=\emptyset$ since $U$ contains all dyadic rationals.

For an arbitrary $S\subset I$, its {\it outer measure} $\mu^*(S)=\inf_{U\supset S}\mu(U)$ over all open sets $U$
containing $S$.
It is easy to see that $\mu^*(\Q\cap I)=0$, and consequently one cannot replace open sets with elementary sets 
in the definition of $\mu^*$.
Clearly, $\mu^*(S\cup T)\le\mu^*(S)+\mu^*(T)$, which turns into equality if $S\cap T=\emptyset$.
Also, $S\subset T$ implies $\mu^*(S)\le\mu^*(T)$.

A subset $L\subset I$ is {\it measurable} if for each $\eps>0$ there exists an open set $U\subset I$ such that 
$\mu^*(L\vartriangle U)<\eps$. 
In particular, $L$ is measurable if $\mu^*(L)=0$ (by considering $U=\emptyset$).
If $L\subset I$ is measurable, its {\it measure} $\mu(L)$ is defined to be $\mu^*(L)$.
Clearly, every subset of a measure zero set is of measure zero (in particular, measurable).

\begin{proposition} (a) $L\subset I$ is measurable if and only if for each $\eps>0$ there exists an elementary 
set $E\subset I$ such that $\mu^*(L\vartriangle E)<\eps$.

(b) If $U_n\subset I$ are open sets such that $\mu^*(L\vartriangle U_n)\to 0$, then $\mu(L)=\lim\mu(U_n)$. 
\end{proposition}

\begin{proof}[Proof. (a)] Let $U\subset I$ be an open set such that $\mu^*(L\vartriangle U)<\eps/2$, 
and let $E,E^+\subset U$ be elementary sets such that $\mu(U)-\mu(E)<\eps$ and $\Cl E\subset E^+$.
Then $L\vartriangle E^+\subset (L\vartriangle U)\cup (U\but E^+)$, where
$\mu^*(U\but E^+)\le\mu^*(U\but\Cl E)=\mu(U\but\Cl E)=\mu(U)-\mu(E)<\eps/2$.
Hence $\mu^*(L\vartriangle E^+)\le\eps$.
\end{proof}

\begin{proof}[(b)] It is easy to see that $|\mu^*(S)-\mu^*(T)|\le\mu^*(S\vartriangle T)$.
Indeed, since $S\subset T\cup (S\vartriangle T)$ and $T\subset S\cup (S\vartriangle T)$,
we have $\mu^*(S)\le\mu^*(T)+\mu^*(S\vartriangle T)$ and $\mu^*(T)\le\mu^*(S)+\mu^*(S\vartriangle T)$.
Then we have $|\mu^*(L)-\mu^*(U_n)|\to 0$ and the assertion follows.
\end{proof}

\begin{proposition} (a) If $L\subset I$ is measurable, then so is $I\but L$, and $\mu(I\but L)=1-\mu(L)$.

(b) A set $S\subset I$ is measurable if and only if $\mu^*(S)=\mu_*(S)\bydef 1-\mu^*(I\but S)$.
\end{proposition}

\begin{proof}[Proof. (a)] Let $E\subset I$ be an elementary set such that $\mu^*(L\vartriangle E)<\eps/2$, 
and let $E^+,E^-$ be elementary sets such that $\Cl E^-\subset E$, $\Cl E\subset E^+$ and  
$\mu(E^+)-\mu(E^-)<\eps/2$.
Then $(I\but L)\vartriangle (I\but\Cl E)=L\vartriangle\Cl E\subset (L\vartriangle E)\cup(\Cl E\but E)$, 
where $\mu^*(\Cl E\but E)\le\mu^*(E^+\but\Cl E^-)=\mu(E^+\but\Cl E^-)=\mu(E^+)-\mu(E^-)<\eps/2$.
Hence $\mu^*\big((I\but L)\vartriangle (I\but\Cl E)\big)<\eps$.
Since $\mu(I\but\Cl E)=1-\mu(E)$ for each $E=E(\eps)$, we get $\mu(I\but L)=1-\mu(L)$.
\end{proof}

\begin{proof}[(b)] The ``only if'' assertion follows from (a).
Conversely, $\mu_*(S)=1-\inf_{U\supset I\but S}\mu(U)=1-\inf_{F\subset S}\mu(I\but F)=
\sup_{F\subset S}\big(1-\mu(I\but F)\big)$ over all closed subsets $F$ of $S$.
By (a) applied to open sets, $1-\mu(I\but F)=\mu(F)$, so $\mu^*(S)=\sup_{F\subset S}\mu(F)$.
If $\mu^*(S)=\mu_*(S)$, then for each $\eps>0$ there exists an open $U\supset S$ and a closed $F\subset S$
such that $\mu(U)-\mu(F)=\big(\mu(U)-\mu^*(S)\big)+\big(\mu_*(S)-\mu(F)\big)<\eps$.
We have $S\vartriangle U=U\but S\subset U\but F$.
Hence $\mu^*(S\vartriangle U)\le\mu^*(U\but F)=\mu(U\but F)=1-\mu\big((I\but U)\cup F\big)=
1-\mu(I\but U)-\mu(F)=\mu(U)-\mu(F)<\eps$.
\end{proof}

It is well-known that measurable sets are closed not only under complement, but also under countable union
(and therefore also under countable intersection).
In particular, they include all Borel sets.
Moreover, $\mu$ is known to be countably additive.
In particular, a countable union of measure zero sets is of measure zero.
It is well-known that non-measurable sets exist in ZFC, but do not exist in ZF with the Axiom of Countable Choice 
(a countable product of nonempty sets is nonempty) and the Axiom of Determinacy.

\begin{lemma}
Every measurable set is the union of a Borel (more precisely, $F_\sigma$) set and a set of measure zero.
\end{lemma}

\begin{proof}
If $L$ is measurable, by the definition of $\mu^*$ there exist open sets $U_n$ containing $L$ and 
such that $\mu(U_n)<\mu(L)+1/n$.
Then $B\bydef \bigcap_n U_n$ is a $G_\delta$ set containing $L$ and such that $\mu(B)=\mu(L)$.
Consequently $Z\bydef B\but L$ is of measure zero. 
Also, $I\but L=(I\but B)\cup Z$.
\end{proof}

\subsection{Metrics on measurable functions}\label{memefu}

Now let $X$ be a metric space and let $X^I$ be the set of all (possibly discontinuous) maps $I\to X$.
Maps $f,g\in X^I$ are called {\it equal almost everywhere} if there exists an $S\subset I$ of measure zero
such that $f|_{I\but S}=g|_{I\but S}$.
This is equivalent to requiring that $\mu^*(|f\neE g|)=0$, where $|f\neE g|$ is the set of all $t\in I$ 
for which $f(t)\ne g(t)$.
Given an $\eps>0$, let us also consider the set 
\[|f\nee\eps g|=\big\{t\in I\ \big|\ d\big(f(t),g(t)\big)>\eps\big\}.\]
Then $|f\neE g|=\bigcup_{\eps>0}|f\nee\eps g|$, and consequently $\mu^*(|f\neE g|)=0$ if and only if
$\mu^*(|f\nee\eps g|)=0$ for each $\eps>0$.

For any $f,g\in X^I$ let \[D(f,g)=\inf\big\{\eps+\mu^*\big(|f\nee\eps g|\big)\ \big|\ \eps>0\big\},\]
\[D'(f,g)=\,\inf\,\big\{\eps>0\ \big|\ \mu^*\big(|f\nee\eps g|\big)\le\eps\big\}.\]

\begin{lemma} \label{Ky Fan}
(a) $D$ and $D'$ are pseudo-metrics on $X^I$.

(b) $\mu^*(|f\nee\lambda g|)\le\lambda$ implies $D'(f,g)\le\lambda$; and $D'(f,g)<\lambda$ implies
$\mu^*(|f\nee\lambda g|)\le\lambda$.

(c) $D$ and $D'$ are uniformly equivalent and descend to metrics (also denoted $D$, $D'$) on the quotient set 
$X^I/_{\sim_0}$, where $\sim_0$ is the equivalence relation of equality almost everywhere.
\end{lemma}

\begin{proof}[Proof. (a)] Clearly, $|f\Nee{\eps+\delta}h|\subset |f\nee\eps g|\cup |g\nee\delta h|$. 
Hence $\mu^*(|f\Nee{\eps+\delta}h|)\le\mu^*(|f\nee\eps g|)+\mu^*(|g\nee\delta h|)$.

Using that $\eps+\delta+\mu^*(|f\Nee{\eps+\delta}h|)\le\eps+\mu^*(|f\nee\eps g|)+\delta+\mu^*(|g\nee\delta h|)$, 
it is easy to see that $D(f,h)\le D(f,g)+D(g,h)$.

Using that $\mu^*(|f\nee\eps g|)\le\eps$ and $\mu^*(|g\nee\delta h|)\le\delta$ imply
$\mu^*(|f\Nee{\eps+\delta}h|)\le\eps+\delta$, it is easy to see that $D'(f,h)\le D'(f,g)+D'(g,h)$.
\end{proof}

\begin{proof}[(b)] The first assertion is trivial.
Since $\mu^*\big(|f\nee\eps g|\big)$ is non-increasing as a function of $\eps$,
we have $D'(f,g)=\sup\big\{\eps>0\mid\mu^*(|f\nee\eps g|)>\eps\big\}$.
This implies the second assertion.
\end{proof}

\begin{proof}[(c)]
Suppose that $f$ and $g$ are not equal almost everywhere.
Then $\mu^*(|f\nee\kappa g|)=\lambda>0$ for some $\kappa>0$.
Hence $\mu^*(|f\nee\eps g|)\ge\lambda$ for each $\eps\le\kappa$.
So we have $\eps+\mu^*(|f\nee\eps g|)\ge\lambda$ if $\eps\le\kappa$, and 
$\eps+\mu^*(|f\nee\eps g|)\ge\kappa$ if $\eps\ge\kappa$.
Consequently $D(f,g)\ge\min(\kappa,\lambda)$.
Thus $D$ descends to a metric on $X^I/_{\sim_0}$.

Since $\mu(f\nee\eps g)\le\eps$ implies $\eps+\mu(f\nee\eps g)\le 2\eps$, we have $D(f,g)\le 2D'(f,g)$. 
By the above, $\mu(f\nee\kappa g)\ge\kappa$ implies $D(f,g)\ge\min(\kappa,\kappa)=\kappa$, and it follows that
$D(f,g)\ge D'(f,g)$.
Thus $D(f,g)$ and $D'(f,g)$ induce the same uniform structure.
In particular, $D'$ also descends to a metric on $X^I/_{\sim_0}$.
\end{proof}

The metric $D'$ on $X^I/_{\sim_0}$ is known as the {\it Ky Fan metric}.
It is easy to see that by sending each $x\in X$ to the constant function $I\to\{x\}\subset X$ we get
an isometric embedding of $X$ in $X^I$ with either of the metrics $D$, $D'$.

\begin{example} Let $X=\{0,1\}$.
Then functions $I\to X$ can be identified with subsets of $I$, and equality almost everywhere corresponds to 
the equivalence relation $\sim_0$ on subsets defined by $A\sim_0 B$ if $\mu^*(A\vartriangle B)=0$.
Clearly, $D(A,B)=D'(A,B)=\mu^*(A\vartriangle B)$.
\end{example}

\subsection{Convergence in measure}

The topology (uniformity) induced by $D$ (or $D'$) on $X^I$ and on $X^I/_{\sim_0}$ is called the topology
(uniformity) of {\it convergence in measure}.

A map $f\in X^I$ is called
\begin{itemize} 
\item a {\it step function} if there exits a finite sequence $0=t_1<\dots<t_n=1$ such that 
$f$ is constant on each $[t_i,t_{i+1})$;
\item a {\it simple function} if it has only countably many values $x_i$, and their point-inverses 
$f^{-1}(x_i)$ are measurable;
\item a {\it Borel} map if $f^{-1}(U)$ is Borel for every Borel set $U$, or equivalently
for each open set $U$;
\item {\it measurable} if $f^{-1}(U)$ is measurable for every Borel set $U$, or equivalently
for each open set $U$.
\end{itemize}

\begin{proposition} \label{closure}
(a) Simple functions in $X^I$ are limits in measure of step functions.

(b) If $X$ is separable (resp.\ compact), measurable maps in $X^I$ are uniform limits, hence also limits in measure, 
of simple functions (resp.\ simple functions with finitely many values).

(c) If $X$ is separable, then so is the set of measurable functions $I\to X$ with the topology of convergence
in measure.

(d) If a sequence of measurable (Borel) maps $f_n\in X^I$ pointwise converges to an $f\in X^I$, then 
$f$ is measurable (resp.\ Borel).

(e) Every measurable map $f\in X^I$ is equal almost everywhere to a Borel map.

(f) If $f,g\in X^I$ are measurable, then $d(f,g)\:I\to I$ is measurable.
In particular, $|f\neE g|$ is measurable and $|f\nee\lambda g|$ is measurable for each $\lambda>0$.

(g) If $f,g\in X^I$ are measurable, then $D'(f,g)\le\lambda$ is equivalent to $\mu(|f\nee\lambda g|)\le\lambda$.
\end{proposition}

Parts (d) and (e) are found in \cite{Sch}, and (b) is implicit there.  

\begin{proof}[Proof. (a)]
Given a simple function $f$ with point inverses $L_i=f^{-1}(x_i)$, $i\in\N$, listed in an order of 
decreasing measure, the simple functions $f_k$ with finitely many values, $k\in\N$, defined by 
$f_k|_{L_i}=x_{\min{i,k}}$ clearly converge to $f$ in measure.
Next let $f$ be a simple function with finitely many point-inverses $L_i=f^{-1}(x_i)$, $i=1,\dots,n$.
For each $i$ and each $k\in\N$ there exists an elementary set $E_{ik}\subset I$ such that 
$\mu^*(L_{ik}\vartriangle E_{ik})\le 1/nk$, where $L_{ik}=L_i\but\Cl(E_{1k}\cup\dots\cup E_{i-1,k})$.
The sets $D_{ik}\bydef E_{ik}\but\Cl(E_{1k}\cup\dots\cup E_{i-1,k})$ are disjoint and satisfy
$\mu^*(L_{ik}\vartriangle D_{ik})\le\mu^*(L_{ik}\vartriangle E_{ik})\le 1/nk$.
Hence $\mu^*(\bigcup_{i=1}^n L_{ik}\vartriangle D_{ik})\le 1/k$.

Each $L_i\vartriangle D_{ik}$ lies in the union of $L_{ik}\vartriangle D_{ik}$
and $K\bydef L_i\cap\Cl(E_{1k}\cup\dots\cup E_{i-1,k})=L_i\cap\Cl(D_{1k}\cup\dots\cup D_{i-1,k})$.
Since $K$ lies in $L_i$, it is disjoint from $L_1\cup\dots\cup L_{i-1}$, and consequently also from
$L_{1k}\cup\dots\cup L_{i-1,k}$.
Since each point of $K$ also lies in some $\Cl D_{jk}$ for $j<i$, it must be contained in 
$L_{jk}\vartriangle\Cl D_{jk}$.
Thus $L_i\vartriangle D_{ik}$ lies in $\bigcup_{j=1}^i L_{jk}\vartriangle \Cl D_{jk}$.
Consequently, $\bigcup_{i=1}^n L_i\vartriangle D_{ik}\subset\bigcup_{i=1}^n L_{ik}\vartriangle\Cl D_{ik}$.
Hence $\mu^*(\bigcup_{i=1}^n L_i\vartriangle D_{ik})\le 1/k$.

Let us note that $D_{0k}\bydef I\but\Cl(D_{1k}\cup\dots\cup D_{nk})$ lies in $\bigcup_{i=1}^n L_i\vartriangle D_{ik}$
since the $L_i$ cover $I$.
Each $D_{ik}\cap (0,1)$ is a finite union of open intervals; let $C_{ik}$ be the union of $D_{ik}\cap (0,1)$ with
the left endpoints of these intervals.
Also, let us add $\{1\}$ to $C_{0k}$.
Then the step functions $f_k$ defined by $f_k|_{C_{ik}}=x_i$, where $x_0\in X$ can be chosen at random, 
converge to $f$ in measure.
\end{proof}

\begin{proof}[(b)]
For each $k\in\N$ let $C_k$ be a countable cover of $X$ by open sets of diameters $\le 1/n$.
If $X$ is compact, we assume that $C_k$ is finite.
If $C_k=\{U_1,U_2,\dots\}$, let $V_i=U_i\but(U_1\cup\dots\cup U_{i-1})$ and let $v_i\in V_i$.
Each $V_i$ is Borel, so each $L_i\bydef f^{-1}(V_i)$ is measurable.
Then the simple functions $f_k$ defined by $f_k|_{L_i}=v_i$ uniformly converge to $f$.
\end{proof}

\begin{proof}[(c)] If $X$ is countable, then it is easy to see that there are only countably many step functions 
in $X^I$ with jumps at rational points.
This implies the assertion.
\end{proof}

\begin{proof}[(d)] 
For an open $U\subset X$ and an $\eps>0$ let $U^\eps$ be the set of all $x$ such that $U$ contains the
closed ball of radius $\eps$ centered at $x$.
It is not hard to see that $\phi^{-1}(U)=\bigcup_{k\in\N}\bigcup_{n\in\N}\bigcap_{m\ge n}\phi_m^{-1}(U^{1/k})$.
Indeed, if $\phi(x)\in U$, then $\phi(x)\in U^{1/2k}$ for some $k\in\N$.
Then there exists an $n\in\N$ such that $\phi_m(x)\in U^{1/k}$ for all $m>n$.
Conversely, if there exist $k$ and $n$ such that $\phi_m(x)\in U^{1/k}$ for all $m>n$, then 
$\phi(x)\in\Cl U^{1/k}\subset U$. 
Since each $\phi_m$ is measurable (Borel), $\bigcup_{k\in\N}\bigcup_{n\in\N}\bigcap_{m\ge n}\phi_m^{-1}(U^{1/k})$
is a measurable (Borel) set.
Hence $\phi$ is measurable (Borel). 
\end{proof}

\begin{proof}[(e)] By (b), $f$ is a uniform limit of simple functions $f_n$.
Let $L_{nk}=f_n^{-1}(x_{nk})$ be the point-inverses of $f_n$.
Each $L_{nk}=B_{nk}\cup Z_{nk}$, where $B_{nk}$ is Borel and $Z_{nk}$ is of measure zero.
Then $B_n\bydef \bigcup_k Z_{nk}$ is also of measure zero.
Since $B_n=I\but\bigcup_k B_{nk}$, it is Borel.
Then $B\bydef \bigcup_n Z_n$ is also a Borel set of measure zero.
Let $K_{nk}=L_{nk}\but B$.
Since $K_{nk}=B_{nk}\but B$, it is Borel.
Let us define a Borel simple function $g_n\:I\to X$ by $g_n|_{K_{nk}}=f_n|_{K_{nk}}$ for each $k$
and $g_n(B)=z$, where $z\in X$ is some fixed point (the same for all $n$).
Since each $g_n|_B=g_1|_B$ and each $g_n|_{I\but B}=f_n|_{I\but B}$, the maps $g_i$ converge to
a map $g$ such that $g(B)=z$ and $g|_{I\but B}=f|_{I\but B}$.
Since $B$ is of measure zero, $g$ is equal to $f$ almost everywhere.
Also, by (d), $g$ is Borel.
\end{proof}

\begin{proof}[(f)] Since $f$ and $g$ are measurable, by (b) and (c) so is $f\x g\:I\to X\x X$.
Since $d\:X\x X\to [0,\infty)$ is continuous, $d\circ (f\x g)$ is measurable.
Hence $|f\neE g|$ and $|f\nee\lambda g|$ are measurable.
\end{proof}

\begin{proof}[(g)]
By Lemma \ref{Ky Fan}(b), it suffices to show that $D'(f,g)=\lambda$ implies $\mu(|f\nee\lambda g|)\le\lambda$.
Assuming that $D'(f,g)=\lambda$, there exists a sequence $\eps_n\to\lambda$ such that $\eps_n\ge\lambda$ and
$\mu(|f\nee{\eps_n} g|)\le\eps_n$ for each $n$.
Since $|f\nee\lambda g|=\bigcup_n|f\nee{\eps_n} g|$, we have $\mu(|f\nee{\eps_n} g|)\to\mu(|f\nee\lambda g|)$.
Hence $\mu(|f\nee\lambda g|)\le\lambda$.
\end{proof}

\subsection{Convergence almost everywhere}

A sequence of maps $f_n\in X^I$ is {\it (pointwise) convergent almost everywhere} to
an $f\in X^I$ if there exists an $S\subset I$ of measure zero such that
$f_n|_{I\but S}$ pointwise converge to $g|_{I\but S}$.
This is equivalent to requiring that $\mu^*(|f_n\neC f|)=0$, where $|f_n\neC f|$ is the set of all $t\in I$ for 
which $f_n(t)\not\to f(t)$.

\begin{lemma} \label{convergences-lemma}
Let us consider sequences $f_n\in X^I$ and $g_n\in X^I$ and a map $f\in X^I$.

(a) $D(f_n,g_n)\to 0$ if and only if 
\[\mu^*\big(|f_n\nee\eps g_n|\big)\to 0\quad\text{for each}\quad \eps>0.\tag{$*$}\] 
In particular, $f_n$ converge in measure to $f$ if and only if
$\mu^*(|f_n\nee\eps f|)\to 0$ for each $\eps>0$.
\vspace{-10pt}

(b) If the $f_n$ are measurable and converge almost everywhere to $f$, then $f$ is measurable.

(c) $d(f_n,g_n)\to 0$ almost everywhere if 
\[\mu^*\Big(\bigcup_{m\ge n} |f_m\nee\eps g_m|\Big)\to 0\quad\text{for each}\quad \eps>0.\tag{$**$}\]
The converse holds when the $f_n$ and $g_n$ are measurable.

In particular, $f_n$ converge almost everywhere to $f$ if $\mu^*\big(\bigcup_{m\ge n} |f_m\nee\eps f|\big)\to 0$ 
for each $\eps>0$.
The converse holds when the $f_n$ are measurable.

(d) If the $f_n$ are measurable and converge almost everywhere to $f$, then the $f_n$ converge 
in measure to $f$.
\end{lemma}

The converse to (d) is false.
Indeed, for each $k$ let $n_k=1+\dots+k=k(k+1)/2$, and for each $i=1,\dots,k$ let
$f_{n_{k-1}+i}:I\to I$ be the indicator function of $[(i-1)/k,\,i/k]$.
Then the $f_i$ converge to $0$ in measure, but $f_i(x)\not\to 0$ for each $x\in [0,1]$.

\begin{proof}[Proof. (a)] $D'(f_n,g_n)\to 0$ means that for each $\eps>0$, almost all $n$ (= all except for
finitely many) satisfy $\mu^*(|f_n\nee\eps g_n|)\le\eps$.
Condition ($*$) means that for each $\eps>0$ and each $\lambda>0$, almost all $n$ satisfy
$\mu^*(|f_n\nee\eps g_n|)\le\lambda$.
Clearly, ($*$) implies $D'(f_n,g_n)\to 0$.
Conversely, suppose that $D'(f_n,g_n)\to 0$.
If $\eps\le\lambda$, then $\mu^*(|f_n\nee\eps g_n|)\le\eps\le\lambda$; and if 
$\lambda\le\eps$, then $\mu^*(|f_n\nee\eps g_n|)\le\mu^*(|f_n\nee\lambda g_n|)\le\lambda$.
Hence ($*$) holds.
\end{proof}

\begin{proof}[(b)]
We are given an $S\subset I$ of measure zero such that $f_{n_k}|_{I\but S}$ pointwise converge to $f|_{I\but S}$.
Clearly, $f|_S$ is measurable.
By the proof of Lemma \ref{closure}(d), $f|_{I\but S}$ is also measurable.
Hence $f$ is measurable.
\end{proof}

\begin{proof}[(c)]
By the definition of limit (for sequences of reals), $|d(f_n,g_n)\neC 0|=\bigcup_{\eps>0} |d(f_n,g_n)\nec\eps 0|$, 
where $|d(f_n,g_n)\nec\eps 0|=\bigcap_{n\in\N}\bigcup_{m\ge n} |f_m\nee\eps g_m|$.
Hence $\mu^*(|d(f_n,g_n)\neC 0|)=0$ if and only if $\mu^*(|d(f_n,g_n)\nec\eps 0|)=0$ for each $\eps>0$.
On the other hand, we have $\mu^*(|d(f_n,g_n)\nec\eps 0|)=
\mu^*\big(\bigcap_{n\in\N}\bigcup_{m\ge n} |f_m\nee\eps g_m|\big)\le
\lim_n\mu^*\big(\bigcup_{m\ge n} |f_m\nee\eps g_m|\big)$.
The latter inequality turns into equality if $\mu^*=\mu$, which is the case when the $f_m$ and $g_m$ are measurable.
\end{proof}

\begin{proof}[(d)] This follows from (a) and (c), since ($**$) trivially implies ($*$).
\end{proof}

\begin{theorem} \label{convergences} (a) If a sequence of maps $f_n\in X^I$ converges in measure to an $f\in X^I$, 
then every infinite subsequence of the $f_n$ contains a subsequence that converges almost everywhere to $f$.
The converse holds if the $f_n$ are measurable.

(b) Given sequences of maps $f_n\in X^I$ and $g_n\in X^I$, if $D(f_n,g_n)\to 0$, then every infinite sequence
$n_k$ contains a subsequence $n_{k_i}$ such that $d(f_{n_{k_i}},g_{n_{k_i}})\to 0$ almost everywhere.
The converse holds if the $f_n$ and $g_n$ are measurable.
\end{theorem}

The case $X=\R$ of Theorem \ref{convergences}(a) is noted in passing in \cite{BP1} (repeated in \cite{BP}*{\S VI.7})
as an easy consequence of standard textbook material.
So it is, but, surprisingly, the result does not seem to be stated in textbooks.  
The proof of the general case is similar to that of the case $X=\R$.

\begin{proof} Part (a) is a special case of (b), so it suffices to prove (b).

Suppose that $D(f_n,g_n)\not\to 0$.
Then condition ($*$) is not satisfied.
Hence there exists an $\eps>0$, a $\lambda>0$ and an infinite subsequence $n_k$ such that 
$\mu^*(|f_{n_k}\nee\eps g_{n_k}|)>\lambda$ for each $k$.
Then for every subsequence $n_{k_i}$ we have
$\mu^*\big(\bigcup_{j\ge i} |f_{n_{k_j}}\nee\eps g_{n_{k_j}}|\big)
\ge\mu^*(|f_{n_{k_i}}\nee\eps g_{n_{k_i}}|)>\lambda$
for each $i$.
Hence ($**$) is not satisfied for this subsequence.

Conversely, suppose that $D(f_n,g_n)\to 0$.
Then for any given subsequence $n_k$ we also have $D(f_{n_k},g_{n_k})\to 0$.
Let $F_k=f_{n_k}$ and $G_k=g_{n_k}$.
Since the $F_k$ and $G_k$ satisfy condition ($*$), there exist $k_1<k_2<\dots$ such that
$\mu^*(|F_{k_i}\nee{1/i} G_{k_i}|)<2^{-i}$ for each $i\in\N$.
Then $\mu^*\big(\bigcup_{i\ge m}|F_{k_i}\nee{1/i} G_{k_i}|\big)\le\sum_{i\ge m}\mu^*(|F_{k_i}\nee{1/i} G_{k_i}|)<
\sum_{i\ge m}2^{-i}=2^{-m+1}$ for each $m\in\N$.
Given an $\eps>0$, we have $|F_{k_i}\nee\eps G_{k_i}|\subset|F_{k_i}\nee{1/i} G_{k_i}|$ as long as $1/i<\eps$.
Hence $\mu^*\big(\bigcup_{i\ge m}|F_{k_i}\nee\eps G_{k_i}|\big)\le
\mu^*\big(\bigcup_{i\ge m}|F_{k_i}\nee{1/i} G_{k_i}|\big)\le 2^{-m-1}$ as long as $1/m<\eps$.
It follows that the $F_{k_i}$ and $G_{k_i}$ satisfy ($**$).
\end{proof}

Since (pointwise) convergence almost everywhere has nothing to do with any metric on $X$, we obtain

\begin{corollary} \label{measure-independence}
(a) \cite{BP1}  The topology of convergence in measure on the set of measurable functions $I\to X$ depends only 
on the topology (and not the metric) of $X$.

(b) The pre-uniformity of convergence in measure on the set of measurable functions $I\to X$ depends only on 
the uniform structure (and not the metric) of $X$.
\end{corollary}

\begin{corollary}
If a sequence of measurable maps $f_n\in X^I$ converges in measure to an $f\in X^I$, then $f$ is measurable.
\end{corollary}

\begin{proof} 
By Theorem \ref{convergences}(a), some subsequence $f_{n_k}$ converges to $f$ almost everywhere.
Now the assertion follows from Lemma \ref{convergences-lemma}(b).
\end{proof}

\subsection{Cauchy sequences}

A sequence of maps $f_n\in X^I$ is called {\it fundamental in measure} if it is a fundamental (= Cauchy)
sequence with respect to the metric $D$ (or equivalently with respect to its underlying uniform structure).
This means for each $\eps>0$ there exists a $k$ such that for all $m,n>k$, $D(f_m,f_n)<\eps$.
This is equivalent to requiring that $D(f_m,f_n)\to 0$ as $m,n\to\infty$, and also (by considering the negations)
to requiring that for every increasing subsequences $m_k$ and $n_k$, $D(x_{m_k},x_{n_k})\to 0$ as $k\to\infty$.

On the other hand, a sequence of maps $f_n\in X^I$ is {\it (pointwise) fundamental almost everywhere}
if there exists an $S\subset I$ of measure zero such that $f_n|_{I\but S}$ is pointwise fundamental.
This is equivalent to requiring that $\mu^*(|d(f_m,f_n)\neC 0|)=0$, where $|d(f_m,f_n)\neC 0|$ is the set of all 
$t\in I$ for which $d\big(f_m(t),f_n(t)\big)\not\to 0$ as $m,n\to\infty$.

\begin{lemma} \label{convergences-lemma'}
(a) A sequence of maps $f_n\in X^I$ is fundamental in measure if and only if for each $\eps>0$,
\[\mu^*\big(|f_m\nee\eps f_n|\big)\to 0\quad\text{as}\quad m,n\to\infty.\tag{$*'$}\] 

(b) A sequence of maps $f_n\in X^I$ is fundamental almost everywhere if for each $\eps>0$,
\[\mu^*\Big(\bigcup_{m,n\ge k}|f_m\nee\eps f_n|\Big)\to 0\quad\text{as}\quad k\to\infty.\tag{$**'$}\]
The converse holds if the $f_n$ are measurable.

(c) If a sequence of measurable maps $f_n\in X^I$ is fundamental almost everywhere, then it is fundamental 
in measure.
\end{lemma} 

The proof is a trivial modification of that of Lemma \ref{convergences-lemma}.

\begin{theorem} \label{convergences'}
(a) If a sequence of maps $f_n\in X^I$ is fundamental in measure, then it contains a subsequence that
is fundamental almost everywhere.

(b) If $X$ is complete, then the set of measurable functions $I\to X$ is complete with respect to $D$.
\end{theorem}

The ``converse'' to part (a) like in Theorem \ref{convergences}(a) is false.
Indeed, let $f_n\:I\to\{-1,1\}$ be defined by $f_n(x)=(-1)^n$.
Then the $f_n$ are measurable, and each infinite subsequence $f_{n_k}$ contains a uniformly convergent
subsequence $f_{n_{k_i}}$ (indeed, at least one of the sets $\{k\mid n_k\text{ is even}\}$ and
$\{k\mid n_k\text{ is odd}\}$ is infinite, and so can be enumerated by a subsequence).
In particular, $f_{n_{k_i}}$ is fundamental almost everywhere.
However, the sequence $f_n$ is not fundamental in measure.

\begin{proof}[Proof. (a)] We have $D'(f_m,f_n)\to 0$ as $m,n\to\infty$.
So for each $\eps>0$, there exists a $k$ such that all $m,n>k$ satisfy $\mu^*(|f_m\nee\eps f_n|)\le\eps$.
Then there exist $k_1<k_2<\dots$ such that for each $l$ we have
$\mu^*(|f_m\nee{2^{-l}} f_n|)\le 2^{-l}$ whenever $m,n\ge k_l$.
Let $F_i=f_{k_i}$.
Then, in particular, $\mu^*(|F_i\nee{2^{-i}} F_{i+1}|)\le 2^{-i}$ for each $i$.
Given $l<i<j$, since $2^{-i}+\dots+2^{-j+1}\le\sum_{p>l} 2^{-p}=2^{-l}$, we have
$|F_i\nee{2^{-l}} F_j|\subset |F_i\nee{2^{-i}} F_{i+1}|\cup\dots\cup |F_{j-1}\nee{2^{-j+1}} F_j|$.
Hence $\bigcup_{i,j>l}|F_i\nee{2^{-l}} F_j|\subset\bigcup_{i>l}|F_i\nee{2^{-i}} F_{i+1}|$.
Therefore $\mu^*\big(\bigcup_{i,j>l}|F_i\nee{2^{-l}} F_j|\big)\le\sum_{i>l}\mu^*(|F_i\nee{2^{-i}} F_{i+1}|)
\le\sum_{i>l} 2^{-p}=2^{-l}$.
If $2^{-l}<\eps$, we have $|F_i\nee\eps F_j|\subset|F_i\nee{2^{-l}} F_j|$ and consequently
$\mu^*\big(\bigcup_{i,j>l}|F_i\nee\eps F_j|\big)\le 2^{-l}$.
It follows that the $F_i$ satisfy ($**'$).
\end{proof}

\begin{proof}[(b)] We need to show that if a sequence of measurable maps $f_n\in X^I$ is fundamental in measure, 
then it converges in measure to some measurable $f\in X^I$.
By (a), $f_n$ contains a subsequence $f_{n_k}$ such that $f_{n_k}|_{I\but S}$ is pointwise fundamental for some 
$S\subset I$ of measure zero.
Since $X$ is complete, $f_{n_k}|_{I\but S}$ pointwise converges to some map $I\but S\to X$.
Let us extend the latter to map $f\:I\to X$ in an arbitrary way.
Since $S$ is of measure zero, $f_{n_k}$ converges to $f$ almost everywhere.
Then by Lemma \ref{convergences-lemma}(b,d), $f_{n_k}$ converges in measure to $f$, and $f$ is measurable.
Since $f_n$ is fundamental in measure, it also converges in measure to $f$.
\end{proof}

\subsection{Lebesgue integral}

Let us now briefly recall the basics of Lebesgue integral for measurable functions $I\to [a,b]$.
First let $f\:I\to [a,b]$ be a simple function, defined by $f(L_i)=t_i$, where $L_1,L_2,\dots\subset I$ are 
disjoint measurable sets with $\bigcup_i L_i=I$.
Let $\int_I f(t)\,dt=\sum_i\mu(L_i)t_i$, where the sum is finite since its absolute value is bounded 
above by $1\cdot\max(|a|,|b|)$.

\begin{lemma} \label{integral}
If $f,g\:I\to [a,b]$ are simple functions such that $D'(f,g)\le\lambda$, then 
$\Big|\int_I f(t)\,dt-\int_I g(t)\,dt\Big|\le\lambda(1+b-a)$.
\end{lemma}

\begin{proof}
Since two countable sets have countable product, we may assume that $f$ and $g$ are defined by $f(L_i)=s_i$ 
and $g(L_i)=t_i$, where $L_1,L_2,\dots\subset I$ are disjoint measurable sets with $\bigcup_i L_i=I$.
Then $\int_I f(t)\,dt-\int_I g(t)\,dt=\sum_i \mu(L_i)(s_i-t_i)$.
Let $S=\big\{i\ \big|\ |s_i-t_i|\le\lambda\big\}$.
Then by Lemma \ref{convergences-lemma}(b), $\mu(\bigcup_{i\notin S} L_i)=\mu(|F\nee\lambda G|)\le\lambda$.
Hence $\sum_{i\in S}\mu(L_i)|s_i-t_i|\le 1\cdot\lambda$ and
$\sum_{i\notin S}\mu(L_i)|s_i-t_i|\le\lambda(b-a)$.
Thus $\Big|\int_I f(t)\,dt-\int_I g(t)\,dt\Big|\le\lambda(1+b-a)$.
\end{proof}

Now by Lemma \ref{closure}(a,b), every measurable function $f\:I\to [a,b]$ is a limit in measure of simple 
functions $f_i\:I\to [a,b]$ (in fact, these can be chosen to be step functions.)
By Lemma \ref{integral}, the sequence of real numbers $I_{f_n}\bydef \int_I f_i(t)\,dt$ is fundamental.
We define $\int_I f(t)\,dt$ to be their limit.
Given another sequence of simple functions $g_i\:I\to [a,b]$ converging in measure to $f$, the sequence
$I_{f_1},I_{g_1},I_{f_2},I_{g_2},\dots$ is also fundamental, so $\lim I_{f_n}=\lim I_{g_n}$.

\subsection{The $L_p$ metric}
If $X$ has diameter $\le 1$, then the function $d_{fg}\:I\to I$, defined by $d_{fg}(t)=d\big(f(t),g(t)\big)$, 
is bounded and measurable, and hence Lebesgue integrable.
Let \[L_1(f,g)=\int_I d_{fg}(t)\,dt\]
and more generally \[L_p(f,g)=\left(\int_I\big(d_{fg}(t)\big)^p\,dt\right)^{1/p}\]
for $1\le p<\infty$.

\begin{proposition} \label{L_p metric}
If $X$ has diameter $\le 1$, then

(a) $L_1$ is a pseudo-metric on the set of measurable functions $I\to X$, and is uniformly 
equivalent to $D$;

(a) $L_p$ is a pseudo-metric on the set of measurable functions $I\to X$, and is uniformly 
equivalent to $L_1$.
\end{proposition}

\begin{proof}[Proof. (a)] By integrating $d_{fh}(t)\le d_{fg}(t)+d_{gh}(t)$ we get $L_1(f,h)\le L_1(f,g)+L_1(g,h)$.

By the definition of $D(f,g)$, for each $\eps>0$ there exists a $\lambda>0$ such that 
$\lambda+\mu(|f\nee\lambda g|)\le D(f,g)+\eps$.
Let $S=|f\nee\lambda g|$.
Then $L_1(f,g)=\int_S d_{fg}(t)\,dt+
\int_{I\but S}d_{fg}(t)\,dt\le\mu(S)\sup_{t\in S}d_{fg}(t)+
\mu(I\but S)\sup_{t\notin S}d_{fg}(t)\le\mu(S)\cdot 1+1\cdot\lambda\le D(f,g)+\eps$.
Thus $L_1(f,g)\le D(f,g)$.

Next let $D_*(f,g)=\sup\big\{\eps\mu^*(|f\nee\eps g|)\mid\eps>0\big\}$.
Then for each $\eps>0$ there exists a $\lambda>0$ such that 
$\lambda\mu(|f\nee\lambda g|)\ge D_*(f,g)-\eps$.
Let $S=|f\nee\lambda g|$.
Then $L_1(f,g)\ge\int_S d_{fg}(t)\,dt\ge\mu(S)\inf_{t\in S} d_{fg}(t)\,dt\ge\mu(S)\lambda\ge D_*(f,g)-\eps$.
Thus $L_1(f,g)\ge D_*(f,g)$.

On the other hand, since $D'(f,g)=\sup\big\{\eps>0\mid\mu^*(|f\nee\eps g|)>\eps\big\}$
and $\mu(f\nee\eps g)>\eps$ implies $\eps\mu(f\nee\eps g)>\eps^2$, we have 
$D_*(f,g)\ge \big(D'(f,g))^2$.
Thus $L_1(f,g)\ge\big(D'(f,g))^2$.
\end{proof}

\begin{proof}[(b)] We have $L_p(f,g)=||d_{fg}||_p$, where $||\phi||_p=L_p(\phi,0)$ for a real-valued function 
$\phi$.
The Minkowski inequality $||d_{fh}||_p\le||d_{fg}||_p+||d_{gh}||_p$ is the triangle inequality for $L_p$.
Let $p<q$.
Since $X$ has diameter $\le 1$, we have $d_{fg}^p\ge d_{fg}^q$ and consequently 
$\big(L_p(f,g)\big)^p\ge\big(L_q(f,g)\big)^q$.
On the other hand, the H\"older inequality $||\phi\psi||_1\le ||\phi||_\alpha||\psi||_\beta$, where
$\alpha^{-1}+\beta^{-1}=1$, $\alpha,\beta\ge 1$, 
implies $||1\cdot d_{fg}^p||_1\le ||1||_{q/(q-p)}||d_{fg}^p||_{q/p}$.
Since $||1||_{q/(q-p)}=\mu(I)=1$, by raising to the power of $1/p$ we get $L_p(f,g)\le L_q(f,g)$.
\end{proof}

\section{The space of measurable functions}

\subsection{Hartman--Mycielski construction}

If $X$ is a metrizable uniform (topological) space, let $MF(X)\subset X^I$ denote the set of all measurable 
maps $I\to X$ with the uniformity (topology) of convergence in measure.
Let $MF_0(X)=MF(X)/_{\sim_0}$, where $\sim_0$ denotes, as before, equality almost everywhere.
By the above, the set of all Borel maps $I\to X$ surjects onto $MF_0(X)$.

By Proposition \ref{L_p metric}, $MF_0([-1,1])$ is uniformly homeomorphic to the unit ball of the 
Banach space $L_p([0,1])$ for each $1\le p<\infty$.
Since $L_2([0,1])$ is linearly isometric to $l_2$ (as Hilbert spaces with orthonormal bases of the same 
cardinality), its unit ball is also uniformly homeomorphic to that of each $l_p$, $1\le p<\infty$ 
(see \cite{BL}*{Theorem 9.1}).

Let $HM(X)\subset MF_0(X)$ consist of the equivalence classes of step functions $I\to X$.
These equivalence classes can be identified with step functions $[0,1)\to X$ or equivalently 
with step functions $S^1\to X$.
$HM(X)$ is known in the literature as the Hartman--Mycielski construction (or functor), after their 
1958 paper \cite{HaM}.
Let $HM_n(X)$ be the subspace of $HM(X)$ consisting of step functions $[0,1)\to X$ with at most $n$ jumps.
It is easy to see that $i_X\:X\to HM(X)$, defined by $i_X(x)(t)=x$ for all $t\in [0,1)$,
is a uniform homeomorphism onto $HM_0(X)$, and that $HM_0(X)$ is closed in $MF_0(X)$.

If $X$ is separable or complete, then by the above, $MF_0(X)$ is also separable or complete.
In fact, when $X$ is complete, $MF_0(X)$ is the completion of $HM(X)$.
In particular, if $X$ is completely metrizable, then $HM(X)$ is dense in $MF_0(X)$.

\begin{example} It is easy to see that $MF_0(X)$ is not compact already for $X=\{0,1\}$.
Indeed, let $f_i\:[0,1)\to\{0,1\}$ be the $i$th digit in the standard binary expansion
(with $0.\alpha_1\dots\alpha_n1000\dots$ preferred to $0.\alpha_1\dots\alpha_n0111\dots$).
Clearly, $L_1(f_i,f_j)=\frac12$ for $i\ne j$.
\end{example}

\subsection{Milgram's classifying space}
Surprisingly, the following simple observation does not seem to appear in the literature:

\begin{theorem} \label{EG}
If $G$ is a discrete group, then each $HM_n(G)$ is $G$-homeomorphic to the $n$-skeleton of Milgram's $EG$.
\end{theorem}

Let us recall Milgram's classifying space construction (see \cite{Se}).
Let $C_G$ be the category with one object, with morphisms given by the elements of $G$ and with
composition given $g\circ h=hg$.
Let $\bar C_G$ be the category with objects given by the elements of $G$, with a unique morphism
$\phi^g_h\:g\to h$ between each (ordered) pair of objects, and with composition given by
$\phi^h_k\circ\phi^g_h=\phi^g_k$.
Clearly, $F\:\bar C_G\to C_G$, $\phi^g_h\mapsto g^{-1}h$, is a functor.
The nerve $NC_G$ is Milgram's construction of $BG$, whereas $N\bar C_G$ is the corresponding $EG$.

\begin{proof}
Let us recall McCord's description of Milgram's construction \cite{Mc}.
Let $G$ be a (discrete) group and $X$ be set with a basepoint $b$.
Let $G[X]=\bigoplus_{x\in X}G_x$, where each $G_x$ is a copy of $G$.
Thus if $G$ is abelian, elements of $G[X]$ are finite formal linear combinations $g_1x_1+\dots+g_nx_n$, 
where each $g_i\in G$ and each $x_i\in X$.
If $G$ is non-abelian, the $x_i$ must be assumed pairwise distinct.
Let $G_b[X]=G[X]/G_b$.
Let $I=[0,1]$ and let $p$ be the composition $G_0[I]\to G[I]/G[\partial I]=G_{\bar 0}[S^1]$, where 
$S^1=I/\partial I$ and $\bar 0$ denotes the image of $0$ in $S^1$.
For each $n\ge 0$ let $G^n[S^1]$ be the subset of $G[S^1]$ consisting of all linear combinations 
of the form $g_1x_1+\dots+g_nx_n$, and let $G_{\bar 0}^n[S^1]$ be its image in $G_{\bar 0}[S^1]$.
Let $G_0^n[I]=p^{-1}(G_{\bar 0}^n[S^1])$, and let $p_n\:G_0^n[I]\to G_{\bar 0}^n[S^1]$ be the restriction of $p$.

Let $\Delta^n=\{(t_1,\dots,t_n)\in I^n\mid t_1\le\dots\le t_n\}$.
Let $\nu_m$ be the composition $G^n\x\Delta^n\to G^n[I]\to G_b^n[I]$, where the first map is given by 
$(g_1,\dots,g_n,t_1,\dots,t_n)\mapsto g_1t_1+\dots+g_nt_n$ if $G$ is abelian, and in general by
$(g_1,\dots,g_n,t_1,\dots,t_n)\mapsto h_1s_1+\dots+h_ks_k$, where each $s_i=t_{\mu_{i-1}+1}=\dots=t_{\mu_i}$ 
and $s_1<\dots<s_n$, and each $h_i=g_{\mu_i}g_{\mu_i-1}\cdots g_{\mu_{i-1}+1}$.
Let us topologize $G_0^n[I]$ and $G_{\bar 0}^n[S^1]$ with the quotient topologies from $\nu_n$ and $p_n\nu_n$, 
respectively.
The inclusions $G_0^0[I]\subset G_0^1[I]\subset\dots$ and 
$G_{\bar 0}^0[S^1]\subset G_{\bar 0}^1[S^1]\subset\dots$ are clearly embeddings.
It is not hard to see that the $n$-skeleton of $NC_G$ is homeomorphic to $G_{\bar 0}^n[S^1]$
and the $n$-skeleton of $N\bar C_G$ is homeomorphic to $G_0^n[I]$, so that
the map $p_n\:G_0^n[I]\to G_{\bar 0}^n[S^1]$ is induced by the functor $F$.

Now a step function $\phi\in HM_n(G)$ may be thought of as a sum of ``$G$-valued Heaviside functions'',
whose derivative is then the sum of ``$G$-valued $\delta$-functions''.
More precisely, if $\phi(t)=g_i$ for each $t\in [t_{i-1},t_i)$, where $0=t_0<t_1<\dots<t_n<t_{n+1}=1$, 
let $\phi'=(g_2^{-1}g_1)t_1+\dots+(g_{n+1}^{-1}g_n)t_n+g_{n+1}t_{n+1}$.
Then $\phi'\in G_0^n[I]$, and it is not hard to see that the map $HM_n(G)\to G_0^n[I]$, $\phi\mapsto\phi'$,
is an equivariant homeomorphism.
\end{proof}

\subsection{Theory of retracts}

\begin{lemma} \label{linear homotopy}
The map $MF_0(X)\x MF_0(X)\x I\to MF_0(X)$, $([f],[g],t)\mapsto [h_t^{f,g}]$, where 
$h_t^{f,g}\:[0,1]\to X$ is defined by 
\[h_t^{f,g}(s)=\begin{cases} 
f(s),\text{ if }s<1-t;\\
g(s),\text{ if }s\ge 1-t,
\end{cases}\]
is uniformly continuous.
\end{lemma}

Let us note that if $f$ and $g$ are step functions, then so is $h_t^{f,g}$ for each $t\in I$.

\begin{proof} Let us fix a bounded metric on $X$.
Then $L_1(h_t^{f,g},h_t^{f',g'})\le L_1(f,g)+L_1(f',g')$
and $L_1(h_t^{f,g},h_{t'}^{f,g})\le |t-t'|\Delta$, where $\Delta$ is the diameter of $X$.
\end{proof}

\begin{proposition} \label{lec}
If $X$ is a metric space, then $MF_0(X)$ and $HM(X)$ are uniformly contractible and uniformly 
locally contractible.
\end{proposition}

\begin{proof} Let us pick an arbitrary basepoint $[f]\in MF_0(X)$.
By Lemma \ref{linear homotopy} the map $MF_0(X)\x I\to MF_0(X)$, $([g],t)\mapsto [h^{f,g}_t]$, 
is uniformly continuous.
Thus $MF_0(X)$ is uniformly contractible; similarly, so is $HM(X)$.

Given uniformly continuous maps $F,G\:Y\to MF_0(X)$, let $H\:Y\x I\to MF_0(X)$
be defined by $(y,t)\mapsto [h_t^{f,g}]$, where $[f]=F(y)$ and $[g]=G(y)$.
By Lemma \ref{linear homotopy} $H$ is uniformly continuous.
Also, $L_1(f,h_t^{f,g})\le L_1(f,g)$ for each $t\in I$.
So if $F$ and $G$ are $\eps$-close, then $H$ is an $\eps$-homotopy.
Thus $MF_0(X)$ is uniformly locally contractible; similarly, so is $HM(X)$.
\end{proof}

It is shown in \cite{Nhu5} that if $X$ is a nonempty metrizable space, then $MF_0(X)$ and $HM(X)$ are ARs
and consequently (using two results of Toru\'nczyk), $MF_0(X)$ is homeomorphic to a Hilbert space.
When $X$ is also separable and contains at least two points, $MF_0(X)\cong\ell_2$ by an earlier result of 
Bessaga and Pe\l czy\'nski \cite{BP1}, \cite{BP}.
See also \cite{RR} concerning the homeomorphism type of $HM(X)$.

\section{Metric-free treatment}

\subsection{Convergence in measure}
If $C$ is a cover of a set $X$, and $f,g\in X^I$, let 
\[|f\nee C g|=\big\{t\in I\ \big|\ \text{no }V\in C\text{ contains both }f(t)\text{ and }g(t)\big\},\]
\[U^C_\eps(f)=\big\{g\mid \mu^*\big(|f\nee C g|\big)<\eps\big\}.\]

Let $X$ be a pre-uniform space.
Given a uniform cover $C$ of $X$ and an $\eps>0$, let $C_\eps^I$ be the cover of $X^I$ by the sets $U^C_\eps(f)$
for all $f\in X^I$.

\begin{lemma} \label{general-measurable}
(a) The covers $C_\eps^I$ form a basis of a pre-uniformity on $X^I$.

(b) If $X$ is a metric space, this pre-uniformity is induced by the pseudo-metric $D'$.
\end{lemma}

The pre-uniformity of (a) will be called the pre-uniformity of {\it convergence in measure} 
(cf.\ \cite{FR}*{2.7(4)}, where vicinities are used instead of uniform covers).
Let us note that Lemma \ref{general-measurable}(c) yields a different proof and a strengthening of 
Corollary \ref{measure-independence} (b).

\begin{proof}[Proof. (a)]
If $E$ refines $C$, then $|f\nee C g|\subset |f\nee E g|$ and consequently
$\mu^*\big(|f\nee E g|\big)\ge\mu^*\big(|f\nee C g|\big)$.
Hence any two covers $C_\eps^I$ and $D_\delta^I$ are both refined by $E_\gamma^I$, where
$E$ is a common refinement of $C$ and $D$ and $\gamma=\min(\delta,\eps)$.

If $D$ barycentrically refines $C$, then $|g\nee C h|\subset |f\nee D g|\cup |f\nee D h|$ and consequently
$\mu^*\big(|f\nee D g|\big)+\mu^*\big(|f\nee D h|\big)\ge\mu^*\big(|g\nee C h|\big)$.
Hence $g,h\in U^D_{\eps/2}(f)$ implies $h\in U^C_\eps(f)$, and therefore $D_{\eps/2}^I$ barycentrically 
refines $C_\eps^I$.
\end{proof}

\begin{proof}[(b)]
If $D_\eps$ is the cover of $X$ by all sets of diameter $\le\eps$, then $|f\nee {D_\eps} g|=|f\nee \eps g|$.
Hence by Lemma \ref{Ky Fan}(b), $U^{D_\eps}_\eps(f)$ contains the $\eps^-$-ball centered at $f$ and is contained 
in the $\eps^+$-ball centered at $f$, with respect to $D'$, whenever $\eps^-<\eps<\eps^+$.
Thus the covers $(D_\eps)_\eps^I$ form a basis of the uniformity induced by $D'$.
Now any uniform cover $C$ of $X$ is refined by $D_\lambda$ for some $\lambda>0$, and hence by the proof of (a),
$C_\eps^I$ for any $\eps>0$ is refined by $(D_\delta)_\delta^I$, where $\delta=\min(\lambda,\eps)$. 
\end{proof}

\begin{lemma} \label{general-measurable2} Let $X$ be a topological space.

(a) The sets $U^C_\eps(f)$ for all open covers $C$ of $X$, all $f\in X^I$ and all $\eps>0$
form a basis of a topology on $X^I$.

(b) If $f\:I\to X$ is measurable and $C$ is a countable open cover of $X$, then $U^C_\eps(f)$ contains 
a step function for each $\eps>0$.  
\end{lemma}

This topology of (a) will be called the topology of {\it convergence in measure}. 
Part (b) is proved similarly to Lemma \ref{closure}(a,b).

\begin{proof}[Proof of (a)]
If $g\in U^C_\eps(f)\cap U^D_\delta(h)$, then there exists a $\gamma>0$ such that
$\mu^*\big(|f\nee C g|\big)<\eps-\gamma$ and $\mu^*\big(|h\nee D g|\big)<\delta-\gamma$.
Then any $g'\in U^{C\cap D}_\gamma(g)$ satisfies 
$\mu^*\big(|g\nee C g'|\big)\le\mu^*\big(|g\nee {C\cap D} g'|\big)<\gamma$
and hence $\mu^*\big(|f\nee C g'|\big)<\eps$.
Thus $U^{C\cap D}_\gamma(g)\subset U^C_\eps(f)$, and similarly $U^{C\cap D}_\gamma(g)\subset U^D_\delta(h)$.
\end{proof}

\subsection{Comparison of topologies}

\begin{theorem} \label{t-vs-u}
Let $uX$ be a pre-uniform space and $X$ its underlying topological space.
The topology $T_{uX}$ of the pre-uniformity of convergence in measure coincides with the topology $T_X$ 
of convergence in measure

(a) on the set $HM'(X)$ of all step functions $I\to X$;

(b) on the set $MF(X)$ of all measurable maps $I\to X$, if $X$ is fully normal and Lindel\"of.
\end{theorem}

A topological space is called {\it fully normal} if every its open cover has an open barycentric refinement.
By a well-known result of Stone \cite{En}*{5.1.12}, for $T_1$ spaces full normality is equivalent 
to paracompactness.
A topological space is called {\it Lindel\"of} if every its open cover has a countable open refinement.
By a well-known result of Morita \cite{En}*{3.8.11}, regular Lindel\"of spaces are paracompact, and
by a result of Willard \cite{En}*{5.1.26}, separable paracompact spaces are Lindel\"of.

Let us note that (b) yields a different proof of Corollary \ref{measure-independence}(a) for 
separable metrizable spaces.

\begin{proof}[Proof. (a)] 
A basis $\beta_0(f)$ of neighborhoods of $f$ in $T_{uX}$ consists of the sets $\st(f,C_\eps^I)$ for 
all uniform covers $C$ of $uX$ and all $\eps>0$.
It suffices to consider only a basis of uniform covers of $uX$; we may assume that is contains only open covers 
(see \cite{I3}*{I.19}).

A basis $\beta_1(f)$ of neighborhoods of $f$ in $T_X$ consists of the sets $U^C_\eps(g)$ that contain $f$, 
for all open covers $C$ of $X$, all $g\in X^I$ and all $\eps>0$.
Since $\st(f,C_\eps^I)$ is a union of such sets, we have $\beta_0(f)\subset\beta_1(f)$, and consequently
$T_{uX}\subset T_X$.

Another basis $\beta_2(f)$ of neighborhoods of $f$ in $T_X$ consists of the sets $U^C_\eps(f)$, for 
all open covers $C$ of $X$ and all $\eps>0$.
Indeed, $\beta_1$ refines $\beta_2$ (that is, every element of $\beta_2(f)$ 
contains some element of $\beta_1(f)$) since $\beta_2(f)\subset\beta_1(f)$.
Conversely, let us show that $\beta_2(f)$ refines $\beta_1(f)$, 
If $U^C_\eps(g)$ contains $f$, then it contains $U^C_\eps(g)\cap U^C_\eps(f)$, which by 
the proof of Lemma \ref{general-measurable2}(a) in turn contains $U^C_\delta(f)$ for some $\delta>0$.

Now let us show that if $f$ is a step function, then $\beta_0(f)$ refines $\beta_2(f)$.
Let $C$ be an open cover of $X$.
If $x_1,\dots,x_n\in X$ are the values of $f$, let $V_i=\st(x_i,C)$.
Then $|f\nee C g|=\{t\in I\mid f(t)=x_i,\ g(t)\notin V_i\}$.
Since $V_i$ is open and the topology of $X$ is induced by the pre-uniformity of $uX$, there exists 
a uniform cover $D_i$ of $uX$ such that $\st(x_i,D_i)\subset V_i$.
If $D$ is a common uniform refinement of $D_1,\dots,D_n$, then $W_i\bydef \st(x_i,D)\subset V_i$ for each $i$.
Hence $|f\nee D g|=\{t\in I\mid f(t)=x_i,\ g(t)\notin W_i\}$ contains $|f\nee C g|$.
Therefore $\mu^*\big(|f\nee D g|\big)\ge\mu^*\big(|f\nee C g|\big)$, and consequently
$U^D_\eps(f)\subset U^C_\eps(f)$.
If $E$ is a uniform barycentric refinement of $D$, then by the proof of Lemma \ref{general-measurable}(a), 
$\st(f,E^I_{\eps/2})\subset U^D_\eps(f)$.
Thus $\st(f,E^I_{\eps/2})\subset U^C_\eps(f)$, so $\beta_0(f)$ refines $\beta_2(f)$, 
and consequently $T_X|_{HM'(X)}=T_{uX}|_{HM'(X)}$.
\end{proof}

\begin{proof}[(b)]
We continue the proof of (a).
If $X$ is Lindel\"of and fully normal and $f$ is a measurable map, then yet another basis $\beta_3(f)$ of 
neighborhoods of $f$ in $T_X$ consists of the sets $U^C_\eps(g)$ containing $f$, for all step functions $g$, 
all open covers $C$ of $X$ and all $\eps>0$. 
Indeed, $\beta_1(f)$ refines $\beta_3(f)$ since $\beta_3(f)\subset\beta_1(f)$.
Let us show that $\beta_3(f)$ refines $\beta_2(f)$.
Suppose that we are given an open cover $C$ and an $\eps>0$.
Since $X$ is fully normal and Lindel\"of, there exists a countable open barycentric refinement $D$ of $C$.
Then by Lemma \ref{general-measurable2}(b), $U^D_{\eps/2}(f)$ contains a step function $g$.
Then $U^D_{\eps/2}(g)$ contains $f$ and is therefore contained in $\st(f,D_{\eps/2}^I)$.
By the proof of Lemma \ref{general-measurable}(a), the latter in turn lies in $U^C_\eps(f)$. 

Finally, let us show that $\beta_0(f)$ refines $\beta_3(f)$.
Let $C$ be an open cover of $X$, let $g$ be a step function with values $x_1,\dots,x_n\in X$, and let
$V_i=\st(x_i,C)$.
Similarly to the proof of (a), there exists a uniform cover $D$ of $uX$ such that $\st(x_i,D)\subset V_i$ 
for each $i$.
Let $E$ be obtained from $D$ by removing all elements that meet any of the $x_i$, and inserting the $V_i$
instead.
Since $\st(x_i,D)\subset V_i$, $E$ is refined by $D$, and hence is a uniform cover.
Since $\st(x_i,D)=V_i=\st(x_i,C)$, we have $|g\nee E h|=\{t\in I\mid g(t)=x_i,\ h(t)\notin V_i\}=|g\nee C h|$.
Hence $U^E_\eps(g)=U^C_\eps(g)$.
Since $U^E_\eps(g)$ contains $f$, it contains $U^E_\eps(g)\cap U^E_\eps(f)$, which by 
the proof of Lemma \ref{general-measurable2}(a) in turn contains $U^E_\delta(f)$ for some $\delta>0$.
Finally, if $F$ is a uniform cover star-refining $E$, then by the proof of Lemma \ref{general-measurable}(a), 
$\st(f,F^I_{\delta/2})\subset U^E_\delta(f)$
Thus $\st(f,F^I_{\delta/2})\subset U^C_\eps(g)$, so $\beta_0(f)$ refines $\beta_3(f)$, and consequently
$T_X|_{MF(X)}=T_{uX}|_{MF(X)}$.
\end{proof}

\begin{lemma} \label{open-in-HM}
Let $X$ be a $T_1$ topological space and $U_1,\dots,U_n$ be pairwise disjoint open subsets of $X$.
Let $V_{U_1,\dots,U_n}$ be the set of all step functions $f\:[0,1)\to X$ such that each $U_i$ contains 
at least one value $f(s_i)$ of $f$, where the $s_i$ satisfy $s_1<\dots<s_n$.
Then $V_{U_1,\dots,U_n}$ is open in $HM(X)$.
\end{lemma}

\begin{proof} Let $f\in V_{U_1,\dots,U_n}$. 
Then there exist $s_1<t_1\le\dots\le s_n<t_n$ such that $f$ is constant on each $[s_i,t_i)$, and each $f(s_i)\in U_i$.
Let $L=\min(t_1-s_1,\dots,t_n-s_n)$.
Since $X$ is $T_1$, the sets $U_1,\dots,U_n$ and $X\but\{f(s_1),\dots,f(s_n)\}$ form an open cover $D$ of $X$.
It is easy to see that $U^D_L(f)\subset V_{U_1,\dots,U_n}$.
By the proof of Theorem \ref{t-vs-u}(a), the sets $U^C_\eps(f)$, where $C$ is an open cover of $X$ and
$\eps>0$, form a basis of neighborhoods of $f$ in $HM(X)$.
So $V_{U_1,\dots,U_n}$ contains a neighborhood of every its point, and therefore is open. 
\end{proof}

\newpage
\part{REVIEW OF MEASURE SPACES} \label{prob-measures}

This chapter is largely expository and brings together material from rather diverse fields, including 
functional analysis, Lipschitz geometry, applied optimization and probability theory (not to mention 
measure theory).
However, most proofs (and statements) have been either reworked to various degrees or found from scratch.
As a result, some errors in the literature have been corrected (see Theorem \ref{sequential} and 
Theorem \ref{KR-complete}) and some arguments have hopefully been simplified or made more explicit.

\section{Measures and charges}

\subsection{Definitions}
A {\it (finite signed Borel) measure} $\mu$ on a topological space $X$ is a countably additive real-valued 
function on the $\sigma$-algebra $\B(X)$ of all Borel subsets of $X$.
A measure $\mu$ on $X$ is called a {\it probability measure} if it is nonnegative and $\mu(X)=1$.

For each $x\in X$ the {\it Dirac measure} $\delta_x$ is defined by 
\[\delta_x(A)=
\begin{cases}
1&\text{if } x\in A;\\
0&\text{if } x\notin A.
\end{cases}\]
Scalar multiples $\lambda\delta_x$, $\lambda\in\R$, of Dirac measures are also known as ``atomic'' measures, 
and their finite linear combinations $\sum_i\lambda_i\delta_{x_i}$, $\lambda_i\in\R$, as ``molecular'' measures.

Given a measure $\mu$ on $X$, every continuous map $f\:X\to Y$ induces a measure $f_*\mu$ on $Y$,
defined by $f_*\mu(A)=\mu\big(f^{-1}(A)\big)$.
Also, if $Z$ is a Borel subset of $X$, the inclusion $i\:Z\to X$ induces a measure $i^!\mu$ on 
the subspace $Z$, defined by $i^!\mu(A)=\mu(A)$.
Clearly, if $\mu$ is nonnegative, then so are $f_*\mu$ and $i^!\mu$, and if $\mu$ is a probability measure, 
then so is $f_*\mu$.
It is easy to see that $i^!i_*\nu=\nu$ for any measure $\nu$ on $Z$.

\begin{lemma} {\rm \cite{AlAD}*{Theorem 9.4}, \cite{Bog}*{3.1.1}} \label{partition}
If $\mu$ is a measure on $X$, there exists a partition of $X$ into disjoint Borel sets $X^+$, $X^-$ such that 
$\mu(A)\ge 0$ for each Borel $A\subset X^+$ and $\mu(A)\le 0$ for each Borel $A\subset X^-$.
\end{lemma}

From this lemma we have the measures $\mu_+\bydef i_*i^!\mu$ and $\mu_-\bydef j_*j^!\mu$ on $X$ induced by the inclusions 
$i\:X_+\to X$ and $j\:X_-\to X$.
Clearly, $\mu_+$ is nonnegative, $\mu_-$ is nonpositive and $\mu=\mu_++\mu_-$.
The nonnegative measure $|\mu|\bydef \mu_+-\mu_-$ is called the {\it total variation} of $\mu$.
Clearly, $||\mu||\bydef |\mu|(X)$ defines a norm on the vector space of all measures on $X$.

A measure $\mu$ on $X$ is called 
\begin{itemize}
\item {\it regular}, if for each Borel set $A$, $|\mu|(A)=\sup_Z|\mu|(Z)$ over all closed $Z\subset A$;
\item {\it Radon}, if for each Borel set $A$, $|\mu|(A)=\sup_K|\mu|(K)$ over all compact $K\subset A$;
\item {\it $\tau$-additive} (or $\tau$-smooth), if for each decreasing net of closed subsets $Z_\alpha\subset X$ 
with $\bigcap_\alpha Z_\alpha=\emptyset$ the numerical net $\mu(Z_\alpha)$ converges to $0$.
\end{itemize}
It is easy to see that a finitely additive real-valued function $\mu$ on Borel subsets of $X$ is 
countably additive (=``$\sigma$-additive'') if and only if for each decreasing sequence of 
Borel subsets $Z_n\subset X$ with $\bigcap_{n=1}^\infty Z_n=\emptyset$ the numerical sequence 
$\mu(Z_n)$ converges to $0$.
Clearly, a regular measure is $\tau$-additive if and only if it satisfies the same property with ```sequences''
replaced by ``nets''.

\begin{lemma}\label{regularity1} {\rm (see \cite{Bog}*{7.1.7})} Every measure on a metrizable space is regular.
\end{lemma}

\begin{proof} Given a Borel subset $A\subset X$, let $U_n$ be the open $\frac1n$-neighborhood of $X\but A$.
By the countable additivity of $|\mu|$, $|\mu|(U_n)\to|\mu|(X\but A)$.
Hence $|\mu|(X\but U_n)\to|\mu|(A)$. 
\end{proof}

\subsection{Radon measures}

If $\mu$ is a measure on $X$ and $Z$ is a subspace of $X$, the inclusion $i\:Z\to X$ induces a measure $i^!\mu$ 
on $Z$, defined by $i^!\mu(A)=\inf\{\mu(B)\mid B\in\B(X),\,A\subset B\}$%
\footnote{Let us note that ``$A\subset B$'' can be replaced with ``$A=Z\cap B$'' here, since open 
(and consequently Borel) subsets of $Z$ are intersections of $Z$ with open (respectively, Borel) subsets of $X$.}
if $\mu$ is nonnegative and by $i^!\mu(A)=i^!\mu_+-i^!(-\mu_-)$ in the general case.
When $Z$ is a Borel subset of $X$, the new definition of $i^!\mu$ clearly reduces to the previous one.
For an arbitrary $Z\subset X$, we still have $i^!i_*\nu=\nu$ for any measure $\nu$ on $Z$.

The inclusion $i\:Z\to X$ also induces a measure $i^?\mu$ on $Z$, defined by 
$i^?\mu(A)=\sup\{\mu(B)\mid B\in\B(X),\,B\subset A\}$ if $\mu$ is nonnegative and by 
$i^?\mu(A)=i^?\mu_+-i^?(-\mu_-)$ in the general case.
When $Z$ is a Borel subset of $X$, we clearly have $i^?\mu=i^!\mu$.

\begin{lemma} \label{Radon} Let $X$ be a Polish (=separable completely metrizable) space, or more generally a 
separable absolute Borel space.

(a) Every measure on $X$ is Radon.

(b) Let $\nu$ be a measure on a subset $Z$ of $X$ and let $i\:Z\to X$ be the inclusion map.
Then $\nu$ is Radon if and only if $\nu=i^?i_*\nu$.
\end{lemma}

Both assertions are implicit in \cite{Ban1}; see also \cite{Bog}*{7.1.7}, \cite{Va}*{Appendix, Remark (b)}.

An {\it absolute Borel} space is a metrizable space whose every homeomorphic image in every metrizable space 
is Borel.
By a theorem of Lavrentiev (see \cite{BP}*{Theorem VIII.1.1}) every Borel subset of every completely metrizable 
space is an absolute Borel space.

\begin{proof} Let us first prove the ``only if'' assertion of (b) in the case where $X$ is compact.
By Lemma \ref{regularity1} $i_*\nu$ is regular.
Since $X$ is compact, $i_*\nu$ is Radon.
Hence $i^?i_*\nu$ is also Radon.
Thus $\nu=i^?i_*\nu$ implies that $\nu$ is Radon.

Next let us prove (a).
Since $X$ is separable and metrizable, it admits an embedding $g$ in the Hilbert cube $I^\infty$.
Since $X$ is absolute Borel, $g(X)$ is Borel in $I^\infty$.
Hence any measure $\nu$ on $X$ satisfies $g^?g_*\nu=g^!g_*\nu=\nu$.
Then by the compact case of (b) $\nu$ is Radon.

Finally, let us prove (b).
By (a) $i_*\nu$ is Radon, and therefore $i^?i_*\nu$ is also Radon.
Thus $\nu=i^?i_*\nu$ implies that $\nu$ is Radon.
The converse implication follows from the fact that compact subsets of $Z$ are Borel subsets of $X$.
\end{proof}

\subsection{$\tau$-additive measures}

\begin{lemma} \label{tau2} {\rm (see \cite{Bog}*{7.2.2(iv)}, \cite{Va}*{Corollary IV to Theorem I.25})}
Every measure on a separable metrizable space is $\tau$-additive.
\end{lemma}

\begin{proof} Given a decreasing net of closed sets $Z_\alpha\subset X$ with $\bigcap_\alpha Z_\alpha=\emptyset$,
the open cover of $X$ by the sets $X\but Z_\alpha$ has a countable subcover $\{X\but Z_{\alpha_n}\mid n\in\N\}$
(see \cite{En}*{4.1.16}).
Since the net $Z_\alpha$ is decreasing, the sequence $Z_{\alpha_n}$ can be amended so as to be decreasing.
Then every measure $\mu$ on $X$ satisfies $|\mu|(Z_{\alpha_n})\to 0$ by the countable additivity of $|\mu|$.
Hence $|\mu|(Z_\alpha)\to 0$.
\end{proof}

\begin{lemma} \label{support} {\rm (see \cite{Bog}*{7.2.9})}
If $\mu$ is a $\tau$-additive measure on a metrizable space $X$, then there exists a smallest (by inclusion) 
closed subset $Z\subset X$ such that $|\mu|(X\but Z)=0$.
\end{lemma}

\begin{proof} Let $Z$ be the intersection of all closed sets $Z_\alpha$ such that $|\mu|(X\but Z_\alpha)=0$.
The sets $Z_\alpha$ form a decreasing net, since their finite intersections are of the form $Z_\alpha$ 
by the finite additivity of $|\mu|$.
By Lemma \ref{regularity1} $\mu$ is regular, so, given an $\eps>0$, there exists an open set $U\supset Z$ 
such that $|\mu|(U\but Z)<\eps$.
Since $\mu$ is $\tau$-additive, $|\mu|(Z_\alpha\but U)\to 0$.
Since $\eps>0$ is arbitrary, it follows that $|\mu|(Z_\alpha\but Z)\to 0$.
Hence $|\mu|(X\but Z)=0$.
\end{proof}

This closed set $Z$ provided by Lemma \ref{support} is called the {\it support} of the $\tau$-additive 
measure $\mu$.
By Lemma \ref{tau} below, it is always separable.

\begin{lemma}\label{tau} {\rm (see \cite{Bog}*{proof of 7.2.10}, \cite{Va}*{Corollary to Theorem I.27})}
Let $\mu$ be a measure on a metrizable space $X$.
Then $\mu$ is $\tau$-additive if and only if there exists a separable closed set $Z\subset X$ 
such that $|\mu|(X\but Z)=0$.
\end{lemma}

\begin{proof} To prove the ``if'' assertion, let $i\:Y\to X$ be the inclusion.
By Lemma \ref{tau2} $i^!\mu$ is $\tau$-additive.
Hence so is $i_*i^!\mu$.
Since $|\mu|(X\but Y)=0$, we have $i_*i^!\mu=\mu$.

Conversely, since $\mu$ is $\tau$-additive, by Lemma \ref{support} it has support $Z$.
Suppose that $Z$ is not separable.
Then it contains an uncountable collection of disjoint open subsets $U_\lambda$ (see \cite{En}*{4.1.16}).
Each $Z\but U_\lambda$ is closed in $Z$ and hence also in $X$.
If $|\mu|(U_\lambda)=0$, then $|\mu|(Z\but U_\lambda)=0$, contradicting the minimality of $Z$. 
So each $|\mu|(U_\lambda)>0$.

Using the axiom of choice, we may assume that the indexing set $\Lambda$ of $U_\lambda$'s is well-ordered.
Let $V_\lambda=\bigcup\{U_\kappa\mid\kappa\le\lambda\}$.
Then $\lambda\mapsto|\mu|(V_\lambda)$ defines an order-preserving injection $f\:\Lambda\to\R$.
Let $f'(\lambda)$ be any rational number between $f(\lambda)$ and $f(\lambda+1)$.
Thus we get an injection $f'\:\Lambda\to\Q$, which is a contradiction.
\end{proof}

\begin{lemma}\label{regularity} \ 
(a) {\rm (see \cite{Bog}*{7.2.2(i)})} Every Radon measure is $\tau$-additive.

(b) Let $X$ be a completely metrizable space, or more generally an absolute Borel space.
Then every $\tau$-additive measure on $X$ is Radon.
\end{lemma}

\begin{proof}[Proof. (a)]
Given an $\eps>0$, let $K$ be a compact subset such that $\mu(X\but K)<\eps$.
Given a decreasing net of closed sets $Z_\alpha$ with $\bigcap_\alpha Z_\alpha=\emptyset$, the open cover 
of $K$ by the sets $K\but Z_\alpha$ has a finite subcover.
Since the sets $K\but Z_\alpha$ form an increasing net, actually $K=K\but Z_\alpha$ for some $\alpha$.
Then $\mu(Z_\alpha)<\eps$, and consequently $\mu(Z_\beta)<\eps$ for each $\beta\ge\alpha$.
\end{proof}

\begin{proof}[(b)]
Let $\mu$ be a $\tau$-additive measure on $X$.
By Lemma \ref{tau} there exists a separable closed subset $Z\subset X$ such that $|\mu|(X\but Z)=0$.
Let $i\:Z\to X$ be the inclusion.
Clearly, $i^!\mu$ is $\tau$-additive.
Hence by Lemma \ref{Radon} $i^!\mu$ is Radon.
Therefore $i_*i^!\mu$ is also Radon.
Since $|\mu|(X\but Y)=0$, we have $i_*i^!\mu=\mu$.
\end{proof}

\begin{example}\label{tau-example}
Non-$\tau$-additive measures exist (see \cite{Va}*{Remark IV to Theorem I.28}).
However, if $X$ is a metric space and the minimal cardinality $\frak m$ of the base of a topology on $X$
is uncountable, then the existence of a non-$\tau$-additive measure on $X$ is equivalent to 
{\it $\sigma$-measurability} of $\frak m$, that is, the existence of a probability measure defined on all 
subsets of a set of cardinality $\frak m$ and vanishing on all singletons (see \cite{Bog}*{7.2.10}).
The existence of a $\sigma$-measurable cardinal cannot be proved in ZFC (nor even be proved in ZFC to be
consistent with ZFC) as it implies the existence of a weakly inaccesible cardinal (see \cite{Jech}*{10.1}), 
which implies the consistency of ZFC.
However, the existence of a $\sigma$-measurable cardinal is implied by the existence of a so-called
measurable cardinal (see \cite{Jech}*{\S10}), which is one of the best known large cardinal axioms.
\end{example}

\subsection{Charges}

Given a topological space $X$, let $\B_\omega(X)$ be the algebra generated by all closed%
\footnote{Closed sets in this capacity are often replaced by functionally closed sets 
(= zero-sets of real-valued functions) in the literature (\cite{Va} and \cite{Bog}, but not \cite{Par}).
Let us note that in every metric space, every closed set $Z$ is functionally closed, namely $Z=f^{-1}(0)$,
where $f(x)=d(x,Z)$.}
subsets of $X$.
It is a subalgebra of the Borel $\sigma$-algebra $\B(X)$.

A {\it charge} (cf.\ \cite{AlAD}; also called ``additive set function'') on $X$ is a bounded finitely 
additive function $\B_\omega(X)\to\R$. 
(Let us note that every countably additive function $\B(X)\to\R$ is obviously bounded.)

If $\mu$ and $\nu$ are charges on $X$, we write $\mu\le\nu$ if $\mu(A)\le\nu(A)$ for all $A\in \B_\omega(X)$.
It is easy to see that charges on $X$ form a partially ordered vector space over $\R$ (i.e., a vector space and 
a poset where $\mu\le\nu$ implies $\mu+\lambda\le\nu+\lambda$ and $r\mu\le r\nu$ for each $r\in\R$).

If $\mu$ is a charge on $X$, let $\mu_+(A)=\sup\{\mu(B)\mid B\in \B_\omega(X),\,B\subset A\}$ 
and $\mu_-(A)=\inf\{\mu(B)\mid B\in \B_\omega(X),\,B\subset A\}$ for each $A\in B_\omega(X)$.
Clearly, $\mu_+$ and $\mu_-$ are charges on $X$, and we have $\mu_+\ge 0$ and $\mu_-\le 0$.
Let us note that $\mu_-=-(-\mu)_+$.

\begin{lemma}\label{lattice} Let $\mu$ be a charge on $X$.

(a) The poset of charges on $X$ is a lattice, with $\mu_+=\mu\lor 0$ and $\mu_-=\mu\land 0$.

(b) \cite{YH}*{1.12} {\rm (see also \cite{DuS}*{III.1.8})} $\mu=\mu_++\mu_-$.
\end{lemma}

Thus charges on $X$ form a Riesz space, i.e.\ a partially ordered vector space whose underlying poset is 
a lattice.

\begin{proof}[Proof. (a)] Clearly, $\mu_+\ge\mu$ and $\mu_+\ge 0$.
If $\nu$ is a charge on $X$ satisfying $\nu\ge\mu$ and $\nu\ge 0$, then for each $A,B\in \B_\omega(X)$ with 
$B\subset A$ we have $\nu(A)=\nu(B)+\nu(A\but B)\ge\mu(B)+0$, and therefore $\nu\ge\mu_+$.
Thus the least upper bound of $\mu$ and $0$ exists and equals $\mu_+$.
Similarly, the greatest lower bound of $\mu$ and $0$ exists and equals $\mu_-$.

A charge $\lambda$ on $X$ satisfies $\lambda\ge\mu$ and $\lambda\ge\nu$ if and only if it satisfies 
$\lambda-\nu\ge\mu-\nu$ and $\lambda-\nu\ge 0$.
Hence the least upper bound of $\mu$ and $\nu$ exists and equals $(\mu-\nu)_++\nu$.
Similarly, the greatest lower bound of $\mu$ and $\nu$ exists and equals $(\mu-\nu)_-+\nu$.
\end{proof}

\begin{proof}[(b)]
By (a) $\mu\lor 0=\mu_+$, and by the proof of (a), $\lambda\lor\nu=(\lambda-\nu)_++\nu$.
Hence we also have $0\lor\mu=(-\mu)_++\mu=\mu-\mu_-$.
Since $\mu\lor\nu=\nu\lor\mu$, we get $\mu_+=\mu-\mu_-$.
\end{proof}

\begin{remark}
Here is an alternative proof of (b).
Since charges on $X$ form a partially ordered abelian group, which by (a) is a lattice, they 
satisfy%
\footnote{Since $\lambda\mapsto\kappa+\lambda$ is an automorphism of the poset of charges on $X$, and
$\lambda\mapsto-\lambda$ is an anti-automorphism of this poset, $\lambda\mapsto\kappa-\lambda$ is 
an anti-automorphism of the poset.
Hence we have $\kappa-(\mu\land\nu)=(\kappa-\mu)\lor(\kappa-\nu)$, and it remains to set $\kappa=\mu+\nu$.}
the so-called modular law $\mu+\nu=(\mu\land\nu)+(\mu\lor\nu)$.%
\footnote{Its best known special cases are $mn=\gcd(m,n)\lcm(m,n)$ for $m,n\in\Z_+$ and 
$x+y=\max(x,y)+\min(x,y)$ for $x,y\in\R$.}
Hence $\mu=\mu+0=(\mu\land 0)+(\mu\lor 0)=\mu_-+\mu_+$.
\end{remark}

The nonnegative charge $|\mu|\bydef \mu_+-\mu_-$ is called the {\it total variation} of $\mu$.
Clearly, $||\mu||\bydef |\mu|(X)$ defines a norm on the vector space of all charges on $X$.
A charge $\mu$ on $X$ is called {\it regular} if for each $A\in \B_\omega(X)$, 
$|\mu|(A)=\sup_Z|\mu|(Z)$ over all closed $Z\subset A$.

A non-countably additive regular charge is ``constructed'', using the axiom of choice, in 
Example \ref{charge-example} below. 

\begin{lemma} \label{Hahn-Jordan} Let $\mu$ be a charge on $X$.

(a) \cite{AlAD}*{Lemma 6.8} If $\mu$ is regular, then so are $\mu_+$ and $\mu_-$.

(b) If $\mu$ is regular, for each $\eps>0$ there exist disjoint closed sets $Z_+,Z_-\subset X$ such that
$\mu_+(X)\le\mu(Z_+)+\eps$ and $\mu_-(X)\ge\mu(Z_-)-\eps$.
\end{lemma}

\begin{proof}[Proof. (a)] Since $\mu$ is regular, for every $A\in \B_\omega(X)$ and $\eps>0$ there exists 
a closed $Z\subset A$ such that $|\mu|(A)\le|\mu|(Z)+\eps$, that is, 
$\mu_+(A)-\mu_-(A)\le\mu_+(Z)-\mu_-(Z)+\eps$.
On the other hand, since $\mu_+$ is nonnegative and additive, $\mu_+(A)\ge\mu_+(Z)$; and 
similarly $\mu_-(A)\le\mu_-(Z)$.
By adding the first and third inequalities, we get $\mu_+(A)\le\mu_+(Z)+\eps$.
From this and the second inequality it follows that $\mu_+(A)=\sup_Z\mu_+(Z)$ over all closed $Z\subset A$.
Since $\mu_+$ is nonnegative, $|\mu_+|=\mu_+$.
Hence $|\mu_+|(A)=\sup_Z|\mu_+|(Z)$.
By symmetry, $\mu_-$ is also regular.
\end{proof}

\begin{proof}[(b)] Let $\delta=\frac\eps3$.
By the definition of $\mu_+$ and $\mu_-$ there exist $A_+,A_-\in \B_\omega(X)$ such that 
$\mu_+(X)\le\mu(A_+)+\delta$ and $\mu_-(X)\ge\mu(A_-)-\delta$.
Let $A=A_+\cap A_-$.
If $\mu(A)\le 0$, let $B_+=A_+\but A$ and $B_-=A_-$; else let $B_+=A_+$ and $B_-=A_-\but A$.
In each case $\mu(A_+)\le\mu(B_+)$ and $\mu(A_-)\ge\mu(B_-)$; in addition, $B_+\cap B_-=\emptyset$.
Thus we get $\mu_+(X)\le\mu(B_+)+\delta$ and $\mu_-(X)\ge\mu(B_-)-\delta$.
Now by additivity, $\mu_+(B_+)\le\mu_+(X)$ and $\mu_-(B_-)\ge\mu_-(X)$.
It follows that $\mu_-(B_+)\ge-\delta$ and $\mu_+(B_-)\le\delta$.

By (b) there exist closed sets $Z_+\subset B_+$ and $Z_-\subset B_-$ 
such that $\mu_+(B_+)\le\mu_+(Z_+)+\delta$ and $\mu_-(B_-)\ge\mu_-(Z_-)-\delta$.
On the other hand, by additivity $\mu_+(Z_-)\le\mu_+(B_-)\le\delta$ and $\mu_-(Z_+)\ge\mu_-(B_+)\ge-\delta$.
It follows that $\mu(B_+)\le\mu_+(B_+)\le\mu(Z_+)+2\delta$, and similarly $\mu(B_-)\ge\mu(Z_-)-2\delta$.
Consequently $\mu_+(X)\le\mu(B_+)+\delta\le\mu(Z_+)+3\delta$ and similarly $\mu_-(X)\ge\mu(Z_-)-3\delta$.
\end{proof}

\begin{lemma} \label{extension} \ 
\begin{enumerate}
\item Every countably additive charge $\mu$ on $X$ uniquely extends to a measure $\bar\mu$ on $X$.
\item Moreover, $|\bar\mu|$ is an extension of $|\mu|$, and in particular $||\bar\mu||=||\mu||$.
\item If $X$ is metrizable, every countably additive charge on $X$ is regular.
\item If $X$ is compact, every regular charge on $X$ is countably additive.
\item A regular charge is countably additive if and only if $\mu(Z_n)\to 0$ for every decreasing sequence of 
closed sets $Z_n$ with $\bigcap_{n=1}^\infty Z_n=\emptyset$.
\end{enumerate}
\end{lemma}

\begin{proof} (1,2) See \cite{DuS}*{III.5.9}. 
(3) By (1) and \ref{regularity1}. 
(4) See \cite{DuS}*{III.5.13}.
(5) See \cite{AlAD}*{Theorem 9.2}.
\end{proof}

\section{Integration}

\subsection{Charges vs.\ functionals}
For any nonnegative charge $\mu$ on a topological space $X$ we define the Riemann integral $\int_X f\, d\mu$ 
in the usual way.
We consider finite partitions of $X$ into elements of $\B_\omega(X)$ and form the upper and lower Darboux sums. 
If the infimum of all the upper Darboux sums equals the supremum of all lower Darboux sums,
this number is denoted $\int_X f\, d\mu$, and $f$ is said to be integrable with respect to $\mu$.
It is not hard to see that all bounded continuous functions are integrable (cf.\ \cite{Par}*{p.\ 35};
see also \cite{DS}*{\S3.2}, \cite{Bog}*{4.7(ix)} for a Lebesgue style approach).
For an arbitrary charge $\mu$ we set $\int_X f\, d\mu=\int_X f\, d\mu_+-\int_X f\,d(-\mu_-)$.

\begin{lemma} \ \label{distinguishability}
(a) \cite{AlAD}*{Lemma 7.11} If $\mu$ and $\nu$ are regular charges on a metrizable space $X$ such that 
$\int_X f\, d\mu=\int_X f\, d\nu$ for every continuous $f\:X\to[0,1]$, then $\mu=\nu$.

(b) If $\mu$ and $\nu$ are measures on a metric space $X$ such that 
$\int_X f\, d\mu=\int_X f\, d\nu$ for every $1$-Lipschitz $f\:X\to[0,1]$, then $\mu=\nu$.
\end{lemma}

See also \cite{Par}*{II.5.7 and II.5.9} for the case of non-negative measures and charges.

\begin{proof} By considering $\mu-\nu$ we may assume that $\nu=0$.
Let $Z$ be any closed subset of $X$, and let $\eps>0$.

In (a), by regularity there exists an open neighborhood $U$ of $Z$ such that $|\mu|(U\but Z)<\eps$.
Let $f\:X\to [0,1]$ be a continuous function sending $Z$ to $1$ and $X\but U$ to $0$.
Then the hypothesis implies $\int_X f\,d\mu=0$.

In (b), let $U_n$ be the open $\frac1n$-neighborhood of $Z$. 
Then $\lim_{n\to\infty}|\mu|(U_n)=|\mu|(Z)$ due to the countable additivity of $|\mu|$.
Hence $|\mu|(U_N\but Z)<\eps$ for some $N\in\N$.
Let $U=U_N$, and let us define $f\:X\to [0,1]$ by $f(x)=1-\min\big(1,Nd(x,Z)\big)$.
Thus $f$ sends $Z$ to $1$ and $X\but U$ to $0$.
The triangle axiom implies that $|f(x)-f(y)|\le Nd(x,y)$, so $f$ is $N$-Lipschitz.
Then $\frac{f}{N}$ is $1$-Lipschitz.
Hence $\int_X f\,d\mu=N\int_X \frac{f}{N}\,d\mu=0$.

On the other hand, in both (a) and (b) we have $\int_X f\,d\mu=\mu(Z)+\int_{U\but Z}f\,d\mu+0$.
Also $\big|\int_{U\but Z}f\,d\mu\big|\le|\mu|(U\but Z)<\eps$.
Hence $|\mu(Z)|<\eps$.
But $\eps>0$ was arbitrary, so $\mu(Z)=0$.
Thus $\mu$ vanishes on closed sets.
Hence by regularity (using \ref{regularity1} for the purposes of (b)) it vanishes identically. 
\end{proof}

If $X$ is a topological space, let $C_b(X)$ be the Banach (=complete normed) space of all continuous bounded 
real-valued functions $X\to\R$ with the norm $||f||=\sup_{x\in X}|f(x)|$.
If $X$ is a uniform space, let $U_b(X)$ be the closed subspace of $C_b(X)$ consisting of all uniformly 
continuous functions.
If $X$ is a metric space, let $\Lip_b(X)$ be the subspace of $U_b(X)$ consisting of all Lipschitz functions.
For a normed space $V$ its dual $V^*$ is the normed space of all continuous ($\Leftrightarrow$ bounded)
linear functionals $\Phi\:V\to\R$ with the norm $||\Phi||=\sup_{||f||\le1}|\Phi(f)|$.
We recall that $V^*$ is always complete (so, a Banach space), even if $V$ is not complete.
If $X$ is a metric space, restriction of functionals yields continuous (in fact, 1-Lipschitz) linear maps 
$C_b^*(X)\to U_b^*(X)\to \Lip_b^*(X)$, which are surjective by the Hahn--Banach theorem.
The map $U_b^*(X)\to\Lip_b^*(X)$ is also injective since $\Lip_b(X)$ is dense in $U_b(X)$
(see Lemma \ref{approximation}(a) below).

\begin{lemma} \label{norm}
Given a charge $\mu$ on a metric space $X$, let us define a linear functional $\Phi_\mu\:C_b(X)\to\R$ by 
$\Phi_\mu(f)=\int_X f\, d\mu$.

(a) If $\mu$ is a regular charge, then $||\Phi_\mu||=||\mu||$; in particular, $\Phi_\mu\in C_b^*(X)$.

(b) If $\mu$ is a measure, then $||\Phi_\mu|_{\Lip_b(X)}||=||\Phi_\mu|_{U_b(X)}||=||\mu||$.
\end{lemma}

The assertion is obvious for any non-negative charge (even not necessarily regular): in this case 
$||\Phi_\mu||=\Phi_\mu(1)=\mu(X)=||\mu||$, and the same applies to $\Phi_\mu|_{\Lip_b(X)}$ and
$\Phi_\mu|_{U_b(X)}$.

The second assertion of (b) is contained in \cite{Pac2}*{7.4}.

\begin{proof}[Proof. (a)] Clearly, $|\Phi_\mu(f)|=\big|\int_Xf\,d\mu\big|\le||f||\cdot||\mu||$, and so 
$||\Phi_\mu||\le||\mu||$.

Conversely, by \ref{Hahn-Jordan}(b) for each $\eps>0$ there exist disjoint closed sets $Z_+,Z_-\subset X$ 
such that $\mu_+(X)\le\mu(Z_+)+\eps$ and $\mu_-(X)\ge\mu(Z_-)-\eps$.
Since $\mu(Z_+)\le\mu_+(Z_+)\le\mu_+(Z_+\cup Z_-)$, we get $\mu_+(X)\le\mu_+(Z_+\cup Z_-)+\eps$ and 
similarly $\mu_-(X)\ge\mu_-(Z_+\cup Z_-)-\eps$.
Hence $|\mu|(X)\le|\mu|(Z_+\cup Z_-)+2\eps$.
Thus by additivity $|\mu|\big(X\but(Z_+\cup Z_-)\big)\le 2\eps$.

Let $f\:X\to[-1,1]$ be a continuous function sending $Z_+$ to $1$ and $Z_-$ to $-1$.
Then $\int_Xf\,d\mu=\mu(Z_+)-\mu(Z_-)+\int_{X\but(Z_+\cup Z_-)}f\,d\mu\ge\mu_+(X)-\mu_-(X)-4\eps=||\mu||-4\eps$.
Since $\eps$ was arbitrary, we get $||\Phi_\mu||\ge||\mu||$.
\end{proof}

\begin{proof}[(b)] Lemma \ref{partition} yields a partition of $X$ into disjoint Borel sets $X_+$ and $X_-$
such that $\mu_+(X_-)=0$ and $\mu_-(X_+)=0$.
Given an $\eps>0$, by Lemma \ref{regularity1} there exists a closed subset $Z\subset X$ contained in $X_+$ 
such that $\mu_+(X_+\but Z)<\eps$.
Since $Z\subset X_+$, we have $\mu_-(Z)=0$.
Let $U_n$ be the open $\frac1n$-neighborhood of $Z$ in $X$.
Then $\lim_{n\to\infty}\mu_-(U_n)=\mu_-(Z)=0$ due to the countable additivity of $\mu_-$.
Hence $\mu_-(U_N)>-\eps$ for some $N\in\N$.
Let $U=U_N$.
Then $\mu_+(U\but Z)\le\mu_+(X\but Z)<\eps$ and $\mu_-(U\but Z)\ge\mu_-(U)>-\eps$.
Hence $|\mu|(U\but Z)<2\eps$.
Also $\mu_+(X\but U)\le\mu_+(X\but Z)<\eps$.
Let us define $f\:X\to [0,1]$ by $f(x)=1-2\min\big(1,Nd(x,Z_+)\big)$.
Thus $f$ sends $Z$ to $1$ and $X\but U$ to $-1$.
The triangle axiom implies that $|f(x)-f(y)|\le 2Nd(x,y)$, so $f$ is $2N$-Lipschitz.

Now $\int_Z f\,d\mu=\mu(Z)=\mu_+(Z)>\mu_+(X)-\eps$.
Also $\big|\int_{U\but Z} f\,d\mu\big|\le |\mu|(U\but Z)<2\eps$.
Finally, $\int_{X\but U} f\,d\mu=-\mu(X\but U)>-\mu_-(X)-2\eps$, since 
$\mu_-(X\but U)<\mu_-(X)+\eps$ and $\mu_+(X\but U)<\eps$.
Thus $\int_X f\,d\mu>|\mu|(X)-5\eps$.
Since $\eps$ was arbitrary, $||\Phi_\mu|_{\Lip_b(X)}||\ge||\mu||$.
The converse inequality follows from (a).
\end{proof}

Conversely, if $X$ is metrizable, it turns out that for every continuous linear functional $\Phi\:C_b(X)\to\R$
there exists a regular charge $\mu$ on $X$ such that $\Phi_\mu=\Phi$ \cite{AlAD}*{Theorem 7.1} (see also 
\cite{Bog}*{7.9.1}, \cite{Par}*{II.5.7}, \cite{DuS}*{IV.6.2}).
Taking into account Lemma \ref{distinguishability}, we get

\begin{theorem} \label{isomorphism} Let $X$ be a metric space.
The assignment $\mu\mapsto\Phi_\mu$ yields

(a) a linear isometry of the $||\cdot||$-normed space of all regular charges on $X$ with $C_b^*(X)$;

(b) linear isometric embeddings of the $||\cdot||$-normed space of all measures on $X$ in $C_b^*(X)$, in
$U_b^*(X)$ and in $\Lip_b^*(X)$.
\end{theorem}

Part (a) and the first assertion of (b) first appeared in \cite{AlAD}.

\subsection{Locally compact spaces}

\begin{corollary} \label{Riesz-Markov}
Let $X$ be a locally compact metrizable space.
Then the assignment $\mu\mapsto\Phi_\mu$ yields an isomorphism of the $||\cdot||$-normed space of 
all measures on $X$ with $C_0^*(X)$, where $C_0(X)$ is the closure in $C_b(X)$ of the subspace
$C_{00}(X)$ of functions that vanish outside a compact set.
\end{corollary}

Variations of this result are known as the Riesz--Markov representation theorem.

\begin{proof}
Let $X^+=X\cup\{\infty\}$ be the one-point compactification of $X$.
The restriction map $r\:C_b(X^+)\to C_b(X)$ is clearly injective.
Its image is clearly closed and contains $C_{00}(X)$.
Hence it contains $C_0(X)$.
Thus $r$ restricts to an isomorphism between $K\bydef r^{-1}\big(C_0(X)\big)$ and $C_0(X)$.
In particular, $(r|_K)^*\:C_0^*(X)\to K^*$ is an isomorphism.

Clearly, $K$ lies in the kernel of $\Phi_{\delta_\infty}\:C_b(X^+)\to\R$, $f\mapsto f(\infty)$.
Conversely, for any $f\in\ker\Phi_{\delta_\infty}$ and each $n\in\N$ there exists an open neighborhood 
$U_n$ of $\infty$ in $X$ such that $f(U_n)\subset (-\frac1n,\frac1n)$.
We may assume without loss of generality that $\bigcap_i U_i=\{\infty\}$.
Each $K_n\bydef X\but U_n$ is compact, and without loss of generality is disjoint from the closure 
$\bar U_{n+1}$ of $U_{n+1}$.
Let $\phi_n$ be a continuous function sending $K_n$ to $1$ and $\bar U_{n+1}$ to $0$, and let $f_n=\phi_n f$.
Then $f_n\to f$ in $C_b(X^+)$, and therefore $f|_X\in C_0(X)$, i.e.\ $f\in K$.

The short exact sequence $0\to K\to C_b(X^+)\xr{\Phi_{\delta_\infty}}\R\to 0$ leads to a short exact sequence
$0\to\R^*\to C_b^*(X^+)\to K^*\to 0$ (since $\R$, like any field, is injective as a module over itself), 
where the image of $\R^*$ is generated by $\Phi_{\delta_\infty}$.
Let $L$ be the subspace of $C_b^*(X^+)$ consisting of all $\Phi_\mu$ where $\mu$ is a measure on $X^+$
such that $\mu(\{\infty\})=0$. 
Clearly, $L\cap\left<\Phi_{\delta_\infty}\right>=0$.
On the other hand, by Theorem \ref{isomorphism}(a) and Lemma \ref{extension}(4), every element of $C_b^*(X^+)$ 
is of the form $\Phi_\mu$ for some measure $\mu$ on $X$.
Since $\Phi_{\mu-\mu(\{\infty\})\delta_\infty}\in L$, we have $C_b(X^+)=L\oplus\left<\Phi_{\delta_\infty}\right>$.
Hence the restriction map $C_b^*(X^+)\to K^*$ restricts to an isomorphism between $L$ and $K^*$.

Let $M$ be the subspace of $C_b^*(X)$ consisting of all $\Phi_\nu$ where $\nu$ is a measure on $X$.
If $i\:X\to X^+$ denotes the inclusion, $i^!i_*\nu=\nu$ and $i_*i^!\mu=\mu$ as long as $\mu(\{\infty\})=0$.
Hence $r^*\:C_b^*(X)\to C_b^*(X^+)$ restricts to an isomorphism between $M$ and $L$.
\end{proof}

\begin{example} \label{charge-example} \cite{Bog}*{7.10.3}, \cite{Pac2}*{5.7}.
Let $\N^+=\N\cup\{\infty\}$ denote the one-point compactification of the countable discrete space.
The restriction map $C_b(\N^+)\to C_b(\N)$ is obviously injective, and so the dual map
$C_b^*(\N)\to C_b^*(\N^+)$ is surjective by the Hahn--Banach theorem.
Hence the Dirac measure $\delta_\infty$ extends to a continuous linear functional $\Phi$ on $C_b(\N)$.

This functional is not of the form $\Phi_\mu$ for any measure $\mu$ on $\N$.
Indeed, since $\mu$ is countably additive, $\Phi_\mu(f)=\int_\N f\,d\mu=\sum_{i=1}^\infty\mu(\{i\}) f(i)$ 
for every $f\in C_b(\N)$.
So either $\Phi_\mu$ is identically zero or there exists an $n\in\N$ such that $\Phi_\mu(f)$ depends 
linearly on the value of $f(n)$.
However, for functions $f\:\N\to\R$ that extend continuously over $\N^+$, 
$\delta_\infty=\lim_{n\to\infty} f(n)$ does not depend on $f(n)$ for any particular $n\in\N$.
Hence $\Phi_\mu$ is not an extension of $\delta_\infty$. 

Let $\mu(A)=\Phi(\chi_A)$, where $\chi_A$ is the indicator function of the set $A\subset\N$.
It is easy to see that $\mu$ is a charge; since the topology of $\N$ is discrete, it is regular.
By the above, $\mu$ is not countably additive.
Since $\Phi$ is an extension of $\delta_\infty$, $\mu(A)=0$ and $\mu(\N\but A)=1$ for all finite $A$.
Using a more precise version of the Hahn--Banach theorem, $\Phi$ can be chosen so that $\mu\ge 0$
(see \cite{Bog}*{1.12.28}).
In fact, given any non-principal ultrafilter $F$ on $\N$, we get an instance of such a $\mu$ 
by setting $\mu(A)=0$ if $A\in F$ and $\mu(A)=1$ if $A\notin F$. 

Let us note that $C_b(\N)=l_\infty$ and $C_b(\N^+)=c$.
It is well-known that $c^*\simeq l_1$%
\footnote{In fact, by Theorem \ref{isomorphism}(a) and Lemma \ref{extension}(4), $c^*$ is isomorphic to 
the space of all measures on $\N^+$.
If $\mu$ is a measure on $\N^+$, let $\mu_x=\mu(\{x\})$ for each $x\in\N^+$.
Then $|\mu|(\{x\})=|\mu_x|$ for each $x\in\N^+$, and by the countable additivity
$|\mu_\infty|=\lim_{n\to\infty}|\mu|(\N^+\but[n])$, where $[n]=\{1,\dots,n\}$.
Also, each $|\mu|(\N^+\but[n])=|\mu|(\N^+)-\sum_{i=1}^n|\mu_i|$. 
Hence $|\mu|(\N^+)=|\mu_\infty|+\sum_{i=1}^\infty|\mu_i|$, and it follows that $c^*\simeq l_1$.}
and consequently also $c^*_0\simeq l_1$, and it is easy to see that $l_1^*\simeq l_\infty$.%
\footnote{Given an $f\in l_1^*$, for each $(x_i)\in l_1$ we have 
$f(x_1,\dots,x_n,0,0,\dots)\to f(x_1,x_2,\dots)$ by the continuity of $f$.
Hence $f$ is determined by its values $f_i$ on the vectors $\delta_i=(0,\dots,0,1,0,0,\dots)$.
Thus the norm-preserving linear map $l_\infty\to l_1^*$, 
$(f_i)\mapsto\big((x_i)\mapsto \sum_{i=1}^\infty f_ix_i\big)$, is a bijection.}
By the proof of Theorem \ref{Riesz-Markov}, the functionals in $C_b^*(\N)$ that are represented by measures 
on $\N$ correspond precisely to the image of the natural injection of $l_1\simeq c_0^*$ in 
$l_1^{**}\simeq l_\infty^*$.
\end{example}

\subsection{Integration of measures}

\begin{theorem} \label{representability} \cite{AlAD}, \cite{Bog}
Let $X$ be a metrizable space and let $\Phi\in C_b^*(X)$.

(a)  The following are equivalent:
\begin{itemize}
\item $\Phi=\Phi_\mu$ for some measure $\mu$ on $X$;
\item if $f_n\in C_b(X)$, $||f_n||\le 1$, pointwise converge to $f\in C_b(X)$, then 
$\Phi(f_n)\to\Phi(f)$;
\item if $f_n\in C_b(X)$ pointwise decrease and pointwise converge to $0$, then $\Phi(f_n)\to 0$.
\end{itemize}

(b) The following are equivalent:
\begin{itemize}
\item $\Phi=\Phi_\mu$ for some $\tau$-additive measure $\mu$ on $X$;
\item if a net $f_\alpha\in C_b(X)$ pointwise decreases and pointwise converges to $0$, then 
the numerical net $\Phi(f_\alpha)$ converges to $0$.
\end{itemize}

(c) The following are equivalent:
\begin{itemize}
\item $\Phi=\Phi_\mu$ for some Radon measure $\mu$ on $X$;
\item $\Phi$ is continuous on the unit ball of $C_b(X)$ in the compact-open topology on $C_b(X)$.
\end{itemize}
\end{theorem}

Let us recall that for a metric space $X$, the compact-open topology on $C(X)$ (and hence also on $C_b(X)$
and on $U_b(X)$) coincides with the topology of uniform convergence on compact sets.

\begin{theorem} \label{representability2} \cite{Pac2}*{5.3, 7.4, 7.6}
Let $X$ be a metrizable uniform space, $\Phi\in U_b^*(X)$.

(a)  The following are equivalent:
\begin{itemize}
\item $\Phi=\Phi_\mu$ for some measure $\mu$ on $X$;
\item if $f_n\in U_b(X)$ pointwise decrease and pointwise converge to $0$, then $\Phi(f_n)\to 0$.
\end{itemize}

(b) The following are equivalent:
\begin{itemize}
\item $\Phi=\Phi_\mu$ for some $\tau$-additive measure $\mu$ on $X$;
\item if a net $f_\alpha\in U_b(X)$ pointwise decreases and pointwise converges to $0$, then 
the numerical net $\Phi(f_\alpha)$ converges to $0$.
\end{itemize}

(c) The following are equivalent:
\begin{itemize}
\item $\Phi=\Phi_\mu$ for some Radon measure $\mu$ on $X$;
\item $\Phi$ is continuous on the unit ball of $U_b(X)$ in the compact-open topology on $U_b(X)$.
\end{itemize}
\end{theorem}

\subsection{Functions vs.\ sets}
For future use we note the following consequence of Lemma \ref{norm}.

\begin{lemma} \label{estimate}
(a) Given a regular charge $\nu$ on a metric space $X$, an open subset $U\subset X$ and 
an $\eps>0$, there exists a continuous $f\:X\to [-1,1]$ such that any charge $\mu$ on $X$ satisfying
$\big|\int_X f\,d(\mu-\nu)\big|<\frac\eps2$ also satisfies $|\mu|(U)>|\nu|(U)-\eps$.

(b) $|\mu|$, $|\nu|$ can be replaced by $\mu_+$, $\nu_+$ (respectively, by $-\mu_-$, $-\nu_-$).

(c) If $\nu$ is countably additive, then $f$ can be chosen to be Lipschitz.
\end{lemma}

Closely related assertions appear in \cite{Va}*{Theorem II.3}, \cite{Bog}*{8.4.7}, \cite{Pac2}*{P.22}.

\begin{proof}[Proof. (a)]
Since $\nu$ is regular, there exists a closed $Z\subset U$ such that $|\nu|(U\but Z)<\frac\eps6$.
Let $\phi\:X\to [0,1]$ be a continuous function sending $Z$ to $1$ and $X\but U$ to $0$.
By Lemma \ref{norm}(a) there exists a continuous $g\:Z\to[-1,1]$ such that 
$\int_Z g\,d\nu>|\nu|(Z)-\frac\eps6$.
Since $Z$ is closed, $g$ extends to a continuous function $\bar g\:X\to [-1,1]$.
Let $f=\phi\bar g$.
We have $\big|\int_{U\but Z}\phi\bar g\,d\nu\big|\le|\nu|(U\but Z)<\frac\eps6$, and consequently
$\int_X f\,d\nu>|\nu|(Z)-\frac{2\eps}6>|\nu|(U)-\frac{3\eps}6$.
Hence on assuming $\big|\int_X f\,d\mu-\int_X f\,d\nu\big|<\frac\eps2$ we have 
\[|\mu|(U)\ge\int_X f\,d\mu\ge\int_X f\,d\nu-\frac\eps2>|\nu|(U)-\eps.\]
\end{proof}

\begin{proof}[(b)]
Let $f_+=\min(f,0)$ and $f_-=\max(f,0)$.
Then $\int_X f_+\,d\nu\le\int_X f_+\,d\nu_+\le\nu_+(U)$ and similarly $\int_X f_-\,d\nu\le-\nu_-(U)$.
On the other hand, $\int_X f_++f_-\,d\nu>|\nu|(U)-\frac\eps2$.
Hence $\int_X f_+\,d\nu>\nu_+(U)-\frac\eps2$ and $\int_X f_-\,d\nu>-\nu_-(U)-\frac\eps2$.

Thus on assuming $\big|\int_X f_+\,d\mu-\int_X f_+\,d\nu\big|<\frac\eps2$ we have 
\[\mu_+(U)\ge\int_X f_+\,d\mu\ge\int_X f_+\,d\nu-\frac\eps2>\nu_+(U)-\eps.\]
Similarly, on assuming $\big|\int_X f_-\,d\mu-\int_X f_-\,d\nu\big|<\frac\eps2$ we have 
$-\mu_-(U)>-\nu_-(U)-\eps$.
\end{proof}

\begin{proof}[(c)]
If $\nu$ is countably additive, then $Z$ may be chosen to be the complement to the open $\frac1k$-neighborhood
of $X\but U$ for some $k\in\N$, and consequently $\phi$ may be assumed $k$-Lipschitz.
By Lemma \ref{norm}(b) $g$ may also be assumed Lipschitz.
By Lemma \ref{lip-ext} below, $\bar g$ may be assumed Lipschitz as well.
Hence $f$, $f_+$, $f_-$ may be assumed Lipschitz.
\end{proof}

\begin{lemma} \label{lip-ext} {\rm (see \cite{Bog}*{8.10.71}, \cite{We}*{1.33})}
Let $X$ be a metric space and $A$ a subset of $X$.
Then every $k$-Lipschitz function $f\:A\to\R$ extends to a $k$-Lipschitz function $\bar f\:X\to\R$ such that 
$\bar f(X)$ lies in the convex hull of the closure of $f(A)$.
\end{lemma}

\begin{proof} Let us define $g\:X\to\R$ by $g(x)=\sup\{f(a)-kd(x,a)\mid a\in A\}$.

Then $g$ extends $f$. Indeed, if $x\in A$, then $|f(x)-f(a)|\le kd(x,a)$ for each $a\in A$. 
Hence $f(a)-kd(x,a)\le f(x)$ and thus $g(x)=f(x)$.

Also $g$ is $k$-Lipschitz. Indeed, given $x,y\in X$ and a $\gamma>0$, there exists an $a\in A$ such that 
$g(x)\le f(a)-kd(x,a)+\gamma$.
Then $g(y)\ge f(a)-kd(y,a)$.
Hence $g(x)-g(y)\le kd(y,a)-kd(x,a)+\gamma\le kd(x,y)+\gamma$.
Since $\gamma$ was arbitrary, we get $g(x)-g(y)\le kd(x,y)$.
Hence by symmetry $|g(x)-g(y)|\le kd(x,y)$.

Finally, let $\bar f(x)=\max\big(g(x),\,\inf_{a\in A}f(a)\big)$.
Then $\bar f$ is again a $k$-Lipschitz extension of $f$; also, $\bar f(X)$ lies in the convex hull of 
the closure of $f(A)$.
\end{proof}

\section{Measure spaces}

\subsection{Notation}

If regular charges on a topological space $X$ are regarded as continuous linear functionals on $C_b(X)$,
the vector space $M(X)$ of all regular charges on $X$ can be endowed with the weak* topology, 
also called the $C_b(X)$-weak topology, which is given by the family of seminorms $\{|\cdot|_f\mid f\in C_b(X)\}$, where 
$|\mu|_f=\big|\int_X f d\mu\big|$.
This means that basic neighborhoods of $0$ are finite intersections of the $\eps_i$-balls of the 
seminorms $|\cdot|_{f_i}$.
By default we assume $M(X)$ to be endowed with this topology; thus it is a locally convex topological
vector space (by definition).

Let us consider the following vector subspaces of $M(X)$:
\begin{itemize}
\item $M_\sigma(X)$ consisting of all measures (or rather of their restrictions to $\B_\omega(X)$); 
\item $M_\tau(X)$ consisting of all $\tau$-additive measures; 
\item $M_\rho(X)$ consisting of all Radon measures;
\item $M_\delta(X)$ consisting of all finite linear combinations of Dirac measures.
\end{itemize} 
For each symbol $\tt x\in\{\ ,\sigma,\tau,\rho,\delta\}$ and each $r>0$ let 
\begin{align*}
M^+_{\tt x}(X)&=\{\mu\in M_{\tt x}(X)\mid\mu\ge 0\},\\
M^0_{\tt x}(X)&=\{\mu\in M_{\tt x}(X)\mid \mu(X)=0\},\\
M^r_{\tt x}(X)&=\{\mu\in M_{\tt x}(X)\mid\,||\mu||\le r\},\\
\partial M^r_{\tt x}(X)&=\{\mu\in M_{\tt x}(X)\mid\,||\mu||=r\},\\
PM_{\tt x}(X)&=\{\mu\in M_{\tt x}(X)\mid\mu\ge 0,\,||\mu||=1\}.
\end{align*}

If $X$ is a uniform space, then $M_\sigma(X)$ can be re-topologized using the $U_b(X)$-weak topology, 
given by the family of seminorms $\{|\cdot|_f\mid f\in U_b(X)\}$.
If $X$ is a metric space, $M_\sigma(X)$ can also be re-topologized using the $\Lip_b(X)$-weak topology, given by 
the family of seminorms $\{|\cdot|_f\mid f\in\Lip_b(X)\}$.
Obviously, the $\Lip_b(X)$-weak topology on $M_\sigma(X)$ is weaker than the $U_b(X)$-weak topology,
which is in turn weaker than the $C_b(X)$-weak topology.

The seminorms $|\mu|_f$, $f\in C_b(X)$, also determine a uniform structure on $M(X)$, and the seminorms
$|\mu|_f$, where $f$ ranges over $U_b(X)$ or $\Lip_b(X)$, also determine uniform structures on $M_\sigma(X)$.
These are of course the additive uniformities of the corresponding vector space topologies.
However, these uniform structures restricted to $PM_\rho(X)$ are not determined by the corresponding 
topologies of $PM_\rho(X)$.
Indeed, all three weak* topologies, as well as a few further important topologies on $M_\rho(X)$ will turn 
out to coincide (all of them!) on $M^+_\rho(X)$, and in particular on $PM_\rho(X)$ 
(see Theorems \ref{top=unif}(a,b), \ref{KR-metric}(b), \ref{unif-measures}(c), \ref{free-unif}(c) and 
Lemma \ref{opencone}(a) below); but this is not the case for the corresponding additive uniformities 
(see Examples \ref{merging-example} and \ref{unif-comparison} below). 
Thus in order to understand the additive uniformities of $PM_\rho(X)$ we need to understand the corresponding 
topologies not only on $PM_\rho(X)$ but also on the balls $M_\rho^r(X)$.

\subsection{Closed and compact subspaces}
General properties of the weak* topology yield

\begin{lemma}\label{weak*}
(a) The topology of $M(X)$ is weaker than the topology induced by $||\cdot||$. 

(b) Each $M^r(X)$ is compact.

(c) If $X$ is compact and metrizable, then so is each $M^r(X)$.
\end{lemma}

\begin{proof}[Proof. (a)]
This is a consequence of Lemma \ref{norm}(a) along with the general fact that the weak* topology of $V^*$
is weaker than its original topology (induced by the norm).

Alternatively, here is a direct proof.
We have $|\mu-\nu|_f\le||f||\cdot|(\mu-\nu)(X)|\le||f||\cdot|\mu-\nu|(X)=||f||\cdot||\mu-\nu||$.
Hence the $\nu$-centered $\eps$-ball of the seminorm $|\cdot|_f$ contains the $\nu$-centered 
$\frac{\eps}{||f||}$-ball of the norm $||\cdot||$. 
\end{proof}

\begin{proof}[(b)] See \cite{DuS}*{V.4.3}. See also \cite{KF}*{Theorem 20.6 (English transl.)}
for a constructive proof under the hypothesis that $C_b(X)$ is separable (which is satisfied
when $X$ is compact and metrizable).
\end{proof}

\begin{proof}[(c)]
See \cite{KF}*{Theorem 20.5 (English transl.)} or \cite{DuS}*{V.5.1}.

Alternatively, here is a direct proof.
Since $X$ is compact, $C_b(X)$ is separable.
Let $S$ be a countable dense subset of $C_b(X)$.
Then for each $f\in C_b$ and $\eps>0$ there exists a $g\in S$ such that $||f-g||<\frac{\eps}{4r}$.
Hence $|\mu-\nu|_{|f-g|}\le\frac\eps2$ for each $\mu,\nu\in M^r(X)$.
Therefore the $\nu$-centered $\eps$-ball of the seminorm $|\cdot|_f$ in $M^r(X)$ contains
the $\nu$-centered $\frac\eps2$-ball of the seminorm $|\cdot|_g$ in $M^r(X)$.
Thus the countable family of pseudometrics $d_g(x,y)\bydef |x-y|_g$, $g\in S$, determines the topology of $M^r(X)$.
Let $Y_g$ be the copy of $M^r(X)$ endowed with the pseudometric $d_g$.
Then $\prod_{g\in S}Y_g$ is pseudometrizable, moreover its diagonal is metrizable and homeomorphic to $M^r(X)$.
\end{proof}

\begin{example} \label{convergence-example}
Let $\N^+$ be the one-point compactification of the countable discrete space $\N$.

(a) (compare \cite{Bog}*{8.1.4}) Let $i\:\N\to\N^+$ be the inclusion map.
Then the restriction map $i^!\:M^+_\delta(\N^+)\to M^+_\delta(\N)$ is not continuous.
Indeed, $\delta_n\to\delta_\infty$ in $M^+(\N^+)$, but $i^!(\delta_n)=\delta_n\not\to 0=i^!(\delta_\infty)$ in 
$M^+(\N)$.

(b) $\partial M^1_\delta(\N^+)$ is not closed in $M^1_\delta(\N^+)$.
Indeed, let $\mu_n=(\delta_n-\delta_\infty)/2$.
Then $||\mu_n||=1$ for each $n$.
On the other hand, $|\mu_n|_f=|f(n)-f(\infty)|/2\to 0$ for each $f\in C_b(\N^+)$.
Hence $\mu_n\to 0$ in $M(\N^+)$.
\end{example}

\begin{lemma}\label{closed subspaces}
(a) $PM(X)$ is closed in $M^+(X)$.

(b) If $X$ is metrizable, $M^+(X)$ and $M^r(X)$ are closed in $M(X)$.

(c) If $X$ is metrizable and $U\subset X$ is open, for each $r>0$, the sets 
$V\bydef \{\mu\in M(X)\mid \mu(U)\ge 0\}$ and $W\bydef \{\mu\in M(X)\mid\ |\mu|(U)\le r\}$ are closed in $M(X)$.
\end{lemma}

Part (b) is not entirely obvious, as $\mu_\alpha\to\nu$ does not imply either $(\mu_\alpha)_-\to\nu_-$ 
or $||\mu_\alpha||\to||\nu_\alpha||$ by Example \ref{convergence-example}(b).

\begin{proof}[Proof. (a)] For $\mu,\nu\in M^+(X)$ we have 
$|\mu-\nu|_1=|\mu(X)-\nu(X)|=\big|\,||\mu||-||\nu||\,\big|$.
If $\nu\notin PM(X)$, then $\eps\bydef \big|\,||\nu||-1\big|>0$, and hence the $\nu$-centered open $\eps$-ball of 
the seminorm $|\cdot|_1$ is disjoint from $PM(X)$.
Thus $\nu$ is not in the closure of $PM(X)$.
\end{proof}

\begin{proof}[(b)] This is a special case of (c).

Alternatively, $M^r(X)$ is closed because it is compact; or because convex sets are weak* closed
if and only if they are strongly closed (see \cite{DuS}*{V.3.13}).
\end{proof}

\begin{proof}[(c)] Suppose that $\nu\in M(X)\but V$.
Then $\eps\bydef -\nu_-(U)>0$.
By Lemma \ref{estimate}(b) there exists a continuous $f\:X\to [-1,1]$ such that $|\mu-\nu|_f<\frac\eps2$ 
implies $\mu_-(U)<\nu_-(U)+\eps=0$.
However, if $\nu$ lies in the closure of $V$, there exists a $\mu\in V$ 
satisfying $|\mu-\nu|_f<\frac\eps2$.
Then $\mu_-(U)<0$, which is a contradiction.

Suppose that $\nu\in M(X)\but W$.
Then $\eps\bydef |\nu|(U)-r>0$.
By Lemma \ref{estimate}(a) there exists a continuous $f\:X\to [-1,1]$ such that $|\mu-\nu|_f<\frac\eps2$ 
implies $|\mu|(U)>|\nu|(U)-\eps=r$.
However, if $\nu$ lies in the closure of $W$, there exists a $\mu\in W$ 
satisfying $|\mu-\nu|_f<\frac\eps2$.
Then $|\mu|(U)>r$, which is a contradiction.
\end{proof}

\begin{remark} \label{closed subspaces2}
Clearly, the proof of (a) works as well for the $U_b(X)$-weak and $\Lip_b(X)$-weak topologies.
The proofs of (b) and (c) work for these topologies if each $M^*(X)$ is replaced by $M^*_\sigma(X)$.
\end{remark}

\begin{corollary} \label{compact} (a) $PM(X)$ is compact.

(b) If $X$ is compact, then $PM_\rho(X)$ and $PM_\tau(X)$ are compact.
\end{corollary}

\begin{proof}[Proof. (a)]
By Lemma \ref{closed subspaces}(a,b) $PM(X)$ is closed in $M(X)$ and therefore also in $M^1(X)$, 
which is compact by Lemma \ref{weak*}(b).
\end{proof}

\begin{proof}[(b)] By Lemmas \ref{extension}(4), \ref{Radon}(a) and \ref{regularity},
$PM_\rho(X)=PM_\tau(X)=PM(X)$.
\end{proof}

\subsection{Dirac measures}

\begin{lemma} {\rm (see \cite{Par}*{Lemma 6.1}, \cite{Bog}*{8.9.2})} \label{closed emb}
Let $X$ be a metrizable space.
The assignment $x\mapsto\delta_x$ determines an embedding $X$ onto a closed subset of $PM_\tau(X)$.
\end{lemma}

Let us note that $PM_\tau(X)$ is closed in $M_\tau(X)$ by Lemma \ref{closed subspaces}(a,b).

\begin{proof}
For each $f\in C_b(X)$ we have $\int_X f\,d\delta_x=f(x)$.
Hence the topology on $X$ induced from $PM_\tau(X)$ is given by the family of semi-metrics 
$d_f(x,y)=|\delta_x-\delta_y|_f=|f(x)-f(y)|$, where $f\in C_b(X)$.
The $y$-centered $\eps$-ball $\{x\in X\mid\,|f(x)-f(y)|<\eps\}$ of $d_f$ contains the $y$-centered 
$\delta$-ball $\{x\in X\mid d(x,y)<\delta\}$ of any metric $d$ that metrizes $X$, for some $\delta>0$.
Conversely, the $y$-centered $\eps$-ball of $d$ contains the $y$-centered $\eps$-ball of $d_{f_y}$, where 
$f_y\in C_b(X)$ is defined by $f_y(x)=\min\big(d(x,y),1\big)$.
Thus $d$ and the family $d_f$ determine the same neighborhoods of points.

Suppose that some $\nu\in PM_\tau(X)$ is contained in the closure of $X$ but not in $X$.
Since $\nu$ is $\tau$-additive, by Lemma \ref{support} it has support, which then contains 
at least two distinct points $z_1$, $z_2$.
Let $U_1$ and $U_2$ be disjoint open neighborhoods of $z_1$ and $z_2$.
By the definition of support we have $\eps\bydef \min\big(\nu(U_1),\nu(U_2)\big)>0$.
By Lemma \ref{estimate}(a) there exists a continuous $f_i\:X\to [-1,1]$ such that 
$|\mu-\nu|_{f_i}<\frac\eps2$ implies $|\mu|(U_i)>\nu(U_i)-\eps\ge 0$.
On the other hand, since $\nu$ lies in the closure of $X$, there exists an $x\in X$ such that 
$|\delta_x-\nu|_{f_i}<\frac\eps2$ for $i=1,2$.
Then we get $|\delta_x|(U_i)>0$ for $i=1,2$, that is, $x\in U_1$ and $x\in U_2$, which is a contradiction.
\end{proof}

\begin{example} \label{Stone-Cech}
(a) The closure of the image of $X$ in $PM(X)$ is nothing but the Stone--\v Cech compactification%
\footnote{$\beta X$ is the unique (up to homeomorphism keeping $X$ fixed) compact space containing $X$
as a dense subspace and such that every (continuous) map from $X$ to a compact space $K$ extends over $\beta X$.}
$\beta X$ of $X$.

Indeed, a standard construction of $\beta X$ is by embedding $X$ in the $C_X$-indexed product of copies of 
$[-1,1]$, where $C_X$ is the set of all continuous maps $X\to [-1,1]$, by means of 
$x\mapsto\big(f(x)\big)_{f\in C_X}$ and then taking the closure.

Now $C_X$ is nothing but the unit ball of $C_b(X)$.
Since a linear function on $C_b(X)$ is determined by its restriction to $C_X$, the unit ball $C'_X$ of 
$C_b^*(X)$ with the weak topology is the same as the space of all continuous maps $C_X\to [-1,1]$ with 
the topology of pointwise convergence.
Since $C'_X$ is compact, it is closed in the larger space of all (not necessarily continuous) maps 
$C_X\to [-1,1]$ with the topology of pointwise convergence.
But the latter space is the same as $\prod_{C_X}[-1,1]$.

(b) Let us mention a well-known but important precaution regarding the topology of non-metrizable spaces.
Since $\beta\N$ is compact, every sequence in $\N$, being a net, has a subnet converging to 
a point in $\beta\N\but\N$.
But this subnet is never a subsequence!

Indeed, suppose that a sequence $x_n\in\N$ converges to an $x\in\beta\N\but\N$.
We may assume that the $x_n$ are pairwise distinct.  
Let us choose some non-convergent sequence $y_n\in[-1,1]$, for instance, $y_n=-1$ or $1$
according to whether $n$ is odd or even.
The map $f\:\N\to [-1,1]$, defined by $f(x_n)=y_n$ and by $f(z)=0$ for all other $z\in\N$, is continuous.
So by the universal property of the Stone--\v Cech compactification%
\footnote{Explicitly, in the notation of (a) we have $f\in C_\N$. 
Define $\bar f$ to be the restriction to $\beta\N$ of the projection of $\prod_{C_\N}[-1,1]$ to the $f$th factor.
It extends $f$ in the sense that $\bar f(\delta_n)=f(n)$ for all $n\in\N$.}
it has a continuous extension $\bar f\:\beta\N\to [-1,1]$.
Then $y_n\to\bar f(x)$, which is a contradiction.

(c) If $X$ is a uniform space, then the closure of the image of $X$ in $U_b^*(X)$ with 
the $U_b(X)$-weak topology is the Samuel compactification of $X$ \cite{Pac2}*{6.37(2)}.
The Samuel compactification of $X$ is the unique (up to homeomorphism keeping $X$ fixed) compact space 
$\gamma X$ containing $X$ as a dense subset and such that the inclusion $X\to\gamma X$ is 
a uniformly continuous homeomorphism (but of course not a uniform embedding, unless $X$ is precompact) 
and every uniformly continuous map of $X$ into a compact space $K$ extends to a continuous map $\gamma X\to K$
(see \cite{I3}*{II.32}).
In particular, for $X=\N$ this reduces to the universal property of the Stone--\v Cech compactification, 
so $\gamma\N=\beta\N$.
\end{example}

\subsection{Molecular measures}

\begin{lemma} \label{dirac} {\rm (e.g.\ \cite{Bog}*{8.1.6(i)})} 

(a) $M_\delta(X)$ is dense in $M(X)$.

(b) $M^+_\delta(X)$ is dense in $M^+(X)$.

(c) $PM_\delta(X)$ is dense in $PM(X)$.

(d) $M^0_\delta(X)$ is dense in $M^0(X)$.

(e) $\partial M^r_\delta(X)$ is dense in $\partial M^r(X)$ for each $r>0$.

(f) $M^0_\delta(X)\cap\partial M^r_\delta(X)$ is dense in $M^0(X)\cap\partial M^r(X)$ for each $r>0$.
\end{lemma}

Let us note that e.g.\ (a) yields a representation of any regular charge as the $C_b(X)$-weak limit of
a net (but not necessarily a sequence) of Dirac measures.
Clearly, Lemma \ref{dirac} implies its own versions for $U_b(X)$-weak and $\Lip_b(X)$-weak topologies.

\begin{proof}[Proof. (a)--(d)]
Given a $\nu\in M(X)$, a finite collection $f_1,\dots,f_n\in C_b(X)$ and an $\eps>0$, let us find 
a linear combination $\mu$ of Dirac measures such that $|\mu-\nu|_{f_i}<\eps$ for each $i$.
Let $N=\max(||f_1||,\dots,||f_n||)$ and let us pick points $-N=t_0<\dots<t_m=N$ such that 
each $t_{i+1}-t_i<\frac\eps{2||\nu||}$.
Let $A_{ij}=f_i^{-1}\big([t_j,t_{j+1})\big)$ and let $g_i=\sum_{j=1}^m t_j\chi_{A_{ij}}$.
Then $||f_i-g_i||<\frac\eps{2||\nu||}$. 
Let $B_{i_1,\dots,i_n}=A_{1i_1}\cap\dots\cap A_{ni_n}$.
The sets $B_J$, $J=(i_1,\dots,i_n)$, are pairwise disjoint, and their union is $X$.
Let $S=\{J\mid B_J\ne\emptyset\}$.
For each $J\in S$ let us pick some $x_J\in B_J$ and set $\mu=\sum_{J\in S}\nu(B_J)\delta_{x_J}$.
Then $\mu(X)=\sum_{J\in S}\nu(B_J)=\nu(X)$, and in particular $\nu\in M^0(X)$ implies $\mu\in M^0(X)$.
Also, $||\mu||=\sum_{J\in S}|\nu(B_j)|\le\sum_{J\in S}|\nu|(B_J)=||\nu||$.
Let us note that if $\nu\ge 0$, then all the coefficients $\nu(B_J)$ of the linear combination are nonnegative
and $||\mu||=||\nu||$.
In particular, $\nu\in PM(X)$ implies $\mu\in PM(X)$.
Since each $A_{ij}$ is a union of $B_J$'s (which are pairwise disjoint), we have $\mu(A_{ij})=\nu(A_{ij})$.
Hence $\int_X g_i\,d\mu=\int_X g_i\,d\nu$ for each $i$.
On the other hand, $\big|\int_X f_i-g_i\,d\mu\big|\le||\mu||\cdot||f_i-g_i||\le||\nu||\frac\eps{2||\nu||}
=\frac\eps2$.
Similarly, $\big|\int_X f_i-g_i\,d\nu\big|\le\frac\eps2$.
Hence $|\mu-\nu|_{f_i}=\big|\int_X f_i\,d(\mu-\nu)\big|\le\eps$.
\end{proof}

\begin{proof}[(e,f)] Let us slightly amend the previous construction.
Let $S'$ be the set of all multi-indices $J$ such that $B_J$ contains at least two distinct points,
denoted $x^+_J$ and $x^-_J$.
Let $\mu=\sum_{J\in S\but S'}\nu(B_J)\delta_{x_J}+\sum_{J\in S'}\nu^+(B_J)\delta_{x^+_J}+\nu^-(B_J)\delta_{x^-_J}$.
Then $\mu_\pm(B_J)=\nu_\pm(B_J)$ for each $J\in S$.
Hence $||\mu||=||\nu||$.
The rest of the proof proceeds like before.
\end{proof}

\begin{remark} Here is an alternative proof of Lemma \ref{dirac}(a).
Since $M(X)$ is endowed with the $C_b(X)$-weak topology, the natural embedding $E\:C_b(X)\to M^*(X)$, 
given by $E(f)(\Phi)=\Phi(f)$, is an isomorphism (see \cite{Scha}*{IV.1.2}).
Given an $\omega\in M^*(X)$ that vanishes on $M_\delta(X)$, we have $\omega=E(f)$ for some $f\in C_b(X)$
and $0=\omega(\delta_x)=E(f)(\delta_x)=\delta_x(f)=f(x)$ for each $x\in X$.
Hence $f=0$ and consequently $\omega$ vanishes on the entire $M(X)$.
Therefore $M_\delta(X)$ is dense in $M(X)$ (see \cite{Pac2}*{P.6}).
\end{remark}

\subsection{Sequential completeness}

\begin{theorem} \label{convergence}
Let $X$ be a metrizable space and suppose that a sequence $\mu_n$ converges to a $\mu$ in $M(X)$.

(a) If each $\mu_n$ is countably additive, then so is $\mu$.

(b) If each $\mu_n$ is $\tau$-additive, then so is $\mu$.
\end{theorem}

Let us note that both (a) and (b) fail for nets in place of sequences by Lemma \ref{dirac}
and Examples \ref{charge-example} and \ref{tau-example}.

Part (a) is proved in {\rm \cite{AlAD}*{Theorem 19.3} (see also \cite{Va}*{Theorem II.19})},
part (b) is found in \cite{Va}*{Theorem II.20}.
  
\begin{proof}[Proof. (a)] Let $Z_k$ be a decreasing sequence of closed sets with 
$\bigcap_{k=1}^\infty Z_k=\emptyset$.
By Lemma \ref{extension}(5), it suffices to show that $\mu(Z_k)\to 0$ as $k\to\infty$.
Let us fix some metric on $X$ and let $U_k$ be the open $\frac1k$-neighborhood of $Z_k$.
It is not hard to see that $\bigcap_{k=1}^\infty U_k=\emptyset$.
Indeed, given an $x\in X$, there exists a $k$ such that $d(x,Z_k)>0$.
If $l\ge\frac1{d(x,Z_k)}$, then $x$ does not lie in the open $\frac1l$-neighborhood of $Z_k$,
which in turn contains $U_{\max(k,l)}$.

Let $C_1(X)$ denote $\{f\in C_b(X)\mid\,||f||\le 1\}$ endowed with the sup metric. 
Since each $C_1(X\but U_k)$ is complete and the restriction maps $\dots\to C_1(X\but U_2)\to C_1(X\but U_1)$ are 
uniformly continuous, their inverse limit $\Lambda$ is also complete.
The natural uniformly continuous map $\lambda\:C_1(X)\to\Lambda$ is clearly a bijection; in particular, 
$\Lambda$ is nonempty.
For any measure $\nu$ on $X$ it is not hard to see that the map $e_\nu\:\Lambda\to\R$ given by 
$\lambda(f)\mapsto \int_X f\,d\nu$ is uniformly continuous.
Indeed, given an $\eps>0$, by the countable additivity of $\nu$ there exists a $k$ such that 
$|\nu|(U_k)<\frac\eps4$.
Then there exists a $\delta>0$ such that $\delta$-close points in $\Lambda$ have $\frac\eps{2||\nu||}$-close 
images in $C_1(X\but U_k)$.
Hence whenever $\lambda(f)$ and $\lambda(g)$ are $\delta$-close, 
$\big|\int_{X\but U_k}f-g\,d\nu\big|<\frac\eps{2||\nu||}|\nu|(X\but U_k)\le\frac\eps2$.
Also $\big|\int_{U_k}f-g\,d\nu\big|\le 2|\nu|(U_k)<\frac\eps2$.
Thus $\big|\int_X f\,d\nu-\int_X g\,d\nu\big|<\eps$.

Now let us fix a $\delta>0$.
By the above, $\{\lambda(f)\in\Lambda\mid\,\big|\int_X f\,d\nu\big|\le\delta\}=e_\nu^{-1}([-\delta,\delta])$
is a closed subset of $\Lambda$ for each measure $\nu$ on $X$.
Hence 
\[Q_n\bydef \{\lambda(f)\in\Lambda\mid\,\left|\int_X f\,d(\mu_l-\mu_m)\right|\le\delta
\text{ for all }l,m\ge n\},\]
being an intersection of sets of the form $e_{\mu_l-\mu_m}^{-1}([\delta,\delta])$, is also a closed subset 
of $\Lambda$.
Since the sequence $\mu_n$ is Cauchy with respect to the family of seminorms $|\cdot|_f$, $f\in C_1(X)$,
we have $\bigcup_{n=1}^\infty Q_n=\Lambda$.
Hence by the Baire category theorem some $Q_n$ has a nonempty interior.
This $Q_n$ contains the open $\gamma$-neighborhood of $\lambda(f)$ for some $\gamma>0$ and some $f\in C_1(X)$.
By Lemma \ref{A.11} there exists a $k$ and a $\beta>0$ such that $\beta$-close points in $C_1(X\but U_k)$
have $\gamma$-close preimages in $\Lambda$.
Thus we have shown that for each $\delta>0$ there exist $k,n\in\N$, $\beta>0$ and $f\in C_1(X)$ such that every 
$g\in C_1(X)$ satisfying $|g(x)-f(x)|<\beta$ for all $x\in X\but U_k$ also satisfies
$\big|\int_X g\,d(\mu_l-\mu_m)\big|\le\delta$ for all $l,m\ge n$.

Let us temporarily fix $l$ and $m$ and write $\nu=\mu_l-\mu_m$.
By Lemma \ref{regularity1} $U_k$ contains a closed set $Z$ such that $|\nu|(U_k\but Z)<\delta$.
By Lemma \ref{norm}(a) there exists a continuous function $h\:Z\to[-1,1]$ such that
$\big|\int_Z h\,d\nu\big|>|\nu|(Z)-\delta$.
It is not hard to construct continuous functions $f_0,f_1\:X\to [-1,1]$ such that $f_0$ sends $Z$ to $0$,
$f_1$ extends $h$ and each $f_i$ equals $f$ on $X\but U_k$;%
\footnote{Namely, since $Z$ is closed, $h$ extends to a continuous function $\bar h\:X\to[-1,1]$. 
Let $\phi\:X\to [0,1]$ be a continuous function sending $X\but U_k$ to $1$ and $Z$ to $0$.
Let $h_0=(1-\phi)\bar h$ and $f_0=\phi f$.
Then $h_0$ is a continuous function $X\to [-1,1]$ equal to $h$ on $Z$ and sending $X\but U_k$ to $0$, and
$f_0$ is a continuous function $X\to [-1,1]$ equal to $f$ on $X\but U_k$ and sending $Z$ to $0$.
Let $f_1=f_0+h_0$.
Since $|f_0|\le\phi$ and $|h_0|\le 1-\phi$, we have $||f_1||\le 1$.}
as a byproduct of the construction we also obtain that $h_0\bydef f_1-f_0$ satisfies $||h_0||\le 1$.
Then $\big|\int_X f_i\,d\nu\big|\le\delta$ for each $i$.
Hence $\big|\int_X h_0\,d\nu\big|\le 2\delta$.
On the other hand, $\int_{X\but U_k} h_0\,d\nu=0$, 
$\big|\int_{U_k\but Z} h_0\,d\nu\big|\le|\nu|(U_k\but Z)<\delta$
and $\big|\int_Z h_0\,d\nu\big|>|\nu|(Z)-\delta$.
Hence $|\nu|(Z)<4\delta$.
Consequently $|\nu|(U_k)<5\delta$.
Thus we have shown that for each $\delta>0$ there exist $k,n\in\N$ such that $|\mu_l-\mu_m|(U_k)<5\delta$ 
for all $l,m\ge n$.

Since $\mu_n$ is countably additive, there exists a $K\ge k$ such that $|\mu_n|(U_K)<\delta$.
Then for each $m\ge n$ we have $|\mu_m|(U_K)\le|\mu_n|(U_K)+|\mu_m-\mu_n|(U_K)<6\delta$.
Thus we have shown that for each $\delta>0$ there exist $K,n\in\N$ such that $|\mu_m|(U_K)<6\delta$ 
for all $m\ge n$.%
\footnote{The hypothesis $m\ge n$ is actually redundant here, for we may assume that $K$ is chosen so that
$|\mu_m|(U_K)<6\delta$ for all $m<n$ (using the countable additivity of $\mu_1,\dots,\mu_{n-1}$).
This shows that the sequence $\mu_n$ has no ``eluding load'' in the terminology of A. D. Alexandrov \cite{AlAD}.}
Then by Lemma \ref{closed subspaces}(c) $|\mu|(U_K)\le 6\delta$.
Thus $|\mu|(U_i)\to 0$ as $i\to\infty$.
In particular, $|\mu|(Z_i)\to 0$.
Hence also $\mu(Z_i)\to 0$.
\end{proof}

\begin{proof}[(b)] Lemma \ref{tau} yields separable closed sets $Z_n\subset X$ such that 
$|\mu_n|(X\but Z_n)=0$ for each $n$.
Let $Z$ be the closure of their union.
Then $Z$ is separable and each $|\mu_n|(X\but Z)=0$.
Hence by Lemma \ref{closed subspaces}(c) $|\mu|(X\but Z)=0$.
On the other hand, by (a) $\mu$ is countably additive.
Hence by Lemma \ref{tau} $\mu$ is $\tau$-additive.
\end{proof}

\begin{corollary} \label{weak*complete}
If $X$ is a metrizable space, then $M(X)$, $M_\sigma(X)$ and $M_\tau(X)$ are sequentially complete.
\end{corollary}

Let us note that $X$ is not assumed to be completely metrizable.
In fact, it should be noted that the uniformity of $M(X)$ is determined by the topology of $X$.

\begin{proof} By Theorem \ref{convergence} $M_\sigma(X)$ and $M_\tau(X)$ are sequentially closed in $M(X)$,
so it remains to show that $M(X)$ is sequentially complete.
Let $\mu_n$ be a Cauchy sequence of regular charges on $X$ (with respect to the weak* uniformity).
Then the bounded linear operators $\Phi_{\mu_n}$ on $C_b(X)$ are pointwise Cauchy and therefore pointwise
converge to some linear operator $\Phi$ (since $\R$ is complete).
Since $C_b(X)$ is complete, by the Banach--Steinhaus Theorem \cite{Scha}*{III.4.2} $\Phi$ is bounded.
Hence by Theorem \ref{isomorphism}(a) $\Phi=\Phi_\mu$ for some regular charge $\mu$.
By the definition of weak* topology, $\mu_n\to\mu$ in $M(X)$.
\end{proof}

\subsection{Subbasic neighborhoods}

\begin{lemma} \label{subbase} \cite{Va}*{Remark III to Theorem II.2}, \cite{DS}*{1.1}, \cite{Bog}*{8.2.1}
If $X$ is metrizable, a subbase of neighborhoods of a $\nu\in M^+(X)$ is given by each
of the following three families of sets.

(a) The sets
\begin{align*}
B^+_{U,\eps}(\nu)&\bydef \{\mu\in M^+(X)\mid\,\nu(U)-\mu(U)<\eps\}\\
B^-_{Z,\eps}(\nu)&\bydef \{\mu\in M^+(X)\mid\,\mu(Z)-\nu(Z)<\eps\},\,
\end{align*}
for all $\eps>0$, all closed $Z\subset X$ and all open $U\subset X$.

(b) The sets $B^+_{U,\eps}(\nu)$ and $B^+_{X,\eps}(\nu)\cap B^-_{X,\eps}(\nu)$ for all $\eps>0$ and 
all open $U\subset X$.

(c) The sets $B^-_{Z,\eps}(\nu)$ and $B^+_{X,\eps}(\nu)\cap B^-_{X,\eps}(\nu)$ for all $\eps>0$ and 
all closed $Z\subset X$.
\end{lemma}

In other words, part (a) is saying that a subset of $M^+(X)$ is a neighborhood of $\nu$ if and only if it 
contains a finite intersection of sets of the form $B^+_{U,\eps}(\nu)$ and $B^-_{Z,\eps}(\nu)$.

Let us note that $B^+_{X,\eps}(\nu)\cap B^-_{X,\eps}(\nu)=\{\mu\in M^+(X)\mid\,|\mu(X)-\nu(X)|<\eps\}$
is nothing but the $\nu$-centered $\eps$-ball of the seminorm $|\cdot|_1$.

\begin{proof}[Proof. (a)] Given an open $U\subset X$ and an $\eps>0$, by \ref{regularity1} there exists 
a closed $Z\subset U$ such that $\nu(U\but Z)<\frac\eps2$.
Let $f\:X\to [0,1]$ be a continuous function sending $Z$ to $1$ and $X\but U$ to $0$.
Then $\nu(U)-\mu(U)\le\nu(Z)-\mu(U)+\frac\eps2\le\int_X f\,d\nu-\int_X f\,d\mu+\frac\eps2$.
Hence $B^+_{U,\eps}(\nu)$ contains the $\nu$-centered $\frac\eps2$-ball
$\{\mu\in M^+(X)\mid\,\big|\int_X f\,d\nu-\int_X f\,d\mu\big|\le\frac\eps2\}$ of $|\cdot|_f$.

Given a closed $Z\subset X$ and an $\eps>0$, similarly (by considering the complements) there exists 
an open $U\supset Z$ such that $\nu(U\but Z)<\frac\eps2$.
Let $f\:X\to [0,1]$ be a continuous function sending $Z$ to $1$ and $X\but U$ to $0$.
Then $\mu(Z)-\nu(Z)\le\mu(Z)-\nu(U)+\frac\eps2\le\int_X f\,d\mu-\int_X f\,d\nu+\frac\eps2$.
Hence $B^-_{Z,\eps}(\nu)$ contains the same $\nu$-centered $\frac\eps2$-ball of $|\cdot|_f$.

Conversely, let $f\in C_b(X)$ and $\eps>0$ be given.
First suppose that $\delta>0$ and $a,b\in\R$ are such that $\nu\big(f^{-1}(a)\big)=\nu\big(f^{-1}(b)\big)=0$.
Let $Z=f^{-1}\big([a,b]\big)$ and $U=f^{-1}\big((a,b)\big)$, and also let $A=f^{-1}\big([a,b)\big)$. 
If $\mu\in B^+_{U,\delta}(\nu)\cap B^-_{Z,\delta}(\nu)$, then $\mu(A)\ge\mu(U)>\nu(U)-\delta=\nu(A)-\delta$
and similarly $\mu(A)\le\mu(Z)<\nu(Z)+\delta=\nu(A)+\delta$.
Hence $|\mu(A)-\nu(A)|<\delta$.
Thus $B_{A,\delta}(\nu)\bydef \{\mu\in M^+(X)\mid\,|\mu(A)-\nu(A)|<\delta\}$ contains 
$B^+_{U,\delta}(\nu)\cap B^-_{Z,\delta}(\nu)$.
So it suffices to show that every neighborhood of $\nu$ in $M^+(X)$ contains an intersection of sets of 
the form $B_{A,\delta}(\nu)$.

Let us pick points $-||f||=t_0<\dots<t_n=||f||$ such that each $\nu\big(f^{-1}(t_i)\big)=0$ and
each $t_{i+1}-t_i<\frac\eps4$.
Let $A_i=f^{-1}\big([t_i,t_{i+1})\big)$.
Then $\big|\big|f-\sum_{i=1}^m t_i\chi_{A_i}\big|\big|<\frac\eps4$. 
Hence \[\left|\int\limits_X f\,d(\mu-\nu)\right|
\le\left|\int\limits_X \sum_{i=1}^m t_i\chi_{A_i}\,d(\mu-\nu)\right|+\frac\eps2
\le\sum_{i=0}^m t_i|(\mu-\nu)(A_i)|+\frac\eps2.\]
Let $\delta=\frac{\eps}{2(t_0+\dots+t_n)}$.
Assuming that $\mu\in\bigcap_{i=1}^nB_{A_i,\delta}(\nu)$, we have $|\mu(A_i)-\nu(A_i)|<\delta$ for each $i$.
Hence $\sum_{i=0}^m t_i|(\mu-\nu)(A_i)|<\frac\eps2$. 
Thus $\big|\int_X f\,d(\mu-\nu)\big|<\eps$.
In other words, $\mu$ belongs to the $\nu$-centered $\frac\eps2$-ball of $|\cdot|_f$.
\end{proof}

\begin{proof}[(b)] It suffices to show that a subbase of neighborhoods of a $\nu\in M^+(X)$ is given by 
the sets $B^+_{U,\eps}(\nu)$ and $B^-_{X,\eps}(\nu)$ for all $\eps>0$ and all open $U\subset X$.
Indeed, let a closed $Z\subsetneqq X$ and an $\eps>0$ be given.
Let $U=X\but Z$.
Then \[\mu(Z)-\nu(Z)=\big(\mu(X)-\mu(U)\big)-\big(\nu(X)-\nu(U)\big)=
\big(\mu(X)-\nu(X)\big)+\big(\nu(U)-\mu(U)\big).\]
Hence $B^-_{X,\frac\eps2}(\nu)\cap B^+_{U,\frac\eps2}(\nu)\subset B^-_{Z,\eps}(\nu)$.  
\end{proof}

\begin{proof}[(c)] Similarly to (b).
\end{proof}

\subsection{Embeddings}
Given a continuous map $\phi\:X\to Y$, the induced map $\phi_*\:M(X)\to M(Y)$ is continuous, since
$|\phi_*\mu|_f=|\mu|_{f\phi}$ for any $f\in C_b(Y)$ and $\mu\in M(X)$.

\begin{lemma} \label{emb-measures} {\rm \cite{DS}*{1.2 and Remark} 
(see also \cite{Bog}*{8.9.1}, \cite{Ban1}*{1.4, 1.5}).}
Let $i\:X\subset Y$ be an inclusion, where $Y$ is metrizable.

(a) $i_*\:M^+(X)\to M^+(Y)$ is an embedding.

(b) If $X$ is closed in $Y$, then $i_*\:M(X)\to M(Y)$ is an embedding.

(c) If $X$ is not closed in $Y$, then $i_*\:M(X)\to M(Y)$ is not an embedding.
\end{lemma}
 
\begin{proof}[Proof. (a)] For each open $U\subset X$ we have $U=V\cap X$ for some open $V\subset Y$.
Given a $\nu\in M^+(X)$, the set $B^+_{U,\eps}(\nu)=\{\mu\in M^+(X)\mid\,\nu(U)-\mu(U)<\eps\}$
coincides with the preimage $\{\mu\in M^+(X)\mid\,(i_*\nu)(V)-(i_*\mu)(V)<\eps\}$ of $B^+_{V,\eps}(i_*\nu)$
in $M^+(X)$.
Similarly, $B^-_{X\but U,\eps}(\nu)$ coincides with the preimage of $B^-_{Y\but V,\eps}(i_*\nu)$ in $M^+(X)$.
\end{proof}

\begin{proof}[(b)]
Let us show that $(i_*)^{-1}$ is continuous.
Indeed, given a $\nu\in M(X)$ and an $f\in C_b(X)$, let $\bar f\in C_b(Y)$ be an extension of $X$
(using that $\big[-||f||,||f||\big]\subset\R$ is an absolute retract).
Then $|\mu-\nu|_f=\big|\int_X\bar f\,d(\mu-\nu)\big|=
\big|\int_Y\bar f\,d(i_*\mu-i_*\nu)\big|=|i_*\mu-i_*\nu|_{\bar f}$.
Hence the preimage of the $\nu$-centered $\eps$-ball of $|\cdot|_f$ under $(i_*)^{-1}$
contains the $i_*\nu$-centered $\eps$-ball of $|\cdot|_{\bar f}$ (actually, they coincide).
\end{proof}

\begin{proof}[(c)] There exists a sequence of pairwise distinct points $x_n\in X$ converging to a $y\in Y\but X$.
Let $\mu_n=\delta_{x_{2n}}-\delta_{x_{2n-1}}$.
Then $\mu_n\to 0$ in $M(Y)$ (see Example \ref{convergence-example}(b)).
However, $C_b(X)$ contains a $g$ such that $g(x_{2n})=1$ and $g(x_{2n-1})=-1$.
Then $|\mu_n|_g=2$ for each $n$.
Hence $\mu_n\not\to 0$ in $M(X)$.
\end{proof}

\subsection{Open cone lemma}
The {\it open cone} $C^\circ Q$ over a topological space $Q$ consists of an embedded copy of $Q\x (0,\infty)$ 
along with a singleton $\{v\}$ (the ``cone vertex''), whose base of neighborhoods is given by 
the sets $U_\eps\bydef \{v\}\cup Q\x(0,\eps)$.

\begin{lemma} \label{opencone} 
Let $\phi$ be the bijection from $M(X)$ to the open cone $C^\circ\big(\partial M^1(X)\big)$, defined by 
$\phi(0)=v$ and $\phi(\mu)=\big(\frac\mu{||\mu||},||\mu||\big)$ for $\mu\ne 0$.

(a) \cite{DS}*{1.3} $\phi$ restricts to a homeomorphism between $M^+(X)$ and $C^\circ\big(PM(X)\big)$.

(b) $\phi$ restricts to a homeomorphism between each $\partial M^r(X)$ and $\partial M^1(X)\x\{r\}$.

(c) The composition \[\chi\:C^\circ\big(\partial M^1(X)\big)\xr{\phi^{-1}}M(X)\xr\psi M^+(X)\x M^+(X),\]
where $\psi(\mu)=(\mu_+,-\mu_-)$, is an embedding.
In particular, $\psi$ embeds each $\partial M^r(X)$.
\end{lemma}

We will refer to the preimage in $M(X)$ of the topology of $C^\circ\big(\partial M^1(X)\big)$ under
the bijection $\phi$ as {\it the open cone topology} on $M(X)$.
Thus part (c) says that $\psi$ is an embedding with respect to the open cone topology on $M(X)$ and the usual
(i.e., $C_b(X)$-weak) topology on both copies of $M^+(X)$.

The same arguments work to prove the uniform and Lipschitz analogues of (a), (b) and (c), where 
each $M^*(X)$ is replaced by the corresponding $M_\sigma^*(X)$, re-topologized by the $U_b(X)$- or 
$\Lip_b(X)$-weak topology.

\begin{proof}[Proof. (a)]
Clearly $\phi^{-1}$ is continuous on the complement of $\{v\}$.
(The details can be found in the proof of (c).)
Also by \ref{weak*}(a) each weak* neighborhood of $0$ contains a $||\cdot||$-neighborhood of $0$, 
so $\phi^{-1}$ is continuous at $v$. 

It remains to show that the restriction of $\phi$ to $M^+(X)$ is continuous.
For this, it suffices to show that the map $M^+(X)\to[0,\infty)$, $\mu\mapsto||\mu||$, is continuous.
Indeed, for each $\nu\in M^+(X)$ and each $\eps>0$, if $\mu$ belongs to the $\nu$-centered $\eps$-ball 
$\{\mu\in M^+(X)\mid\,|\mu(X)-\nu(X)|<\eps\}$ of $|\cdot|_1$, then it satisfies 
$\big|\,||\mu||-||\nu||\,\big|<\eps$.
\end{proof}

\begin{proof}[(b)] This follows from the fact that for each $r>0$ the map $f_r\:M(X)\to M(X)$, 
$\mu\mapsto r\mu$, is a homeomorphism. 
Indeed, its restriction $\partial M^1(X)\to\partial M^r(X)$ can be identified with the inverse of
$\phi|_{\partial M_r}$.
\end{proof}

\begin{proof}[(c)] It is clear that $\psi$ is injective, and consequently so is $\chi$.
Since $\mu=\mu_+-(-\mu_-)$ and $||\mu||=||\mu_+||+||{-\mu_-}||$, we have
$\chi^{-1}(\mu,\nu)=\big(\frac{\mu-\nu}{||\mu||+||\nu||},||\mu||+||\nu||\big)$.
By the proof of (a) the map $M^+(X)\to[0,\infty)$, $\mu\mapsto||\mu||$, is continuous, so we get
that $\chi^{-1}$ is continuous.

To show that $\chi$ is continuous it suffices to show that its compositions with the projections
$\pi_\pm$ onto the factors of $M_+(X)\x M_+(X)$ are continuous.
By symmetry it suffices to show the continuity of $\pi_+\chi$.

It is easy to see that $\pi_+\chi$ is continuous at the cone vertex.
Indeed, given an $\eps>0$ and an $f\in C_b(X)$, if $||\mu||<\frac\eps{||f||}$, then 
$|\mu_+|_f\le ||\mu_+||\cdot||f||\le||\mu||\cdot||f||<\eps$.

To show that $\pi_+\chi$ is continuous on the complement to the cone vertex we will use Lemma \ref{subbase}.
Let $\nu\in M(X)\but\{0\}$, an $\eps>0$ and an open set $U\subset X$ be given.
We may assume that $\eps<||\nu||$.
By Lemma \ref{estimate}(b) there exists a continuous%
\footnote{If $\nu$ is $\sigma$-additive, then by Lemma \ref{estimate}(c) we may assume that $f$ is Lipschitz.} 
function $f\:X\to [-1,1]$ such that $|\mu-\nu|_f<\frac\eps2$ implies $\mu_+(U)>\nu_+(U)-\eps$.
Suppose that $\big|\,||\mu||-||\nu||\,\big|<\frac{\eps}4$ 
(which already implies that $||\mu||>\frac{3\eps}4$ and in particular $\mu\ne 0$) and 
$\big|\frac\mu{||\mu||}-\frac\nu{||\nu||}\big|_f<\frac\eps{4||\nu||}$, and also that
$\big|\frac\mu{||\mu||}-\frac\nu{||\nu||}\big|_1<\frac\eps{4||\nu||}$.

We have $\big|\frac\mu{||\nu||}-\frac\mu{||\mu||}\big|_f=
\frac{\left|\vphantom{^i}\,||\mu||-||\nu||\,\right|}{||\mu||\cdot||\nu||}|\mu|_f\le
\frac{\eps/4}{||\mu||\cdot||\nu||}||\mu||<\frac\eps{4||\nu||}$.
Hence $\big|\frac\mu{||\nu||}-\frac\nu{||\nu||}\big|_f<\frac\eps{2||\nu||}$.
Therefore $|\mu-\nu|_f<\frac\eps2$.
Similarly $|\mu-\nu|_1<\frac\eps2$.
By our choice of $f$, from $|\mu-\nu|_f<\frac\eps2$ we get $\mu_+(U)>\nu_+(U)-\eps$.
Thus $\mu_+\in B^+_{U,\eps}(\nu_+)$ in the notation of Lemma \ref{subbase}.

Finally, $|\mu(X)-\nu(X)|=|\mu-\nu|_1<\frac\eps2$ and 
$\big|\,|\mu|(X)-|\nu|(X)\,\big|=\big|\,||\mu||-||\nu||\,\big|<\frac{\eps}4$.
Since $\mu_+=\mu+|\mu|$, we obtain $|\mu_+(X)-\nu_+(X)|<\eps$.
Hence $\mu_+\in B^+_{X,\eps}(\nu_+)\cap B^-_{X,\eps}(\nu_+)$ in the notation of Lemma \ref{subbase}.
\end{proof}

\subsection{Metrizability}

\begin{theorem} \label{abstract-metr}
If $X$ is separable metrizable, then so is $M^+_\sigma(X)$ and each $\partial M^r_\sigma(X)$.
\end{theorem}

The assertion on $M^+_\sigma(X)$ is proved, for instance, in \cite{Va}*{Theorem II.13}, \cite{DS}*{1.3}.

The following proof of Theorem \ref{abstract-metr} is far from being explicit. 
An explicit proof of the separability can be obtained from Lemma \ref{dirac}.
An explicit proof of the metrizabililty is given in Theorem \ref{KR-metric} below.

\begin{proof}
Let $i$ be an embedding of $X$ in the Hilbert cube $I^\infty$.
Then by Lemma \ref{emb-measures}(a), $i_*\:PM_\sigma(X)\to PM_\sigma(I^\infty)$ is an embedding.
On the other hand, $PM_\sigma(I^\infty)$ is separable and metrizable by \ref{weak*}(b,c).
Hence so is $PM_\sigma(X)$.
Then by \ref{opencone}(a) $M^+_\sigma(X)$ is also separable and metrizable.
Consequently by \ref{opencone}(b) $\partial M^r_\sigma(X)$ is also separable and metrizable for each $r>0$.
\end{proof}

The second assertion of Theorem \ref{abstract-metr} implies

\begin{corollary} If $X$ is separable metrizable, then so is the open cone topology on $M_\sigma(X)$.
\end{corollary}

\subsection{Comparison of topologies}

\begin{lemma} \label{approximation} \cite{Pac1}*{Lemma 1}, \cite{Pac2}*{5.14, 5.20}.
Let $X$ be a metric space.

(a) $\Lip_b(X)$ is dense in $U_b(X)$.

(b) Given a compact $K\subset X$, an $f\in C_b(X)$ and an $\eps>0$, there exists a $g\in\Lip_b(X)$  
such that $||g||\le ||f||$, $g(x)\ge f(x)$ for all $x\in X$ and $g(x)\le f(x)+\eps$ for all $x\in K$.
\end{lemma}

\begin{proof}[Proof. (a)] Let $f\in U_b(X)$ and $\eps>0$.
Then there exists a $\delta>0$ such that $d(x,y)<\delta$ implies $|f(x)-f(y)|<\eps$.
Let $k=\frac{2||f||}\delta$.
Let us define $g\:X\to\R$ by \[g(x)=\sup\,\{f(y)-kd(x,y)\mid y\in X\}.\]

By considering $y=x$ we get $g(x)\ge f(x)$.
It is also easy to see that $g(x)\le f(x)+\eps$.
Indeed, if $d(x,y)<\delta$, then $f(y)-kd(x,y)\le f(y)\le f(x)+\eps$; and if $d(x,y)\ge\delta$, 
then $kd(x,y)\ge 2||f||\ge |f(x)-f(y)|$ and consequently $f(y)-kd(x,y)\le f(x)\le f(x)+\eps$.

Let us show that $g$ is $k$-Lipschitz.
Given $x,x'\in X$ and a $\gamma>0$, there exists a $y\in X$ such that $g(x)\le f(y)-kd(x,y)+\gamma$.
Then $g(x')\ge f(y)-kd(x',y)$.
Hence $g(x)-g(x')\le kd(x',y)-kd(x,y)+\gamma\le kd(x,x')+\gamma$.
Since $\gamma$ was arbitrary, we get $g(x)-g(x')\le kd(x,x')$.
Hence by symmetry $|g(x)-g(x')|\le kd(x,x')$.
\end{proof}

\begin{proof}[(b)] Since $K$ is compact, $f|_K$ is uniformly continuous, so there exists a $\delta>0$ such that 
$d(x,y)<\delta$ implies $|f(x)-f(y)|<\frac\eps2$ for all $x,y\in K$.
Also, for each $x\in K$ there exists a $\delta_x>0$ such that $d(x,y)<\delta_x$ implies
$|f(x)-f(y)|<\frac\eps2$ for all $y\in X$.
We may assume that $\delta_x<\frac\delta2$.
Since $K$ is compact, there exists a finite subset $Z\subset K$ such that $K$ is contained in the set $U$ of 
all $x\in X$ such that $d(x,z)<\delta_z$ for some $z\in Z$.
Since $U$ is open, it contains the $\gamma$-neighborhood of $K$ for some $\gamma>0$.%
\footnote{Let $\gamma=\frac12\inf\,\{d(x,\,X\but U)\mid x\in K\}$.
Since $K$ is compact, $\gamma=\frac12d(x_0,\,X\but U)$ for some $x_0\in K$, and since $X\but U$ is closed, 
$d(x_0,\,X\but U)>0$.
Thus $\gamma>0$.}
We may assume that $\gamma<\frac\delta2$.

Thus given $x\in K$ and $y\in X$ such that $d(x,y)<\gamma$, there exists a $z\in Z$ such that $d(y,z)<\delta_z$.
Then $d(x,z)<\gamma+\delta_z\le\delta$, and hence $|f(x)-f(y)|<\frac\eps2+\frac\eps2=\eps$.

Let $k=\frac{2||f||}{\gamma}$.
Let us define $g,g'\:X\to\R$ by $g'(x)=\sup\,\{f(y)-kd(x,y)\mid y\in X\}$ and by 
$g(x)=\min\big(g'(x),||f||\big)$.
Then by the proof of (a), $g'$ is $k$-Lipschitz, $g'(x)\ge f(x)$ for all $x\in X$ and $g'(x)\le f(x)+\eps$
for all $x\in K$.
It follows that $g$ satisfies the same properties, an additionally $||g||\le||f||$.
\end{proof}

\begin{theorem} \label{top=unif} Let $X$ be a metric space and let $r>0$.

(a) The $U_b(X)$- and $\Lip_b(X)$-weak topologies coincide on $M^+_\sigma(X)$ and on $M^r_\sigma(X)$.

(b) The $C_b(X)$- and $U_b(X)$-weak topologies coincide on $M^+_\sigma(X)$ and on $\partial M^r_\sigma(X)$.
\end{theorem}

Example \ref{comparison-example}(b,c) below shows that in general, the $U_b(X)$- and $\Lip_b(X)$-weak 
topologies do not coincide on $M_\sigma(X)$; and the $C_b(X)$- and $U_b(X)$-weak topologies do not 
coincide on any $M^r_\sigma(X)$.

The proof of (a) is based on the approximation lemma (Lemma \ref{approximation}); the proof of (b) 
is based on the subbasis lemma (Lemma \ref{subbase} below).
We also include an alternative proof, based on the approximation lemma, of the following special case 
(b$_\rho$) of (b): {\it The $C_b(X)$- and $U_b(X)$-weak topologies coincide on $M^+_\rho(X)$.}
 
Assertions (a) and (b$_\rho$) are contained in \cite{Pac2}*{5.17, 5.22}; see also \cite{Bog}*{8.3.1}, 
\cite{LeC}*{Lemma 5}, \cite{Pac2}*{5.19} concerning (b).
 
\begin{proof}[Proof. (a)] Let $\nu\in M^r(X)$.
For each $\mu\in M^r_\sigma(X)$ we have $||\mu-\nu||\le||\mu||+||\nu||\le 2r$.
Let $f\in U_b(X)$ and $\eps>0$.
Then by \ref{approximation}(a) there exists a $g\in L_b(X)$ such that $||f-g||<\frac{\eps}{4r}$.
Then $\big|\int_X (f-g)\,d(\mu-\nu)\big|\le 2r\frac{\eps}{4r}=\frac\eps2$.
Hence by the triangle inequality, the $\nu$-centered $\eps$-ball 
$\{\mu\in M^r_\sigma(X)\mid\,\big|\int_X f\,d(\mu-\nu)\big|\le\eps\}$ of the seminorm $|\cdot|_f$ 
contains the $\nu$-centered $\frac\eps2$-ball 
$\{\mu\in M^r_\sigma(X)\mid\,\big|\int_X g\,d(\mu-\nu)\big|\le\frac\eps2\}$ of $|\cdot|_g$.

The coincidence of the two topologies on $M^+_\sigma(X)$ follows from their coincidence on 
$PM_\sigma(X)=M^1_\sigma(X)\cap M^+_\sigma(X)$ and Lemma \ref{opencone}(a).
\end{proof}

\begin{proof}[(b$_\rho$)]
Let $\nu\in M^+_\rho(X)$, $f\in C_b(X)$ and $\eps>0$.
Since $\nu$ is Radon, there exists a compact $K\subset X$ such that $|\nu|(X\but K)<\frac{\eps}{8||f||}$.
By \ref{approximation}(b) there exist $g_+,g_-\in U_b(f)$ such that $g_-\le f\le g_+$, both $g_+$ and $g_-$ are
$\frac{\eps}{4||\nu||}$-close to $f$ on $K$ and $||g_\pm||\le||f||$.
Then we have $\big|\int_K(f-g_\pm)\,d\nu\big|\le\frac{\eps}{4||\nu||}\cdot|\nu|(K)$ and
$\big|\int_{X\but K}(f-g_\pm)\,d\nu\big|\le||f-g_\pm||\cdot\frac{\eps}{8||f||}$, where 
$|\nu|(K)\le||\nu||$ and $||f-g_\pm||\le2||f||$.
Hence $\big|\int_X(f-g_\pm)\,d\nu\big|\le\frac\eps2$.

Suppose that $\mu$ lies in the intersection of the two $\nu$-centered $\frac\eps2$-balls
$\{\mu\in M^+_\rho(X)\mid\,\big|\int_X g_\pm\,d(\mu-\nu)\big|\le\frac\eps2\}$ of the seminorms $|\cdot|_{g_\pm}$.
Then
\[\int\limits_X f\,d\nu-\eps\le\int\limits_X g_-\,d\nu-\frac\eps2\le\int\limits_X g_-\,d\mu
\le\int\limits_X f\,d\mu
\le\int\limits_X g_+\,d\mu\le\int\limits_X g_+\,d\nu+\frac\eps2\le\int\limits_X f\,d\nu+\eps,\]
where the two central inequalities hold due to $\mu\ge 0$.
Hence $\mu$ lies in the $\nu$-centered $\eps$-ball
$\{\mu\in M^+_\rho(X)\mid\,\big|\int_X f\,d(\mu-\nu)\big|\le\eps\}$ of the seminorm $|\cdot|_f$.
\end{proof}

\begin{proof}[(b)]
Let $\nu\in M^+_\sigma(X)$, let $Z$ be a closed subset of $X$ and let $\eps>0$.
By Lemma \ref{subbase}(c) below it suffices to show that each of $B^-_{Z,\eps}(\nu)$
and $B^+_{X,\eps}(\nu)\cap B^-_{X,\eps}(\nu)$ contains a ball of the seminorm $|\cdot|_f$ for some 
$f\in U_b(X)$.

Let $U_n$ be the $1/n$-neighborhood of $Z$. 
Then $\lim_{n\to\infty}\nu(U_n)=\nu(Z)$ due to the countable additivity of $\nu$.
Hence $\nu(U_N\but Z)<\frac\eps2$ for some $N\in\N$.
Let $U=U_N$, and let us define $f\:X\to [0,1]$ by $f(x)=1-\min\big(1,Nd(x,Z)\big)$.
Thus $f$ sends $Z$ to $1$ and $X\but U$ to $0$.
The triangle axiom implies that $|f(x)-f(y)|\le Nd(x,y)$, so $f$ is $N$-Lipschitz.

Let $\mu$ lie the $\nu$-centered $\frac\eps2$-ball 
$\{\mu\in M^+_\sigma(X)\mid\,\big|\int_X f\,d(\mu-\nu)\big|\le\frac\eps2\}$ of the seminorm $|\cdot|_f$.
Then $\mu(Z)\le\int_X f\,d\mu\le\int_X f\,d\nu+\frac\eps2\le\nu(U)+\frac\eps2\le\nu(Z)+\eps$.
Hence $\mu\in B^-_{Z,\eps}(\nu)$. 

Finally, $B^+_{X,\eps}(\nu)\cap B^-_{X,\eps}(\nu)\cap M_\sigma(X)=
\{\mu\in M^+_\sigma(X)\mid\,|\mu(X)-\nu(X)|<\eps\}$
contains the $\nu$-centered $\eps$-ball of the seminorm $|\cdot|_1$ (in fact, they coincide).

The coincidence of the $C_b(X)$-weak and the $U_b(X)$-weak topologies on each $\partial M^r_\sigma(X)$ 
follows from their coincidence on $M^+_\sigma(X)$ and Lemma \ref{opencone}(c).
In more detail, since both $\psi|_{\partial M^r(X)}$ and its uniform version are embeddings, and their 
codomains have the same topology, their domains must also have the same topology. 
\end{proof}

\section{Kantorovich--Rubinstein norm}

\subsection{Definition}

Let $X$ be a metric space.
The normed space $\Lip_b(X)$ of all bounded Lipschitz functions $f\:X\to\R$ is endowed with the norm 
$||f||=\sup_{x\in X}|f(x)|$, but its underlying vector space also carries another seminorm $\Lip(f)$: 
the minimal $k\ge 0$ such that $f$ is $k$-Lipschitz.%
\footnote{Let us note that $\Lip(f)=||f'||$, where $f'\:X\x X\but\Delta\to\R$ is defined by
$f'(x,y)=\frac{f(x)-f(y)}{d(x,y)}$.
Let us also note that $(fg)'=f'g+g'f$ with respect to the obvious $C_b(X)$-bimodule structure on 
$C_b(X\x X\but\Delta)$. 
See \cite{We}*{\S2.4} for a further discussion.}
Then $||f||_{bl}\bydef \max\big(||f||,\Lip(f)\big)$ is a norm on the underlying vector space of $\Lip_b(X)$.
It is convenient to denote the same vector space but with the new norm $||\cdot||_{bl}$ by $\Lip_{bl}(X)$.
The dual norm on $\Lip_{bl}^*(X)$, denoted $||\cdot||_{KR}$, is called the {\it Kantorovich--Rubinstein norm} 
(see \cite{BK}).
Thus $||\Phi||_{KR}$ is the supremum of $|\Phi(f)|$ over all 1-Lipschitz functions $f\:X\to[-1,1]$.
Will will refer to the topology and the uniformity induced by $||\cdot||_{KR}$ on the dual space 
$\Lip^*_{bl}(X)$ and its subspaces as the {\it KR-topology} and {\it KR-uniformity}.

Since Lipschitz functions $X\to [-1,1]$ include 1-Lipschitz functions $X\to [-1,1]$, we have 
$||\Phi||_{KR}\le||\Phi||$ for each linear functional $\Phi$ on $\Lip_b(X)$ (or on $\Lip_{bl}(X)$, which is
the same as a vector space).
In particular, the Kantorovich--Rubinstein norm is defined on the underlying vector space of $\Lip_b^*(X)$.
Since $M_\sigma(X)$ injects into this space, $||\cdot||_{KR}$ is also a norm on $M_\sigma(X)$.
Thus for any measure $\mu$ on $X$, $||\mu||_{KR}$ is the supremum of $\big|\int_X f\,d\mu\big|$ over all 
1-Lipschitz functions $f\:X\to[-1,1]$.

\begin{remark}
A well-known variant of the Kantorovich--Rubinstein norm is the {\it Fortet--Mourier norm} $||\cdot||_{FM}$, 
which is the dual of the norm $||f||_{bl}'\bydef ||f||+\Lip(f)$ on the underlying vector space of $\Lip_b(X)$.
Thus for any measure $\mu$ on $X$, $||\mu||_{FM}$ is the supremum of the same quantity $\big|\int_X f\,d\mu\big|$
over all $(1-\lambda)$-Lipschitz functions $f\:X\to[-\lambda,\lambda]$, where $\lambda\in [0,1]$.
Clearly, $\max(x,y)\le x+y\le 2\max(x,y)$, and consequently 
$||\mu||_{KR}\ge||\mu||_{FM}\ge\tfrac12||\mu||_{KR}$ for each measure $\mu$ on $X$.
\end{remark}

\begin{lemma} \label{KR-restriction} Let $X$ is a metric space, $A$ a subset of $X$ and $i\:A\to Z$ 
the inclusion map.
Then $||\mu||_{KR}=||i_*\mu||_{KR}$ for every measure $\mu$ on $A$.
\end{lemma}

Thus $\mu\mapsto i_*\mu$ defines an isometric embedding $M_\sigma(A)\to M_\sigma(X)$ with respect
to the Kantorovich--Rubinstein norm.

\begin{proof}
By Lemma \ref{lip-ext} every 1-Lipschitz $f\:A\to [-1,1]$ extends to a 1-Lipschitz $\bar f\:X\to[-1,1]$,
and clearly $|\mu|_f=|i_*\mu|_{\bar f}$.
Also, for every 1-Lipschitz $g\:X\to [-1,1]$ we have $|i_*\mu|_g=|\mu|_{g|_Y}$.
\end{proof}

\subsection{Comparison of topologies}

\begin{theorem}\label{KR-metric} Let $X$ be a metric space.

(a) The KR-topology is stronger than the $\Lip_b(X)$-weak topology on $M_\sigma(X)$.

(b) The KR-topology coincides with the $C_b(X)$-weak topology on $M^+_\tau(X)$ and on 
each $\partial M^r_\tau(X)$.

(c) The KR-topology is weaker than the open cone topology on $M_\tau(X)$.

(d) $M_\delta(X)$, $M^+_\delta(X)$, $PM_\delta(X)$ and $M^0_\delta(X)$ are dense respectively in 
$M_\tau(X)$, $M^+_\tau(X)$, $PM_\tau(X)$ and $M^0_\tau(X)$ in the KR-topology. 
\end{theorem}

The first assertion of (b) is proved in \cite{Bog}*{8.2.18, 8.3.2}, and the full assertion of (b)
with $\tau$ replaced by $\rho$ in \cite{Pac2}*{5.37, 5.39}.

The proof of (b) is based on the subbase lemma (Lemma \ref{subbase}).
We also include an alternative proof, based on the approximation lemma (Lemma \ref{approximation}), 
of the following special case (b$_\rho$) of (b): {\it If $X$ is a metric space, the KR-topology 
coincides with the $\Lip_b(X)$-weak topology on $M^+_\rho(X)$.}
Let us note that (b$_\rho$) implies (b) by considering the completion of $X$ and using 
Lemma \ref{regularity}(b) and Lemma \ref{KR-restriction} as well as Theorem \ref{top=unif}.

\begin{proof}[Proof. (a)]
This is proved similarly to the general fact that the weak* topology of $V^*$ is weaker than its original 
topology (induced by the norm).
Indeed, clearly $|\mu-\nu|_f\le ||f||\cdot\sup\{|\mu-\nu|_\phi\mid\phi\in\Lip_b(X),\,||\phi||\le 1,\,
\Lip(\phi)\le 1\}$ for each $1$-Lipschitz $f\in\Lip_b(X)$.
Hence the $\nu$-centered $\eps$-ball of the seminorm $|\cdot|_f$ contains the $\nu$-centered 
$\frac{\eps}{||f||}$-ball of the norm $||\cdot||_{KR}$. 
Finally, if $g$ is $k$-Lipschitz, then $\frac gk$ is $1$-Lipschitz, and $|\cdot|_g=k|\cdot|_{g/k}$. 
\end{proof} 

\begin{proof}[(b$_\rho$)] By (a) the KR-topology is stronger than the $\Lip_b(X)$-weak topology 
on $M^+_\rho(X)$.

Conversely, let $\nu\in M^+_\rho(X)$ and $\eps>0$ be given.
Since $\nu$ is Radon, there exists a compact $K\subset X$ such that $\nu(X\but K)<\frac\eps{10}$.
Since $K$ is compact, there exists a finite $S\subset K$ such that $d(x,S)\le\frac\eps{20||\nu||}$ 
for each $x\in K$. 
Let us define $h\:X\to[0,1]$ by $h(x)=\min\big(\frac\eps{20||\nu||}+d(x,S),1\big)$.
Then $h$ is $1$-Lipschitz and $h(x)\le\frac\eps{10||\nu||}$ for all $x\in K$.
Hence $|\nu|_h\le\nu(K)\frac\eps{10||\nu||}+\nu(X\but K)\le\nu(X)\frac\eps{10||\nu||}+\frac\eps{10}=\frac\eps5$.

Let $L_X$ be the set of all $1$-Lipschitz maps $X\to [-1,1]$, and let $r\:L_X\to L_S$ be defined by
$f\mapsto f|_S$.
Since $L_S=[-1,1]^S$ is compact, it admits a finite cover by open sets $B_i$ of diameters 
$\le\frac\eps{10||\nu||}$ in the $l_\infty$ metric.
For each nonempty preimage $r^{-1}(B_i)$ let us pick one point in it, and let $Q$ be the set of thus selected 
points.
Thus $Q$ is a finite subset of $L_X$ such that for every $f\in L_X$ there exists a $g\in Q$ such that 
$|f(x)-g(x)|<\frac\eps{10||\nu||}$ for each $x\in S$.
Here $f-g$ is a 2-Lipschitz map into $[-2,2]$, so the latter condition in fact implies $|f(x)-g(x)|\le 2h(x)$ 
for each $x\in X$.

Let $\mu\in M_\rho(X)$ be a point in the intersection of the $\nu$-centered $\frac\eps{20}$-balls of
the seminorms $|\cdot|_g$, where $g\in Q\cup\{h\}$.
Thus $|\mu-\nu|_g\le\frac\eps{20}$ for each $g\in Q$ and also $|\mu-\nu|_h\le\frac\eps{20}$. 
Then we have $|\mu|_h=|(\mu-\nu)+\nu|_h\le |\mu-\nu|_h+|\nu|_h\le\frac\eps{20}+\frac\eps5=\frac\eps4$.

Let $f\in L_X$ be given.
Then there exists a $g\in Q$ such that $|f-g|\le 2h$.
We have 
$\big|\int_X f\,d\nu-\int_X g\,d\nu\big|\le\int_X|f-g|\,d\nu\le 2|\nu|_h\le\frac{2\eps}5$.
Similarly
$\big|\int_X f\,d\mu-\int_X g\,d\mu\big|\le\int_X|f-g|\,d\mu\le 2|\mu|_h\le\frac\eps2$.
(It is here that we essentially use that $\mu\ge 0$.)
Since $\big|\int_X g\,d\mu-\int_X g\,d\nu\big|\le\frac\eps{20}$, we conclude that
$|\mu-\nu|_f=\big|\int_X f\,d\mu-\int_X f\,d\nu\big|\le\frac{2\eps}5+\frac\eps2+\frac\eps{20}<\eps$.
Since $f$ is an arbitrary $1$-Lipschitz function with $||f||\le 1$, we get $||\mu-\nu||_{KR}<\eps$.
Thus the $\nu$-centered $\eps$-ball of $||\nu||_{KR}$ is a neigborhood of $\nu$ in $M^+_\sigma(X)$.
\end{proof}

\begin{proof}[(b)] By (a) and Theorem \ref{top=unif}(a,b), the KR-topology is stronger than 
the $C_b(X)$-weak topology on $M^+_\tau(X)$ and on each $\partial M^r_\tau(X)$.

Conversely, let $\nu\in M^+_\tau(X)$ and $\eps>0$ be given.
We may assume that $\frac\eps8\le\nu(X)$.
By Lemma \ref{tau} there exists a separable closed subset $Z\subset X$ such that $\nu(X\but Z)=0$.
Let $\{x_1,x_2,\dots\}$ be a countable dense subset of $Z$.
Set $X_0=\emptyset$.
Assuming that a closed subset $X_{i-1}\subset X$ has been defined, let $n_i$ be the minimal number such that
$x_{n_i}\notin X_{i-1}$.
Let $V_i$ be the closed $\eps_i$-ball of $X\but X_{i-1}$ centered at $v_i\bydef x_{n_i}$, where 
$\eps_i\in (0,\frac\eps{8||\nu||})$ is chosen so that the frontier of $V_i$ in $X\but X_{i-1}$ is 
of measure zero with respect to $\nu$.
Finally, let $X_i=X_{i-1}\cup V_i$.
Thus each $V_i\cap V_j=\emptyset$ whenever $i\ne j$ and $X_\infty\bydef \bigcup_{i=1}^\infty V_i$ contains $Z$.
Hence $\sum_{i=1}^\infty\nu(V_i)=\nu(X_\infty)=\nu(X)$.
Since $\nu\ge 0$, there exists an $r$ such that $\nu(X\but X_r)=\sum_{i>r}\nu(V_i)\le\frac\eps8$.
Let $\mathring V_i$ and $\bar V_i$ be the interior and the closure of $V_i$ in $X$.
Arguing by induction, it is easy to see that each $\nu(\bar V_i\but\mathring V_i)=0$.
Using the notation of Lemma \ref{subbase}, let 
\[\mu\in B^-_{X,\,\frac\eps8}(\nu)\cap\bigcap_{i\le r} B^+_{\mathring V_i,\,\frac\eps{8r}}(\nu)\cap
\bigcap_{i\le r} B^-_{\bar V_i,\,\frac\eps{8r}}(\nu).\]

For each $i\le r$ we have $\mu(V_i)\le\mu(\bar V_i)\le\nu(\bar V_i)+\frac\eps{8r}=\nu(V_i)+\frac\eps{8r}$ and
similarly $\mu(V_i)\ge\mu(\mathring V_i)\ge\nu(\mathring V_i)-\frac\eps{8r}=\nu(V_i)-\frac\eps{8r}$.
Hence $\sum_{i\le r}|\mu(V_i)-\nu(V_i)|\le\frac\eps8$.

Also we get $\nu(X_r)-\mu(X_r)=\sum_{i=1}^r\nu(V_i)-\mu(V_i)\le\frac\eps8$, and since 
$\mu\in B^-_{X,\,\frac\eps8}$, we also have $\mu(X)-\nu(X)\le\frac\eps8$.
By adding up these inequalities, we get $\mu(X\but X_r)-\nu(X\but X_r)\le\frac\eps4$.
Hence $\mu(X\but X_r)\le\frac\eps4+\frac\eps8$.
Since each $|\mu(V_i)-\nu(V_i)|\le|\mu(V_i)|+|-\nu(V_i)|=\mu(V_i)+\nu(V_i)$, we have
$\sum_{i>r}|\mu(V_i)-\nu(V_i)|\le\mu(X\but X_r)+\nu(X\but X_r)\le
\left(\frac\eps4+\frac\eps8\right)+\frac\eps8=\frac\eps2$.
In fact, all of this also makes sense if we write $V_\infty=X\but X_\infty$ and let $i$ assume the value 
$\infty$ as well.
To avoid confusion, we will not use such convention, and instead denote $X\but X_\infty$ by $W$.
Then the new reading of the latter inequality takes the form 
$\sum_{i>r}|\mu(V_i)-\nu(V_i)|+\mu(W)\le\frac\eps2$ (using that $\nu(W)=0$).
To summarize, we obtain $\sum_{i=1}^\infty|\mu(V_i)-\nu(V_i)|+\mu(W)\le\frac\eps8+\frac\eps2$.

Let $\mu'=\sum_{i=1}^\infty\mu(V_i)\delta_{v_i}$ and $\nu'=\sum_{i=1}^\infty\nu(V_i)\delta_{v_i}$.
If $f\:X\to [-1,1]$ is $1$-Lipschitz, then 
\[|\nu-\nu'|_f=\left|\sum_{i=1}^\infty\int\limits_{V_i} f-f(v_i)\,d\nu\right|\le
\sum_{i=1}^\infty\nu(V_i)\sup_{v\in V_i}|f(x)-f(v_i)|\le\nu(X)\frac\eps{8||\nu||}=\frac\eps8.\]
Similarly 
\[|\mu-\mu'|_f\le\mu(X)\frac\eps{8||\nu||}+\left|\int\limits_{W}f\,d\mu\right|\le
\big(\nu(X)+\tfrac\eps8\big)\frac\eps{8||\nu||}+\mu(W)\le\frac\eps4+\mu(W),\]
using that $\frac\eps8\le\nu(X)$.

Finally, $|\mu'-\nu'|_f=\big|\sum_{i=1}^\infty f(v_i)\big(\mu(V_i)-\nu(V_i)\big)\big|\le
\sum_{i=1}^\infty|\mu(V_i)-\nu(V_i)|\le\frac\eps8+\frac\eps2-\mu(W)$.
Thus $|\mu-\nu|_f\le\frac\eps8+\big(\frac\eps4+\mu(W)\big)+\big(\frac\eps8+\frac\eps2-\mu(W)\big)=\eps$.
Since $f$ is an arbitrary $1$-Lipschitz function with $||f||\le 1$, we get $||\mu-\nu||_{KR}\le\eps$.
Thus the closed $\nu$-centered $\eps$-ball of $||\nu||_{KR}$ is a neigborhood of $\nu$ in $M^+_\tau(X)$.

It remains to show that the topology of $||\cdot||_{KR}$ is weaker than the $C_b(X)$-weak topology on 
each $\partial M^r_\tau(X)$.
By Lemma \ref{opencone}(b), the map $\psi\:\partial M^r_\tau(X)\to M^+_\tau(X)\x M^+_\tau(X)$,
$\psi(\mu)=(\mu_+,-\mu_-)$, is continuous.
By the above, the topology of $M^+_\tau(X)$ is induced by $||\cdot||_{KR}$.
Finally, $||\mu-\nu||_{KR}=||\mu_+-\nu_++\mu_--\nu_-||_{KR}\le||\mu_+-\nu_+||_{KR}
+||\mu_--\nu_-||_{KR}$, and the assertion follows.
\end{proof}

\begin{proof}[(c)] 
This is similar to the proof of the second assertion of (b), using that by Lemma \ref{opencone}(b), 
the map $\psi\:M_\tau(X)\to M^+_\tau(X)\x M^+_\tau(X)$, $\psi(\mu)=(\mu_+,-\mu_-)$, is continuous 
with respect to the open cone topology on $M_\tau(X)$ and the usual (i.e., $C_b(X)$-weak) topology 
on both copies of $M^+_\tau(X)$.
\end{proof}

\begin{proof}[(d)] By Lemma \ref{dirac}(e), $\partial M^r_\delta(X)$ is dense in $\partial M^r_\tau(X)$ 
in the $C_b(X)$-weak topology for each $r>0$.
By (b), the same holds in the KR-topology.
Hence also $M_\delta(X)$ is dense in $M_\tau(X)$ in the KR-topology.
The other assertions similarly follow from Lemma \ref{dirac}(c,f).
\end{proof}

Theorem \ref{KR-metric}(a) and Remark \ref{closed subspaces2} have the following

\begin{corollary}\label{closed subspaces3} Let $X$ be a metric space.

(a) $PM_\sigma(X)$ is closed in $M^+_\sigma(X)$ in the KR-topology.

(b) $M^+_\sigma(X)$ and $M^r_\sigma(X)$ are closed in $M_\sigma(X)$ in the KR-topology.

(c) If $U\subset X$ is open, for each $r>0$, the sets 
$V\bydef \{\mu\in M_\sigma(X)\mid \mu(U)\ge 0\}$ and $W\bydef \{\mu\in M_\sigma(X)\mid\ |\mu|(U)\le r\}$ are closed in 
$M_\sigma(X)$ in the KR-topology.
\end{corollary}

From Lemma \ref{closed emb}, Theorem \ref{KR-metric}(b) and Corollary \ref{closed subspaces3}(a,b) we get

\begin{corollary} \label{KR-closed} $X$ is closed in $M_\tau(X)$ in the KR-topology.
\end{corollary}

\subsection{Comparison of topologies II}

\begin{example} \label{comparison-example}
(Compare \cite{Bog}*{Remark before 8.10.47} and \cite{BK}*{Remark after 1.1.3}.) 
Let $X$ be a metric space and let $x_n\in X$ be a convergent sequence with limit $x\in X$.

(a) Let $\mu_n=\delta_{x_n}-\delta_x$.
Then $||\mu_n||=2$ for each $n$.
On the other hand, we have $||\mu_n||_{KR}=\min\big(2,d(x_n,x)\big)\to 0$ as $n\to\infty$.
Let us also note that $\mu_n\to 0$ in either of the $C_b(X)$-, $U_b(X)$- and $\Lip_b(X)$-weak topologies
(see Example \ref{convergence-example}(b)).
However, $\mu_n\not\to0$ in the open cone topology, due to $||\mu_n||\not\to 0$.
Thus the KR-topology differs from the open cone topology --- even though the latter is metrizable 
on $M_\tau(X)$ by the second assertion of Theorem \ref{KR-metric}(b).

(b)  Let $\mu_n=\frac{\delta_{x_n}-\delta_x}{\sqrt{d(x_n,x)}}$.
Then $||\mu_n||_{KR}=\frac{\min(2,\,d(x_n,x))}{\sqrt{d(x_n,x)}}\to 0$ as $n\to\infty$.
On the other hand, let $f\in U_b(X)$ be defined by $f(p)=\min\big(\sqrt{d(p,x)},h\big)$, where
$h=\sup_{n\in\N}\sqrt{d(x_n,x)}$.%
\footnote{Let us note that $\sqrt\cdot$ is uniformly continuous, since it is 
$1$-Lipschitz on $[1,\infty)$, whereas $[0,1]$ is compact.}
Then $|\mu_n|_f=\frac{\sqrt{d(x_n,x)}}{\sqrt{d(x_n,x)}}=1$.
Thus $\mu_n\not\to 0$ in the $C_b(X)$- and $U_b(X)$-weak topologies.
However, $\mu_n\to 0$ in the $\Lip_b(X)$-weak topology by Theorem \ref{KR-metric}(a).

(c) Let $\mu_n=\delta_{x_n}-\delta_{y_n}$, where $x_n=n$ and $y_n=n+\frac1n\in\R$.
Then $||\mu_n||_{KR}=d(x_n,y_n)=\frac1n\to 0$ as $n\to\infty$.
Also, $||\mu_n||=2$, so each $\mu_n$ as well as $0$ belong to $M^2_\sigma(X)$.
On the other hand, let $f\:X\to[0,1]$ be a continuous function sending $\{2,3,\dots\}$ to $1$ and
$\{2+\frac12,3+\frac13,\dots\}$ to $0$.
Then $|\mu_n|_f=|f(n)-f(n+\frac1n)|=1$ for $n\ge 2$, so $\mu_n\not\to0$ in the $C_b(X)$-weak topology.
However, $\mu_n\to 0$ in the $\Lip_b(X)$- and $U_b(X)$-weak topologies by Theorem \ref{KR-metric}(a)
and by the second assertion of Theorem \ref{top=unif}(a).
\end{example}

\begin{remark} Let us summarize the relationships between various topologies on $M_\tau(X)$:

\begin{enumerate}
\item[$\Lip_b:$] the $\Lip_b(X)$-weak topology;
\item[$U_b:$] the $U_b(X)$-weak topology;
\item[$C_b:$] the $C_b(X)$-weak topology;
\item[$KR:$] the KR-topology;
\item[$OC:$] the topology of the open cone over $\partial M^1_\tau(X)$ under either of $\Lip_b$, $U_b$, $C_b$;
\item[$TV:$] the topology of the total variation norm or the norms of $C_b^*(X)$, $U_b^*(X)$, $\Lip_b^*(X)$.
\end{enumerate}

Here the three open cone topologies coincide, and the four norms also coincide.
With this notation, we have $\Lip_b<U_b<C_b<OC<TV$  and also $\Lip_b<KR<OC<TV$ on $M_\tau(X)$.
The topologies $KR$, $OC$ and $TV$ are metrizable on $M_\tau(X)$, as long as $X$ is metrizable.
We also have $\Lip_b=U_b=C_b=KR=OC$ on $M^+_\tau(X)$ and on each $\partial M^r_\tau(X)$; also, $\Lip_b=U_b$ 
on each $M^r_\tau(X)$.
\end{remark}

\begin{theorem} \label{sequential} Let $X$ be a metric space.

(a) The KR-topology and the $\Lip_b(X)$-weak topology on $M_\rho(X)$ have the same convergent
sequences.%
\footnote{By this we mean that if a sequence $x_n\in M_\rho(X)$ converges to an $x\in M_\rho(X)$ in one of the two
topologies, then $x_n\to x$ also in the other topology.}

(b) The KR-topology and the $U_b(X)$-weak topology on each $M_\rho^r(X)$ have the same convergent sequences.
\end{theorem}

Let us note that by Example \ref{comparison-example}(b) the KR-topology and the 
$U_b(X)$-weak topology have different convergent sequences in $M_\rho(X)$, contrary to the claims 
in \cite{Pac1}, \cite{Pac2}*{Corollary 5.43}.%
\footnote{Let us discuss \cite{Pac2}*{Proof of 5.43}.
Although the author does not explicitly state Lemma \ref{cont-bij}, he apparently tries to apply it.
However, in his situation it does not quite apply because the hypothesis that $f$ is continuous need not 
be satisfied on the entire $M_r(X)$, as shown by Example \ref{comparison-example}(b).}
By Example \ref{comparison-example}(c) the KR-topology and the $C_b(X)$-weak topology have 
different convergent sequences in each $M_\rho^r(X)$.

Although the argument below corrects an error in \cite{Pac2}*{Proof of Corollary 5.43}, the substantial part 
of the proof is not covered below; see \cite{Pac2}*{Proof of Theorem 5.41}.
See also \cite{Bog}*{Hint to 8.10.134}, which seems to lead to an easier proof (albeit still quite difficult, 
as it depends on the full Prokhorov theorem).

\begin{proof}[Proof. (a)] If a sequence $x_k\in M_\rho(X)$ converges to an $x\in M_\rho(X)$ in the KR-topology, 
then by Theorem \ref{KR-metric}(a) $x_k\to x$ also in the $\Lip_b(X)$-weak topology.

Conversely, it is shown in \cite{Pac2}*{Theorem 5.41 (vii)$\Rightarrow$(iii)} that if $X$ is complete, then
a $||\cdot||$-bounded subset $U$ of $M_\rho(X)$ that is relatively compact in the $\Lip_b(X)$-weak topology
is relatively compact also in the KR-topology.
By Theorem \ref{KR-metric}(a) the same assertion also holds with ``closed'' in place of ``relatively compact''; 
hence ``relatively'' can be dropped.
The hypothesis that $X$ be complete can also be dropped, by considering the completion of $X$ (see details in
\cite{Pac2}*{Proof of 5.43}).
Now Lemma \ref{cont-bij} can be applied to the identity map from $M_\rho(X)$ with the KR-topology to
$M_\rho(X)$ with the $\Lip_b(X)$-weak topology, since this map is continuous by Theorem \ref{KR-metric}(a).
\end{proof}

\begin{proof}[(b)] This follows from (a) and Theorem \ref{top=unif}(a). 
\end{proof}

\begin{lemma} \label{cont-bij} Let $f\:Y\to Y'$ be a continuous bijection between Hausdorff spaces.
Suppose that whenever $K'\subset Y'$ is compact, $f^{-1}(K')$ is compact.
Then $f$ identifies convergent sequences in $Y$ and convergent sequences in $Y'$.
\end{lemma}

\begin{proof} First let us show that $f^{-1}$ is continuous on every compact $K'\subset Y'$.
Let $Z$ be a closed subset of $K\bydef f^{-1}(K')$.
Since $K$ is compact, so is $Z$.
Since $f$ is continuous, $f(Z)$ is compact.
Since $K'$ is Hausdorff, $f(Z)$ is closed in $K'$.

If a sequence $x_k\in Y$ converges to an $x\in Y$, then $f(x_k)$ converges to $f(x)$ by
the continuity of $f$.
Conversely, if a sequence $x'_k\in Y'$ converges to an $x'\in Y$, then 
$K'\bydef \{x_i\mid i\in\N\}\cup\{x'\}$ is compact.
Then by the above $f^{-1}$ is continuous on $K'$.
Hence $f^{-1}(x'_k)$ converges to $f^{-1}(x)$. 
\end{proof}

\begin{remark}
In the case $X=\N$ (see Example \ref{charge-example} concerning this case), Theorem \ref{sequential}(a) also 
follows from Schur's theorem that every weakly (i.e., $l_\infty$-weakly) convergent 
sequence in $l_1$ converges in the norm of $l_1$ (see \cite{Pac2}*{P.15(2)}).
Even in this case, the corresponding assertion for nets (rather than sequences) fails 
(see \cite{Pac2}*{P.15(4)}).
\end{remark}

\subsection{Comparison of uniformities}

From Theorem \ref{sequential} we easily get, using the linear structure of $M_\rho(X)$:

\begin{corollary} \label{merging} Let $X$ be a metric space.

(a) The KR-uniformity and the $\Lip_b(X)$-weak uniformity on $M_\rho(X)$ have the same merging 
pairs of sequences.

(b) The KR-uniformity and the $U_b(X)$-weak uniformity on each $M_\rho^r(X)$ have the same merging 
pairs of sequences.
\end{corollary}

\begin{example} \label{merging-example}
(a) Let $X$ be any metric space that contains a convergent sequence $x_n\to x$.
Let $\mu_n=\frac{\delta_{x_n}}{\sqrt{d(x_n,x)}}$ and $\nu_n=\frac{\delta_x}{\sqrt{d(x_n,x)}}$.
Thus $\mu_n,\nu_n\in M_\delta^+(X)$.
By Example \ref{comparison-example}(b) $\mu_n$ and $\nu_n$ are merging in the KR-uniformity
and in the $\Lip_b(X)$-weak uniformity, but are not merging in the $C_b(X)$-weak and $U_b(X)$-weak
uniformities.

(b) (compare \cite{DDF}*{1.1}) Let $\mu_n=\delta_n$ and $\nu_n=\delta_{n+\frac1n}$.
Thus $\mu_n,\nu_n\in PM_\delta(\R)$.
By Example \ref{comparison-example}(c) $\mu_n$ and $\nu_n$ are merging in the KR-uniformity
and in the $\Lip_b(X)$-weak and $U_b(X)$-weak uniformities, but are not merging in the $C_b(X)$-weak
uniformity.
\end{example}

\begin{theorem} {\rm (see \cite{Bog}*{8.10.43, 8.3.2})}
Let $X$ be a metric space, and for any $A\subset X$ let $A_\eps$ denote the open $\eps$-neighborhood of $A$.
The {\rm L\'evy--Prokhorov metric} on $PM_\sigma(X)$, defined by
\[d_{LP}(\mu,\nu)=\inf\{\eps>0\mid\forall B\in\B(X)\ \,\mu(B)\le\nu(B_\eps)+\eps\text{ and }
\nu(B)\le\mu(B_\eps)+\eps\},\]
is uniformly equivalent to the metric of the Kantorovich--Rubinstein norm.
\end{theorem}

\subsection{Invariance of KR-uniformity}

\begin{theorem} \label{unif-well-def}
Let $X$ be a metrizable uniform space.
The KR-uniformity on each $M^r_\sigma(X)$ does not depend on the choice of a metric on $X$.
\end{theorem}

The case of $PM_\sigma(X)$ appears in \cite{Ban2}*{4.21} with a somewhat different proof.

Since a metrizable uniformity is determined by its notion of merging pairs of sequences,
Corollary \ref{merging}(b) provides an alternative proof of Theorem \ref{unif-well-def} for Polish spaces.

\begin{proof} Suppose that $d$ and $d'$ are uniformly equivalent metrics on $X$.
Let $||\cdot||_{KR}$ and $||\cdot||'_{KR}$ be the corresponding Kantorovich--Rubinstein norms on 
$M_\sigma(X)$.
Given an $\eps>0$, there exists a $\delta>0$ such that $d'(x,y)<\delta$ implies $d(x,y)<\frac\eps{4r}$.
We may assume that $\delta\le\frac\eps4$.
Given $\lambda,\nu\in M^r_\sigma(X)$, let $\mu=\lambda-\nu$.
We have $||\mu||\le||\lambda||+||\nu||\le 2r$.

Suppose that $||\mu||_{KR}>\eps$.
Then there exists a $1$-Lipschitz function $f\:(X,d)\to[-1,1]$ such that $|\mu|_f>\eps$. 
Let us note that if $d'(x,y)<\delta$, then $d(x,y)<\frac\eps{4r}$ and consequently 
$|f(x)-f(y)|<\frac\eps{4r}$.
Let $k=\frac2\delta$.
By the proof of Lemma \ref{approximation}(a), using that $d'(x,y)<\delta$ implies
$|f(x)-f(y)|<\frac\eps{4r}$, there exists a $k$-Lipschitz function $g\:(X,d')\to[-1,1]$
such that $||f-g||<\frac\eps{4r}$.
Then $|\mu|_{f-g}\le||\mu||\cdot||f-g||\le 2r\frac\eps{4r}=\frac\eps2$.
Hence $|\mu|_g\ge |\mu|_f-|\mu|_{f-g}>\eps-\frac\eps2=\frac\eps2$.
Therefore $|\mu|_{g/k}>\frac\eps{2k}=\frac{\eps\delta}4\ge\delta^2$.
Hence $||\mu||'_{KR}>\delta^2$.
Thus we have proved that $||\mu||'_{KR}\le\delta^2$ implies $||\mu||_{KR}\le\eps$.
\end{proof}

\begin{remark} \label{well-def-remark}
The proof of Theorem \ref{unif-well-def} works if measures are replaced by arbitrary
functionals with finite Kantorovich--Rubinstein norm.
Thus it shows that the KR-uniformity is well-defined on each subset $S_r=\{\Phi\mid\,||\Phi||\le r\}$ of 
$\Lip_{bl}^*(X)$ for every metrizable uniform space $X$.
\end{remark}

\begin{example} \label{unif-equiv-metr}
The KR-uniformity of $M^+_\delta([0,1])$ (and consequently also the KR-topology on $M_\delta([0,1])$) depends 
on the choice of a metric on $[0,1]$.

Indeed, let $\mu_n=n\delta_{1/n}$ and $\nu_n=n\delta_0$.
The (uniform) homeorphism $h\:[0,1]\to[0,1]$, $t\mapsto t^2$, takes these measures onto $\mu'_n=n\delta_{1/n^2}$ 
and $\nu'_n=n\delta_0$, in the sense that $\mu'_n=h_*\mu_n$ and $\nu'_n=h_*\nu_n$.
For the usual metric on $[0,1]$ we have $||\mu_n-\nu_n||_{KR}=n\frac1n=1$ and 
$||\mu'_n-\nu'_n||_{KR}=n\frac1{n^2}=\frac1n$.
Hence for the metric $d(x,y)=|h(x)-h(y)|$ we have $||\mu_n-\nu_n||_{KR}=\frac1n\to 0$ as $n\to\infty$.
Thus the sequences $\mu_n$ and $\nu_n$ merge with respect to $d$, but do not merge with respect to
the usual metric on $[0,1]$.
\end{example}

\subsection{Completeness of KR-uniformity}

\begin{theorem} \label{KR-complete}
If $X$ is a complete metric space, then $M^+_\tau(X)$ and $PM_\tau(X)$ are complete in the KR-uniformity.
\end{theorem}

Proofs of this theorem and of the deeper theorem asserting completeness of each $M^r_\tau(X)$
in the KR-uniformity can be found in \cite{Bog}*{8.9.4} and \cite{Pac2}*{5.30};
see also \cite{Par}*{Proof of Theorem II.6.7}.
The proof given below has a more topological flavor, and corrects and fills in some gaps in 
\cite{Ban2}*{Proof of 4.6}.%
\footnote{Let us discuss \cite{Ban2}*{Proof of 4.6}.
The implication ``... $\mu(C_n)\ge 1-\eps$ for every $n\in\N$. Taking into account that
$K=\bigcap_{n=1}^\infty C_n$, we conclude that $\mu(K)\ge 1-\eps$'' is incorrect, as the $C_n$
need not satisfy $C_1\supset C_2\supset\dots$.
Also, the assertion ``One can easily check that $K=\bigcap_{i=1}^\infty C_n$'' is actually 
a strengthened form of Proposition \ref{corona}(c), whose proof is far from being routine.} 
A somewhat different proof in the locally compact case can be obtained similarly to the proof of
Theorem \ref{convergence}.

\begin{proof}
By Lemma \ref{closed subspaces}(a) $PM_\tau(X)$ is closed in $M^+_\tau(X)$ in the $C_b(X)$-weak topology,
hence by Theorem \ref{KR-metric}(b) also in the KR-topology.
So it suffices to prove the first assertion of the theorem.

We first consider the case of a separable $X$.
Suppose that a sequence $\mu_n\in M^+_\sigma(X)$ is Cauchy in the KR-uniformity.
Since $|\mu-\nu|_1=\big|\,||\mu||-||\nu||\,\big|$ for nonnegative $\mu$ and $\nu$, the sequence $\mu_n$
is also Cauchy with respect to the total variation norm $||\cdot||$, and hence belongs to $M^r_\sigma(X)$ 
for some $r$.
Let $i$ be an embedding of $X$ in the Hilbert cube $I^\infty$.
By Lemma \ref{weak*}(b) $M^r(I^\infty)$ is compact, so the sequence $i_*\mu_n$ has a subsequence 
$i_*\mu_{n_k}$ that converges to some $\nu\in M^r(I^\infty)=M^r_\sigma(I^\infty)$ in 
the $C_b(I^\infty)$-weak topology.%
\footnote{Let us note the implicit use of Theorem \ref{isomorphism}(a) to extract $\nu$ out of a functional, 
and of Lemma \ref{extension}(4) to establish that $\nu$ is a measure.}
By Lemma \ref{closed subspaces}(b) $\nu$ is nonnegative.
Suppose that we have shown that $\nu(I^\infty\but X)=0$.
Then $\nu=i_*\mu$, where $\mu=j^!\nu\in M^+_\sigma(X)$. 
By Lemma \ref{emb-measures}(a) $i_*\:M^+_\sigma(X)\to M^+_\sigma(I^\infty)$ is an embedding, so 
then $\mu_{n_k}$ converges to $\mu$ in the $C_b(X)$-weak topology.
Then by Theorem \ref{KR-metric}(b) $\mu_{n_k}\to\mu$ also in the KR-topology.
Since $\mu_n$ is Cauchy in the KR-uniformity, it must also converge to $\mu$ in the KR-topology.
Thus $M^+_\sigma(X)$ complete in the KR-uniformity.

It remains to show that $\nu(I^\infty\but X)=0$.
Let us write $\nu_k=\mu_{n_k}$.
We are given that $i_*\nu_k$ converge to $\nu$ in the $C_b(I^\infty)$-weak topology
and that $\nu_k$ are Cauchy in the KR-uniformity.
Then for each $\eps>0$ there exists an $n$ such that for all $k,l\ge n$ and every 1-Lipschitz
$f\:X\to [-1,1]$ we have $|\nu_i-\nu_k|_f<\eps$.

Given an $\eps>0$ and an $n\in\N$, there exists an $l_n\in\N$ such that for all $k>l_n$ and 
every $2^n$-Lipschitz%
\footnote{Let us note that $l_n$ depends not on $f$, but only on its Lipschitz constant.
This is why the present proof is considerably easier than that of the Prokhorov Theorem (see 
\cite{Bog}*{8.6.2}).}
 $f\:X\to [-1,1]$ we have $|\nu_k-\nu_{l_n}|_{2^{-n}f}<4^{-n}\eps$, and consequently 
$|\nu_k-\nu_{l_n}|_f<2^{-n}\eps$.
Let us write $\lambda_n=\nu_{l_n}$.
Thus $|\lambda_k-\lambda_n|_f<2^{-n}\eps$ for all $k>n$ and every $2^n$-Lipschitz $f\:X\to [-1,1]$.
Since $\lambda_n$ is Radon, there exists a compact $K_n\subset X$ such that $\lambda_n(X\but K_n)<2^{-n}\eps$.
Let $U_n$ be the open $2^{-n}$-neighborhood of $K_n$.
Let $f\:X\to[0,1]$ be a $2^n$-Lipschitz function sending $X\but U_n$ to $0$ and $K_n$ to $1$; namely,
let $f(x)=\min\big(1,2^nd(x,\,X\but U_n)\big)$.
Then from $|\lambda_k-\lambda_n|_f<2^{-n}\eps$ we get
$\lambda_k(U_n)\ge|\lambda_k|_f\ge|\lambda_n|_f-2^{-n}\eps\ge\lambda_n(K_n)-2^{-n}\eps\ge\lambda_n(X)-2^{1-n}\eps$ 
for all $k>n$.

Thus $\lambda_k(X\but U_n)\le 2^{1-n}\eps$ for all $n<k$.
Let $V_n=\bigcup_{i=1}^n U_i$.
Then $\lambda_k(X\but V_n)\le (2^0+\dots+2^{1-n})\eps<2\eps$ for all $k>n$.
Let $Z_n$ be the closure of $V_n$ in $I^\infty$.
Hence $i_*\lambda_k(I^\infty\but Z_n)\le i_*\lambda_k(I^\infty\but V_n)=\lambda_k(X\but V_n)<2\eps$
for all $k>n$.
Since $i_*\lambda_k$ converge to $\nu$ in the $C_b(I^\infty)$-weak topology,
by Lemma \ref{closed subspaces}(c) $\nu(I^\infty\but Z_n)\le 2\eps$.
Let $Z=\bigcap_{i=1}^\infty Z_n$.
Since $\nu$ is countably additive and $Z_1\supset Z_2\supset\dots$, we have $\nu(I^\infty\but Z)\le 2\eps$.
On the other hand, by Proposition \ref{corona}(c) $Z\subset X$ (using that $X$ is complete).
Hence $\nu(I^\infty\but X)\le 2\eps$.
Since $\eps>0$ is arbitrary, $\nu(I^\infty\but X)=0$.

Finally, we consider the general case (where $X$ is not assumed to be separable). 
Suppose that a sequence $\mu_n\in M^+_\tau(X)$ is Cauchy in the KR-uniformity.
Lemma \ref{tau} yields separable closed sets $Z_n\subset X$ such that $|\mu_n|(X\but Z_n)=0$ for each $n$.
Let $Z$ be the closure of their union.
Then $Z$ is separable and complete and each $|\mu_n|(X\but Z)=0$.
Let $i\:Z\to X$ be the inclusion and let $\lambda_n=i^!\mu_n\in M^+_\sigma(Z)$.
Since $|\mu_n|(X\but Z)=0$, we have $i_*\lambda_n=\mu_n$.
By Lemma \ref{KR-restriction} the sequence $\lambda_n$ is Cauchy in the KR-uniformity.
Hence by the above it converges to some $\lambda\in M^+_\sigma(Z)$ in the KR-topology.
Then by Lemma \ref{KR-restriction} $i_*\lambda_n\to i_*\lambda$ in the KR-topology,
where $i_*\lambda$ is $\tau$-additive by Lemma \ref{tau}.
\end{proof}

The completion of the entire $M_\tau(X)$ in the KR-uniformity (which is not a uniform unvariant of $X$)
will be discussed in \S\ref{Lipschitz measures}.

\begin{example} \label{incompleteness}
(compare \cite{We}*{3.24}) $M_\tau([0,1])$ is not complete in the KR-uniformity
(with respect to the usual metric on $[0,1]$).

Indeed, let $\mu_n=\delta_{2^{-2n+1}}-\delta_{2^{-2n}}$ and $\nu_{mn}=\sum_{i=m}^n\mu_i$.
Then $||\mu_n||_{KR}=|2^{-2n+1}-2^{-2n}|=2^{-2n}$.
Hence 
$||\nu_{1n}-\nu_{1m}||_{KR}=||\nu_{mn}||_{KR}\le\sum_{i=m}^n 2^{-2n}\le\sum_{i=m}^\infty 4^{-n}=\frac43 4^{-m}$.
Thus $\nu_{1n}$ is a Cauchy sequence in the KR-uniformity on $M_\tau([0,1])$. 

Let us show that $\nu_{1n}$ has no limit in $M_\tau([0,1])$ and more generally $\Phi_n\bydef \Phi_{\nu_{1n}}$
has no limit in $\Lip_b^*([0,1])$ in the KR-topology.
Suppose on the contrary that $||\Phi_n-\Phi||_{KR}\to 0$, where $\Phi\in\Lip_b^*([0,1])$.

The usual partial order $f\le g\Leftrightarrow\forall x\ f(x)\le g(x)$ on $\Lip_b([0,1])$ admits suprema
and infima, defined by $(f\lor g)(x)=\max(f(x),g(x))$ and $(f\land g)=\min(f(x),g(x))$.
It also yields the dual partial order $\Phi\le\Psi\Leftrightarrow\forall f\ge 0\ \,\Phi(f)\le\Psi(f)$
on $\Lip_b^*([0,1])$, which also admits suprema and infima, and in particular the operator $+$, defined by
$\Phi_+(f)=\sup\{\Phi(g)\mid 0\le g\le f\}$.

Let $f_n\:[0,1]\to[0,1]$ be a $2^{2n}$-Lipschitz function that vanishes on $[0,2^{-2n}]$ and at the points
$2^{-2},2^{-4},\dots,2^{-2n+2}$, and takes the value $1$ at the points $2^{-1},2^{-3},\dots,2^{-2n+1}$.
Then $\Phi_m(f_n)=\int_{[0,1]}f_n\,d\nu_{1m}=n$ for all $m\ge n$.
Since $|\Phi_m(f_n)-\Phi(f_n)|\le 2^{2n}||\Phi_m-\Phi||_{KR}\to 0$ as $m\to\infty$, we get $\Phi(f_n)=n$.
Then $\Phi_+(1)\ge\Phi_+(f_n)\ge\Phi(f_n)=n$ for all $n\in\N$, which is a contradiction.
\end{example}

\section{Kantorovich metric} \label{Kantorovich}

\subsection{Bounded metrics}

Let $Z$ be a metric space.
The vector space $\Lip(Z)$ of all (possibly unbounded) Lipschitz functions $Z\to\R$ contains the $1$-dimensional
subspace $\R$ of $0$-Lipschitz (i.e.\ constant) functions.
Let $\Lip_{qc}(Z)$ be the quotient space $\Lip(Z)/\R$ with the norm 
$\Lip([f])=\inf_{x\ne y}\big|\big|\frac{f(x)-f(y)}{d(x,y)}\big|\big|$.
Let $||\cdot||_K$ denote the dual norm on $\Lip_{qc}^*(Z)$.

Let us assume that $Z$ is a bounded metric space of diameter $2R$.
Let $\lambda\in M^0_\sigma(Z)$; that is, $\lambda$ is a measure on $Z$ with $\lambda(Z)=0$.
It is easy to see that $\lambda(Z)=0$ implies $|\lambda|_{f+c}=|\lambda|_f$ for any $c\in\R$ and 
any $f\in\Lip(Z)$.
Hence $\lambda$ is identified with a linear functional on $\Lip_{qc}(Z)$.
Thus $||\lambda||_K$ is the supremum of $|\lambda|_f=\big|\int_Z f\,d\lambda\big|$ 
over all $1$-Lipschitz $f\:Z\to\R$.

Since all $1$-Lipschitz functions $Z\to\R$ include all $1$-Lipschitz functions $Z\to[-1,1]$,
we have $||\lambda||_K\ge||\lambda||_{KR}$.
On the other hand, if $f$ is $1$-Lipschitz, then $f(Z)$ is of diameter $\le 2R$, and consequently lies 
in $[c-R,c+R]$ for some $c\in\R$.
Since $|\lambda|_{f-c}=|\lambda|_f$, we may assume that $c=0$.
Hence $||\lambda||_K$ equals the supremum of $|\lambda|_f$ over all $1$-Lipschitz $f\:Z\to[-R,R]$.
Thus
\[||\lambda||_{KR}\le||\lambda||_K\le \max(R,1)||\lambda||_{KR}.\]
In particular, $||\lambda||_K=||\lambda||_{KR}$ if $R\le 1$.

By assigning $||\mu-\nu||_K$ to given probability measures $\mu$, $\nu$ we define the {\it Kantorovich metric} 
(also known as the Wasserstein/Vasershtein metric) on $PM_\sigma(Z)$ (see \cite{Kan}, \cite{BK}).
We have proved

\begin{proposition} \label{K=KR}
If $Z$ is a metric space of diameter $\le 2$ (resp.\ of finite diameter), then the Kantorovich metric 
on $PM_\sigma(Z)$ is equal (resp.\ uniformly equivalent) to the Kantorovich--Rubinstein metric, 
and the Kantorovich norm on $M^0_\sigma(Z)$ is equal (resp.\ uniformly equivalent) to 
the Kantorovich--Rubinstein norm.
\end{proposition}

\subsection{Dirac measures}
Now let $Z$ be an arbitrary (possibly unbounded) metric space.
Let $\lambda\in M^0_\delta(Z)$; that is, $\lambda$ is a molecular measure on $Z$ with $\lambda(Z)=0$.
Then $\lambda$ still satisfies $|\lambda|_{f+c}=|\lambda|_f$, and so is identified with a linear functional 
on $\Lip_{qc}(Z)$.
Let us note that we still have $||\lambda||_K\ge||\lambda||_{KR}$.

Like before, the Kantorovich metric is also defined on each $\partial M^r_\delta(Z)$, in particular 
on $PM_\delta(Z)$.
With this metric, the embedding $Z\to PM_\delta(Z)$, $x\mapsto\delta_x$, is an isometry onto its image.
Indeed, if $f\:Z\to\R$ is a 1-Lipschitz function, for each $a,b\in Z$ we have 
$|\delta_a-\delta_b|_f=|f(a)-f(b)|\le d(a,b)$, where the inequality turns into equality for $f(x)=d(a,x)$.
Hence $||\delta_a-\delta_b||_K=d(a,b)$.

Since every metric is uniformly equivalent to a bounded metric, we obtain 

\begin{proposition} \label{embKR}
The embedding $Z\to PM_\delta(Z)$, $x\mapsto\delta_x$, is an isometry onto its image with respect to
the Kantorovich metric, and a uniform embedding with respect to the KR-uniformity.
\end{proposition}

\subsection{Unbounded metrics}
Let $Y$ be a metric space of diameter at most $2$.
Then by adjoining to $Y$ one additional point, denoted $\infty$, with $d(\infty,y)=1$ for all $y\in Y$,
we obtain a new metric space $Y^+$.

\begin{lemma} If $Y$ is a metric space of diameter $\le2$, then $M_\sigma(Y)$ with 
the Kantorovich--Rubinstein norm is linearly isometric to $M^0_\sigma(Y^+)$ with the Kantorovich norm.
\end{lemma}

\begin{proof}
Every measure $\mu\in M_\sigma(Y)$ uniquely extends to a measure $\mu^+\in M^0_\sigma(Y^+)$ by setting 
$\mu^+=\mu-\mu(Y)\delta_\infty$.
Since $Y^+$ is of diameter $\le 2$, we have $||\mu^+||_K=||\mu^+||_{KR}$.
Hence it suffices to show that $||\mu^+||_{KR}=||\mu||_{KR}$.

Given a $1$-Lipschitz function $f\:Y\to[-1,1]$, let us extend it to 
a $1$-Lipschitz function $f^+\:Y^+\to[-1,1]$ by setting $f^+(\infty)=0$.
This yields $||\mu||_{KR}\le||\mu^+||_{KR}$.
Given a $1$-Lipschitz function $F\:Y^+\to[-1,1]$, let us define $f\:Y\to[-1,1]$ by
$f(y)=F(y)-F(\infty)$.
Then $\int_Y f\,d\mu=\int_Y F\,d\mu-F(\infty)\mu(Y)=\int_{Y^+} F\,d\mu^+$.
This yields $||\mu||_{KR}\ge||\mu^+||_{KR}$.
\end{proof}

Given an arbitrary metric space $X$, let $X_2$ be the metric space obtained by replacing the given 
metric $d$ with the metric $d'$ defined by $d'(x,y)=\min\big(d(x,y),2\big)$.
Thus $X_2$ is of diameter at most $2$.
It is easy to see that a map $f\:X\to [-1,1]$ is $1$-Lipschitz with respect to $d$ if and only if
it is $1$-Lipschitz with respect to $d'$.
Thus $\Lip_{bl}^*(X)$ and $\Lip_{bl}^*(X_2)$ are linearly isometric.
In particular, $M_\sigma(X)$ and $M_\sigma(X_2)$ with the Kantorovich--Rubinstein norms 
are linearly isometric.
Thus we obtain

\begin{corollary} If $X$ is a metric space, then $M_\sigma(X)$ with the Kantorovich--Rubinstein norm
is linearly isometric to $M^0_\sigma(X_2^+)$ with the Kantorovich norm.
\end{corollary}

A version of this observation is found in \cite{We}*{2.14}.

\subsection{Arens--Eels construction} \label{AE-construction}

The Arens--Eels space $AE(Z)$ (going back to \cite{AE} and \cite{Mi}) is the completion of $M_\delta^0(Z)$ 
with respect to $||\cdot||_K$. 
Let us recall that the dual of a normed space is always complete.
So $AE(Z)$ may also be defined as the closure of $M_\delta^0(Z)$ in $\Lip_{qc}^*(Z)$.
When $Z$ is bounded, $\Lip_{qc}^*(Z)$ contains $M_\sigma^0(Z)$, and in this case $AE(Z)$ contains $M_\tau^0(Z)$ 
by Theorem \ref{KR-metric}(d).

Upon fixing a basepoint $pt\in Z$ we also have the space $\Lip_{pt}(Z)$ of all (possibly unbounded) Lipschitz 
functions $(Z,pt)\to(\R,0)$ with the norm $\Lip(f)=\inf_{x\ne y}\big|\big|\frac{f(x)-f(y)}{d(x,y)}\big|\big|$.
Clearly, $\Lip_{pt}(Z)$ is linearly isometric to $\Lip_{qc}(Z)$ via $f\mapsto [f]$, with inverse given by
$[f]\mapsto f-f(pt)$.
Hence the dual space $\Lip_{qc}^*(Z)$ is linearly isometric to $\Lip_{pt}^*(Z)$ via 
$\Phi\mapsto\big(f\mapsto\Phi([f])\big)$, with inverse given by
$\Phi\mapsto\Big([f]\mapsto\Phi\big(f-f(pt)\big)\Big)$.
This linear isometry restricts to the vector space isomorphism given by the composition 
$M_\delta^0(Z)\to M_\delta(Z)\to M_\delta(Z)/\left<\delta_{pt}\right>$.
In particular, $AE(Z)$ is identified with the closure of $M_\delta(Z)/\left<\delta_{pt}\right>$ in
$\Lip_{pt}^*(Z)$.
The dual norm of $\Lip_{pt}^*(Z)$ will be still denoted $||\cdot||_K$.

Using the basepoint $pt$, we can also isometrically embed $PM_\delta(Z)$ with the Kantorovich metric
into $M_\delta^0(Z)$ with the Kantorovich norm, via $\mu\mapsto\mu-\delta_{pt}$. 
By composing this embedding with the linear isometry $\Lip_{qc}^*(Z)\to\Lip_{pt}^*(Z)$ we obtain that
$\mu\mapsto[\mu]$ yields an isometric embedding of $PM_\delta(Z)$ onto the subspace
$PM_\delta(Z)/\left<\delta_{pt}\right>$ of $\Lip_{pt}^*(Z)$.

For a metric space $Y$ of diameter $\le 2$, $k$-Lipschitz functions $(Y^+,\infty)\to(\R,0)$
can be identified with $k$-Lipschitz functions $Y\to[-k,k]$.
In particular, the underlying vector space of $\Lip_\infty(Y^+)$ gets identified with that of $\Lip_{bl}(Y)$.
Moreover, the norm $\Lip(\cdot)$ of $\Lip_\infty(Y^+)$ gets identified with the norm $||\cdot||_{bl}$
of $\Lip_{bl}(Y)$.
Consequently, the norm $||\cdot||_{\Lip_\infty}$ of $\Lip_\infty^*(Y^+)$ is identified with 
the Kantorovich--Rubinstein norm $||\cdot||_{KR}$ of $\Lip_{bl}^*(Y)$.
From this and Theorem \ref{KR-metric}(d) we obtain

\begin{proposition} \label{AE-KR}
Let $Y$ be a metric space of diameter $\le 2$.
Then $AE(Y^+)$ is linearly isometric to the closure of $M_\tau(Y)$ in $\Lip_{bl}^*(Y)$.
\end{proposition}

Taking into account the linear isometry between $\Lip_{bl}^*(X)$ and $\Lip_{bl}^*(X_2)$, we also get

\begin{corollary} \label{AE-KR2}
Let $X$ be a metric space.
Then $AE(X_2^+)$ is linearly isometric to the closure of $M_\tau(X)$ in $\Lip_{bl}^*(X)$.
\end{corollary}

An intrinsic description of the closure of $M_\tau(X)$ in $\Lip_{bl}^*(X)$ will be discussed
in \S\ref{Lipschitz measures}.

\begin{example} (compare \cite{We}*{3.10}) If $\N$ is endowed with the metric such that $d(x,y)=2$ whenever 
$x\ne y$, then $M_\delta(\N)$ (with the Kantorovich--Rubinstein metric) is linearly isometric with $l_1^f$, 
and so its completion $AE(\N^+)$ is linearly isometric with $l_1$.
\end{example}

\subsection{Universal property}
The Arens--Eels space $AE(Z)$ is also known as the ``Lipschitz-free Banach space'', due to the following

\begin{theorem} \label{Lipschitz extension}
If $(Z,pt)$ is a pointed metric space and $B$ is a Banach space, every Lipschitz map $f\:(Z,pt)\to (B,0)$ 
uniquely extends to a continuous linear map $\bar f\:AE(Z)\to B$.
Moreover, $||\bar f||=\Lip(f)$.
\end{theorem}

This result was obtained independently by Flood \cite{Fl}*{3.1.4} and Pestov \cite{Pe} and also 
rediscovered by Weaver (see \cite{We}).

\begin{proof} 
Since $Z$ spans $M_\delta^0(Z)$ as a vector space, $f$ uniquely extends to a linear map 
$F\:M_\delta^0(Z)\to B$.
Let us show that $||F||=k$, where $k=\Lip(f)$.

Let $\lambda=\sum_i\lambda_i\delta_{x_i}\in M_\delta^0(Z)$.
By the Hahn--Banach theorem there exists a continuous linear functional $\phi\:B\to\R$ such that
$||F(\lambda)||=\phi\big(F(\lambda)\big)$ and $||\phi||\le 1$.
Thus $\phi$ is a $1$-Lipschitz map.
Hence $\phi f$ is $k$-Lipschitz.
On the other hand, $\phi F(\lambda)=\int_X \phi f\,d\lambda$, since each 
$\phi F(\delta_{x_i})=\phi f(x_i)=\int_X\phi f\,d\delta_{x_i}$.
Since $\phi f\:(Z,pt)\to (B,0)$ is $k$-Lipschitz, $\int_X \phi f\,d\lambda\le k||\lambda||_{\Lip_{pt}}$.
Thus $||F||=\sup_{\lambda\ne 0}\frac{||F(\lambda)||}{||\lambda||_{\Lip_{pt}}}\le k$.
But $||F||=k'<k$ would imply that $F$, and hence $f$, is $k'$-Lipschitz, contradicting $k=\Lip(f)$.
So $||F||=k$.
In particular, $F$ is continuous, and therefore uniformly continuous.
Since $M_\delta^0(Z)$ is dense in $AE(Z)$ and $B$ is complete, $F$ uniquely extends to a continuous 
linear map $\bar f\:AE(Z)\to B$ with $||\bar f||=||F||=\Lip(f)$.
\end{proof}

\begin{example} (compare \cite{We}*{3.11}) Let $\R$ be endowed with the usual metric and with basepoint 
at $0$.
Let us define a linear map $T$ from $M_\delta^0(\R)$ to $L_1(\R)$ by 
\[T(\delta_x-\delta_0)=\begin{cases}
\chi_{[0,x)}&\text{ if }x>0,\\
0&\text{ if }x=0,\\
-\chi_{[x,0)}&\text{ if }x<0.
\end{cases}
\]
Then $T(\delta_x-\delta_y)=\chi_{[x,y)}$ whenever $x>y$.
Since $T$ is $1$-Lipschitz on the basis vectors, by Theorem \ref{Lipschitz extension} 
$||T(\lambda)||=||\lambda||_K$ whenever $\lambda(X)=0$.
Hence $T$ extends to a linear isometry between $AE(\R)$ and $L_1(\R)$.
\end{example}

It is shown in \cite{DTZ} that $AE(l_1)$ is also linearly isometric with $L_1(\R)$.

\subsection{Transportation of masses}

\begin{theorem} \label{transportation} {\rm (see \cite{Bog}*{8.10.45}, \cite{Ed})} 
Let $(X,D)$ be a bounded metric space and let $\mu,\nu\in PM_\rho(X)$.
Then \[||\mu-\nu||_K=\inf_\lambda\int\limits_{X\x X} D\ d\lambda,\]
where the infimum is over all $\lambda\in PM_\rho(X\x X)$ such that $\mu(B)=\lambda(B\x X)$ and 
$\nu(B)=\lambda(X\x B)$ for all $B\in\B(X)$. 
Moreover, the infimum is attained.
\end{theorem}

If we understand $\mu$ and $\nu$ as distributions of masses (for instance, $\mu$ is the produce of iron mines
and $\nu$ is the demand of steel factories), the infimum can be interpreted as the minimal amount of work 
needed to transport $\mu$ into $\nu$; and a measure $\lambda$ on $X\x X$ at which the infimum
is attained --- as an optimal transportation plan.

The case of molecular measures also holds for unbounded metric spaces:

\begin{theorem} \label{transportation2} {\rm (see \cite{CDW}*{\S1})}
Let $X$ be a metric space and $\lambda\in M_\delta^0(X)$.
Then \[||\lambda||_K=||\lambda||_K'\bydef \inf\sum_{k=1}^r|\lambda_k| d(x_k,y_k),\]
where the infimum is over all representations $\lambda=\sum_{k=1}^r\lambda_k(\delta_{x_k}-\delta_{y_k})$.
\end{theorem}

\begin{proof}[Proof of Theorem \ref{transportation2}]
Clearly, $||\delta_x-\delta_y||_K=d(x,y)$ for all $x,y\in X$.
Hence by Lemma \ref{KR-max} $||\lambda||_K\le ||\lambda||_K'$ for all $\lambda\in M_\delta^0(X)$.

It is easy to see that $||\cdot||_K'$ is a seminorm, and since it is bounded below by the norm
$||\cdot||_K$, it is a norm.
Also, from $d(x,y)\le ||\delta_x-\delta_y||_K\le ||\delta_x-\delta_y||_{K'}\le d(x,y)$
we get $||\delta_x-\delta_y||_{K'}=d(x,y)$.
Then by Theorem \ref{Lipschitz extension}, the isometric embedding $X\to\big(M_\delta(X),||\cdot||_{K'}\big)$ 
extends to a linear 1-Lipschitz map $\big(M_\delta(X),||\cdot||_K\big)\to \big(M_\delta(X),||\cdot||_{K'}\big)$.
Being the identity on the basis vectors, this linear map must be the identity.
Hence $||\lambda||_K\ge ||\lambda||_K'$ for all $\lambda\in M_\delta^0(X)$.
\end{proof}

\begin{lemma} \label{KR-max} \cite{AE} If $||\cdot||_\circ$ is a seminorm on $M_\delta^0(X)$ such that 
$||\delta_x-\delta_y||_\circ\le d(x,y)$ for all $x,y\in X$, then $||\lambda||_\circ\le ||\lambda||_K'$
for all $\lambda\in M_\delta^0(X)$.
\end{lemma}

\begin{proof} Let $\lambda=\sum_i\lambda_i(\delta_{x_i}-\delta_{y_i})$.
Then 
$||\lambda||_\circ\le\sum_i|\lambda_i|\cdot||\delta_{x_i}-\delta_{y_i}||_\circ\le\sum_i|\lambda_i|d(x_i,y_i)$.
Since this holds for all representations $\lambda=\sum_i\lambda_i(\delta_{x_i}-\delta_{y_i})$, 
we get $||\lambda||_\circ\le||\lambda|_K'$.
\end{proof}

\begin{remark} If we redefine $AE(X)$ as the completion with respect to $||\cdot||_K'$ (rather than
$||\cdot||_K$), then Theorem \ref{Lipschitz extension} can be proved without using the Hahn--Banach theorem. 
Namely, in the notation of the proof of Theorem \ref{Lipschitz extension}, let us define 
a seminorm $||\cdot||_\circ$ on $M_\delta^0(X)$ by $||\lambda||_\circ=\frac1k||F(\lambda)||$.
Then $||\delta_x-\delta_y||_\circ=\frac1k||f(x)-f(y)||\le d(x,y)$.
Hence by Lemma \ref{KR-max} $||\lambda||_\circ\le||\lambda||_K'$, that is,
$\frac{||F(\lambda)||}{||\lambda||_{K'}}\le k$.
Thus $||F||\le k$.
\end{remark}

\begin{proposition} \label{point masses}
Let $\nu=\sum_{i=1}^l\lambda_i\delta_{a_i}-\sum_{j=1}^m\mu_j\delta_{b_j}\in M^0_\delta(X)$,
where each $\lambda_i>0$ and each $\mu_j>0$.
Then $||\nu||_K=\sum_{i=1}^l\sum_{j=1}^m\nu_{ij} d(a_i,b_j)$ for some representation
$\nu=\sum_{i=1}^l\sum_{j=1}^m\nu_{ij}(\delta_{a_i}-\delta_{b_j})$ where each $\nu_{ij}\ge 0$.
\end{proposition}

Let us note that such a representation is equivalent to a solution of the following system of equations: 
$\lambda_i=\sum_{j=1}^m\nu_{ij}$ for $1\le i\le l$ and $\mu_j=\sum_{i=1}^l\nu_{ij}$ for $1\le j\le m$.

\begin{proof} We may assume without loss of generality that $a_i\ne b_j$ for each $i,j$.

Suppose that $\nu=\sum_{k=1}^r\nu_k(\delta_{x_k}-\delta_{y_k})$, where each $\nu_k>0$.
Given a $p\in X$, let $S^+_p=\{k\in[r]\mid x_k=p\}$ and $S^-_p=\{k\in [r]\mid y_k=p\}$.
We may assume that $S^+_p\cap S^-_p=\emptyset$ for each $p\in X$.
Let $N^\pm_p=\sum_{k\in S^\pm_p}\nu_k$ and let $N=\sum_{k=1}^r\nu_k d(x_k,y_k)$.
Clearly, $N^+_{a_i}-N^-_{a_i}=\lambda_i$, $N^-_{b_j}-N^+_{b_j}=\mu_j$ and $N^+_p-N^-_p=0$ if
$p\notin\{a_1,\dots,a_l,\,b_1,\dots,b_m\}$.
For such a $p$, we may assume that after a partitioning of each $\nu_k$ into a sum of positive reals 
(which partitioning does not affect $N$ and each $N^\pm_p$), there is a bijection $f\:S^+_p\to S^-_p$ 
such that $\nu_k=\nu_{f(k)}$.
Since also $y_{f(k)}=p=x_k$ for each $k\in S^+_p$, we have 
$\sum_{k\in S_p\cup T_p}\nu_k(\delta_{x_k}-\delta_{y_k})=\sum_{k\in S_p}\nu_k(\delta_{x_{f(k)}}-\delta_{y_k})$ 
and $\sum_{k\in S_p\cup T_p}\nu_k d(x_k,y_k)\ge\sum_{k\in S_p}\nu_k d(x_{f(k)},y_k)$ due to
$d(x_{f(k)},y_{f(k)})+d(x_k,y_k)=d(x_{f(k)},p)+d(p,y_k)\ge d(x_{f(k)},y_k)$.
Thus we may assume without increasing $N$ that $N^+_p=N^-_p=0$.
By a similar argument, we may also assume without increasing $N$ that $N^-_{a_i}=0$ for each $i\in [n]$ 
and $N^+_{b_j}=0$ for each $j\in [m]$.
Thus each $x_k\in\{a_1,\dots,a_l\}$ and each $y_k\in\{b_1,\dots,b_m\}$.
Finally, after collecting terms of the form $\nu_k(\delta_{a_i}-\delta_{b_j})$ (which does not affect $N$),
we may assume without increasing $N$ that $\nu=\sum_{i\in[l],\,j\in[m]}\nu_{ij}(\delta_{a_i}-\delta_{b_j})$.

Thus $||\nu||_K=\inf\sum_{i=1}^l\sum_{j=1}^m\nu_{ij} d(a_i,b_i)$, where the infimum is over the set 
$P\subset\R^{lm}$ of all $lm$-tuples $(\nu_{ij})_{i\in[l],\,j\in[m]}$ such that each $\nu_{ij}\ge 0$ and
$\nu=\sum_{i=1}^l\sum_{j=1}^m\nu_{ij}(\delta_{a_i}-\delta_{b_j})$. 
Since the latter equation is equivalent to the system of equations $\lambda_i=\sum_{j=1}^m\nu_{ij}$ and 
$\mu_j=\sum_{i=1}^l\nu_{ij}$, in fact $P$ is a convex polytope lying in the box 
$\prod_{i\in[l],\,j\in[m]}[0,\min(\lambda_i,\mu_j)]$.
In particular $P$ is compact, and so $\sum_{i=1}^l\sum_{j=1}^m\nu_{ij} d(a_i,b_i)$ attains its infimum on $P$.
\end{proof}

\begin{proposition} \label{K-closed} (a) $X$ is closed in $PM_\delta(X)$ in the Kantorovich metric. 

(b) $X-\delta_{x_0}$ is closed in $M^0_\delta(X)$ in the Kantorovich norm.
\end{proposition}

\begin{proof}[Proof. (a)] If $\lambda=\sum_{i=1}^l\lambda_i\delta_{a_i}\in PM_\delta(X)$, where each 
$\lambda_i>0$, then by Proposition \ref{point masses} $||\lambda-\delta_x||_K=\sum_{i=1}^l\lambda_id(a_i,x)$.
It follows that if $\delta_{x_n}\to\lambda$ in the Kantorovich metric, then $x_n\to a_i$ for each $i$
such that $\lambda_i>0$.
\end{proof}

\begin{proof}[(b)]
Suppose that we are given a sequence of points $x_n\in X$ such that $\delta_{x_n}-\delta_{x_0}$ converges to 
some $\lambda\in M^0_\delta(X)$ in the Kantorovich norm.
Suppose that $\lambda_+=\sum_{i=1}^l\lambda^+_i\delta_{a_i}$ and $-\lambda_-=\sum_{j=1}^m\lambda^-_j\delta_{b_j}$,
where each $\lambda^+_i>0$ and each $\lambda^+_j>0$.
Let us note that $a_i\ne b_j$ for all $i,j$.
By Proposition \ref{point masses} for each $n$ there exists a $\mu_n\in M^0_\delta(X)$ such that
$(\mu_n)_+=\sum_{i=1}^l\mu^+_{ni}\delta_{a_i}$ and $-(\mu_n)_-=\sum_{j=1}^m\mu^-_{nj}\delta_{b_j}$, where each 
$\lambda^+_i\ge\mu^+_{ni}\ge 0$ and each $\lambda^+_j\ge\mu^-_{nj}\ge 0$, and such that
$||\lambda-\delta_{x_n}+\delta_{x_0}||_K=||\lambda-\mu_n||_K+
\sum_{i=1}^l\mu^+_{ni}d(a_i,x_n)+\sum_{j=1}^m\mu^-_{nj}d(b_j,x_0)$.
Let $\eps=\min_{i,j}d(a_i,b_j)$ and let $\gamma=\min(\min_i\lambda^+_i,\min_j\lambda^-_j)$.
For a sufficiently large $n$ we have $||\lambda-\delta_{x_n}+\delta_{x_0}||_K\le\eps\gamma/2$ and consequently
$||\lambda-\mu_n||_K\le\eps\gamma/2$.
Then each $\lambda^+_i-\mu^+_{ni}\le\gamma/2\le\lambda^+_i/2$ and each 
$\lambda^-_j-\mu^-_{nj}\le\gamma/2\le\lambda^-_j/2$.
Hence each $\mu^+_{ni}\ge\lambda^+_i/2$ and each $\mu^-_{nj}\ge\lambda^-_j/2$.
Therefore $||\lambda-\delta_{x_n}+\delta_{x_0}||_K\ge 
\frac12\sum_{i=1}^l\lambda^+_id(a_i,x_n)+\frac12\sum_{j=1}^m\lambda^-_jd(b_j,x_0)$.
Since the left hand side tends to $0$, it follows that each $b_j=x_0$ and that $x_n\to a_i$ for each $i$.
Hence $\lambda=\delta_x-\delta_{x_0}$ for some $x\in X$.
\end{proof}

\section{Measure-like spaces}

\subsection{Lipschitz measures} \label{Lipschitz measures}

Let $X$ be a metric space, and let us fix a 1-Lipschitz function $h\:X\to[0,\infty)$.
Let $\Lip_h(X)$ be the vector space consisting of all $k$-Lipschitz functions $f\:Z\to\R$ such that 
$|f(x)|\le kh(x)$ for all $x\in X$, where $k\in [0,\infty)$.
We endow $\Lip_h(X)$ with the norm $||f||_h$ equal to the infimum of all $k$ satisfying the above. 

\begin{example} (a) $\Lip_1(X)=\Lip_{bl}(X)$, where $1$ denotes the constant function $x\mapsto 1$.

(b) Let $pt\in X$ and let $h_{pt}$ be defined by $h_{pt}(x)=d(pt,x)$.
Then $\Lip_{h_{pt}}(X)=\Lip_{pt}(X)$.
\end{example}

Let $L_{X,h}$ be the unit ball of $\Lip_h(X)$, in other words, the set of all $1$-Lipschitz functions 
$f\:X\to\R$ that are bounded by $h$ (in the sense that $|f(x)|\le h(x)$ for all $x\in X$).
The dual norm $||\cdot||_h$ on $\Lip_h^*(Z)$ is defined, as usual, by setting $||\Phi||_h$ to be 
the supremum of $|\Phi(f)|$ over all $f\in L_{X,h}$.
Thus $||\cdot||_1=||\cdot||_{KR}$ and $||\cdot||_{h_{pt}}=||\cdot||_K$.

A linear functional $\Phi$ on $\Lip_h(X)$ is called an {\it $h$-Lipschitz measure} on $X$ if $\Phi$ restricted 
to $L_{X,h}$ is continuous in the topology of pointwise convergence on $L_{X,h}$.
This topology coincides with the compact-open topology on $L_{X,h}$, since a sequence of 1-Lipschitz functions 
that pointwise converges to a 1-Lipschitz function is easily seen to do so uniformly on compact sets.

Since $h,\frac h2,\frac h3,\dots$ pointwise converge to $0$, every $h$-Lipschitz measure $\Phi$ 
is continuous in the topology of the norm $||\cdot||_h$ restricted to $L_{X,h}$.
Since $L_{X,h}$ is the unit ball of this norm, $\Phi$ is continuous on the entire normed space $\Lip_h(X)$.
Thus it belongs to the dual space $\Lip_h^*(X)$, which carries the dual norm $||\cdot||_h$.
Let $M_h(X)$ be the normed space of all $h$-Lipschitz measures on $X$ with this norm.

Example \ref{incompleteness} shows that $M_1([0,1])$ is not a subset of $\Lip_b^*([0,1])$.
In particular, some $1$-Lipschitz measures on $[0,1]$ are not measures.
On the other hand, by Theorem \ref{representability2}(c) all Radon measures are $1$-Lipschitz measures.

\begin{theorem} \label{lip-m} {\rm \cite{Pac2}*{Exercises 5.33, 10.13} (see also \cite{We}*{3.3, 3.23})} 
Let $X$ be a metric space and $h\:X\to [0,\infty)$ be a 1-Lipschitz function.

(a) $M_h^*(X)=\Lip_h(X)$ as vector spaces.

(b) $M_\delta(X)/M_\delta\big(h^{-1}(0)\big)$ is dense in $M_h(X)$.

(c) $M_h(X)$ is complete.
\end{theorem}

Let us note that $X$ is not assumed to be complete.

\begin{proof}[Proof. (a)]
The assertion will follow from the Mackey--Arens theorem \cite{Scha}*{IV.3.2} (see also \cite{Pac2}*{P.8}) once 
we check that its hypotheses are satisfied. 

It is easy to see that $M_h(X)$-weak convergence is equivalent to pointwise convergence in $L_{X,h}$.
Indeed, if a sequence $f_n\in L_{X,h}$ pointwise converges to an $f\in L_{X,h}$, then by the definition of 
$M_h(X)$ we have $\Phi(f_n)\to\Phi(f)$ for each $\Phi\in M_h(X)$, that is, $f_n\to f$ also $M_h(X)$-weakly.
Conversely, assuming $M_h(X)$-weak convergence $f_n\to f$, for each $x\in X\but h^{-1}(0)$ we have 
$\delta_x(f_n)\to\delta_x(f)$, that is, $f_n(x)\to f(x)$.
Also $f_n(x)=f(x)=0$ for each $x\in h^{-1}(0)$.

On the other hand, it is easy to see that $L_{X,h}$ with the topology of pointwise convergence is closed in 
the $X$-indexed product of copies of $[-1,1]$, and therefore compact by Tikhonov's theorem.
Then by the above $L_{X,h}$ is compact in the $M_h(X)$-weak topology.
Also, $L_{X,h}$ is absolutely convex, that is, if $f,g\in L_{X,h}$ and $|\lambda|+|\mu|\le 1$, then
$\lambda f+\mu g\in L_{X,h}$.

Finally, it is easy to see that $\Phi_n\to\Phi$ in the topology of $||\cdot||_h$ on $M_h^*(X)$ if and only if 
$\Phi_n\to\Phi$ uniformly on the sets $rL_{X,h}$, $r>0$.
Indeed, the former means that for each $\eps>0$ there exists an $n_0$ such that for all $n\ge n_0$ and all 
$f\in L_{X,h}$ we have $|\Phi_n(f)-\Phi(f)|<\eps$.
This is equivalent to saying that for each $\eps>0$ and $r>0$ there exists an $n_0$ such that for all 
$n\ge n_0$ and all $f\in L_{X,h}$ we have $|\Phi_n(f)-\Phi(f)|<\frac1r\eps$; or in other words, for all 
$g\in rL_{X,h}$ we have $|\Phi_n(g)-\Phi(g)|<\eps$.
\end{proof}

\begin{proof}[(b)] Part (a) says in more detail that the natural embedding $E\:\Lip_h(X)\to M_h^*(X)$
of vector spaces, given by $E(f)(\Phi)=\Phi(f)$, is surjective.
Given an $\omega\in M^*_h(X)$ that vanishes on $M_\delta(X)/M_\delta\big(h^{-1}(0)\big)$, we have $\omega=E(f)$ 
for some $f\in\Lip_h(X)$
and $0=\omega(\delta_x)=E(f)(\delta_x)=\delta_x(f)=f(x)$ for each $x\in X\but h^{-1}(0)$.
Also trivially $f(x)=0$ for all $x\in h^{-1}(0)$.
Hence $f=0$ and consequently $\omega=0$, so $\omega$ vanishes on the entire $M_h(X)$.
Therefore $M_\delta(X)/M_\delta\big(h^{-1}(0)\big)$ is dense in $M_h(X)$ by a well-known lemma of
functional analysis (see \cite{Pac2}*{P.6}).
\end{proof}

\begin{proof}[(c)] $\Lip_h^*(X)$, being the dual of a normed space, is complete.
So it suffices to show that $M_h(X)$ is closed in $\Lip_h^*(X)$.

Suppose that $\Phi_n\to\Phi$ in $\Lip_h^*(X)$ and each $\Phi_n$ is continuous in the topology of 
pointwise convergence on $L_{X,h}$.
Given a pointwise convergent sequence $f_m\to f$ in $L_{X,h}$, we need to show that $\Phi(f_m)\to\Phi(f)$. 

Given an $\eps>0$, there exist an $n$ such that $||\Phi-\Phi_n||_h\le\frac\eps3$ and an $m_0$ such that
$|\Phi_n(f)-\Phi_n(f_m)|\le\frac\eps3$ for all $m>m_0$.
Since $f$ and $f_m$ belong to $L_{X,h}$, we must have 
$|\Phi(f)-\Phi_n(f)|\le\frac\eps3$ and $|\Phi(f_m)-\Phi_n(f_m)|\le\frac\eps3$.
Hence $|\Phi(f)-\Phi(f_m)|\le\eps$.
\end{proof}

Since $M_\delta(X)$ is KR-dense in $M_\tau(X)$ by Theorem \ref{KR-metric}(d), we have

\begin{corollary} \label{lip-closure}
Let $X$ be a metric space.

(a) $M_1(X)$ coincides with the closure of $M_\tau(X)$ in $\Lip_{bl}^*(X)$.

(b) For each $pt\in X$, $M_{h_{pt}}(X)$ coincides with the closure of $M_\delta(X)/\left<\delta_x\right>$ 
in $\Lip_{pt}^*(X)$.
\end{corollary}

Of course, the closure of $M_\delta(X)/\left<\delta_x\right>$ in $\Lip_{pt}^*(X)$ is nothing but 
a copy of $AE(X)$.
We can thus rephrase (b) by saying that $AE(X)$ consists of those linear functionals on the space 
$\Lip(X)$ of all Lipschitz functions $X\to\R$ that vanish on constant functions and are continuous on 
the set of all $1$-Lipschitz functions with the topology of pointwise convergence.
(Here linear functionals of $\Lip(X)$ that vanish on constant functions are identified with
linear functionals on $\Lip_{qc}(X)$.)
From this we also get

\begin{corollary} \label{PM-completion}
Let $PM_l(X)$ be the set of all linear functionals on $\Lip(X)$ that evaluate to $1$ on 
the constant function $1$, are nonnegative on nonnegative Lipschitz functions, and are continuous on 
the set of all $1$-Lipschitz functions with the topology of pointwise convergence.

Then the completion of $PM_\delta(X)$ in the Kantorovich metric is $PM_l(X)$ with the metric 
$d(\Phi,\Psi)=||\Phi-\Psi||_K$.
\end{corollary}

Let us recall that by Theorems \ref{KR-metric}(d) and \ref{KR-complete} the completion of $PM_\delta(X)$ 
in the Kantorovich--Rubinstein metric is $PM_\tau(X)$ with the Kantorovich--Rubinstein metric.

\subsection{Uniform measures}
Let $X$ be a uniform space and let $D$ be the set of all uniform pseudo-metrics on $X$.
(When $X$ is metrizable, we will need the entire $D$ and not just one fixed metric; see Example
\ref{unif-equiv-metr}.)
For each pseudometric $d\in D$ let $X_d$ denote the corresponding metric space, that is, the quotient 
of $X$ by the equivalence relation $x\sim y$ if $d(x,y)=0$, endowed with the metric $\bar d([x],[y])=d(x,y)$.
Clearly, every $k$-Lipschitz function $f\:X\to\R$ with respect to $d$ is the composition of the quotient map 
$q_d\:X\to X_d$ and a $k$-Lipschitz function $f'\:X_d\to\R$.

For each $d\in D$ the restriction map $U_b^*(X)\to\Lip_b^*(X_d)$ sends a functional $\Phi$ to
the functional $\Phi_d$ defined by $\Phi_d(f)=\Phi(fq_d)$.
Since $||f||_{KR}\le ||f||$, we also have the inclusion $\Lip_b^*(X_d)\subset\Lip_{bl}^*(X_d)$.
Their composition $\alpha_d\:U_b^*(X)\to\Lip_{bl}^*(X_d)$ induces a seminorm, denoted $||\cdot||_{KR(d)}$, 
on the underlying vector space of $U_b^*(X)$.
In fact, the continuous linear maps $\alpha_d$, $d\in D$, amount to
a single a continuous linear map 
\[\alpha\:U_b^*(X)\to\prod_{d\in D}\Lip_{bl}^*(X_d),\]
which is injective by Lemma \ref{UEB}(a).
Each $\Lip_{bl}^*(X_d)$ contains the closure $M_1(X_d)$ of $M_\tau(X_d)$, and we define
\[M_u(X)=\alpha^{-1}\left(\prod_{d\in D}M_1(X_d)\right),\]
endowed with the the family of seminorms $||\cdot||_{KR(d)}$, $d\in D$.
Thus $\alpha$ restricts to a linear homeomorphism between $M_u(X)$ and its image in
$\prod_{d\in D}M_1(X_d)$, topologized as a subset of the product.
Let us note that each $\alpha_d\:M_u(X)\to M_1(X_d)$ is a continuous linear map, 
as a restriction of the projection $\prod_{d\in D}M_1(X_d)\to M_1(X_d)$.

Let $M_u^r(X)=\{\Phi\in M_u(X)\mid\,||\Phi||\le r\}$ and
$M_1^r(X_d)=\{\Phi\in M_1(X_d)\mid\,||\Phi||\le r\}$ for each $r>0$. 
Let us note that each $M_1^r(X_d)$ lies in $\Lip_b^*(X)$, but $M_1(X_d)$ need not lie in $\Lip_b^*(X)$
(see Example \ref{incompleteness} and Corollary \ref{lip-closure}(a)).
Thus $M_1(X_d)\ne\bigcup_{r>0}M_1^r(X_d)$ in general.
However, $M_u(X)=\bigcup_{r>0}M_u^r(X)$ since $M_u(X)$ lies in $U_b^*(X)$.

\begin{theorem} \label{metrization}
Let $X$ be a metrizable uniform space, metrized by a metric $d$. 
Then $\alpha_d\:M_u^r(X)\to M_1^r(X_d)$ is a uniform homeomorphism for each $r>0$.
In particular, the uniformity of $M_u^r(X)$ is induced by the Kantorovich--Rubinstein metric 
associated to $d$.
\end{theorem}

A related construction in the case of $PM_\tau$ appears in \cite{Ban2}*{discussion preceding 4.20}.

\begin{proof} Given $p,q\in D$, write $p\le q$ if $p(x,y)\le q(x,y)$ for all $x,y\in X$.
Then $D$ is a poset under $\le$.
Let $p+q\in D$ be defined by $(p+q)(x,y)=p(x,y)+q(x,y)$.
Then $p+q$ is an upper bound of $p$ and $q$, and so $D$ is directed.
Since $X$ is metrizable by the metric $d$, every pseudo-metric $p\in D$ is bounded above by the metric $d+p$,
so the subset $D_0$ of $D$ consisting of metrics is cofinal in $D$.
Whenever $p\le q$ in $D$, we have a $1$-Lipschitz map $\beta^{pq}\:X_p\to X_q$.

In general, let $\beta\:X\to Y$ be a $1$-Lipschitz map between metric spaces.
Then every $k$-Lipschitz map $f\:Y\to [-k,k]$ yields a $k$-Lipschitz map $f\beta\:X\to [-k,k]$.
This yields $\beta_*\:M_1(X)\to M_1(Y)$, defined by $(\beta_*\Phi)(f)=\Phi(f\beta)$. 
Clearly, $||\beta_*\Phi||_{KR}\le||\Phi||_{KR}$ for each $\Phi$, so $\beta_*$ is $1$-Lipschitz.

Thus we get the inverse system $\big(M_1(X_p),\beta^{pq}_*;D\big)$, which we take to be in the category
uniform spaces.
Its limit is, obviously, nothing but $\alpha\big(M_u(X)\big)$.%
\footnote{Of course, this assertion also holds in the category of locally convex topological vector spaces.}
By the same token, $\alpha\big(M_u^r(X)\big)$ is the inverse limit of $\big(M_1^r(X_p),\beta^{pq}_*;D\big)$.
On the other hand, the limit does not change if we restrict the inverse system over the cofinal subset $D_0$.
By Remark \ref{well-def-remark} each $\beta^{pq}_*\:M_1^r(X_p)\to M_1^r(X_q)$, where $p,q\in D_0$, 
is a uniform homeomorphism.
Hence $\beta^{\infty d}_*\:\alpha\big(M_u^r(X)\big)\to M_1^r(X_d)$ is also a uniform homeomorphism.
\end{proof}

\begin{example}
By Example \ref{unif-equiv-metr} the KR-topology on $M_\delta([0,1]_d)$ depends on the choice of
the metric $d$. 
Hence $\alpha_d\:M_u([0,1])\to M_1([0,1]_d)$ is generally not a homeomorphism, even when restricted 
to $M_\delta([0,1])$.

It should be noted in this connection that, similarly to Example \ref{convergence-example},
the sets $M_\delta^r([0,1])\but\partial M_\delta^r([0,1])$ are not open in $M_\delta([0,1])$.
\end{example}

There is also an alternative description of $M_u(X)$.
A family of functions $S\subset U_b(X)$ is called an {\it UEB set} if it is uniformly equicontinuous 
and uniformly bounded (with respect to the usual sup norm on $U_b(X)$).
A linear functional $\Phi$ on $U_b(X)$ is called a {\it uniform measure} if $\Phi$ is continuous on each 
UEB set $S\subset U_b(X)$ in the pointwise topology on $S$.
The {\it UEB topology (uniformity)} on the set of all linear functionals on $U_b(X)$ is the topology 
(uniformity) of uniform convergence on UEB sets.

\begin{proposition} $M_u(X)$ is the vector space of all uniform measures, endowed with the UEB topology 
(or uniformity).
\end{proposition}

\begin{proof}
It follows from Lemma \ref{UEB}(a) that a set $S\subset U_b(X)$ is UEB if and only if is lies in  
some ball of some $\Lip_{bl}(X_d)$.

Consequently the topology (uniformity) induced on $U_b^*(X)$ by the family of seminorms $||\cdot||_{KR(d)}$,
$d\in D$, coincides (see the proof of Theorem \ref{lip-m}(a)) with the UEB topology (uniformity).

Also, a linear functional $\Phi$ on $U_b(X)$ is a uniform measure if and only if for each $d\in D$ 
the restriction of $\Phi_d$ to the unit ball $L_{X_d}$ of $\Lip_{bl}(X_d)$ is continuous in 
the pointwise topology on $L_{X_d}$ (in other words, $\Phi_d$ is a 1-Lipschitz measure).

Finally, each uniform measure $\Phi$ on $X$ actually belongs to $U_b^*(X)$.
Indeed, given a sequence $f_n\in U_b(X)$ that uniformly converges to an $f_0\in U_b(X)$, let 
$d(x,y)=\sup_i\big(f_i(x)-f_i(y)\big)$.
Then $d\in D$ and each $f_i$ factors as $X\xr{q_d} X_d\xr{f_i'}\R$, where $f_i'\in L_{X_d}$.
Hence $\Phi_d(f_i')\to\Phi_d(f_0')$, that is, $\Phi(f_i)\to\Phi(f_0)$.
\end{proof}

\begin{remark} Let us note that a basis of uniform covers for the UEB uniformity on $U_b^*(X)$ is given by 
$C_{V,\eps}=\{B_{V,\eps}(\Phi)\mid\Phi\in U_b^*(X)\}$ for all $\eps>0$ and all UEB sets $V$, where
$B_{V,\eps}(\Phi)=\{\Psi\in U_b^*(X)\mid \forall f\in V\ |\Phi(f)-\Psi(f)|<\eps\}$.
\end{remark}

By Theorem \ref{representability2}(c) $M_\rho(X)\subset M_u(X)$ (as sets).
Also, since each $\alpha_d$ sends $M_\tau^+(X)$ into $M_\tau^+(X,d)$, which lies in $M_1(X,d)$, 
we get that $M_\tau^+(X)\subset M_u(X)$ (as sets).

\begin{theorem} \label{unif-measures} Let $X$ be a uniform space.

(a) \cite{Pac2}*{6.6 and subsequent remark} $M_u(X)$ is complete and $M_\delta(X)$ is dense in $M_u(X)$. 

(b) \cite{Pac2}*{6.10(2)} If $X$ is a complete metric space, $M_u(X)=M_\rho(X)$ (as sets). 

(c) \cite{Pac2}*{6.13, 6.15} The UEB topology coincides with the $U_b(X)$-weak topology on $M_u^+(X)$
and on each $\partial M_u^r(X)$.

(d) \cite{Pac2}*{6.18} The UEB topology and the $U_b(X)$-weak topology have the same convergent sequences
in $M_u(X)$.
Consequently, the UEB uniformity and the $U_b(X)$-weak uniformity have the same merging pairs of sequences
in $M_u(X)$.

(e) \cite{Pac2}*{6.5} $M_u^*(X)=U_b(X)$.
\end{theorem}

\begin{theorem} \cite{Pac2}*{10.24, 6.40} Let $X$ be a uniform space, $E$ a complete locally convex 
vector space and $f\:X\to E$ a uniformly continuous map such that $f(X)$ is bounded (that is, for each 
neighborhood $U$ of $0$ there exists an $r>0$ such that $f(X)\subset rU$).
Then $f$ extends to a continuous linear map $M_u(X)\to E$.
\end{theorem}

\subsection{Comparison}

\begin{example} \label{unif-comparison}
Let us compare various uniformities on Dirac measures.
Let $X$ be a metric space and $E=E_X\:X\to PM_\delta(X)$ be the standard embedding, $x\mapsto\delta_x$.

(a) By Proposition \ref{embKR} $E_X$ is a uniform embedding with respect to the KR-uniformity,
which is given by the norm $||\cdot||_{KR}$.
It follows that $E_X$ is a uniform embedding also with respect to the uniformity given by the 
family of seminorms $||\cdot||_{KR(d)}$, $d\in D$, that is, the UEB uniformity.

By Lemmas \ref{closed emb} and \ref{closed subspaces}(a) $E(X)$ is closed in $M_\tau^+(X)$ in 
the $C_b(X)$-weak topology, and hence by Theorems \ref{top=unif}(b), \ref{unif-measures}(c) 
and \ref{KR-metric}(b) also in the KR-topology and in the UEB topology.

On the other hand, by Theorems \ref{lip-m}(c) and \ref{unif-measures}(a) $M_1(X)$ is complete in 
the KR-uniformity and $M_u(X)$ is complete in the UEB uniformity.
So the closures of $E(X)$ in $M_1(X)$ (in the KR-uniformity) and in $M_u(X)$ (in the UEB uniformity) 
are uniform copies of the completion of $X$.

In contrast to the UEB uniformity, the KR-uniformity on $M_\delta^+(X)$ depends on the choice of
a metric (see Example \ref{unif-equiv-metr}); however, by Theorem \ref{unif-measures}(d)
and Corollary \ref{merging}(b) on $PM_\rho(X)$ they have the same merging pairs of sequences.

(b) By using either Theorem \ref{unif-measures}(d) or Corollary \ref{merging}(b), we get that
$E_X$ is a sequentially uniform embedding with respect to the $U_b(X)$-weak uniformity
(that is, both $E_X$ and $E_X^{-1}$ take merging pairs of sequences into merging pairs of sequences).
Similarly, Corollary \ref{merging}(a) yields that $E_X$ is a sequentially uniform embedding with 
respect to the $\Lip_b(X)$-weak uniformity.

In fact, it is not hard to check directly that $E_X$ is uniformly continuous with respect to 
the $U_b(X)$-weak and $\Lip_b(X)$-weak uniformities.
However, $E_X^{-1}$ is not uniformly continuous unless $X$ is precompact.
(Indeed, by Theorem \ref{weak*}(b) $M^1(X)$ is compact and hence $E(X)$ is precompact in either
of the $U_b(X)$-weak and $\Lip_b(X)$-weak uniformities.)

For example, if $X$ is the countable uniformly discrete space $\N$, then $E(\N)$ with the $U_b(\N)$-weak 
uniformity looks as follows: all Cauchy sequences are eventually constant, but there exist Cauchy nets
that are not eventually constant.
In fact, we recover Example \ref{Stone-Cech}(b), since by Example \ref{Stone-Cech}(c) the $U_b(\N)$-weak 
closure of $E(\N)$ in $U_b^*(\N)$ is the Stone--\v Cech compactification of $\N$.

By Theorem \ref{unif-measures}(c) the closures of $E(X)$ in $M_u(X)$ in the $U_b(X)$-weak topology
and in the UEB topology coincide as sets and have the same topology.
Hence the closure $\gamma X$ of $E(X)$ in $U_b^*(X)$ with the $U_b(X)$-weak topology, which is also known 
as the Samuel compactification of $X$ (see Example \ref{Stone-Cech}(c)), contains a topological copy
of the completion of $X$.

(c) The $C_b(X)$-weak uniformity is determined solely by the topology of $X$.
Moreover, by Corollary \ref{weak*complete} $PM_\tau(X)$ is sequentially complete in the $C_b(X)$-weak 
uniformity, regardless of whether $X$ is complete.
By contrast, if $X$ is an incomplete metric space, then $PM_\tau(X)$ is not complete (and hence, being
metrizable, is not sequentially complete) in the KR- and UEB uniformities, because $E(X)$ is closed 
in $PM_\tau(X)$ in the KR- and UEB topologies.

Let us also note that while the KR-uniformity on each $M_\sigma^r(X)$ depends only on the uniform 
structure of the metric space $X$ (Theorem \ref{unif-well-def}), the KR-topology on $M_\sigma^1(X)$ 
and even on $M_\delta^1(X)$ depends not only on the topology of $X$.
Indeed, if $h\:X\to Y$ is a homeomorphism of metric spaces that is not uniformly continuous, then
$h_*\:M_\sigma^1(X)\to M_\sigma^1(Y)$ is also not uniformly continuous, because $X$ and $Y$ are
uniformly embedded onto the sets of Dirac measures.
Hence some merging pair of sequences $\delta_{x_i}$ and $\delta_{y_i}$ is sent by $h_*$ to 
a non-merging pair, and consequently the sequence $\delta_{x_i}-\delta_{y_i}$, which converges to $0$ in 
$M_\delta^1(X)$, is sent to a non-convergent sequence in $M^1_\delta(Y)$.
\end{example}

\subsection{Free uniform measures}

Let $X$ be uniform space and let $D$ be the set of all uniform pseudo-metrics on $X$.
Let $U(X)$ be the (algebraic) vector space of all (possibly unbounded) uniformly continuous maps $X\to\R$.
A family of functions $S\subset U(X)$ is called an {\it UE set} if it is uniformly equicontinuous 
and pointwise bounded.
By Lemma \ref{UEB}(b) a set $S\subset U(X)$ is UE if and only if there exist a $d\in D$ and 
a $1$-Lipschitz function $h\:X_d\to[0,\infty)$ such that $S$ lies in the set $L_{X_d,h}$ of all 
$h$-bounded $1$-Lipschitz functions $X_d\to\R$.

Given a linear functional $\Phi\:U(X)\to\R$, its restriction $\Phi_d$ over $\Lip_h(X_d)$
is defined, as before, by $\Phi_d(f)=\Phi(fq_d)$, where $q_d\:X\to X_d$ is the quotient map.
Hence for each $d\in D$ we obtain a seminorm, denoted $||\cdot||_{h,d}$, on 
the algebraic dual $U^*(X)$.
By the above, a set $S\subset U(X)$ lies in some ball of $\Lip_h(X_d)$ for some $d\in D$ and 
some $1$-Lipschitz function $h\:X_d\to[0,\infty)$ if and only if $S$ is UE.
Therefore the topology (uniformity) induced on $U^*(X)$ by the family of 
seminorms $||\cdot||_{h,d}$, where $d\in D$ and $h\:X_d\to[0,\infty)$ is a $1$-Lipschitz function,
coincides (see the proof of Theorem \ref{lip-m}(a)) with the topology (uniformity) of uniform 
convergence on UE sets, also called the {\it UE topology (uniformity)}.
A basis of uniform covers for the UE uniformity on $U^*(X)$ is given by 
$C_{V,\eps}=\{B_{V,\eps}(\Phi)\mid\Phi\in U^*(X)\}$ for all $\eps>0$ and all UE sets $V$, where
$B_{V,\eps}(\Phi)=\{\Psi\in U^*(X)\mid \forall f\in V\ |\Phi(f)-\Psi(f)|<\eps\}$.

A linear functional $\Phi$ on $U(X)$ is called a {\it free uniform measure} if it is continuous on each 
UE set $S\subset U(X)$ in the pointwise topology on $S$.
By the above, this is equivalent to saying that for each $d\in D$ and each $1$-Lipschitz function 
$h\:X_d\to[0,\infty)$ the restriction of $\Phi_d$ to the unit ball $L_{X_d,h}$ of $\Lip_h(X_d)$ is continuous in 
the pointwise topology on $L_{X_d,h}$ (i.e., $\Phi_d$ is an $h$-Lipschitz measure).
Since every UEB subset of $U_b(X)$ is an UE subset of $U(X)$, the restriction of every free uniform measure 
over $U_b(X)$ is a uniform measure.

Let $M_{fu}(X)$ be the vector space of all free uniform measures on $X$, endowed with the family of 
seminorms $||\cdot||_{h,d}$, where $d\in D$ and $h\:X_d\to[0,\infty)$ is a $1$-Lipschitz function.
It is easy to see that the restriction map $M_{fu}(X)\to M_u(X)$ is injective.
Indeed, suppose that $\Phi\in M_{fu}(X)$ and $\Phi(f)=0$ for all $f\in U_b(X)$.
Given an $f\in U(X)$, let us define a uniformly continuous $f|^{[-n,n]}\:X\to [-n,n]$ by 
$f|^{[-n,n]}(x)=\max\Big(\min\big(f(x),n\big),-n\Big)$.
Then $f|^{[-n,n]}$ pointwise converge to $f$, and the set $\{f,f^{[-1,1]},f^{[-2,2]},\dots\}$ is UE.
Hence $\Phi(f)=\lim_{n\to\infty}\Phi(f|^{[-n,n]})=0$.

\begin{theorem} \label{free-unif}
Let $X$ be a uniform space.

(a) \cite{Pac2}*{10.5} A uniform measure $\Phi$ is the restriction of a free uniform measure if and only if
$\lim_{n\to\infty} \Phi(f|^{[-n,n]})$ exists and is finite for each $f\in U(X)$.

(b) \cite{Pac2}*{10.11} In the UE uniformity, the embedding $X\to M_\delta(X)$ is a uniform embedding,
$M_\delta(X)$ is dense in $M_{fu}(X)$, and $M_{fu}(X)$ is complete.

(c) \cite{Pac2}*{10.16} The UE topology coincides with the $U(X)$-weak topology on $M_{fu}^+(X)$.

(d) The UE topology and the $U(X)$-weak topology have the same convergent sequences in $M_{fu}(X)$.
Consequently, the UE uniformity and the $U(X)$-weak uniformity have the same merging pairs of sequences 
in $M_{fu}(X)$.

(e) \cite{Pac2}*{10.10} $M_{fu}^*(X)=U(X)$.
\end{theorem}

\begin{proof}[Proof of (d)] This follows from \cite{Pac2}*{10.18(v)$\Rightarrow$(iii)} and Lemma \ref{cont-bij}.
\end{proof}

\begin{theorem} \cite{Rai}, \cite{Pac2}*{10.23} If $E$ is a complete locally convex vector space and $f\:X\to E$ 
is a uniformly continuous map, then $f$ extends to a continuous (in the UE topology) linear map $X\to E$.
\end{theorem}

By endowing a given topological space with its fine uniformity, we obtain a similar result
in the topological category, which is due originally to A. A. Markov, Jr.\ (see \cite{Rai}).

\section{Theory of retracts}

\begin{proposition} (a) $PM_\tau(X)$, $PM_\rho(X)$ and $PM_\delta(X)$ are absolute retracts.

(b) Moreover, when endowed with the uniformity of the Kantorovich--Rubinstein norm, they are 
uniformly contractible and uniformly locally contractible.
\end{proposition}

\begin{proof} $PM_*(X)$ are convex subsets of the locally convex vector space $M(X)$ 
(and also of the normed vector space obtained by re-endowing $M(X)$ with the $||\cdot||_{KR}$ norm).
Hence they are absolute extensors by Dugundji's extension theorem (see \cite{Bo}*{III.7.1}, 
\cite{Hu1}*{II.14.1}).
Also they are uniformly contractible and uniformly locally contractible using the usual linear homotopy
(compare the proof of Theorem \ref{lec}).
\end{proof}

\begin{proposition} \label{not-uar}
$PM_\tau(B)$, $PM_\rho(B)$ and $PM_\delta(B)$ with the uniformity of the
Kanto\-rovich--Rubinstein norm are not uniform ARs for a certain separable metric space $B$.
\end{proposition}

This follows easily from Theorem \ref{1-lip-r}(c) below, but we note a slightly different proof.

\begin{proof} Let $V$ be a separable Banach space whose unit ball $B$ is not an AR 
(see Remark \ref{Banach}(c,d)).
Let $f\:B^+\to V$ be defined by the identity on $B$ and by $f(*)=0$.
Then $f$ is $1$-Lipschitz.
By Theorem \ref{Lipschitz extension} $f$ extends to a continuous linear map $\bar f\:AE(B^+)\to V$
with $||\bar f||=1$.
In particular, $\bar f$ is $1$-Lipschitz, and hence uniformly continuous.
On the other hand, Propositions \ref{embKR} and \ref{AE-KR} yield uniform embeddings
$B\subset PM_\delta(B)\subset PM_\rho(B)\subset PM_\tau(B)\subset M_\tau(B)\subset AE(B^+)$
with respect to the Kantorovich--Rubinstein norm.
Since $\bar f$ is linear and $B$ is convex, $\bar f$ sends $PM_\tau(B)$ into $B$.
Given any $C\supset B$, and assuming that $PM_\tau(X)$, $PM_\rho(X)$ or $PM_\delta(X)$ is a uniform AR, 
we also get a uniformly continuous extension $g\:C\to PM_\tau(B)$ of the inclusion $B\subset PM_\delta(B)$.
Thus $\bar fg$ is a uniform retraction of $C$ onto $B$, which is a contradiction. 
\end{proof}

\begin{theorem} \label{1-lip-r}
Let $B$ a Banach space and $C$ be a convex subset of $B$ such that $0\in C$.

(a) \cite{DTZ}*{Lemma 1} $B$ is a $1$-Lipschitz retract of $AE(B)$.

(b) $C$ is a $1$-Lipschitz retract of $PM_\delta(C)$ with the Kantorovich metric.

(c) If $C$ is closed and bounded, then it is also a $1$-Lipschitz retract of $PM_\tau(C)$ with 
the Kantorovich metric.
The latter is complete and coincides with $PM_\rho(C)$.
\end{theorem}

If $C$ is closed (but not necessarily bounded), then by (b) it is also a $1$-Lipschitz retract of 
the completion of $PM_\delta(C)$ in the Kantorovich metric (see Corollary \ref{PM-completion} 
concerning this completion).

\begin{proof} By Theorem \ref{Lipschitz extension} the inclusion $C\subset B$ extends to a 1-Lipschitz 
linear map $f\:AE(C)\to B$.
This completes the proof of (a).

The convex hull of $C$ in $AE(C)$ is the isometric copy $PM_\delta(C)/\left<\delta_0\right>$ of $PM_\delta(C)$ 
with the Kantorovich metric.
Since $C$ is convex, $f$ sends $PM_\delta(C)$ (or rather the said isometric copy) into $C$.
This completes the proof of (b).

Now suppose that $C$ is closed and bounded.
By Theorem \ref{KR-metric}(d) $PM_\delta(C)$ is dense in $PM_\tau(C)$.
Since $C$ is bounded, $PM_\tau(C)$ with the Kantorovich metric is a subset of $AE(C)$.
Since $C$ is closed, $f$ sends $PM_\tau(C)$ into $C$.
Also $C$ is complete.
Hence by Lemma \ref{regularity}(b) $PM_\tau(C)=PM_\rho(C)$ and by Theorem \ref{KR-complete} $PM_\tau(C)$ 
is complete in the KR-uniformity.
Since $C$ is bounded, the latter coincides with the uniformity of the Kantorovich metric.
\end{proof}

\begin{proposition} \label{PM-homotopy-complete} 
If $X$ contains at least two distinct points $p$, $q$, then $PM_\tau(X)\but X$ is homotopy complete in 
$PM_\tau(X)$.
\end{proposition}

\begin{proof} Let $\lambda=\frac12(\delta_p+\delta_q)$.
Then $H_t\:PM_\tau(X)\to X$ defined by $H_t(\mu)=(1-t)\mu+t\lambda$ is a uniform homotopy
and $H_t\big(PM_\tau(X)\big)$ is disjoint from $X$ for $t>0$.
\end{proof}

\end{document}